\documentclass[11pt, notitlepage]{article}
\usepackage[utf8]{inputenc}
\usepackage[T1]{fontenc}
\usepackage[english]{babel}
\usepackage{graphicx}
\usepackage{cancel}
\usepackage{stmaryrd}
\usepackage{amsmath,amsfonts,amssymb,yhmath}
\usepackage{dsfont}
\usepackage{titlesec}
\usepackage{tikz}
\usepackage{pgfplots}
\usetikzlibrary{patterns}
\usetikzlibrary{positioning,arrows,arrows.meta}
\titleformat{\chapter}[display]{\bfseries}{\huge\chaptertitlename~\thechapter}{14pt}{\LARGE}

\usepackage{hyperref}
\usepackage{amsthm}
  


\newtheorem{Def}{Definition}[section]
\newtheorem{Lem}[Def]{Lemma}
\newtheorem{Th}[Def]{Theorem}
\newtheorem{Pro}[Def]{Proposition}
\newtheorem{Cor}[Def]{Corollary}
\newtheorem{Rq}[Def]{Remark}
\newcommand{\R}{\mathbb{R}}
\newcommand{\K}{\widehat{\mathbb{P}}_0}

\title{Asymptotic properties of small data solutions of the Vlasov-Maxwell system in high dimensions}

\author{Léo Bigorgne\footnote{Laboratoire de Mathématiques, Univ. Paris-Sud, CNRS, Université Paris-Saclay, 91405 Orsay.}}
\date{}

\begin{document}
    \maketitle
    
\begin{abstract}

We prove almost sharp decay estimates for the small data solutions and their derivatives of the Vlasov-Maxwell system in dimension $n \geq 4$. The smallness assumption concerns only certain weighted $L^1$ or $L^2$ norms of the initial data. In particular, no compact support assumption is required on the Vlasov or the Maxwell fields. The main ingredients of the proof are vector field methods for both the kinetic and the wave equations, null properties of the Vlasov-Maxwell system to control high velocities and a new decay estimate for the velocity average of the solution of the relativistic massive transport equation.

We also consider the massless Vlasov-Maxwell system under a lower bound on the velocity support of the Vlasov field. As we prove in this paper, the velocity support of the Vlasov field needs to be initially bounded away from $0$. We compensate the weaker decay estimate on the velocity average of the massless Vlasov field near the light cone by an extra null decomposition of the velocity vector.
\end{abstract}

    \tableofcontents

\section{Introduction}

In this paper, we study the asymptotic properties of the small data solutions of the Vlasov-Maxwell system in dimensions $n \geq 4$. For $K$ species, this system is given by\footnote{During this article, we will use the Einstein summation convention. For instance, $e^kJ(f_k)_{\nu}=\sum_{k=1}^K e^kJ(f_k)_{\nu}$. Roman indices goes from $1$ to $n$ and greek indices from $0$ to $n$. Moreover, we raise and lower indices with respect to the Minkowski metric.}
\begin{eqnarray}\label{syst1}
   \sqrt{m_k^2+|v|^2} \partial_t f_k+v^i \partial_i f_k + e_k v^{\mu}{F_{\mu}}^{ j} \partial_{v^j}f_k & = & 0,  \\ \label{syst2}
    \nabla^{\mu} F_{\mu \nu} & = &  e^kJ(f_k)_{\nu}, \\ \label{syst3}
    \nabla^{\mu} {}^* \! F_{\mu \alpha_1 ... \alpha_{n-2}} & = & 0 ,
\end{eqnarray} with initial data,
\begin{eqnarray}    
    f_k(0,.,.) & = & f_{0k}, \\ 
    F(0,.) & = & F_0. 
\end{eqnarray} 

This is a classical model in plasma physics and we refer to \cite{Glassey} for an introduction to its analysis. Here,
\begin{itemize}
\item $m_k \in \R_+$ and $e_k \in \R^*$ are the mass and the charge of the particles of the species $k \in \llbracket 1, K \rrbracket$. The function $f_k(t,x,v)$ is their velocity distribution, which is a non-negative function.
\item The Maxwell field is described in geometric form by the 2-form $F(t,x)$ and its Hodge dual ${}^* \! F(t,x)$.
\item The $(n+1)$-current $J(f_k)_{\nu}$ in equation $(2)$ is given by
$$ J(f_k)^{\nu}(t,x)= \int_{v \in \R^n} \frac{v^{\nu}}{v^0_k} f_k(t,x,v) dv, \hspace{5mm} \text{where} \hspace{5mm} v^0_k:=\sqrt{m_k^2+|v|^2}.$$
\item The variable $t$ will be taken in $\mathbb{R}_+$, $x$ will be taken in $\R^n$ and for the species $k$, $v$ will be taken either in $\R^n$, if $m_k \neq 0$, or in $\R^n \setminus \{0 \}$, if $m_k=0$.
\end{itemize}  
In the $3$ dimensional case, we can express the system in terms of the electric and the magnetic vector fields through the relations
$$E^i=F_{0i} \hspace{3mm} \text{and} \hspace{3mm} B^i=-{}^* \! F_{0i}$$
so that the Vlasov-Maxwell equations take the familiar form

\begin{flalign*}
 & \hspace{2cm} \sqrt{m_k^2+|v|^2} \partial_t f_k+v^i \partial_i f_k + e_k(E+v \times B) \cdot \nabla_v f_k = 0, & \\
& \hspace{2cm} \nabla \cdot E = e^k J(f_k)^0, \hspace{1cm} \partial_t E^j = (\nabla \times B)^j -e^kJ(f_k)^j, & \\
& \hspace{2cm} \nabla \cdot B = 0, \hspace{2,4cm} \partial_t B = - \nabla \times E. &
\end{flalign*} 

\subsection{Global in time solutions for the Vlasov-Maxwell system}

The global existence for classical solutions to the Vlasov-Maxwell system is still an open problem. They are known to be global in certain particular cases such as under a translation symmetry hypothesis on the initial data in one of the space variables. The pioneer works on this two and one half dimensional case originated from Glassey-Schaeffer in \cite{GlSch2.5} and required a compact support assumption in $v$. The result obtained recently by Luk-Strain allows data with non-compact velocity support \cite{LukStrain}. The solutions to the Vlasov-Maxwell system also appear to be global when they arise from pertubation of spherically symmetric initial data\footnote{Recall that for spherically symmetric solutions, the Vlasov-Maxwell system reduces to the relativistic Vlasov-Poisson system.} (see \cite{Rein2}). 

For the general case, several continuation criteria are known. The first one, obtained by Glassey and Strauss in \cite{GlStrauss} (see also \cite{BGP2} and \cite{KlSta} for alternative proofs), expresses that $C^1$ solutions to the Vlasov-Maxwell system arising from compactly initial data do not develop singularities as long as the velocity supports of the particle densities $f_k$ remain bounded. An improved continuation criteria requires the finiteness of
\begin{equation}\label{crit}
\left\| \sqrt{1+|v|^2}^{\theta} f_k \right\|_{L^{\infty}\left( [0, T^*[, L^q_x L^1_v \right)},
\end{equation}
for a certain $q$ and $\theta$, in order to extend the solution beyond $T^* >0$. Let us mention \cite{Pallard1} for the cases $6 \leq q \leq \infty$ and $\theta > \frac{4}{q}$, \cite{SAI} for $q = + \infty$ and $\theta = 0$ as well as \cite{Pallard2} for $q=6$ and $\theta=0$. Earlier results of Glassey-Strauss required the boundedness of \eqref{crit} for $q= \infty$ and $\theta =1$ and cover data with non-compact support in $v$ (see \cite{GlStracrit}). Recently, Luk-Strain removed in \cite{LukStrain} all compact support assumptions and extend the continuation criteria \eqref{crit} for $6 \leq q \leq + \infty$ and $\theta > \frac{2}{q}$.

\subsection{Previous work on small data solutions of the Vlasov-Maxwell system}

Global existence for small data in dimension $3$ was first established by Glassey-Strauss in \cite{GSt} under a compact support assumption (in $x$ and in $v$). In particular they proved $\int_v f dv \lesssim \frac{1}{(1+t)^3}$, coinciding with the linear decay, but they did not control $\partial_{\mu_1}...\partial_{\mu_p}\int_v f dv $. They also proved decay estimates for the electromagnetic field and its derivatives of first order, but not for the derivatives of higher order. A similar result was proved in \cite{GSc} for the nearly neutral case, i.e.\@ $\sum_k e_k m_k^3f_{0k}(x,m_kv)$ has to be small and not the individual particle densities. The result established by Schaeffer in \cite{Sc} allowed particles with high velocity but still requires the data to be compactly supported in space\footnote{Note also that when the Vlasov field is not compactly supported (in $v$), the decay estimate obtained in \cite{Sc} for its velocity average contains a loss.}. Finally, let us also mention the earlier result of Bardos-Degond for the more classical Vlasov-Poisson system \cite{Bardos}. Under a smallness assumption, they established that the solution of the system is global in time and proved that $\int_v f dv \lesssim \frac{1}{(1+t)^3}$ but they did not obtain informations on the derivatives of $f$. They also proved decay estimates for the electric field up to second order.

\subsection{Optimal gradient estimates for Vlasov systems}

Due to the linearity of the Maxwell equations and the elliptic nature of the Poisson equation or a nonresonant phenomenon\footnote{According to \cite{BGP}, the velocity averages of the solutions of a system coupling a linear wave equation with a transport equation are such that the velocity averages are more regular than expected if the speed of propagation of the wave equation is strictly larger than the speed of the particles governed by the transport equation.}, the previous results are established without essentially commuting the Vlasov equation and controlling higher derivatives of the solutions. For the Vlasov-Poisson system with small data, the sharp time decay estimates $\left|\int_v \partial_{\mu_1}...\partial_{\mu_p} f dv\right| \lesssim \frac{1}{(1+t)^{3+p}}$ was proved\footnote{A similar result is established in \cite{Yukawa}, using the same techniques, for the Vlasov-Yukawa system in dimension $2$.} in \cite{HRV}. A similar result was obtained in \cite{Poisson} using a vector field method which led to global bounds for the solution and optimal space and time decay rates for the velocity averages. In the same spirit, optimal decay estimates was proved for the derivatives of the solutions of the Vlasov-Nordström system in \cite{FJS} and \cite{FJS2}. The stability of the Minkowski space for the Einstein-Vlasov system was recently, and independently, proved in \cite{FJS3} and \cite{Lindblad}. For both of them, vector field methods was a crucial point in the proof and led in particular to almost optimal decay rates for the derivatives of the solutions.

The goal of this paper is to prove almost optimal decay for the small data solutions and their derivatives of the Vlasov-Maxwell system in dimension $n \geq 4$ without any compact support assumption on the initial data.
\subsection{The vector field method for Vlasov fields}

In this paper, we will use vector field methods to derive decay estimates for both the electromagnetic field and the Vlasov field. The vector field method was originally developped by Klainerman in \cite{Kl85} to study wave equations and was adapted to cover the Maxwell equations (and the spin $2$ equations) in $3d$ in \cite{CK}. More recently, the method was extended for the free relativistic transport equation in \cite{FJS}.

As in \cite{Kl85}, these methods are based on energy estimates, commutation vector fields and weighted Sobolev inequalities. For the transport operator $v^{\mu}\partial_{\mu}$, the set $\K$ of commutation vector fields used in \cite{FJS} is composed of the scaling vector field $S=x^{\mu} \partial_{\mu}$ and the complete lifts of the generators of the Poincaré group, that is to say the transalations
$$\partial_{\mu}, \hspace{2mm} 0 \leq \mu \leq n,$$
the complete lifts of the rotations
$$ \widehat{\Omega}_{ij}=x^i\partial_{j}-x^j \partial_i+v^i \partial_{v^j}-v^j \partial_{v^i}, \hspace{2mm} 1 \leq i < j \leq n$$
and the complete lifts of the Lorentz boosts
$$\widehat{\Omega}_{0k}=t\partial_{k}+x^k \partial_t+v^0 \partial_{v^k}, \hspace{2mm} 1 \leq k \leq n.$$
In \cite{FJS}, vector field methods are applied to derive the behavior of the solutions to the Vlasov-Nordström system in the future the hyperboloid\footnote{The use of a hyperboloidal foliation in order to establish decay estimates was introduced in \cite{Kl93} in the context of the Klein-Gordon equation.} $t^2-r^2=1$. However, without any compact support assumption, one cannot reduce the study of a solution in the future of a $t=constant$ slice to its study in the future of a hyperboloidal slice (see for instance \cite{FJS}, Appendix $A$, for more details). In order to remove all compact support assumption on the data, one of the goal of this paper is to start from a $t=0$ slice and adapt the vector field method for transport equations to the more common foliation $\left( \{t \} \times \R^n \right)_{t \geq 0}$. Note that \cite{FJS2} (respectively \cite{FJS3}) use slight modifications of the commutation vector fields\footnote{The modified vector fields are built in order to compensate the worst source terms in the commuted transport equations.} of the operator $v^{\mu} \partial_{\mu}$ in order to study the small data solutions of the Vlasov-Nordström (respectively Einstein-Vlasov) system in $3d$. They also use a hyperboloidal foliation and restrict the study of the solutions to the future of a hyperboloid.

\subsection{The Lorenz gauge}

Recall that a $1$-form $A$ is a potential of the electromagnetic field $F$ if $F=dA$ (or $F_{\mu \nu} = \partial_{\mu} A_{\nu}-\partial_{\nu} A_{\mu}$). If moreover
$$\partial^{\mu} A_{\mu} =0,$$
we say that $A$ satisfies the Lorenz gauge condition. As the energy momentum tensor of $F$ is not traceless in dimension $n \geq 4$ and the Morawetz vector field $K_0:=(t^2+r^2)\partial_t+2tr\partial_r$ is merely a \textit{conformal} Killing vector field, we encounter the same difficulty in using it as for the wave equation in $3d$ (see \cite{Sogge}, Chapter $II$ for more details). To circumvent this difficulty, we will consider in this paper the Vlasov-Maxwell system in the Lorenz gauge. $A_{\mu}$ will then satisfy the equation
\begin{equation}\label{eqLorenz}
 \square A_{\mu} = e^k\int_v \frac{v_{\mu}}{v^0} f_k dv.
\end{equation}
We also make fundamental use of the Lorenz gauge to establish the optimal decay rate on the component $\alpha$ of the null decomposition of the electromagnetic field, as the method used in \cite{CK} cannot be reproduced in dimension $n \neq 3$.

\subsection{Results for the massive Vlasov-Maxwell system}

We will consider weighted $L^2$ norms to control $A$ and $F$ such as\footnote{For a tensor $G$ and a multi-index $\beta=(\beta_1,...,\beta_p)$,
$$\mathcal{L}_{Z^{\beta}}G= \mathcal{L}_{Z^{\beta_1}}...\mathcal{L}_{Z^{\beta_p}}G,$$
while $$ Z\mathcal{L}_{Z^{\beta}}(G)_{\mu}=Z[\mathcal{L}_{Z^{\beta}}(G)_{\mu}].$$}
$$\widetilde{\mathcal{E}}_N[A] = \sum_{\mu=0}^n \sum_{|\beta| \leq 1} \sum_{|\gamma| \leq N} \|Z^{\beta} \mathcal{L}_{Z^{\gamma}}(A)_{\mu}\|^2_{L^2(\R^n)} $$ as well as weighted $L^1$ norms for the Vlasov field, such as
$$ \mathbb{E}^q_{N}[g](t)=\sum_{|\beta| \leq N} \int_{\R_x^n} \int_{\R^n_v } (v^0)^q |\widehat{Z}^{\beta} g| dv dx+\int_{C_u(t)} \hspace{-0.5mm} \int_{\R^n_v } (v^0)^{q-1}v^{\underline{L}} |\widehat{Z}^{\beta} g | dv d C_u(t), $$ 
where $C_u(t):= \{(s,y)  \in \mathbb{R}_+ \times \mathbb{R}^n \hspace{1mm} / \hspace{1mm} s \leq t, \hspace{0.5mm} s-|y|=u \}$. For the Vlasov field, we also use extra norms with the additional weights $v^{\mu}$, $x^iv^j-x^jv^i$ or $tv^i-x^iv^0$. See Definitions \ref{norm1}, \ref{norm2} and Section \ref{poids} for an introduction to the other norms and the weights.

We are now ready to present our main result for the massive Vlasov-Maxwell system (for a detailled version, see Theorem \ref{massif}).

\begin{Th}\label{intromasstheo}

Let $n \geq 4$, $K \geq 2$, and $N \geq  \frac{5}{2}n+1 $. Let $(f_0,F_0)$ be an initial data set for the massive Vlasov-Maxwell system with $K$ species. Let $(f,F)$ be the unique classical solution to the system and let $A$ be a potential in the Lorenz gauge. There exists $\epsilon >0$ such that\footnote{A smallness condition on $F$, which implies $\widetilde{\mathcal{E}}_N[A](0) \leq \epsilon$, is given in Proposition \ref{LorenzPot}.}, if 

$$ \widetilde{\mathcal{E}}_N[A](0) + \mathcal{E}_N[F](0) +\sum_{k=1}^K \mathbb{E}^2_{N+n,1}[f_k](0) \leq \epsilon,$$
then $(f,F)$ exists globally in time and verifies the following estimates. 
\begin{itemize}
\item Energy bounds for $A$, $F$ and $f_k$ such as  $\mathbb{E}^2_{N}[f_k] \lesssim \epsilon$ on $\R_+$.
\item Sharp pointwise decay estimates on the null decomposition of $\mathcal{L}_{Z^{\beta}}(F)$ and for the velocity average of $f_k$ and its derivatives. For instance, $$ \forall \hspace{0.5mm} |\beta| \leq N-\frac{3n+2}{2}, \hspace{1mm} (t,x) \in \mathbb{R}_+ \times \mathbb{R}^n, \hspace{2mm} \int_{v \in \mathbb{R}^n  } | \widehat{Z}^{\beta} f_k | dv \lesssim \frac{\epsilon}{(1+t+|x|)^{n}}.$$
\end{itemize}
\end{Th}

\begin{Rq}\label{introneutral}
Since we suppose that the initial energy on $F$ is finite, we are necessarily in the neutral case\footnote{In other words, the total charge verifies
$$e^k \int_{x \in \R^n} \int_{v \in \R^n} f_{0k} dv dx =0.$$} when the dimension is $n = 4$.
On the other hand, when the total charge is non zero, Gauss's law implies that the energy $\mathcal{E}_N[F]$ is infinite. We refer to \cite{Bieri} and \cite{Yang} for a study of the Maxwell-Klein-Gordon system with a non-zero total charge.

\end{Rq}

\begin{Rq}
Thanks to the vector field method and in view of the definition of our norms, we do not need any compact support assumption on the initial data. We also automatically obtain improved decay rates for the derivatives of both the electromagnetic field and the velocity averages of the particle densities. For instance, for all $|\beta| \leq N- \frac{3n+2}{2}$ and $(t,x) \in \mathbb{R}_+ \times \mathbb{R}^n$,
$$  \left| \partial^{\beta}_{t,x}\int_{v \in \mathbb{R}^n  }  f_k  dv \right| \lesssim \frac{\epsilon}{(1+t+|x|)^{n}(1+|t-|x||)^{|\beta|}}$$ 
and (see Section \ref{optidecayderiv}), assuming more decay on the initial data,
$$  \left| \partial^{\beta}_{t,x}\int_{v \in \mathbb{R}^n  }  f_k  dv \right| \lesssim \frac{\epsilon}{(1+t+|x|)^{n+|\beta|}}.$$

\end{Rq}

\begin{Rq}
Notice that in dimension $n=4$, it is sufficient for $\widehat{Z}^{\beta}f_k$ to initially decay faster than $(1+|v|)^{-6-\delta}(1+|x|)^{-5-\delta}$, with $\delta >0$, for our theorem to apply. In \cite{Sc}, which concerns the $3d$ case, the main result requires the initial particle densities to be compactly supported in $x$ and to decay faster than $(1+|v|)^{-q}$, with $q > 60 + 12\sqrt{17}$.
\end{Rq}

\subsection{Results for the massless Vlasov-Maxwell system}

We now present an elusive version of our main result for massless particles (we refer to Theorem \ref{Theomassless} for more details).

\begin{Th}\label{intromasslessTheo}
Let $n \geq 4$, $K \geq 2$, $N \geq 6n+3$ and $R>0$. Let $(f_0,F_0)$ be an initial data set for the massless Vlasov-Maxwell system with $K$ species, $(f,F)$ be the unique classical solution to the system and $A$ be a potential in the Lorenz gauge. There exists $\epsilon >0$ and $R>0$ such that, if 
$$ \widetilde{\mathcal{E}}_N[A](0) + \mathcal{E}_N[F](0) + \sum_{k=1}^K \mathbb{E}^0_{N+n,1}[f_k](0) \leq \epsilon,$$
$$\forall \hspace{0.5mm} 1 \leq k \leq K, \hspace{1cm} \mathrm{supp}(f_{0k}) \subset \{(x,v) \in \mathbb{R}^n_x \times \mathbb{R}^n_v \setminus \{ 0 \} \hspace{1mm} / \hspace{1mm} |v| \geq R \} ,$$
then $(f,F)$ exists globally in time and verifies the following properties.
\begin{itemize}
\item $f_k(.,.,v)$ vanishes for all $|v| \leq \frac{R}{2}$.
\item Energy bounds are propagated for $F$ and $f_k$. For instance, if $n=4$, $\mathcal{E}_{N-8}[F](t) \lesssim \epsilon (1+t)$ for all $t \in [0,T]$.
\item Pointwise decay estimates on the null decomposition of $\mathcal{L}_{Z^{\beta}}(F)$ and for the velocity average of $f$ and its derivatives. For instance,
 $$ \forall \hspace{0.5mm} |\beta| \leq N-2n, \hspace{0.5mm} (t,x) \in \mathbb{R}_+ \times \mathbb{R}^n, \hspace{2mm} \int_{v \in \mathbb{R}^n \setminus \{0 \} } | \widehat{Z}^{\beta} f_k | dv \lesssim \frac{\epsilon}{\tau_+^{{n-1}}\tau_-}.$$
\end{itemize}
\end{Th}
\begin{Rq}
The hypotheses on the velocity supports of the particle densities appear to be necessary (see Section \ref{section8}).
\end{Rq}
\begin{Rq}
We are not able to obtain optimal decay estimates for the electromagnetic field in dimension $n=4$ with our reasoning since the velocity average of the Vlasov field has a weaker decay rate near the light cone when the mass is zero (this is related to the estimate \eqref{eq:decayintro} mentionned below, which only applies to massive particles).
\end{Rq}

\subsection{The main difficulties and ingredients of our proof}
\subsubsection{High velocities and null properties of the system}
As we use vector field methods, we are brought to commute the equations and prove global bounds on the solutions through energy estimates. After commuting the Vlasov equation once, we are led to estimate terms that could be written schematically as $$ \int_0^t \int_x \int_v |v\mathcal{L}_Z(F)\partial_v f| dv dx ds.$$ Unfortunately, $\partial_{v^i}$, for $1 \leq i \leq n$, is not part of the commutation vector fields for the transport equation. We rewrite them in terms of the elements of $\K$ as
\begin{equation}\label{derivv}
\partial_{v^i} f_k=\frac{1}{v^0}(\widehat{\Omega}_{0i} f_k-t\partial_i f_k-x^i \partial_t f_k),
\end{equation}
so that $\partial_v f$ essentially behaves like $t\partial_{\mu}f$, which is consistent with the behavior of solutions to the free transport equation. This leads us to estimate
\begin{equation}\label{quantitytobound}
\int_0^t \int_x \int_v (s+|x|)|\mathcal{L}_Z(F)\partial f| dv dx ds.
\end{equation} 
As a solution to a wave equation, $\mathcal{L}_Z(F)$ only decays near the light cone as $\frac{1}{(1+t+|x|)^{\frac{n-1}{2}}}$ and we cannot prove by a naive estimate that, in dimensions $n \leq 5$, \eqref{quantitytobound} is uniformly bounded. However, if $f$ is initially compactly supported, one can expect (for, say, sufficiently small data) the characteristics of the transport equation to have velocities bounded away from $1$, and thus the Vlasov field support (in $x$) to be ultimately remote from the light cone. Now, assuming enough initial decay on the Maxwell field, one can prove that $$ |\mathcal{L}_Z(F)| \lesssim (1+s+|x|)^{-\frac{n-1}{2}}(1+|s-|x||)^{-\frac{3}{2}}$$
which, combined with the support properties of $f$, leads to
$$ \int_x \int_v (s+|x|)|\mathcal{L}_Z(F) \partial f| dv dx \lesssim  (1+s)^{-\frac{n}{2}} $$ and \eqref{quantitytobound} is then uniformly bounded in dimensions\footnote{Note that our proof would lead to a $\sqrt{t}$-loss in dimension $3$ which is not sufficient to prove the uniform boundedness of \eqref{quantitytobound}.} $ n \geq 4$.

In our work, we do not make any compact support assumption. Instead, we make crucial use of null properties of the Vlasov-Mawxell system\footnote{The null structure of the Vlasov-Nordström system is also a main ingredient of the proof of \cite{FJS} for the dimension $n=4$. } to deal with the high velocities. More precisely, certain null components of the velocity vector $v$, the derivatives of the electromagnetic field (as $\mathcal{L}_Z (F)$) or the vector $(0, \partial_{v^1} f,..., \partial_{v^n} f)$ behave better than others and the structure of the system is such that there is no product involving only terms with a bad behavior. Taking advantage of the null structure allows us either
\begin{itemize}
\item to transform a $t-r$ decay in a $t+r$ one. For instance,
$$ |\rho(\mathcal{L}_Z(F))| = \left| \frac{x^i}{r}\mathcal{L}_Z(F)_{i0} \right| \lesssim (1+s+|x|)^{-\frac{n+1}{2}}(1+|s-|x||)^{-\frac{1}{2}}.$$
\item To transform a $t+r$ loss in a $t-r$ loss using $\frac{x^i}{r} \partial_{v^i} f \sim (t-r) \partial f$.
\item Or to exploit the $t-r$ decay. For instance, we will control
$$ \int_{C_u(t)} \int_{v \in \R^n} \frac{v^{\underline{L}}}{v^0} |\partial f| dv dC_u(t) \leq \mathbb{E}[f](t),$$
so that, by the change of variables $(\underline{u},u)=(s+r,s-r)$,
$$ \int_0^t \int_x \frac{1}{(1+|s-|x||)^2}\int_v \frac{v^{\underline{L}}}{v^0} |\partial f| dv dx ds \leq \int_{u=-\infty}^t \frac{\mathbb{E}[f](t)}{(1+|u|)^2} du.$$
\end{itemize}

\subsubsection{Improved decay estimates}

$\bullet$ \textit{For massive particles.} In \cite{FJS}, two Klainerman-Sobolev inequalities for velocity averages of Vlasov fields were obtained. They imply in particular that, for $g$ a solution to $v^{\mu} \partial_{\mu} (g)=0$,
\begin{eqnarray}
\nonumber \forall \hspace{0.5mm} (t,x) \in \R_+ \times \R^n, \hspace{-5.5mm} & & \\ \nonumber
 \nonumber  \int_{\R^n_v} |g|(t,x,v) dv \hspace{-1mm} & \lesssim &  \hspace{-1mm} \sum_{\begin{subarray}{} \widehat{Z}^{\beta} \in \K^{k} \\ \hspace{1mm} |\beta| \leq 3 \end{subarray}} \frac{\int_{\R^n_y} \int_{\R^n_v}  |\widehat{Z}^{\beta} g |(0,y,v) dv dy}{(1+t+|x|)^{n-1}(1+|t-|x||)},   \\
 \nonumber \forall \hspace{0.5mm} (t,x) \in \R_+ \times \R^n, \hspace{-5.5mm} & & t^2-|x|^2 \geq 1, \\
\nonumber  \int_{\R^n_v} |g|(t,x,v) dv \hspace{-1mm} & \lesssim & \hspace{-1mm} \frac{1}{(1+t)^n} \hspace{-0.5mm} \sum_{\begin{subarray}{} \widehat{Z}^{\beta} \in \K^{k} \\ \hspace{1mm} |\beta| \leq 3 \end{subarray}} \int_{\R^n_y} \int_{\R^n_v}  |\widehat{Z}^{\beta} g |(\sqrt{1+|y|^2},y,v) \frac{dvdy}{\sqrt{1+|y|^2}}.
\end{eqnarray}
The first one has the advantage to be based on the foliation $(\{t \} \times \R^n)_{t \geq 0}$ but provides a weak estimate near the light cone. The second one gives a stronger decay rate near the light cone but is based on a hyperboloidal foliation. In this paper, in order to remove all compact support assumption on the data and start from a $t=0$ slice, we will prove and use a refined version of the Klainerman-Sobolev inequalities of \cite{FJS}. Our estimate will imply that, for $g$ a solution to $v^{\mu} \partial_{\mu} (g)=0$ and for all $(t,x) \in \R_+ \times \R^n$,

\begin{equation}\label{eq:decayintro}
  \int_{\R^n_v} |g|(t,x,v) dv \hspace{1mm} \lesssim \hspace{1mm}  \sum_{\begin{subarray}{} \widehat{Z}^{\beta} \in \K^{k} \\ \hspace{1mm} |\beta| \leq 3 \end{subarray}} \frac{\int_{\R^n_y} \int_{\R^n_v} |v^0|^2 (1+|y|)   |\widehat{Z}^{\beta} g |(0,y,v) dv dy}{(1+t+|x|)^n}.
\end{equation}
Compared to the Klainerman-Sobolev inequalities proved in \cite{FJS}, \eqref{eq:decayintro} cumulates the advantages of giving a strong decay in the whole spacetime and being adapted to the foliation $(\{t \} \times \R^n ) _{t \geq 0}$. On the other hand, our estimate is not a pure Sobolev inequality (we used the transport equation satisfied by $g$ to establish it).

\begin{Rq}
In the exterior of the light cone (where $ t \leq |x|$), one can in fact obtain arbitrary decay provided we consider additional decay on the initial data (see Section \ref{indefinitedecay}).
\end{Rq}
\noindent $\bullet$ \textit{For massless particles.} Unfortunately, \eqref{eq:decayintro} does not apply to massless particles. Instead, we use weights $z \in \mathbf{k}_0$ preserved by the relativistic transport operator $|v|\partial_t+v^i \partial_i$ in order to gain additional decay. More precisely, in the same spirit as the derivative $\partial_t+\partial_r$ (respectively $\partial_t-\partial_r$) provides an extra decay in $t+r$ (respectively $t-r$) for, say, a solution to $\square u =0$, one has
$$v^0-\frac{x^i}{r}v_i \leq \frac{|v|}{1+t+r} \sum_{z \in \mathbf{k}_0} |z| \hspace{1cm} \text{and} \hspace{1cm} v^0+\frac{x^i}{r}v_i \leq \frac{|v|}{1+|t-r|} \sum_{z \in \mathbf{k}_0} |z|.$$

\subsubsection{The problem of the small velocities}

For the massless Vlasov-Maxwell system, another problem arises from the small velocities since $v^0=|v|$ is not bounded by below. The velocity part $V$ of the characteristics of
$$\partial_t f+\frac{v^i}{|v|}\partial_i f+\frac{v^{\mu}}{|v|}{F_{\mu }}^j \partial_{v^j} f=0$$
solves the ordinary differential equation
$$\dot{V^j}=\frac{V^{\mu}}{|V|}{F_{\mu }}^j,$$
which can lead to $V=0$ in finite time\footnote{Note that this difficulty does not appear in the Einstein-Vlasov system. Indeed, as in \cite{FJS3}, the Vlasov equation can be written, for a metric $g$, as
$$v_{\mu}g^{\mu \nu} \partial_{\nu}f-\frac{1}{2}v_{\mu}v_{\nu}\partial_i g^{\mu \nu} \partial_{v_i}f=0,$$
so these situations should be compared to the two ordinary differential equations
$$\dot{y}=1 \hspace{3mm} \text{and} \hspace{3mm} \dot{y}=y.$$}. More precisely, we prove in Section \ref{section8} that there exists smooth initial data such that the particle densities $f_k$ do not vanish for small velocities and for which the massless Vlasov-Maxwell system does not even admit a local classical solution.

An important step of the proof of Theorem \ref{intromasslessTheo} then consists in proving that the velocity supports of $f_k$ remain bounded by below. For this, we make crucial use of the smallness of assumption on the electromagnetic field as well as its strong decay rate\footnote{In dimension $n \geq 4$, $\| F\|_{L^{\infty}_x} \lesssim (1+t)^{-\frac{3}{2}}$ is time integrable.}

\subsubsection{The perspective of the three dimensional case}

Nevertheless, even in making use of the null properties of the system, our proof does not work in dimension $3$. One way to treat the $3d$ massive case would be to use modified vector fields in the spirit of \cite{Poisson} for the Vlasov-Poisson system and \cite{FJS2} for the Vlasov-Nordström system. This method led to the proof of the stability of the Minkowski space for the Einstein-Vlasov system (cf \cite{FJS3}, \cite{Lindblad}), providing sharp estimates on both the Vlasov field and the metric.

\subsection{Structure of the paper}

In section \ref{section2} we introduce the notations used in the paper, the commutation vector fields and the Vlasov-Maxwell system. In Section \ref{section3} we establish various energy estimates for solutions to the relativistic transport equation or the Maxwell equations. Section \ref{section4} contains an integral estimate, some ways to estimate the $v$ derivatives and the tools to prove pointwise decay estimates for the electromagnetic field. Section \ref{section5} is devoted to our new decay estimate for the solution of a massive relativistic transport equation. In Section \ref{section6} (respectively \ref{section7}), we prove the global existence and the asymptotic properties of the small data solutions of the massive (respectively massless) Vlasov-Maxwell system, which is Theorem \ref{intromasstheo} (respectively Theorem \ref{intromasslessTheo}). In Section \ref{section8}, we prove that there exists smooth initial data which do not vanish for small velocities and for which the massless Vlasov-Maxwell system do not admit a local classical solution.

\subsection{Acknowledgements}

This work forms part of my Ph.D. thesis and I obviously would like to thank my advisor Jacques Smulevici for his guidance.

\section{Notations and preliminaries}\label{section2}

\subsection{Basic notations}\label{notations}

Throughout this article we work on the $n+1$ dimensional Minkowski spacetime $(\mathbb{R}^{n+1}, \widetilde{\eta})$ and we consider two types of coordinates on it. The Cartesian coordinates $(t,x)$, in which $\widetilde{\eta}=diag(-1,1,...,1)$, and null coordinates which are defined by

$$u=t-r, \hspace{5mm} \underline{u}=t+r,$$
and spherical variables $(B,C,D,...)$ (always denoted by capital Latin letters\footnote{The letter $A$ will be reserved for the potential of the electromagnetic.}) which are spherical coordinates on the $(n-1)$-dimensional spheres $t$, $r=constant$. These coordinates are defined globally on $\R^{n+1}$ apart from the usual degeneration of spherical coordinates and at $r=0$. The null derivatives $L$ and $\underline{L}$ are defined as 

$$L=\partial_t+\partial_r, \hspace{5mm} \underline{L}=\partial_t-\partial_r,$$
and we designate by $(e_B,e_C,e_D,...)$ an orthonormal basis on the spheres $(t,r)=constant$. We will use the weights

$$\tau_+^2=1+\underline{u}^2 \hspace{5mm} \text{and} \hspace{5mm} \tau_-^2=1+u^2.$$
For a 2-form $F_{\mu \nu}$, its Hodge dual is denoted by ${}^* \! F$, with \begin{equation}\label{Hodgedef}
{}^* \! F_{\lambda_1 ... \lambda_{n-1} } = \frac{1}{2}F^{\mu \nu} \varepsilon_{\mu \nu \lambda_1 ... \lambda_{n-1}},
\end{equation}
where $\varepsilon_{ \lambda_1 ... \lambda_{n+1}}$ is the Levi-Civita symbol. As in \cite{CK}, we consider its null-decomposition given by $$\alpha_B (F) = F_{BL}, \hspace{2mm} \underline{\alpha}_B (F) = F_{B \underline{L}}, \hspace{2mm} \rho(F) = \frac{1}{2} F_{L\underline{L}}, \hspace{2mm} \sigma_{BD} (F) = F_{BD}.$$
We also associate to a 2-form $F$ its energy-momentum tensor

$$T[F]_{\mu \nu} = F_{\mu \beta} {F_{\nu}}^{\beta}- \frac{1}{4}\eta_{\mu \nu} F_{\rho \sigma} F^{\rho \sigma}.$$
We use Greek letters to denote spacetime indices and Latin letters for space indices.
The velociy vector $(v^{\beta})_{0 \leq \beta \leq n }$ is parametrized by $(v^i)_{1 \leq i \leq n}$ and $v^0=\sqrt{m^2+|v|^2}$. When we study massive particles, we often take $m=1$ for simplicty so that $v^0=\sqrt{1+|v|^2}$. On the other hand, for massless particles $v^0=|v|$.

 We designate the null components of the velocity vector $(v^{\beta})_{0 \leq \beta \leq n }$ by $(v^L,v^{\underline{L}},v^B,...)$, i.e.
$$v=v^L L+v^{\underline{L}} \underline{L}+v^Be_B. $$
In particular,

$$v^L=\frac{v^0+v^r}{2} \hspace{5mm} \text{and} \hspace{5mm} v^{\underline{L}}=\frac{v^0-v^r}{2}.$$
We now introduce several subsets of $\mathbb{R}_+ \times \mathbb{R}^n$ depending on $t \in \mathbb{R}_+$ or $u \in \mathbb{R}$. Let $\Sigma_t$, $C_u(t)$ and $V_u(t)$ be defined as

$$ \Sigma_t := \{t\} \times \mathbb{R}^n, \hspace{2mm} C_u(t):= \{(s,y)  \in \mathbb{R}_+ \times \mathbb{R}^n / \hspace{1mm} s \leq t, \hspace{1mm} s-|y|=u \}$$

and

$$V_u(t) := \{ (s,y) \in \mathbb{R}_+ \times \mathbb{R}^n / \hspace{1mm} s \leq t, \hspace{1mm} s-|y| \leq u \}.$$
The volum form on $C_u(t)$ is given by $dC_u(t)=\sqrt{2}r^{n-1}d\underline{u}d \mathbb{S}^{n-1}$, where $ d \mathbb{S}^{n-1}$ is the standard volume form on the $n-1$ dimensional unit sphere.
\vspace{5mm}

\begin{tikzpicture}
\draw [-{Straight Barb[angle'=60,scale=3.5]}] (0,-0.3)--(0,5);
\fill[color=gray!35] (1,0)--(4,3)--(9.8,3)--(9.8,0)--(1,0);
\node[align=center,font=\bfseries, yshift=-2em] (title) 
    at (current bounding box.south)
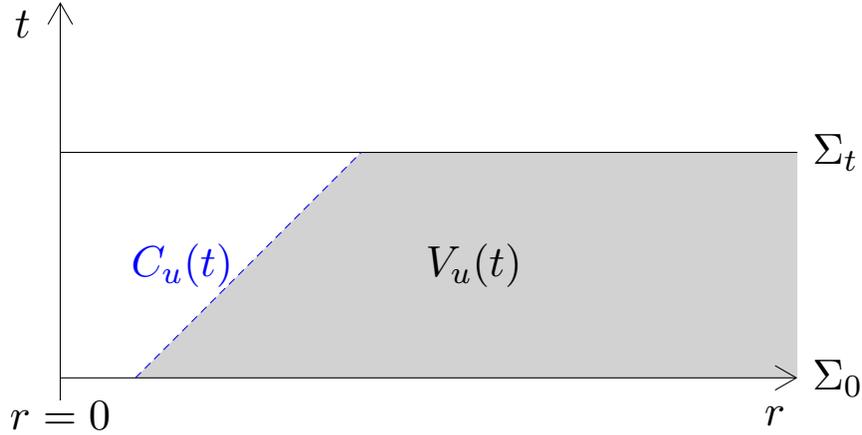
    {The sets $\Sigma_t$, $C_u(t)$ and $V_u(t)$};
\draw (0,3)--(9.8,3) node[scale=1.5,right]{$\Sigma_t$};
\draw [-{Straight Barb[angle'=60,scale=3.5]}] (0,0)--(9.8,0) node[scale=1.5,right]{$\Sigma_0$};
\draw[densely dashed, blue] (1,0)--(4,3) node[scale=1.5,left, midway] {$C_u(t)$};
\draw (5.5,1.5) node[ color=black!100, scale=1.5] {$V_u(t)$}; 
\draw (0,-0.5) node[scale=1.5]{$r=0$};
\draw (-0.5,4.7) node[scale=1.5]{$t$};
\draw (9.5,-0.5) node[scale=1.5]{$r$};   
\end{tikzpicture}

We will use the notation $Q_1 \lesssim Q_2$ for an inequality of the form $ Q_1 \leq C Q_2$, where $C>0$ is a positive constant independent of the solutions but which could depend on $N \in \mathbb{N}$, the maximal order of commutation, or fixed parameters. Finally we will raise and lower indices with respect to the Minkowski metric $\widetilde{\eta}$. For instance, ${F_{\mu}}^j = \widetilde{\eta}^{j \nu} F_{\mu \nu}$ so that ${F_{\mu}}^j =  F_{\mu j}$ for all $1 \leq j \leq n$.

\subsection{The relativistic transport operator}

For $m>0$, we use the notation $T_m$ to refer to the operator defined, for all $v \in \mathbb{R}^n$, by $$T_m=v^0 \partial_t+v^i \partial_i,$$ with $v^0=\sqrt{m^2+|v|^2}$. \\
For the massless case ($m=0$), the relativistic transport operator $T_0$ is only defined for all $v \in \mathbb{R}^n \setminus \{0 \}$ and we have $$T_m=v^0 \partial_t+v^i \partial_i,$$ with $v^0=|v|$.

To simplify the notation, we will most of the time take either $m=1$ or $m=0$ and we will only use $T_1$ and $T_0$.

\subsection{Vector fields}\label{subsecvector}

\subsubsection{The conformal isometries and their complete lifts}

Let us consider the set $\mathbb{K}$ composed by the generators of the isometries group of Minkowski spacetime and by the scaling vector field. $\mathbb{K}$ contains

\begin{flalign*}
& \hspace{1cm} \text{the translations\footnotemark} \hspace{16mm} \partial_{\mu}, \hspace{2mm} 0 \leq \mu \leq n, & \\
& \hspace{1cm} \text{the rotations} \hspace{24mm} \Omega_{ij}=x^i\partial_{j}-x^j \partial_i, \hspace{2mm} 1 \leq i < j \leq n, & \\
& \hspace{1cm} \text{the hyperbolic rotations} \hspace{5mm} \Omega_{0k}=t\partial_{k}+x^k \partial_t, \hspace{2mm} 1 \leq k \leq n, \\
& \hspace{1cm} \text{the scaling} \hspace{28mm} S=x^{\mu}\partial_{\mu}. &
\end{flalign*}
\footnotetext{In this paper, we will denote $\partial_{x^i}$, for $1 \leq i \leq n$, by $\partial_{i}$ and sometimes $\partial_t$ by $\partial_0$.}

Sometimes we will only use the Poincaré group $\mathbb{P} := \mathbb{K} \setminus \{S \}$ or the set of the generators of the rotation group, $\mathbb{O}$ (composed of the $\Omega_{ij}$). These vector fields will be used as commutators whereas $\partial_t$, $S$ and the vector field $\overline{K}_0$, defined by
$$\overline{K}_0=K_0+\partial_t=\frac{1}{2}\tau_-^2 \underline{L}+\frac{1}{2}\tau_+^2L,$$
will be used as multipliers as in \cite{CK}.

These vector fields are well known to commute with the wave operator, i.e.\@ if a smooth function $u$ satisfies $\square u =0$, then, $$\forall \hspace{0.5mm} Z \in \mathbb{K}, \hspace{3mm} \square Z u =0.$$ We will use them to commute the Maxwell equations. However, as in \cite{FJS}, we use another set of vector fields to study the Vlasov equation. For this, we use the \textit{complete lift} of a vector field, a classical operation in differential geometry (see \cite{FJS}, Appendix $C$ for more details). For us, the following definition in coordinates will be sufficient.

\begin{Def}

Let $\Gamma$ be a vector field of the form $\Gamma^{\beta} \partial_{\beta}$. Then, the complete lift $\widehat{\Gamma}$ of $\Gamma$ is defined by
$$\widehat{\Gamma}=\Gamma^{\beta} \partial_{\beta}+v^{\gamma} \frac{\partial \Gamma^i}{\partial x^{\gamma}} \partial_{v^i}.$$

\end{Def}

We then consider $\widehat{\mathbb{P}}$ the set of the complete lifts of $\mathbb{P}$ given by 
$$\widehat{\mathbb{P}}= \{ \widehat{Z} / \hspace{2mm} Z \in \mathbb{P} \}.$$
The last set of vector fields required is the following 
$$\widehat{\mathbb{P}}_0=\widehat{\mathbb{P}} \cup \{ S \}.$$

We can list the complete lifts that we will manipulate.

\begin{Lem}\label{complift}
For $0 \leq \mu \leq n$,
$$\widehat{\partial_{\mu}}=\partial_{\mu}.$$
For $1 \leq i < j \leq n $,
$$\widehat{\Omega_{ij}} = x^i \partial_j -x^j \partial_i +v^i \partial_{v^j}-v^j \partial_{v^i}.$$
Finally, for $1 \leq k \leq n$,
$$\widehat{\Omega_{0k}}=t \partial_k + x^k \partial_t+v^0 \partial_{v^k}.$$

\end{Lem}

The following lemma is used in \cite{FJS} to prove a Klainerman-Sobolev inequality.

\begin{Lem}\label{lem:v}
Let $f : [0,T[ \times \mathbb{R}^n_x \times P \rightarrow \mathbb{R} $, with $P=\mathbb{R}^n_v$ or $P=\mathbb{R}^n_v \setminus \{0 \}$, be a sufficiently regular function. Almost everywhere, we have
$$\forall \hspace{0.5mm} Z \in \mathbb{P}, \hspace{2mm} \left|Z\left( \int_{v \in P } |f| dv \right) \right| \lesssim \int_{v \in P } |\widehat{Z}f| dv +\int_{v \in P } |f| dv,$$
$$\left| S\left( \int_{v \in P } |f| dv \right) \right| \lesssim \int_{v \in P } |Sf| dv.$$

Similar estimates exist for $\int_{v \in \mathbb{R}^n} (v^0)^k |f| dv$. For instance,
$$\left| S\left( \int_{v \in P } v^0|f| dv \right) \right| \lesssim \int_{v \in P } v^0|Sf| dv.$$
\end{Lem}

\begin{Rq}\label{ineq:nolor}

When $Z \in \mathbb{P}$ is not a Lorentz boost, we have
$$\left|Z\left( \int_{v \in P } |f| dv \right) \right| \lesssim \int_{v \in P } |\widehat{Z}f| dv .$$

\end{Rq}

We consider an ordering on each of the following sets of vector fields : $\mathbb{O}$, $\mathbb{P}$, $\mathbb{K}$, $\widehat{\mathbb{P}}$ and $\widehat{\mathbb{P}}_0$. For simplicity, we introduce $\mathbb{L}$ which represents one of those sets. We can suppose that
$$\mathbb{L}= \{ L^i / \hspace{2mm} 1 \leq i \leq |\mathbb{L}| \}.$$
Let $\beta \in \{1, ..., |\mathbb{L}| \}^r$, with $r \in \mathbb{N}^*$. Then we will denote the differential operator $Z^{\beta_1}...Z^{\beta_r}$ by $Z^{\beta}$. For a vector field $Y$, we will denote by $\mathcal{L}_Y$ the Lie derivative with respect to $Y$ and if $Z^{\gamma} \in \mathbb{K}^{q}$, we will write $\mathcal{L}_{Z^{\gamma}}$ for $\mathcal{L}_{Z^{\gamma_1}}...\mathcal{L}_{Z^{\gamma_q}}$. We can suppose that the orderings are such that if $$\mathbb{P}=\{Z^i / \hspace{2mm} 1 \leq i \leq |\mathbb{P}| \}, $$ then
$$\widehat{\mathbb{P}} = \{ \widehat{Z^i} / 1 \leq i \leq |\mathbb{P}| \} \hspace{3mm} \text{and} \hspace{3mm} \K = \{ \widehat{Z^i} / 1 \leq i \leq |\mathbb{P}|+1 \}, \hspace{3mm} \text{with} \hspace{2mm} \widehat{Z^{| \mathbb{P} |+1}}=S.$$
Note that even if the scaling is not a complete lift, we will for simplicity denote any vector field of $\K$ by $\widehat{Z}$.

We now introduce some pointwise norms.

\begin{Def}
Let $U$ be a smooth $p$-covariant tensor field defined in $\R^n$ or in $\R^{1+n}$. For $k \in \mathbb{N}$, the pointwise norm of $U$ with respect to $\mathbb{O}$, of order $k$, is defined by
$$\left| U \right|_{\mathbb{O},k}=\left( \sum_{|\beta| \leq k} |\mathcal{L}_{Z^{\beta}}U|^2 \right)^{\frac{1}{2}},$$
with $Z^{\beta} \in \mathbb{O}^{|\beta|}$ and where
$$|\mathcal{L}_{Z^{\beta}}U|^2 = \sum_{\lambda_1, ... , \lambda_p} |\mathcal{L}_{Z^{\beta}}(U)_{\lambda_1 ... \lambda_p}|^2,$$
with $\mathcal{L}_{Z^{\beta}}(U)_{\lambda_1 ... \lambda_p}$ the Cartesian components of $\mathcal{L}_{Z^{\beta}}(U)$.

\end{Def}

\subsubsection{Commutation properties}

We have the following commutation relations :

\begin{Lem}

Let $\mathbb{L}$ be either $ \mathbb{O}$, $\mathbb{P}$, $\mathbb{K}$, $\widehat{\mathbb{P}}$ or $\widehat{\mathbb{P}}_0$. Then

$$\forall \; L, L' \in \mathbb{L}, \; \exists C_{LL'\Gamma} \in \mathbb{R}, \; \text{such that} \; [L,L']=\sum_{\Gamma \in \mathbb{L}} C_{LL'\Gamma} \hspace{0.5mm} \Gamma.$$
\end{Lem}

The commutation relations between the vector fields of $\widehat{\mathbb{P}}_0$ and the massive transport operator $T_1$ (or the massless relativistic transport operator) are similar to those between the vector fields of $\mathbb{K}$ and the wave operator. 

\begin{Lem}\label{commutransport}
We have, for $m \in \{0,1 \}$,

$$ \forall \hspace{0.5mm} \widehat{Z} \in \widehat{\mathbb{P}}, \hspace{0.5mm} [T_m,\widehat{Z}]=0 \hspace{3mm} \text{and} \hspace{3mm} [T_m,S]=T_m.$$

\end{Lem}

\begin{proof}
This follows easily from Lemma \ref{complift} and the definition of the relativistic transport operator.
\end{proof}

\subsection{Weights preserved by the flow}\label{poids}

We define, as in \cite{FJS}, the two following sets of weights 
\begin{flalign*}
&  \hspace{2cm} \mathbf{k}_1= \left\{\frac{v^{\mu}}{v^0} \hspace{1mm} / \hspace{1mm} 0 \leq \mu \leq n \right\} \cup \left\{ x^{\mu}\frac{v^{\nu}}{v^0}-x^{\nu}\frac{v^{\mu}}{v^0} \hspace{1mm} / \hspace{1mm} \mu \neq \nu \right\}, & \\
& \hspace{2cm} \mathbf{k}_0= \mathbf{k}_1 \cup \left\{ x^{\mu}\frac{v_{\mu}}{v^0} \right\}. &
\end{flalign*}

These weights are solutions to the free transport equation, i.e.

\begin{equation}\label{poids1}
\forall \hspace{0.5mm} z \in \mathbf{k}_0, \hspace{3mm} T_0(z)=0,
\end{equation}
and
\begin{equation}\label{poids2}
\forall \hspace{0.5mm} z \in \mathbf{k}_1, \hspace{3mm} T_1(z)=0.
\end{equation}

Thus, if $f$ is a regular function and if $z \in \mathbf{k}_1$, then $T_1(zf)=zT_1(f)$. 

Moreover, these weights have also good interactions with the vector fields of $\widehat{\mathbb{P}}_0$.

\begin{Lem}\label{vectorweight}

If $\widehat{Z} \in \widehat{\mathbb{P}}_0$, $m \in \{0,1 \}$ and $z \in \mathbf{k}_m$, then either $$\widehat{Z}(v^0z)=0 \hspace{3mm} \text{or} \hspace{3mm} \widehat{Z}(v^0z) \in v^0\mathbf{k}_m.$$ 
This leads to
$$\forall \hspace{0.5mm} \widehat{Z} \in \K, \hspace{1mm} z \in \mathbf{k}_m, \hspace{3mm} |\widehat{Z}(|z|)| \leq \sum_{z \in \mathbf{k}_m} |z|.$$

\end{Lem}

\begin{proof}

Consider for instance $\widehat{\Omega}_{01}$ and $x^1v^2-x^2v^1$ or $x^2v^3-x^3v^2$. We have

$$\widehat{\Omega}_{01}(x^1v^2-x^2v^1)=tv^2-x^2v^0$$
as well as
$$ \widehat{\Omega}_{01}(x^2v^3-x^3v^2)=0.$$
All the other cases are similar.
\end{proof}

The next proposition shows how these weights can be used to provide us extra decay (at least in the massless case). 

\begin{Pro}\label{extradecay1}

Denoting $x^{\mu}v_{\mu}$ by $s$ and $x^{\nu}v^{\mu}-x^{\mu} v^{\nu}$ by $z_{ \mu \nu}$, we have
$$ 2(t-r)v^L= -\frac{x^i}{r}z_{0i}-s,$$
and
$$ 2(t+r)v^{\underline{L}}= \frac{x^i}{r}z_{0i}-s.$$
We also have
$$\frac{|v^B|}{v^0} \lesssim \frac{1}{\tau_+} \sum_{z \in \mathbf{k}_1} |z|, \hspace{5mm} |v^B| \lesssim \sqrt{v^Lv^{\underline{L}}} \hspace{5mm} \text{and} \hspace{5mm} \frac{m^2}{4v^0} \leq v^{\underline{L}}.$$
\end{Pro}

\begin{Rq}\label{extradecay}
This result should be compared with the identities 
$$(t-r)\underline{L}=S-\frac{x^i}{r}\Omega_{0i},$$
$$(t+r)L=S+\frac{x^i}{r}\Omega_{0i},$$
and
$$re_B=\sum_{1 \leq i < j \leq n} C^{i,j}_B \Omega_{ij},$$
where $C^{i,j}_B$ are bounded functions on the sphere.

\end{Rq}

\begin{proof}

Let us start by the first two equations. On the one hand, 
$$(t^2-r^2)v^0=-x^iz_{0i}-ts.$$
On the other hand,
$$(t^2-r^2)v^r=-t \frac{x^i}{r}z_{0i}-rs.$$
It only remains to take the sum and the difference of these two equations. For the third one, use $|v^B| \leq v^0$ and that $rv^B=C_B^{i,j} z_{ij}$, which implies
\begin{eqnarray}
\nonumber |v^B| \hspace{-2.1mm} & \lesssim & \hspace{-2.1mm} \frac{v^0}{r} \sum_{1 \leq i < j \leq n} |z_{ij}| \\ \nonumber
& = & \hspace{-2.1mm} \frac{v^0}{tr} \hspace{-0.3mm} \sum_{1 \leq i < j \leq n} \left| x^i\left( \frac{v^j}{v^0}t-x^j+x^j \right) \hspace{-0.1mm} - \hspace{-0.1mm} x^j\left( \frac{v^i}{v^0}t-x^i+x^i \right) \right| \lesssim \frac{v^0}{t} \sum_{q=1}^n |z_{0q}|.
\end{eqnarray}
The fourth inequality ensues from $rv^B=C_B^{i,j} z_{ij}$ and
\begin{eqnarray}
\nonumber 4r^2v^Lv^{\underline{L}} & = & m^2r^2+r^2 |v|^2-|x^i|^2|v_i|^2-2\sum_{1 \leq k < l \leq n}x^kx^lv^kv^l \\ \nonumber 
& = & m^2r^2+\sum_{1 \leq k < l \leq n} |z_{kl}|^2,
\end{eqnarray}
since $v^0= \sqrt{m^2+|v|^2}$. Finally, using the Cauchy-Schwarz inequality,
$$2v^{\underline{L}}=v^0-\frac{x^i}{r}v_i \geq \frac{m^2}{v^0+|v|} \geq \frac{m^2}{2v^0}.$$

\end{proof}

As for the sets of vector fields, we consider an ordering on $\mathbf{k}_0$ with $x^{\mu} \frac{v_{\mu}}{v^0}$ being the last weight. It then gives an ordering on $\mathbf{k}_1$ too. If $\mathbf{k}_0= \{ z^i / \hspace{2mm} 1 \leq i \leq |\mathbf{k}_0| \}$ and $\beta \in \{ 1, ..., |\mathbf{k}_0| \}^r$ with $ r \in \mathbb{N}^*$, we denote $z^{\beta_1}...z^{\beta_r}$ by $z^{\beta}$.

\subsection{Decay estimates}

\subsubsection{Norms}

With the vector field method, the pointwise decay estimates are obtained through weighted Sobolev inequalities. In view of the above definitions of the vector fields and weights, we are naturally brought to define the following weighted $L^1$ and $L^2$ norms.

\begin{Def}

Let $u : [0,T[ \times \mathbb{R}^n \rightarrow \mathbb{R}$ be a smooth function. For $ k \in \mathbb{N}$, we define for all $t \in [0,T[$,

$$ \| u \|_{\mathbb{K},k}(t) := \sum_{\mu=0}^n \sum_{|\beta| \leq k } \| \partial_{\mu} Z^{\beta} u(t,.) \|_{L^2(\mathbb{R}^n)} ,$$
with $Z^{\beta} \in \mathbb{K}^{|\beta|}$.

Let $f : [0,T[ \times \mathbb{R}^n_x \times P \rightarrow \mathbb{R} $ be a smooth function, with $P=\R^n_v$ or $P=\R^n_v \setminus \{0 \}$. For $ k \in \mathbb{N}$, we define for all $t \in [0,T[$,

$$\| f \|_{\widehat{\mathbb{P}}_0,k}(t) := \sum_{|\beta| \leq k } \| \widehat{Z}^{\beta} f(t,.,.) \|_{L^1_{x,v}},$$
with $\widehat{Z}^{\beta} \in \widehat{\mathbb{P}}_0^{|\beta|}$.

We also define, for $q \in \mathbb{N}$ and $m \in \{0,1 \}$,

$$\| f \|_{\widehat{\mathbb{P}}_0,k,q,m}(t):=\sum_{|\beta| \leq k } \sum_{|\gamma| \leq q} \|z^{\gamma} \widehat{Z}^{\beta} f(t,.,.) \|_{L^1_{x,v}},$$
with $\widehat{Z}^{\beta} \in \widehat{\mathbb{P}}_0^{|\beta|}$ and $z^{\gamma} \in \mathbf{k}^{|\gamma|}_m$.
\end{Def}

Note that $\| u \|_{\mathbb{K},0}$ corresponds to the energy $\sum_{\mu=0}^n \|\partial_{\mu} u \|_{L^2(\R^n)}$.

\subsubsection{Decay estimates for the velocity averages}

Recall the Klainerman-Sobolev inequality (see \cite{Sogge}, Chapter $II$). For $u$ a sufficiently regular function such that for all $t \in [0,T[$, $\|u\|_{\mathbb{K},\frac{n+2}{2}}(t) < + \infty $, we have
\begin{equation}\label{KSclassic}
 \forall \hspace{0.5mm}(t,x) \in [0,T[ \times \mathbb{R}^n, \hspace{3mm} |\nabla_{t,x} u(t,x)| \lesssim \frac{\|u\|_{\mathbb{K},\frac{n+2}{2}}(t)}{(1+t+|x|)^{\frac{n-1}{2}}(1+|t-|x||)^{\frac{1}{2}}}.
 \end{equation}

In particular, if $\square u=0$ then $\|u\|_{\mathbb{K},\frac{n+2}{2}}$ is constant, as $\square Z^{\beta} u=0$ for all $Z^{\beta} \in \mathbb{K}^{|\beta|}$. It gives us a decay estimate for $\nabla_{t,x}u$.

However if $f$ is a solution to a relativistic transport equation, we cannot expect decay on $\| f \|_{L^{\infty}_{x,v}}$ (even for the free transport equation $T_1(f)=0$ or $T_0(f)=0$). It is only the velocity averages of $f$, such as $\int_v f dv$, that decay. For instance, we have the following classical estimate.

\begin{Lem}

Let $f$ be the solution of $T_1(f)=0$ which satisfies $f(0,.,.)=f_0$, with $f_0$ a smooth function compactly supported in $v$. Then, if $R$ is such that $f_0(.,v)=0$ for all $|v| \geq R$,

$$\forall \hspace{0.5mm} (t,x) \in \mathbb{R}_+ \times \mathbb{R}^n, \hspace{3mm} \int_{v \in \mathbb{R}^n} |f(t,x,v)| dv \leq \frac{\sqrt{1+R^2}^{n+2}}{t^n} \|f_0\|_{L^1_xL^{\infty}_v}.$$

\end{Lem}

\begin{proof} We fix $(t,x) \in \mathbb{R}_+ \times \mathbb{R}^n$. By the method of characteristics, we obtain that

$$\forall \hspace{0.5mm} v \in \mathbb{R}^n, f(t,x,v) = f_0\left(x-\frac{v}{v^0}t,v\right).$$

We now use the change of variables $y=\frac{v}{v^0}$. Then, $$\int_{v \in \mathbb{R}^n} |f(t,x,v)|dv =\int_{|y| < 1} |f_0(x-ty,\frac{y}{\sqrt{1-|y|^2}})|\frac{1}{\sqrt{1-|y|^2}^{n+2}}dy.$$

Using the hypothesis on the support of $f_0$, we have $$\int_{v \in \mathbb{R}^n} |f(t,x,v)|dv \leq \sqrt{1+R^2}^{n+2} \int_{|y| < \frac{R}{\sqrt{1+R^2}}} \|f_0(x-ty,.)\|_{L^{\infty}_v}dy.$$

A last change of variables ($z=x-ty$) gives the result.

\end{proof}

\subsubsection{Klainerman-Sobolev inequalities for velocity averages}

As we can expect decay on the velocity average of a solution of a relativistic transport equation (and not on the solution itself), we will then use the following Klainerman-Sobolev inequalities.

\begin{Th}\label{KS1}

Let $T>0$ and $f$ be a smooth function defined on $[0,T[ \times \mathbb{R}^n_x$ $ \times \mathbb{R}^n_v$ or $[0,T[ \times \mathbb{R}^n_x \times (\mathbb{R}^n_v \setminus \{0 \} )$. Then

$$\forall \hspace{0.5mm} (t,x) \in [0,T[ \times \mathbb{R}^n, \hspace{3mm} \int_{v \in \mathbb{R}^n} |f(t,x,v)| dv \lesssim \frac{\| f \|_{\widehat{\mathbb{P}}_0,n}(t)}{\tau_+^{n-1}\tau_-}.$$

\end{Th}

A proof of this inequality can be found in \cite{FJS} (see Theorem $7$). We then deduce the following result.

\begin{Cor}\label{KS3}

Let $T>0$, $q \in \mathbb{N}$, $m \in \{0,1 \}$ and $f$ be a smooth function defined on $[0,T[ \times \mathbb{R}^n_x \times \mathbb{R}^n_v$ or $[0,T[ \times \mathbb{R}^n_x \times (\mathbb{R}^n_v \setminus \{0 \} )$. Then

$$\forall \hspace{0.5mm}|\gamma| \leq q, \hspace{1mm} (t,x) \in [0,T[ \times \mathbb{R}^n, \hspace{3mm} \int_{v \in \mathbb{R}^n} |z^{\gamma}f(t,x,v)| dv \lesssim \frac{\| f \|_{\widehat{\mathbb{P}}_0,n,q,m}(t)}{\tau_+^{n-1}\tau_-},$$

with $z^{\gamma} \in \mathbf{k}_m^{|\gamma|}$.
\end{Cor}

\begin{proof}

Let $|\beta| \leq n$, $|\gamma| \leq q$, $\widehat{Z}^{\beta} \in \widehat{\mathbb{P}}_0^{|\beta|}$ and $z^{\gamma} \in \mathbf{k}_m^{|\gamma|}$. By Lemma \ref{vectorweight}, we have
\begin{equation}\label{eq:KSw}
|\widehat{Z}^{\beta}(z^{\gamma}f)| \lesssim \sum_{|\beta_0| \leq |\beta|} \sum_{|\gamma_0| \leq |\gamma|} |w^{\gamma_0}\widehat{\Gamma}^{\beta_0} f |,
\end{equation}
with $w^{\gamma_0} \in \mathbf{k}_m^{|\gamma_0|}$ and $\widehat{\Gamma}^{\beta_0} \in \widehat{\mathbb{P}}_0^{|\beta_0|}$.
It only remains to apply Theorem \ref{KS1}.

\end{proof}

\begin{Rq}

All the results of this section are true if we add a $v^0$-weight (we can for instance study $\int_{v \in \mathbb{R}^n} (v^0)^k |f| dv$, for $k \in \mathbb{Z}$). We just need to modify the norms in the same way. For instance,
$$ \forall \hspace{0.5mm} (t,x) \in [0,T[ \times \R^n, \hspace{3mm} \int_{v \in \R^n} (v^0)^k |f| dv \lesssim \frac{\sum_{|\beta| \leq n} \|(v^0)^k \widehat{Z}^{\beta}f(t,.,.) \|_{L^1_{x,v}}}{\tau_+^{n-1}\tau_-}.$$

\end{Rq}

\subsection{The Vlasov-Maxwell system}

\subsubsection{Presentation}

In order to introduce the Vlasov-Maxwell system, we abusively use the notation
$$\nabla_v f= \begin{pmatrix}
0 \\
\frac{\partial f}{\partial v^1}  \\
\vdots \\
\frac{\partial f}{\partial v^n}
\end{pmatrix} .$$

For a sufficiently regular function $f$, we recall that 
$$(J(f)^{\nu})_{0 \leq \nu \leq n }= \begin{pmatrix}
\int_v f dv \\
\int_v f\frac{v^1}{v^0}dv \\
\vdots \\
\int_v f\frac{v^n}{v^0}dv
\end{pmatrix} ,$$
with $v^0=\sqrt{m^2+|v|^2}$, where the mass $m$ depends on the species considered.

Let $K \in \mathbb{N}^*$. The equation $(1)$ of the Vlasov-Maxwell system, for the species $k$, can be rewritten as

\begin{equation}
    T_{m_k}(f_k) + e_kF(v, \nabla_v f_k) =0  .    
\end{equation}

Note that the initial data needs to satisfy $$\nabla^i (F_0)_{i0}= e^k J(f_{0k})_{0} \hspace{3mm} \text{and} \hspace{3mm} \nabla^i ({}^* \! F_0)_{i \alpha_1 ... \alpha_{n-3} 0}=0.$$

It is well known that in $3d$ the electric field and the magnetic field are solutions to a wave equation. In dimension $n$ and in the context of the Vlasov-Maxwell system (and more precisely, with equations \eqref{syst2} and \eqref{syst3}), we have
\begin{equation}\label{waveelec}
 \forall \hspace{0.5mm} 1 \leq i \leq n, \hspace{3mm} \square E^i=\sum_{k=1}^Ke_k\int_{v \in \R^n} \partial_i f_k+\frac{v^i}{\sqrt{m_k^2+|v|^2}} \partial_t f_k dv,
 \end{equation}
with $E^i=F_{0i}$, and\footnote{In dimension $n >3$, the magnetic field is a $2$-form defined by $B_{ij}=-F_{ij}$ but we make the choice to work with $F_{ij}$.}
\begin{equation}\label{wavemagn}
 \forall \hspace{0.5mm} 1 \leq i < j \leq n, \hspace{2mm} \square F_{ij}=\sum_{k=1}^Ke_k\int_{v \in \R^n} \frac{v^j}{\sqrt{m_k^2+|v|^2}}\partial_i f_k-\frac{v^i}{\sqrt{m_k^2+|v|^2}} \partial_j f_k dv.
 \end{equation}

We end this subsection by the following proposition, which gives an alternative form of the Maxwell equation.

\begin{Pro}\label{equivalsyst}

The Maxwell equations

\[
\left \{
\begin{array}{c @{=} c}
    
    \nabla^{\mu} G_{\mu \nu} & M_{\nu} \\
    \nabla^{\mu} {}^* \! G_{\mu \alpha_1 ... \alpha_{n-2}} & 0, \\
\end{array}
\right.
\]

for a $2$-form $G$ and a $1$-form $M$, are equivalent to 

\[
\left \{
\begin{array}{c @{=} c}
    
    \nabla_{[ \lambda} G_{\mu \nu ]} & 0 \\
    \nabla_{ [ \lambda} {}^* \! G_{ \alpha_1 ... \alpha_{n-1} ]} & (-1)^{n+1}\frac{(n-1)!}{2}\varepsilon_{\lambda \alpha_1 ... \alpha_n}M^{\alpha_n}, \\
\end{array}
\right.
\]

\end{Pro}

\begin{proof}

That ensues from straightforward calculations. Let us consider the equation $\nabla^i G_{i0} = M_0$.
For $1 \leq i \leq n$, we denote by $(m_j^i)_{1 \leq j \leq n-1}$ the $n-1$ integers of $\llbracket 1,n \rrbracket \setminus \{ i \}$ ranked in ascending order. We have, without any summation,
$$ {}^* \! G_{m_1^i ... m_{n-1}^i}=G^{0i} \varepsilon_{0im_1^i ... m_{n-1}^i}=G_{i0}\varepsilon_{im_1^i ... m_{n-1}^i}.$$
Hence,
$$\nabla^i G_{i0}=\sum_{i=1}^n \varepsilon_{im_1^i ... m_{n-1}^i} \nabla^i {}^* \! G_{m_1^i ... m_{n-1}^i}=\frac{2}{(n-1)!}\nabla_{[1} {}^* \! G_{2...n]}.$$

It only remains to remark that
$$M_0=(-1)^{n+1} \varepsilon_{1...n0}M^0.$$

For the equation $ \nabla^{\mu} {}^* \! G_{\mu 3 ... n} =0$, we note that

$${}^* \! G_{0 3 ... n}=G_{12}, \hspace{2mm} {}^* \! G_{1 3 ... n}=G_{02}, \hspace{2mm} {}^* \! G_{2 ... n }=G_{10}.$$
So $$\nabla^{\mu} {}^* \! G_{\mu 3 ... n}=\nabla_0 G_{21} + \nabla_1 G_{02}+ \nabla_2 G_{10}.$$

It then comes that
$$\nabla_{[0} G_{12]}=0.$$

The remaining equations can be treated similarly. 

\end{proof}

For the remaining of this section, we consider the maximal smooth solution $(f:=(f_1,..,f_K), F )$ to the Vlasov-Maxwell system, defined on $[0, T [$, arising from initial data $(f_0,F_0)$, so that $f$ is a vector valued field $(f_1,..,f_K)$. However, to lighten the notations, we will often denote (by a small abuse of notation) by $f$ only one of the $f_i$ and we will suppose, without loss of generality for the results establish below, that the charge of the species associated to f is equal to 1.

\subsubsection{The electromagnetic potential}

In order to establish energy estimates for the electromagnetic field, it is useful to introduce a potential in the Lorenz gauge.

\begin{Def}

A $1$-form $A$ is said to be a potential of the electromagnetic field $F$ if
$$F=dA \hspace{3mm} \text{or, in coordinates,} \hspace{3mm} F_{\mu \nu} = \partial_{\mu}A_{\nu}-\partial_{\nu}A_{\mu}.$$
$A$ satisfies the Lorenz gauge condition if moreover
$$\partial^{\mu} A_{\mu}=0.$$

\end{Def}

Every electromagnetic field $F$ defined on $\R^{n+1}$, which is contractible, has a potential since $dF=0$. Furthermore, if $A$ is a potential then, for $\chi$ a regular function, $A+d\chi$ is also a potential. In particular, if $A$ is a potential and $\chi$ solves
$$\square \chi= -\partial^{\mu} A_{\mu} $$
then $A+d\chi$ is a new potential satisfying the Lorenz gauge. The following lemma will be useful to study the derivatives of $F$ in the Lorenz gauge.

\begin{Lem}\label{derivpotential}

If $A$ is a potential satisfying the Lorenz gauge for an electromagnetic field $G$, i.e.
$$dA=G \hspace{3mm} \text{and} \hspace{3mm} \partial^{\mu} A_{\mu}=0,$$
then, for all $Z \in \mathbb{K}$,
$$d\mathcal{L}_{Z}(A)=\mathcal{L}_{Z}G \hspace{3mm} \text{and} \hspace{3mm} \partial^{\mu} \mathcal{L}_{Z}(A)_{\mu}=0.$$

\end{Lem} 

Let us mention the wave equation satisfied by the potential in the Lorenz gauge.

\begin{Pro}\label{CommuA}

Let $(f,F)$ be a solution to the Vlasov-Maxwell system and $A$ be a potential of the electromagnetic field $F$ which satisfies the Lorenz gauge. Then, for all $Z^{\beta} \in \mathbb{K}^{|\beta|}$ and $0 \leq \mu \leq n$, there exists constants $C_{\gamma}^{\mu}$ such that
$$ \square \mathcal{L}_{Z^{\beta}} A_{\mu} = \sum_{|\gamma| \leq |\beta|} C_{\gamma}^{\mu}e^k\int_{v \in \R^n} \frac{v^{\mu}}{v^0} \widehat{Z}^{\gamma} f_k dv,$$
with $\widehat{Z}^{\gamma} \in \K^{|\gamma|}$.

\end{Pro}

\begin{proof}

As $$F_{\mu \nu }=\partial_{\mu} A_{\nu}-\partial_{\nu} A_{\mu} \hspace{3mm} \text{and} \hspace{3mm} \partial^{\mu} A_{\mu}=0,$$ we have for $0 \leq \nu \leq n$
$$\partial^{\mu} \partial_{\mu} A_{\nu} = \nabla^{\mu} F_{\mu \nu}.$$
It remains to apply this to $\mathcal{L}_{Z^{\beta}} A$ (see Lemma \ref{derivpotential}) and to use Proposition \ref{Commuelec} below.

\end{proof}

The following proposition shows how we can construct a potential in the Lorenz gauge which is initially controled by the energy (at the time $0$) of the electromagnetic field.

\begin{Pro}\label{LorenzPot}
We suppose here that $n \geq 4$. Let $N \in \mathbb{N}$ and let $F$ be a closed $2$-form such that all the norms considered below are finite and $F(0,.) \in L^2(\R^n)$. Then, there exists a potential $A$ in the Lorenz gauge such that, for all $|\beta| \leq N$,
\begin{flalign*}
& \|\mathcal{L}_{Z^{\beta}} A \|_{L^2(\R^n)}(0) \lesssim \sum_{\begin{subarray}{l}|\gamma| \leq N-1 \\ \hspace{1mm} 1 \leq i \leq n \end{subarray}} \|(1+|x|)^{|\gamma|+1}\partial^{\gamma}F_{0i}(0,.) \|_{L^2(\R^n)} &
\end{flalign*}
 $$ \hspace{0.7cm} +\sum_{\begin{subarray}{l}|\gamma| \leq  N \\ 1 \leq i \leq n \end{subarray}} \left( \| (1+|x|)^{|\gamma|}\partial^{\gamma} \partial^j F_{ji}(0,.) \|_{L^2_x}+\|(1+|x|)^{|\gamma|+1} \partial^{\gamma} \partial^j F_{ji}(0,.) \|_{L^1_x} \right),$$
with $Z^{\beta} \in \mathbb{K}^{|\beta|}$.

\end{Pro}

We start by a technical lemma.

\begin{Lem}\label{arfin}
Let $G$ such that 
$$\|(1+|x|)G\|_{L^1(\R^n)}+\|G\|_{L^2(\R^n)} < + \infty \hspace{3mm} \text{and} \hspace{3mm} \int_{\R^n} G dx =0.$$
Then, denoting by $\mathcal{F}$ the Fourier transform (in $x$),
$$\left\|\mathcal{F}^{-1}\left( \frac{-1}{|\xi|^2} \mathcal{F}(G) \right) \right\|_{L^2(\R^n)} \lesssim \|(1+|x|)G\|_{L^1(\R^n)}+\|G\|_{L^2(\R^n)}.$$
\end{Lem}

\begin{proof}
We have
\begin{eqnarray}
\nonumber \left\|\mathcal{F}^{-1}\left( \frac{-1}{|\xi|^2} \mathcal{F}(G) \right) \right\|_{L^2(\R^n)} & = & \left\| \frac{1}{|\xi|^2} \mathcal{F}(G) \right\|_{L^2(\R^n)} \\ \nonumber
& \lesssim & \| \mathcal{F}(G) \|_{L^2(|\xi| \geq 1)}+\left\| \frac{1}{|\xi|^4} \mathcal{F}(G)^2 \right\|^{\frac{1}{2}}_{L^1(|\xi| \leq 1)}.
\end{eqnarray}
Note now that $\| \mathcal{F}(G) \|_{L^2(|\xi| \geq 1)} \leq \| G \|_{L^2(\R^n)}$. Finally, as $\|(1+|x|)G\|_{L^1(\R^n)}$ is finite, $\mathcal{F}(G)$ is of class $C^1$ and vanishes at $0$, so, using the mean value theorem,
\begin{eqnarray}
\nonumber \left\| \frac{\mathcal{F}(G)}{|\xi|^4} \right\|^{\frac{1}{2}}_{L^1(|\xi| \leq 1)}  & \lesssim &  \|\nabla_{\xi}\mathcal{F}(G)\|_{L^{\infty}_{\xi}} \left\| \frac{1}{|\xi|^3} \right\|^{\frac{1}{2}}_{L^1(|\xi| \leq 1)} \\ \nonumber
& \lesssim &  \||x|G\|_{L^1(\R^n)},
\end{eqnarray}
since $\|F(g)\|_{L^{\infty}_{\xi}} \leq \|g\|_{L^1_x}$ for any $L^1$ function $g$.
\end{proof}

The first step of the construction of the suitable potential is contained in the following lemma.

\begin{Lem}\label{lempoten2}
There exists a potential $A$ of the electromagnetic field $F$ satisfying the Lorenz gauge and such that
$$A_0(0,.)=0, \hspace{2mm} \partial_t A_0(0,.)=0,$$ and $$\forall \hspace{0.5mm} 1 \leq k \leq n, \hspace{2mm} \|A_k\|_{H^2(\R^n)}(0) \leq  \|  \partial^j F_{jk}(0,.) \|_{L^2_x}+\|(1+|x|)  \partial^j F_{jk}(0,.) \|_{L^1_x}.$$
This implies in particular that
\begin{equation}\label{eq:potential2}
\forall \hspace{1mm} 1 \leq k \leq n, \hspace{3mm} \partial_t A_k(0,.) = F_{0k}(0,.) \hspace{3mm} \text{and} \hspace{3mm} \Delta A_k (0,.) = \partial^i F_{ik}(0,.).
\end{equation}

\end{Lem}

\begin{proof}

Suppose that $A$ exists. As $\partial_t A_0(0.)=0$ and $\partial^{\mu} A_{\mu}=0$, we have $\partial^i A_i(0.)=0$. Combined with $\partial_{\mu} A_{\nu}-\partial_{\nu} A_{\mu} = F_{\mu \nu}$ and $A_0(0,.)=0$, it comes that at $t=0$,

\begin{equation}\label{eq:potential3}
\forall \hspace{1mm} 1 \leq k \leq n, \hspace{3mm} \partial_t A_k = F_{0k} \hspace{3mm} \text{and} \hspace{3mm} \Delta A_k = \partial^i F_{ik}.
\end{equation}
Moreover, recall from the proof of Proposition \ref{CommuA} that
\begin{equation}\label{wagenn}
\forall \hspace{0.5mm} 0 \leq \nu \leq n, \hspace{3mm} \square A_{\nu} = \nabla^{\mu} F_{\mu \nu}.
\end{equation}
We then define $A_{\nu}$ as the solution of the wave equation \eqref{wagenn} such that $A_0(0,.)=0$, $\partial_t A_0(0,.)=0$ and, for all $1 \leq k \leq n$, $$\partial_t A_k(0,.)=F_{0k}(0,.) \hspace{2mm} \text{and} \hspace{2mm} A_k(0,.)=\mathcal{F}^{-1}\left(\frac{-1}{|\xi|^2} \mathcal{F}(\partial^j F_{jk}) \right)(0,.).$$ 
Consequently, according to Lemma \ref{arfin}, $\Delta A_k(0,.)=\partial^j F_{jk}$ and
$$\|A_k\|_{L^2(\R^n)}(0) \leq  \|  \partial^j F_{jk}(0,.) \|_{L^2_x}+\|(1+|x|) \partial^j F_{jk}(0,.) \|_{L^1_x}.$$
From classical elliptic equations theory, we have
$$ \|\nabla^2 A_k \|_{L^2(\R^n)} = \|\partial^j F_{jk} \|_{L^2(\R^n)}$$
and 
$$\nabla A_k \in L^2(\R^n), \hspace{3mm} \text{with} \hspace{3mm} \|\nabla A_k\|_{L^2(\R^n)} \lesssim \| A_k\|_{L^2(\R^n)}+\|\nabla^2 A_k\|_{L^2(\R^n)},$$
which concludes the proof.
\end{proof}

\begin{proof}[Proof of Proposition \ref{LorenzPot}]

We consider the potential $A$ constructed in \newline Lemma \ref{lempoten2}. In what follows, we omit to specify that all the quantities are considered at $t=0$. Since, for instance,
$$\forall \hspace{1mm} 1 \leq i,j \leq n, \hspace{3mm} \Omega_{0i} \Omega_{0j} A=x^i\partial_j A+x^ix^j\partial_t \partial_t A,$$
we have (and it is sufficient) to estimate $\| x^{\beta} \partial_{t,x}^{\gamma} A \|_{L^2(\R^n)}$, with $|\beta| \leq |\gamma| \leq N$, in order to control $ \| \mathcal{L}_{Z^{\xi}} A \|_{L^2(\R^n)}(0)$ for all $Z^{\xi} \in \mathbb{K}^{|\xi|}$, with $|\xi| \leq N$. Note that, as $\partial^{\mu} A_{\mu}=0$,
$$\|x^{\beta}\partial_{t,x}^{\gamma}\partial_tA_0\|_{L^2(\R^n)} \leq \sum_{k=1}^n \|x^{\beta}\partial_{t,x}^{\gamma}\partial_kA_k\|_{L^2(\R^n)},$$
so that, since $A_0=0$, we only have to bound $\| x^{\beta} \partial_{t,x}^{\gamma} A_k \|_{L^2(\R^n)}$, for all $1 \leq k \leq n$.

Let $1 \leq k \leq n$, $|\gamma| \leq N-1$ and $|\beta| \leq |\gamma|+1$. Then, since $\partial_t A_k=F_{0k}$ (see Lemma \ref{lempoten2}),
$$x^{\beta} \partial_{t,x}^{\gamma} \partial_t A_k=x^{\beta} \partial^{\gamma}_{t,x} F_{0k}  , \hspace{3mm} \text{so} \hspace{3mm} \|x^{\beta} \partial_{t,x}^{\gamma} \partial_t A_k \|_{L^2_x} \lesssim \|(1+|x|)^{|\gamma|+1}\partial^{\gamma}F_{0k} \|_{L^2_x}.$$
The remaining case, where there are only spatial translations, is treated in the following lemma.

\begin{Lem}\label{labelrq}

For all $1 \leq k \leq n$, $|\gamma| \leq N$ and $|\beta| \leq |\gamma|$,
$$\|x^{\beta} \partial^{\gamma} A_k \|_{L^2(\R^n)} \lesssim \sum_{|\beta_0| \leq |\gamma_0| \leq N} \| x^{\beta_0}\partial^{\gamma_0} \partial^j F_{jk} \|_{L^2_x}+\|(1+|x|) x^{\beta_0}\partial^{\gamma_0} \partial^j F_{jk} \|_{L^1_x}, $$

where $\gamma$, $\beta \in \mathbb{N}^n$,  $x^{\beta}=x_1^{\beta_1}...x_n^{\beta_n}$ and $\partial^{\gamma}=\partial_1^{\gamma_1}...\partial_n^{\gamma_n}$, so there are no time derivatives.

\end{Lem}

\begin{proof}
We fix $1 \leq k \leq n$ and we proceed by induction on $|\beta|$. As $\Delta A_k= \partial^j F_{jk}$, we have, for all $|\gamma| \leq N-2$,
$$ \forall \hspace{1mm} 1 \leq k \leq n, \hspace{3mm} \Delta \partial^{\gamma} A_k = \partial^{\gamma} \partial^j F_{jk}$$
So, by classical elliptic equations theory,
$$\forall \hspace{0.5mm} |\gamma| \leq N-2, \hspace{3mm} \|\nabla^2 \partial^{\gamma} A_k \|_{L^2(\R^n)} = \|\partial^{\gamma} \partial^j F_{jk}\|_{L^2(\R^n)},$$
implying the result for $|\beta|=0$ (the case of the lower order derivatives is treated in Lemma \ref{lempoten2}).

Let $1 \leq |\beta| \leq N$. We suppose that for all $|\delta| \leq |\gamma| \leq N$ and $|\delta| \leq |\beta|-1$,
$$\|x^{\delta} \partial^{\gamma} A_k \|_{L^2(\R^n)} \lesssim \sum_{|\beta_0| \leq |\gamma_0| \leq N} \| x^{\beta_0}\partial^{\gamma_0} \partial^j F_{jk} \|_{L^2_x}+\| (1+|x|)x^{\beta_0}\partial^{\gamma_0} \partial^j F_{jk} \|_{L^1_x}.$$
Let $\gamma$ be a multi-index such that $|\beta| \leq |\gamma| \leq N$. We have
\begin{equation}\label{eq:forfourier}
\Delta x^{\beta}\partial^{\gamma} A_k = \Delta(x^{\beta})\partial^{\gamma} A_k+2\partial_j(x^{\beta})\partial^j\partial^{\gamma} A_k+x^{\beta}\partial^{\gamma} \partial^j F_{jk}.
\end{equation} 
The first two terms of the right hand side are equal to zero or can be rewritten as a linear combination of terms like
\begin{equation}\label{tost}
\partial^{\gamma_2}(x^{\delta} \partial^{\gamma_1} A_k),
\end{equation}
with $|\gamma_2|=2$, $|\gamma_1| \leq |\gamma|-1$ and $|\delta| \leq |\gamma_1|$. For instance,
$$2\partial_j(x_1^q)\partial^j\partial_2^{q} A_k=2q \partial_1\partial_2 ( x_1^{q-1}\partial_2^{q-1}A_k)-2q(q-1)\partial_2^2(x_1^{q-2}\partial_2^{q-2}A_k).$$
Let $B$ be the right hand side of \eqref{eq:forfourier} and $G=x^{\beta}\partial^{\gamma} \partial^j F_{jk}$. $G$ satisfies the hypothesis of Lemma \ref{arfin} and $B-G$ is a linear combination of terms such as \eqref{tost}, which implies $$\left\|\mathcal{F}^{-1}\left( \frac{-1}{|\xi|^2}\mathcal{F}(B-G) \right)\right\|_{L^2(\R^n)} \lesssim \sum_{\begin{subarray}{l} |\gamma_1| \leq |\gamma|-1 \\ \hspace{1mm} |\delta| \leq |\gamma_1| \end{subarray} } \|x^{\delta} \partial^{\gamma_1} A_k\|_{L^2(\R^n)}.$$ So we only have to prove that
\begin{equation}\label{echu}
x^{\beta}\partial^{\gamma} A_k=\mathcal{F}^{-1}\left( \frac{-1}{|\xi|^2}\mathcal{F}(B) \right),
\end{equation}
or (it is equivalent) that $x^{\beta}\partial^{\gamma} A_k$ is the $L^2$ solution of $\Delta \varphi = B$. Recall that the difference of two solutions of this equation is an harmonic polynomial, so that there exists exactly one $L^2$ solution, given by the right hand side of \eqref{echu}. Consequently, there exists $Q_{k,\beta,\gamma} \in L^2(\R^n)$ and $P_{k,\beta,\gamma}$ an harmonic polynomial function such that
$$x^{\beta} \partial^{\gamma} A_k=Q_{k,\beta,\gamma}+P_{k,\beta,\gamma}.$$
By the induction hypothesis, $x^{\delta} \partial^{\gamma} A_k \in L^2(\R^n)$ for all $|\delta| = |\beta|-1$, so
$$\frac{P_{k,\beta,\gamma}}{1+|x|}=\frac{x^{\beta}}{1+|x|} \partial^{\gamma} A_k-\frac{1}{1+|x|}Q_{k,\beta,\gamma} \in L^2(\R^n).$$
As the dimension is $n \geq 4>1$, $P_{k,\beta,\gamma}$ is necessarily zero.
\end{proof}
\end{proof}

If the dimension $n$ is at least $5$, we can do better.

\begin{Pro}
We suppose here that $n \geq 5$. Let $N \in \mathbb{N}$ and let $F$ be a $2$-form such that all the norms considered below are finite. There exists a potential in the Lorenz gauge such that, for all $|\beta| \leq N$,
\begin{flalign*}
& \|\mathcal{L}_{Z^{\beta}} A \|_{L^2(\R^n)}(0) \lesssim \sum_{\begin{subarray}{l} |\gamma| \leq N-1 \\ \hspace{1mm} 1 \leq i \leq n \end{subarray}} \|(1+|x|)^{|\gamma|+1}\partial^{\gamma}F_{0i}(0,.) \|_{L^2(\R^n)}&
\end{flalign*}
 $$+\sum_{\begin{subarray}{l}|\gamma| \leq  N \\ 1 \leq i \leq n \end{subarray}} \left( \| (1+|x|)^{|\gamma|}\partial^{\gamma} \partial^j F_{ji}(0,.) \|_{L^2(\R^n)} \hspace{-0.5mm}+ \hspace{-0.5mm}\| (1+|x|)^{|\gamma|}\partial^{\gamma} \partial^j F_{ji}(0,.) \|_{L^1(\R^n)} \right),$$
with $Z^{\beta} \in \mathbb{K}^{|\beta|}$.

\end{Pro}

\begin{proof}

The proof is similar to the previous one. The difference comes from the fact $ \xi \mapsto \frac{1}{|\xi|^4}$ is integrable around $0$ in $\R^n$, with $n \geq 5$, which allows us to lower the hypothesis of Lemma \ref{arfin}.

\end{proof}

\subsubsection{Commutation properties}

\subsubsection*{Commutation of the transport equation}

We fix the mass $m \in \R_+$ and we denote by $T_F$ the operator\footnote{Note that if the charge $e$ of the species considered is not equal to $1$, one just has to consider $T_{eF}$ (in other words, one just has to replace $F$ by $eF$).} $$ T_F : g \mapsto T_m(g)+F(v,\nabla_v g),$$ 
so that $T_F(f)=0$.
We are now interested by the nature of the source terms of the equation $T_F(\widehat{Z}f)=G$.

\begin{Lem}\label{Commufsimple} If $\widehat{Z} \in \widehat{\mathbb{P}}$, then

$$T_F(\widehat{Z} f ) = -\mathcal{L}_{Z}(F)(v,\nabla_v f) .$$

For the scaling, we have

$$T_F(S f ) = 2F(v,\nabla_v f)-\mathcal{L}_{S}(F)(v,\nabla_v f) .$$

\end{Lem}

\begin{proof}

First of all, let us consider the scaling. According to Lemma \ref{commutransport},

$$T_m(Sf)=-S(F(v,\nabla_v f))+T_m(f).$$
But, $$ S(F(v,\nabla_v f)) = \mathcal{L}_S (F)(v, \nabla_v f)+F([S,v],\nabla_v f)+F(v, [S,\nabla_v f]).$$
Since $$ [S,v]=-v \hspace{5mm} \text{and} \hspace{5mm} [S, \nabla_v f]=\nabla_v S(f)-\nabla_v f,$$
we obtain
$$T_F(Sf)=2F(v, \nabla_v f)-\mathcal{L}_S (F)(v, \nabla_v f).$$

Now, let $\widehat{Z} \in \widehat{\mathbb{P}}$ and consider $Z_v=\widehat{Z}-Z$. According to lemma \ref{commutransport},

$$T_m(\widehat{Z}f)=-Z(F(v,\nabla_v f))-Z_v(F(v,\nabla_v f)).$$

On the one hand, we have
$$Z_v(F(v,\nabla_v f))=F(Z_v(v),\nabla_v f)+F(v,Z_v(\nabla_v f)).$$
On the other hand we have
$$Z(F(v,\nabla_v f))=\mathcal{L}_Z(F)(v,\nabla_v f)+F([Z,v],\nabla_v f)+F(v,[Z,\nabla_v f]).$$

As $[Z,v]=-Z_v(v)$, $F(Z_v(v),\nabla_v f)$ and $F([Z,v],\nabla_v f)$ cancel. \\
If $\widehat{Z}$ is a translation (we denote it by $\partial$), then $Z_v=0$ and $[Z, \nabla_v f ]=\nabla_v \partial(f)$. Thus
$$T_F(\partial f)=-\mathcal{L}_{\partial}(F)(v,\nabla_v f).$$
If $\widehat{Z}=\widehat{\Omega}_{ij}$, then 
$$Z_v(\nabla_v f)=\nabla_v Z_v(f)+\partial_{v^i} f \partial_j-\partial_{v^j} f \partial_i $$ and $$ [Z,\nabla_v f]=\nabla_v Z(f)-\partial_{v^i} f \partial_j+\partial_{v^j} f \partial_i. $$
Therefore 
$$T_F( \widehat{\Omega}_{ij} f ) = - \mathcal{L}_{\Omega_{ij}}(F)(v,\nabla_v f) .$$
Finally, if $\widehat{Z}= \widehat{\Omega}_{0i}$, then
$$Z_v(\nabla_v f)=\nabla_v Z_v(f)-\partial_{v^i} f \frac{v^k}{v^0}\partial_k \hspace{3mm} \text{and} \hspace{3mm} [Z,\nabla_v f]=\nabla_v Z(f)-\partial_{v^i} f \partial_0. $$
It comes that
$$T_F( \widehat{\Omega}_{0i} f ) = - \mathcal{L}_{\Omega_{0i}}(F)(v,\nabla_v f) + \frac{\partial_{v^i} f}{v^0}F(v,v).$$
It remains to remark that $F(v,v)=0$ for all $v \in \mathbb{R}^n$, as $F$ is a $2$-form.

\end{proof}

Iterating the above, one obtains

\begin{Cor}\label{Commuf}
If $\beta \in \{1,...,|\widehat{\mathbb{P}}_0|\}^r$, with $r \geq 0$, there exist integers $C_{\gamma,\delta}^{\beta}$ such that

$$T_F(\widehat{Z}^{\beta} f)=\sum_{\begin{subarray}{l} |\gamma|+|\delta| \leq r \\ \hspace{1mm} |\delta| \leq r-1 \end{subarray}} C_{\gamma,\delta}^{\beta} \mathcal{L}_{Z^{\gamma}}(F)(v,\nabla_v \widehat{Z}^{\delta}(f)),$$

with $\widehat{Z}^{\beta} \in \widehat{\mathbb{P}}_0^r$, $\widehat{Z}^{\delta} \in \widehat{\mathbb{P}}_0^{|\delta|}$ and $Z^{\gamma} \in \mathbb{K}^{|\gamma|}$.

\end{Cor}

\begin{Rq}\label{remcommuf}

If there is a source term $G$ (such that $T_F(f)=G$), then we need to add a linear combination of terms such as $\widehat{Z}^{\widetilde{\beta}} G$, with $|\widetilde{\beta}| \leq r$, on the right hand side.

\end{Rq}

\subsubsection*{Commutation of the Maxwell equations}

Before studying specifically the Vlasov-Maxwell system, we recall the following general result.

\begin{Pro}

Let $M_{\nu}$ be a smooth $1$-form and $G_{\mu \nu}$ a $2$-form satisfying

\[
\left \{
\begin{array}{c @{=} c}
    
    \nabla^{\mu} G_{\mu \nu} & M_{\nu} \\
    \nabla^{\mu} {}^* \! G_{\mu \alpha_1 ... \alpha_{n-2}} & 0. \\
\end{array}
\right.
\]

Then, for all $Z \in \mathbb{P}$,

\[
\left \{
\begin{array}{c @{=} c}
    
    \nabla^{\mu} \mathcal{L}_Z(G)_{\mu \nu} & \mathcal{L}_Z(M)_{\nu} \\
    \nabla^{\mu} {}^* \! \mathcal{L}_Z(G)_{\mu \alpha_1 ... \alpha_{n-2}} & 0. \\
\end{array}
\right.
\]
For the scaling, we have

\[
\left \{
\begin{array}{c @{=} c}
    
    \nabla^{\mu} \mathcal{L}_S(G)_{\mu \nu} & \mathcal{L}_S(M)_{\nu}+2M_{\nu} \\
    \nabla^{\mu} {}^* \! \mathcal{L}_S(G)_{\mu \alpha_1 ... \alpha_{n-2}} & 0. \\
\end{array}
\right.
\]

\end{Pro}

In the Vlasov-Maxwell system, the source term is $e^kJ(f_k)_{\nu}$ (see \eqref{syst2}), with

$$(J(f_k)^{\nu})_{0 \leq \nu \leq 3 }= \begin{pmatrix}
\int_v f_k dv \\
\int_v f_k\frac{v^1}{v^0}dv \\
\vdots \\
\int_v f_k\frac{v^n}{v^0}dv
\end{pmatrix} ,$$

so we need to compute $\mathcal{L}_{Z} (J(f))$, with $Z \in \mathbb{K}$ and $f$ a regular function.

\begin{Pro}

For all $Z \in \mathbb{P}$,

$$\mathcal{L}_{Z} (J(f)_{\nu})=J(\widehat{Z} f)_{\nu}.$$

For the scaling, we have

$$\mathcal{L}_{S} (J(f)_{\nu})=J(S f)_{\nu}+J(f)_{\nu}.$$

\end{Pro}

\begin{proof}

Let $Z \in \mathbb{K}$,

$$\mathcal{L}_{Z} J(f)_{\nu} = ZJ(f)_{\nu}+J(f)_{\mu}\frac{\partial Z^{\mu}}{\partial x^{\nu}}.$$

So
$$L_{\partial} J(f) =   J(\partial f), \hspace{2mm} L_{S}J(f) =J(Sf)+J(f).$$
If $Z$ is a Lorentz boost, say $x^1 \partial_t+t \partial_1$, then, as

$$\int_v v^0\partial_{v^1}f dv=-\int_v f \frac{v^1}{v^0} dv=-J(f)_1, $$ $$ \int_v v^0\frac{v^i}{v^0}\partial_{v^1}f  dv=-\delta_{1,i}\int_v f dv=\delta_{1,i} J(f)_0 $$
and 
$$J(f)_{\mu}\frac{\Omega_{01}^{\mu}}{\partial x^{\nu}} =J(f)_1 \delta_{\nu,0}+J(f)_0 \delta_{\nu,1},$$
it comes that
$$\mathcal{L}_{\Omega_{01}} J(f) = J(\widehat{\Omega_{01}} f ).$$

The case where $Z$ is a rotation is similar.

\end{proof}

Iterating the above, we obtain the following proposition.

\begin{Pro}\label{Commuelec}

Let $(f,F)$ be a smooth solution of the Vlasov-Maxwell system. For all $\beta \in \{1,..., |\mathbb{K}| \}^r$, with $ r \in \mathbb{N}$, there exist integers $C^{\beta}_{\gamma}$ such that

\begin{eqnarray}
    \nonumber
    \nabla^{\mu} \mathcal{L}_{Z^{\beta}}(F)_{\mu \nu} & = & e^kJ(\widehat{Z}^{\beta}f_k)_{\nu}+\sum_{|\gamma| \leq |\beta|-1} C_{\gamma}^{\beta}e^kJ(\widehat{Z}^{\gamma} f_k)_{\nu}, \\ \nonumber
    \nabla^{\mu} {}^* \! \mathcal{L}_{Z^{\beta}}(F)_{\mu \alpha_1 ... \alpha_{n-2}} & = &0, 
\end{eqnarray}

with $Z^{\beta} \in \mathbb{K}^r$ and $\widehat{Z}^{\gamma} \in \widehat{\mathbb{P}}_0^{|\gamma|}$.

\end{Pro}

\section{Energy estimates for the Vlasov-Maxwell system}\label{section3}

For all this section, we consider a sufficiently regular solution $(f,F)$, on $[0,T[$, to the Vlasov-Maxwell system arising from smooth initial data $(f_0,F_0)$.

\subsection{Energy estimates for the transport equation}

We treat here the massless and the massive case together. As the set $\{v=0 \}$ is of measure zero, we write $\int_{v \in \R^n} h dv$, or merely $\int_{v} h dv$, even when the function $h$ is not defined for $v=0$. We start by introducing the vector field $N^{\mu}(g)$ defined by, for a function $g : [0,T[ \times \mathbb{R}^n_x \times \mathbb{R}_v^n \rightarrow \mathbb{R}$, 
$$N^{\mu}(g):= \int_{v \in \mathbb{R}^n} g v^{\mu} \frac{dv}{v^0}.$$

We have the following energy estimates.

\begin{Pro}\label{energyfsimple}

Let $g$ and $H$ be two smooth functions defined on $[0,T[ \times $ $\mathbb{R}^n_x \times \mathbb{R}_v^n$ such that $T_F(g)=H$ and $k \in \mathbb{Z}$.  Then, for all $t \in [0,T[$,
\begin{eqnarray}
\nonumber \int_{\Sigma_t} \int_{\R^n_v} |g| dv dx+ \sqrt{2} \sup_{u \leq t} \int_{C_u(t)} \int_{\R^n_v} |g| \frac{v^{\underline{L}}}{v^0} dv dC_u(t) & \leq & 2 \int_{\Sigma_0} \int_{\R^n_v} |g| dv dx \\ \nonumber
& &\hspace{-1cm}+ 2\int_0^t \int_{\Sigma_s} \int_{\R^n_v} |H|  \frac{dv}{v^0} dx ds.
\end{eqnarray}
\end{Pro}

\begin{proof} First, let us compute the (euclidian) divergence of $N^{\mu}(|g|)$. Start by noticing that, in $W^{1,1}$,
$$T_m(|g|)=v^{\mu}\partial_{\mu} |g|=\frac{g}{|g|}H-F(v,\nabla_v |g|).$$
By integrations by parts and using $F_{jj}=0$ as well as $v^i v^j F_{ij}=0$ (recall that $F$ is a $2$-form), we have
$$\int_{v } F(v,\nabla_v |g|) \frac{dv}{v^0} =\int_{v } \frac{v^{\mu}}{v^0}{F_{\mu}}^j \partial_{v^j} |g| dv= \int_{v } \frac{v^i v^j}{(v^0)^3}F_{ij} |g| dv=0.$$
Consequently,
\begin{equation}\label{eq:divN}
\partial_{\mu} N^{\mu}(|g|)=\int_{v \in \mathbb{R}^n} \left( \frac{g}{|g|}H-F ( v,\nabla_v |g|) \right) \frac{dv}{v^0}=\int_{v \in \mathbb{R}^n} \frac{g}{|g|}H \frac{dv}{v^0}.
\end{equation}
We now apply the divergence theorem to $N^{\mu}(|g|)$ in several region. Applied to $[0,t] \times \mathbb{R}^n$, it gives
$$ \int_{\Sigma_t} \int_{v } |g| dvdx \hspace{1mm} \leq  \hspace{1mm} \int_{\Sigma_0} \int_{v } |g| dvdx+\int_0^t \int_{\Sigma_s} \int_{v } |H|\frac{dv}{v^0} dx ds .$$
Applied to $V_u(t)$ and using that $\frac{1}{\sqrt{2}}(\partial_t-\partial_r)$ is the outward pointing unit normal field to $C_u(t)$, it gives
$$ \sqrt{2} \int_{C_{u}(t)} \int_{v \in \mathbb{R}^n} |g| \frac{v^{\underline{L}}}{v^0} dv d C_{u}(t)  \hspace{1mm} \leq  \hspace{1mm} \int_{\Sigma_0} \int_{v } |g| dvdx +\int_0^t \int_{\Sigma_s} \int_{v } |H| \frac{dv}{v^0} dx ds . $$
The estimate then ensues from the combination of the two inequalities. 
\end{proof}

This estimate invites us to consider the following energies.

\begin{Def}\label{norm1}

For $N \in \mathbb{N}$ and $ k \in \mathbb{Z}$, we define, for $g$ a sufficiently regular function, 

$$\mathbb{E}^k_N[g](t) = \sum_{\begin{subarray}{l}\widehat{Z}^{\beta} \in \widehat{\mathbb{P}}_0^{|\beta|} \\ \hspace{1mm} |\beta| \leq N \end{subarray}}  \|(v^0)^k \widehat{Z}^{\beta} g\|_{L^1_{x,v}}(t)+\sup_{u \in \mathbb{R}} \int_{C_u(t)} \int_{ \mathbb{R}^n} |\widehat{Z}^{\beta} g| (v^0)^{k}\frac{v^{\underline{L}}}{v^0} dv dC_u(t) .$$

We also need the following norms. For $q \in \mathbb{N}$ and $m \in \{0,1 \}$,

\begin{flalign*}
& \hspace{0mm} \mathbb{E}^k_{N,q,m}[g](t) = \sum_{\begin{subarray}{l} \widehat{Z}^{\beta} \in \widehat{\mathbb{P}}_0^{|\beta|} \\ \hspace{1mm} |\beta| \leq N \end{subarray}} \sum_{\begin{subarray}{l} z^{\gamma} \in \mathbf{k}_m^{|\gamma|} \\ \hspace{1mm} |\gamma| \leq q \end{subarray}} \| (v^0)^kz^{\gamma} \widehat{Z}^{\beta} g\|_{L^1_{x,v}}(t) & 
\end{flalign*}
 $$ \hspace{2cm} + \sum_{\begin{subarray}{l} \widehat{Z}^{\beta} \in \widehat{\mathbb{P}}_0^{|\beta|} \\ \hspace{1mm} |\beta| \leq N \end{subarray}} \sum_{\begin{subarray}{l} z^{\gamma} \in \mathbf{k}_m^{|\gamma|} \\ \hspace{1mm} |\gamma| \leq q \end{subarray}} \sup_{u \in \mathbb{R}} \int_{C_u(t)} \int_{v \in \mathbb{R}^n} |z^{\gamma} \widehat{Z}^{\beta} g| \frac{v^{\underline{L}}}{v^0} (v^0)^k dv dC_u(t).$$

When $k=0$, we drop the dependance in $k$ of the energy norm. For instance, $\mathbb{E}^0_N[g]$ is denoted by $\mathbb{E}_N[g]$.

\end{Def}

The following energy estimates hold.

\begin{Pro}\label{energyf}
Let $g$ and $H$ be such that $T_F(g)=H$. Then, assuming that $g$ and $H$ are sufficiently regular, we have for all $N \in \mathbb{N}$ and for all $ t \in [0,T[$, 

\begin{flalign*}
& \hspace{0mm} \mathbb{E}_{N}[g](t)-2\mathbb{E}_{N}[g](0)  \lesssim  \sum_{|\beta| \leq N} \int_0^t \left\|\frac{1}{v^0} \widehat{Z}^{\beta} H \right\|_{L^1_{x,v}}(s)ds &
\end{flalign*}
 $$\hspace{3.3cm}+\sum_{\begin{subarray}{l}|\gamma|+|\delta| \leq N \\ \hspace{1mm} |\delta| \leq N-1 \end{subarray}} \int_0^t \left\| \mathcal{L}_{Z^{\gamma}}(F)\left(\frac{v}{v^0},\nabla_v \widehat{Z}^{\delta}(g) \right) \right\|_{L^1_{x,v}}(s) ds$$
and 
\begin{flalign*}
& \hspace{0mm} \mathbb{E}^2_{N}[g](t)-2\mathbb{E}^2_{N}[g](0) \lesssim  \sum_{|\beta| \leq N} \int_0^t \|v^0 \widehat{Z}^{\beta} H \|_{L^1_{x,v}}(s)+ \|v^iF_{i0} \widehat{Z}^{\beta} g \|_{L^1_{x,v}} (s) ds &
\end{flalign*}
 $$ \hspace{3.8cm} + \sum_{\begin{subarray}{l}|\gamma|+|\delta| \leq N \\ \hspace{1mm} |\delta| \leq N-1 \end{subarray}} \int_0^t \| \mathcal{L}_{Z^{\gamma}}(F)(v,\nabla_v \widehat{Z}^{\delta}(g))v^0 \|_{L^1_{x,v}}(s) ds,$$
with $\widehat{Z}^{\delta} \in \widehat{\mathbb{P}}_0^{|\delta|}$, $\widehat{Z}^{\beta}$ $\in \widehat{\mathbb{P}}_0^{|\beta|}$ and $Z^{\gamma} \in \mathbb{K}^{|\gamma|}$.
\end{Pro}

\begin{proof}

The first estimate follows from Corollary \ref{Commuf}, Remark \ref{remcommuf} and Proposition \ref{energyfsimple}, applied to $\widehat{Z}^{\beta} g$ for $|\beta| \leq N$. For the second one, apply the same results to $(v^0)^2\widehat{Z}^{\beta} g$ and note that
$$ T_F \left( (v^0)^2 \right) = F \left( v, \nabla_v (v^0)^2 \right) = 2v^{\mu} v^i F_{\mu i}=-2v^{\mu} v^0 F_{\mu 0} = -2 v^i v^0 F_{i0}.$$
\end{proof}

\begin{Rq}
Assuming enough decay on the data, similar inequalities holds for $\mathbb{E}^k_{N}[g]$.
\end{Rq}
We also have an energy estimates which implies the weights transported by the flow. 

\begin{Pro}\label{energypoids}

Let $g$ and $H$ be two sufficiently regular functions such that $T_F(g)=H$. For all $N \in \mathbb{N}$, $m \in \{0,1 \}$ and $ t \in [0,T[$, we have

\begin{flalign*}
& \hspace{0mm} \mathbb{E}_{N,1,m}[g](t)-2\mathbb{E}_{N,1,m}[g](0) \lesssim \sum_{z \in \mathbf{k}_m} \sum_{|\beta| \leq N} \int_0^t \left\|\frac{z}{v^0} \widehat{Z}^{\beta} H \right\|_{L^1_{x,v}}(s)ds  &
\end{flalign*}
 $$\hspace{2.4cm} +\sum_{z \in \mathbf{k}_m} \sum_{|\beta| \leq N} \int_0^t \left\| F\left(\frac{v}{v^0}, \nabla_v z \right) \widehat{Z}^{\beta} g \right\|_{L^1_{x,v}}(s)ds $$
  $$ \hspace{3.5cm} +\sum_{z \in \mathbf{k}_m} \sum_{\begin{subarray}{l}|\gamma|+|\delta| \leq N \\ \hspace{1mm} |\delta| \leq N-1 \end{subarray}} \int_0^t  \left\| z \mathcal{L}_{Z^{\gamma}}(F)\left(\frac{v}{v^0},\nabla_v \widehat{Z}^{\delta} g \right) \right\|_{L^1_{x,v}}(s)ds $$
and
\begin{flalign*} 
& \hspace{0mm} \mathbb{E}^2_{N,1,m}[g](t) - 2\mathbb{E}^2_{N,1,m}[g](0) \lesssim  \sum_{z \in \mathbf{k}_m}\sum_{|\beta| \leq N} \int_0^t \left\|v^0z\widehat{Z}^{\beta} H \right\|_{L^1_{x,v}}(s)ds &
\end{flalign*}
 $$ \hspace{1.4cm} + \sum_{z \in \mathbf{k}_m} \sum_{|\beta| \leq N} \int_0^t \|zv^iF_{i0} \widehat{Z}^{\beta} g \|_{L^1_{x,v}}(s)+ \|v^0 F(v, \nabla_v z) \widehat{Z}^{\beta} g \|_{L^1_{x,v}}(s) ds $$ 
 $$ \hspace{-0.9cm}+ \sum_{z \in \mathbf{k}_m} \sum_{\begin{subarray}{l} |\gamma|+|\delta| \leq N \\ \hspace{1mm} |\delta| \leq N-1 \end{subarray}}  \int_0^t  \| v^0z \mathcal{L}_{Z^{\gamma}}(F)(v,\nabla_v \widehat{Z}^{\delta} g) \|_{L^1_{x,v}}(s)ds ,$$
with $\widehat{Z}^{\delta} \in \widehat{\mathbb{P}}_0^{|\delta|}$, $\widehat{Z}^{\beta}$ $\in \widehat{\mathbb{P}}_0^{|\beta|}$ and $Z^{\gamma} \in \mathbb{K}^{|\gamma|}$.
\end{Pro}

\begin{proof}
Note that, for $z \in \mathbf{k}_m$ and according to equations \eqref{poids1} and \eqref{poids2},
$$ T_F\left( z \widehat{Z}^{\beta} g  \right)= zT_F\left(  \widehat{Z}^{\beta} g  \right)+ T_F(z) \widehat{Z}^{\beta} g =zT_F\left(  \widehat{Z}^{\beta} g  \right)+ F(v, \nabla_v z) \widehat{Z}^{\beta} g.$$
The remaining of the proof is then similar to the one of Proposition \ref{energyf}.
\end{proof}

\subsection{Energy estimates for the wave equation}

Recall that a potential $A$ in the Lorenz gauge satisfies the wave equation \ref{eqLorenz}. In order to bound its $L^2$ norm, we recall here a classical energy estimates for the wave equation  using the vector field $\overline{K}_0$. We mostly follow \cite{Sogge}, Chapter $II$.

During this subsection, we consider $u :[0,T[ \times \R^n \rightarrow \R$ a smooth function such that
$$\|u\|_{L^2(\R^n)}(0)+\sum_{Z \in \mathbb{K}} \|Zu \|_{L^2(\R^n)}(0) < + \infty.$$
We also introduce its energy momentum tensor
$$T_{\mu \nu}[u]= \partial_{\mu} u \partial_{\nu}u-\frac{1}{2}\eta_{\mu \nu} \eta^{\sigma \rho} \partial_{\sigma} u \partial_{\rho} u.$$

Since $\overline{K}_0$ is merely a conformal Killing vector field and as $T_{\mu \nu}[u]$ is not traceless, $T_{\mu \nu}[u] \overline{K}_0^{\nu}$ is not divergence free when $\square u =0$. In fact
$$\nabla^{\mu} (T_{\mu \nu}[u] \overline{K}_0^{\nu}) = \square u \overline{K}_0 u+\frac{1}{2}T_{\mu \nu}[u] \pi^{\mu \nu},$$
with
$$\pi^{\mu \nu}= \partial^{\mu} \overline{K}_0^{\nu}+\partial^{\nu} \overline{K}^{\mu}_0.$$
Since $\overline{K}_0$ is a conformal vector field of conformal factor $4t$, $\pi^{\mu \nu}=4 t\eta^{\mu \nu}$. So

$$\nabla^{\mu} (T_{\mu \nu}[u] \overline{K}_0^{\nu}) = \square u \overline{K}_0 u+(1-n)t \partial^{\mu} u \partial_{\mu} u.$$
The equality
$$t\partial^{\mu} u \partial_{\mu} u=\partial_{\mu} (tu \partial^{\mu} u)-\partial_{\mu}(t)u\partial^{\mu} u-tu \square u=\partial^{\mu} \left(tu\partial_{\mu} u- \frac{1}{2}u^2 \partial_{\mu} t\right)-tu \square u,$$
suggests us to introduce the $1$-form
$$P_{\mu} =T_{\mu \nu}[u] \overline{K}^{\nu}_0+(n-1)tu\partial_{\mu} u- \frac{n-1}{2}u^2 \partial_{\mu} t.$$

Applying the divergence theorem on $[0,t] \times \R^n$ to $T_{\mu0}[u]$ and $P_{\mu}$, we obtain

\begin{Pro}\label{prowave1}

$\forall$ $t \in [0,T[$,
$$\sum_{\mu=0}^n \|\partial_{\mu} u \|_{L^2(\Sigma_t)} \leq \sum_{\mu=0}^n \|\partial_{\mu} u \|_{L^2(\Sigma_0)}+\int_0^t \int_{\Sigma_s} |\square u | d \Sigma_s ds$$
and
$$\int_{\Sigma_t} P_0 d \Sigma_t \leq \int_{\Sigma_0} P_0 d \Sigma_0+\int_0^t \int_{\Sigma_s} |\square u| |\overline{K}_0 u+(n-1)tu|d \Sigma_s ds.$$

\end{Pro}

The first thing to verify is that $\int_{\Sigma_t} P_0 d \Sigma_t$ can be compared with the $L^2$ norm of $u$ (and of its derivatives).

\begin{Pro}\label{prowave2}
We suppose that $n \geq 3$. We have, for all $t \in [0,T[$,
$$ \sum_{|\beta| \leq 1} \| Z^{\beta} u \|^2_{L^2(\R^n)}(t) \lesssim \int_{\Sigma_t } P_0 d \Sigma_t \lesssim \sum_{|\beta| \leq 1} \| Z^{\beta} u \|^2_{L^2(\R^n)}(t).$$

\end{Pro}

\begin{proof}

Let us first remark that
$$P_0=\frac{1}{2}(1+|x|^2+t^2)|\nabla_{t,x} u|^2+2tx^i \partial_i u \partial_t u+(n-1)tu\partial_{t} u- \frac{n-1}{2}u^2 .$$
Moreover,
$$(1+|x|^2+t^2)|\nabla_{t,x} u|^2+4tx^i \partial_i u \partial_t u=|\nabla_{t,x} u|^2+|Su|^2+\sum_{0 \leq \mu < \nu \leq n} |\Omega_{\mu \nu} u |^2$$
together with
\begin{equation}\label{eq2.5}
\int_{\R^n} 2tu\partial_{t} u dx = \int_{\R^n} 2uSu-x^i\partial_i( u^2) dx= \int_{\R^n} 2uSu+nu^2dx
\end{equation}
gives
\begin{equation}
\int_{\Sigma_t} P_0 d \Sigma_t = \frac{1}{2} \int_{\Sigma_t} |\nabla_{t,x} u|^2+|Su+(n-1)u|^2+\sum_{0 \leq \mu < \nu \leq n} |\Omega_{\mu \nu} u |^2 d \Sigma_t.
\end{equation}

This proves the second inequality and reduces the first one to 
$$ \|u\|^2_{L^2(\R^n)}(t)+\|Su \|^2_{L^2(\R^n)}(t) \lesssim \int_{\Sigma_t } P_0 d \Sigma_t.$$

In order to transform $\int_{\R^n} 2tu\partial_{t} u d x$ in an alternative expression, we remark that
$$2u \partial_t u = 2u\frac{1}{r} \Omega_{0r} u-\frac{t}{r^2}x^i \partial_i (u^2), \hspace{3mm} \text{with} \hspace{3mm} \Omega_{0r}=\frac{x^i}{r}\Omega_{0i}.$$

So, by integration by parts,
$$\int_{\R^n} 2tu\partial_{t} u dx=\int_{\R^n} \left( 2\frac{t}{r}u \Omega_{0r} u+(n-2)\frac{t^2}{r^2}u^2 \right)dx.$$
Combined with equation \eqref{eq2.5}, we get
\begin{flalign*}
& \hspace{0mm} \int_{\R^n} \left( 2(n-1)tu\partial_{t} u- (n-1)u^2 \right)dx =\frac{2n-3}{2}\int_{\R^n} \left( 2uSu+nu^2 \right)dx &
\end{flalign*}
 $$\hspace{3.3cm}+\frac{1}{2}\int_{\R^n} \left( 2\frac{t}{r}u \Omega_{0r} u+(n-2)\frac{t^2}{r^2}u^2 \right) dx-(n-1)\int_{\R^n} u^2 dx.$$

It then comes that
\begin{equation}\label{equationtruc}
\hspace{-2.7cm} 2\int_{\Sigma_t} P_0 d x = \int_{\Sigma_t} |Su|^2+2\frac{2n-3}{2}uSu+\frac{2n^2-5n+2}{2}u^2 d x
 \end{equation}
 $$+\int_{\Sigma_t} |\nabla_{t,x} u|^2+\sum_{\mu < \nu } |\Omega_{\mu \nu} u |^2-|\Omega_{0r}u|^2 +|\Omega_{0r}u|^2+\frac{t}{r}u \Omega_{0r} u+(n-2)\frac{t^2}{2r^2}u^2 dx.$$
The second integral on the right hand side of \eqref{equationtruc} is nonnegative since
$$|\Omega_{0r}u|^2 =\left|\frac{x^i}{r} \Omega_{0i}u\right|^2 \leq \sum_{i=1}^n |\Omega_{0i}u|^2$$
and
$$|\Omega_{0r}u|^2+\frac{t}{r}u \Omega_{0r} u+(n-2)\frac{t^2}{2r^2}u^2=\left(\Omega_{0r}u+\frac{t}{2r}u\right)^2+(2n-5)\frac{t^2}{4r^2}u^2.$$
Consequently,
$$ \|u\|^2_{L^2(\R^n)}(t)+\|Su \|^2_{L^2(\R^n)}(t) \lesssim \int_{\Sigma_t } P_0 d \Sigma_t$$
comes from
$$|Su|^2+2\frac{2n-3}{2}uSu+\frac{2n^2-5n+2}{2}u^2=\left( Su+ \frac{2n-3}{2}u \right)^2+\frac{2n-5}{4}u^2$$
and from
\begin{flalign*}
& \hspace{0mm} |Su|^2+2\frac{2n-3}{2}uSu+\frac{2n^2-5n+2}{2}u^2 = &
\end{flalign*}
 $$ \hspace{1.5cm} \left( \frac{2n-3}{\sqrt{4n^2-10n+4}}Su+ \left(n^2-\frac{5}{2}n+1\right)^{\frac{1}{2}}u \right)^2+\frac{2n-5}{4n^2-10n+4}|Su|^2.$$

\end{proof}

\begin{Rq}
We also proved that $$ \int_{\Sigma_t}\frac{t^2}{r^2}u^2 d \Sigma_t \lesssim \int_{\Sigma_t} P_0 d \Sigma_t.$$
\end{Rq}

Finally, we obtain the expected estimate.

\begin{Pro}\label{prowave3}
We have, for all $t \in [0,T[$,
$$ \sum_{|\beta| \leq 1} \| Z^{\beta} u \|^2_{L^2(\R^n)}(t) \lesssim \sum_{|\beta| \leq 1} \| Z^{\beta} u \|^2_{L^2(\R^n)}(0)+ \sum_{|\beta| \leq 1} \int_0^t \int_{\Sigma_s} |Z^{\beta} u| |\tau_+ \square u | dx ds,$$
with $Z^{\beta} \in \mathbb{K}$ if $|\beta|=1$, leading to, for all $t \in [0,T[$,

$$\sum_{|\beta| \leq 1} \| Z^{\beta} u \|_{L^2(\R^n)}(t) \lesssim \sum_{|\beta| \leq 1} \| Z^{\beta} u \|_{L^2(\R^n)}(0)+  \int_0^t  \left\|\tau_+ \square u \right\|_{L^2(\R^n)}  ds.$$

\end{Pro}

\begin{proof}

We have, according to Propositions \ref{prowave1} and \ref{prowave2},
\begin{flalign*}
& \hspace{0mm} \sum_{|\beta| \leq 1} \| Z^{\beta} u  \|^2_{L^2(\R^n)}(t) \lesssim \sum_{|\beta| \leq 1} \| Z^{\beta} u \|^2_{L^2(\R^n)}(0) &
\end{flalign*}
 $$\hspace{1.3cm} + \int_0^t \int_{\Sigma_s} |\square u| |\overline{K}_0 u+(n-1)tu|dx ds.$$
The result then follows from Remark \ref{extradecay}, which gives us
$$|\overline{K}_0 u| \lesssim  \tau_+^2 |Lu|+ \tau_-^2 |\underline{L} u| \lesssim \tau_+ \sum_{Z \in \mathbb{K}} |Zu|.$$

\end{proof}

We now apply this to the electromagnetic potential in the Lorenz gauge. Since we will need to estimate $\|S\left(\mathcal{L}_{Z^{\beta}}(A) \right) \|_{L^2(\R^n)}$ in order to bound the energy of the electromagnetic field $F$ (see Proposition \ref{MoraF} below), we consider the following norms.

\begin{Def}

Let $A$ be a sufficiently regular $1$-form defined on $[0,T[ \times \R^n$. We define, for $N \in \mathbb{N}$ and all $t \in [0,T[$,
$$\widetilde{\mathcal{E}}_N[A](t)= \sum_{\mu=0}^n \sum_{|\beta| \leq 1} \sum_{|\gamma| \leq N} \| Z^{\beta}( \mathcal{L}_{Z^{\gamma}} (A)_{\mu}) \|^2_{L^2(\R^n)}(t).$$

\end{Def}

\begin{Rq}

Note that
$$\sum_{\mu=0}^n \sum_{|\beta| \leq N+1} \|Z^{\beta} A_{\mu} \|_{L^2(\R^n)}^2 \lesssim \widetilde{\mathcal{E}}_N[A] \lesssim \sum_{\mu=0}^n \sum_{|\beta| \leq N+1} \|Z^{\beta} A_{\mu} \|_{L^2(\R^n)}^2.$$
We work with $\widetilde{\mathcal{E}}_N[A]$ as we will apply Proposition \ref{prowave3} to $\mathcal{L}_{Z^{\beta}}(A)_{\mu}$.
\end{Rq}

Using Proposition \ref{prowave3}, we get the following result.

\begin{Pro}\label{energypotential}

Let $N \in \mathbb{N}$ and $A_{\mu}$ be a sufficiently regular $1$-form, defined on $[0,T[ \times \R^n$, such that $\widetilde{\mathcal{E}}_N[A](0)< + \infty.$ Then, $\forall \hspace{1mm} t \in [0,T[$,
$$\sqrt{\widetilde{\mathcal{E}}_N[A]}(t) \lesssim \sqrt{\widetilde{\mathcal{E}}_N[A]}(0)+\sum_{\mu=0}^n \sum_{|\gamma| \leq N}\int_0^t  \left\|\tau_+ \square \mathcal{L}_{Z^{\gamma}} (A)_{\mu} \right\|_{L^2(\Sigma_s)}  ds.$$

\end{Pro}

\subsection{Energy estimates for the Maxwell equations}

We prove three conservation laws for the Maxwell equations, using each time a different multiplier ($\partial_t$, $\overline{K}_0$ or $S$). In the study of the massive case, we will mostly use the one associated to the Morawetz vector field.

For the remaining of this section, we consider a $2$-form $G$ and a current $J$, sufficiently regular and defined on $[0,T[$, such that

\[
\left \{
\begin{array}{c @{=} c}
    
    \nabla^{\mu} G_{\mu \nu} & J_{\nu} \\
    \nabla^{\mu} {}^* \! G_{\mu \lambda_1 ... \lambda_{n-2}} & 0. \\
\end{array}
\right.
\]

The following lemmas hold.

\begin{Lem}\label{divelec}

We have, for all $0 \leq \nu \leq n$,
$$\nabla^{\mu} T[G]_{\mu \nu} = G_{\nu \rho} J^{\rho}.$$

\end{Lem}

\begin{proof}

According to Proposition \ref{equivalsyst},

\begin{eqnarray}
\nonumber G_{\mu \rho} \nabla^{\mu} {G_{\nu}}^{\rho}& = & G^{\mu \rho} \nabla_{\mu} G_{\nu \rho} \\
\nonumber & = & \frac{1}{2} G^{\mu \rho} (\nabla_{\mu} G_{\nu \rho}-\nabla_{\rho} G_{\nu \mu}) \\
\nonumber & = & \frac{1}{2} G^{\mu \rho} \nabla_{\nu} G_{\mu \rho} \\
\nonumber & = & \frac{1}{4} \nabla_{\nu} (G^{\mu \rho} G_{\mu \rho}).
\end{eqnarray}

So,
$$\nabla^{\mu} T[G]_{\mu \nu} = \nabla^{\mu} (G_{\mu \rho}){G_{\nu}}^{\rho}+\frac{1}{4} \nabla_{\nu} (G^{\mu \rho} G_{\mu \rho})-\frac{1}{4}\eta_{\mu \nu} \nabla^{\mu} (G^{\sigma \rho} G_{\sigma \rho})=G_{\nu \rho} J^{\rho}.$$

\end{proof}

\begin{Lem}\label{momentumnull}

We have, denoting by $(\alpha, \underline{\alpha}, \rho, \sigma)$ the null decomposition of $G$,
$$T[G]_{L L}=|\alpha|^2, \hspace{3mm} T[G]_{\underline{L} L}=|\underline{\alpha}|^2 \hspace{3mm} \text{and} \hspace{3mm} T[G]_{L \underline{L}}=|\rho|^2+|\sigma|^2.$$

\end{Lem}

\subsubsection{Using $\partial_t$ as a multiplier}

As we use here the multiplier $\partial_t$, we work with $T[G]_{\mu 0}$. Applying the divergence theorem to $T[G]_{\mu 0}$ on $[0,t] \times \R^n$ and $V_u(t)$, we obtain the following result.

\begin{Pro}\label{energyt0}

For all $t \in [0,T[$,

\begin{flalign*}
& \hspace{0cm} \int_{\Sigma_t} |\alpha|^2+|\underline{\alpha}|^2+2|\rho|^2+2|\sigma|^2 dx =  & \\
& \hspace{3.1cm}  \int_{\Sigma_0} |\alpha|^2+|\underline{\alpha}|^2+2|\rho|^2+2|\sigma|^2 dx+4\int_0^t \int_{\Sigma_s} G_{0 \mu} J^{\mu}  dx ds &
\end{flalign*}
and
\begin{flalign*}
& \hspace{0cm} \sqrt{2} \sup_{u \leq t} \int_{C_u(t)} |\alpha|^2+|\rho|^2+|\sigma|^2 dC_u(t) \leq & \\
& \hspace{2.8cm} \int_{\Sigma_0} |\alpha|^2+|\underline{\alpha}|^2+2|\rho|^2+2|\sigma|^2 dx+4\int_0^t \int_{\Sigma_s} |G_{0 \mu} J^{\mu}|  dx ds. &
\end{flalign*}

\end{Pro}

This explains the introduction of the following norms.

\begin{Def}
Let $N \in \mathbb{N}$. We define, for $t \in [0,T[$,
\begin{flalign*}
& \hspace{1.5cm} \mathcal{E}^0[G](t)= \int_{\Sigma_t}\left( |\alpha|^2+|\underline{\alpha}|^2+2|\rho|^2+2|\sigma|^2 \right)dx &\\
& \hspace{3.2cm}+\sup_{u \leq t} \int_{C_u(t)} \left( |\alpha|^2+|\rho|^2+|\sigma|^2 \right) dC_u(t) &
\end{flalign*}
 and
$$\hspace{-4.3cm} \mathcal{E}_N^0[G](t)= \sum_{|\beta| \leq N} \mathcal{E}^0_N[\mathcal{L}_{Z^{\beta}}(G)](t),$$
with $Z^{\beta} \in \mathbb{K}^{|\beta|}$.
\end{Def}

Using the previous energy identities and commutation formula of Proposition \ref{Commuelec}, we obtain

\begin{Pro}\label{energyt}
For all $N\in \mathbb{N}$ and all $t \in [0,T[$, we have
$$\mathcal{E}^0_N[F](t)-2\mathcal{E}_N^0[F](0) \lesssim \sum_{|\beta|,|\gamma| \leq N} \int_0^t \int_{\Sigma_s} |e^k \mathcal{L}_{Z^{\beta}}(F)_{0 \mu} J(\widehat{Z}^{\gamma} f_k)^{\mu})|dxds,$$
with $Z^{\beta} \in \mathbb{K}^{|\beta|}$ and $\widehat{Z}^{\gamma} \in \K^{|\gamma|}$.
\end{Pro}

\subsubsection{Using $\overline{K}_0$ as a multiplier}

As $T[G]$ is not traceless in dimension $n \geq 4$, $\nabla^{\mu}(T[G]_{\mu \nu}\overline{K}_0^{\nu})$ does not necessarily vanishes when $G$ solves the free Maxwell equations. We then consider, in the spirit of what is done for the wave equation, for $A$ a sufficiently regular potential of $G$ in the Lorenz gauge, the current
$$P_{\mu}=T[G]_{\mu \nu} \overline{K}^{\nu}+(n-3)\Big(tA_{\beta}\partial_{\mu}A^{\beta}-\frac{1}{2}\partial_{\mu}(t)A_{\beta}A^{\beta}-tA_{\beta}\partial^{\beta} A_{\mu}+\partial^{\beta}(t)A_{\beta}A_{\mu}\Big).$$

In order to establish an energy estimate for the electromagnetic field, we compute the divergence of $P_{\mu}$.

\begin{Lem}\label{divFmod}

We have

$$\nabla^{\mu} P_{\mu} = G_{\mu \nu} \overline{K}_0^{\nu} J^{\mu}+(n-3)tA_{\beta} \square A^{\beta}.$$

\end{Lem}

\begin{proof}
We have

$$\nabla^{\mu} ( T[G]_{\mu \nu} \overline{K}^{\nu}_0)= \nabla^{\mu} (T[G]_{\mu \nu}) \overline{K}^{\nu}_0+T[G]_{\mu \nu} \nabla^{\mu} \overline{K}^{\nu}_0.$$

Since $T[G]$ is symmetric, 

$$T[G]_{\mu \nu} \nabla^{\mu} \overline{K}^{\nu}_0 = \frac{1}{2} T[G]_{\mu \nu} \pi^{\mu \nu},$$
with $\pi^{\mu \nu}=\nabla^{\mu} \overline{K}_0^{\nu}+\nabla^{\nu} \overline{K}_0^{\mu}$. As $\overline{K}_0$ is a conformal vector field (of conformal factor $4t$), we have
$$\pi_{\mu \nu } =4t \eta_{\mu \nu} .$$
Thus, 
$$T[G]_{\mu \nu} \nabla^{\mu} \overline{K}^{\nu}_0=2t {T(G)_{\mu}}^{\mu}=\frac{3-n}{2}tG_{\sigma \rho} G^{\sigma \rho}.$$

Now, according to Lemma \ref{divelec}, we obtain that
$$\nabla^{\mu} ( T[G]_{\mu \nu} \overline{K}^{\nu}_0)= G_{\nu \rho}\overline{K}^{\nu}_0 J^{\rho}+\frac{3-n}{2}tG_{\sigma \rho} G^{\sigma \rho}.$$
We now compute the divergence of
$$(n-3) \left( tA_{\beta}\partial_{\mu}A^{\beta}-\frac{1}{2}\partial_{\mu}(t)A_{\beta}A^{\beta}-tA_{\beta}\partial^{\beta} A_{\mu}+\partial^{\beta}(t)A_{\beta}A_{\mu} \right). $$
First,
$$\nabla^{\mu} \left( tA_{\beta}\partial_{\mu}A^{\beta} \right) = -A_{\beta}\partial_{0}A^{\beta}+t\partial^{\mu} A_{\beta} \partial_{\mu} A^{\beta}+tA_{\beta} \square A^{\beta}.$$
Secondly,
$$\nabla^{\mu} \left( \frac{1}{2}\partial_{\mu}(t)A_{\beta}A^{\beta} \right) =   A_{\beta} \partial^{0}  A^{\beta}.$$
We also have, using in particular that in Lorenz gauge $\partial^{\mu} A_{\mu}=0$,
\begin{eqnarray}
\nonumber \nabla^{\mu} \left( tA_{\beta}\partial^{\beta} A_{\mu} \right) & = &- A_{\beta}\partial^{\beta} A_{0}+t\partial^{\mu} (A_{\beta})\partial^{\beta} A_{\mu}+tA_{\beta} \partial^{\beta} \partial^{\mu} A_{\mu} \\ \nonumber
& = & - A_{\beta}\partial^{\beta} A_{0}+t\partial^{\mu} (A_{\beta})\partial^{\beta} A_{\mu} .
\end{eqnarray}
Finally
\begin{eqnarray}
\nonumber \nabla^{\mu} \left( \partial^{\beta}(t)A_{\beta}A_{\mu} \right) & = &-\partial^{\mu}(A_{0})A_{\mu} -A_{0} \partial^{\mu} A_{\mu} \\ \nonumber 
& = & -\partial^{\mu}(A_{0})A_{\mu} .
\end{eqnarray}
Hence,
\begin{flalign*}
& \hspace{0mm} (n-3) \nabla^{\mu}\left( tA_{\beta}\partial_{\mu}A^{\beta}-\frac{1}{2}\partial_{\mu}(t)A_{\beta}A^{\beta}-tA_{\beta}\partial^{\beta} A_{\mu}-\partial^{\beta}(t)A_{\beta}A_{\mu} \right) =& 
\end{flalign*}
 $$\hspace{4.4cm} (n-3)tA_{\beta} \square A^{\beta}+(n-3)t(\partial_{\mu}A_{\beta}\partial^{\mu}A^{\beta}-\partial_{\mu}A_{\beta} \partial^{\beta}A^{\mu}). $$
And, since $G_{\mu \nu}=\partial_{\mu} A_{\nu}-\partial_{\nu} A_{\mu}$,
$$\frac{1}{2}G_{\mu \beta}G^{\mu \beta}=\partial_{\mu}A_{\beta}\partial^{\mu}A^{\beta}-\partial_{\mu}A_{\beta} \partial^{\beta}A^{\mu},$$
which gives us the result.

\end{proof}

We are now ready to prove the following energy estimate.

\begin{Pro}\label{proelec}

For all $t \in [0,T[$,
$$\int_{\Sigma_t} \tau_+^2 |\alpha|^2+\tau_-^2| \underline{\alpha}|^2+(\tau_+^2+\tau_-^2)(|\rho|^2+|\sigma|^2) d \Sigma_t+(n-3)^2 \sum_{\mu=0}^n \| A_{\mu} \|^2_{L^2(\Sigma_t)}  \leq  $$ $$ \hspace{2.2cm} \int_{\Sigma_0} (1+r^2)(|\alpha|^2+|\underline{\alpha}|^2+|\rho|^2+|\sigma|^2) d \Sigma_0 +4\int_0^t \int_{\Sigma_s} |\overline{K}_0^{\nu} G_{\nu \mu} J^{\mu} | dx ds$$ $$ \hspace{1.6cm}+(n-3)\sum_{\mu=0}^n  \| S A_{\mu} \|^2_{L^2(\Sigma_t)} +4(n-3)\int_0^t \int_{\Sigma_s} s|A_{\mu} \square A^{\mu}| dx ds.$$
\end{Pro}

\begin{proof}

In order to apply the divergence theorem to $P_{\mu}$ in $[0,t] \times \mathbb{R}^n$, we transform
$$\int_{\mathbb{R}^n} \left( t A_{\beta} \partial_t A^{\beta}-\frac{1}{2}A_{\beta}A^{\beta}-tA_{\beta} \partial^{\beta} A_{0}-A_0^2 \right) dx.$$
On the one hand, let us notice that
$$-\frac{1}{2}A_{\beta}A^{\beta}-A_0^2=-\frac{1}{2} \sum_{\beta=0}^{n} A_{\beta}^2.$$
On the other hand,
\begin{equation}\label{eq:pegement1}-t\int_{\mathbb{R}^n} A_{\beta} \partial^{\beta} A_0 dx=-t\int_{\mathbb{R}^n} A_0 \partial^0 A_0 dx+t\int_{\mathbb{R}^n} \partial^j( A_j) A_0 dx=t\partial_t\int_{\mathbb{R}^n} A_0^2dx,
\end{equation}
since $\partial^{\mu} A_{\mu}=0$ in the Lorenz gauge.
As 
\begin{equation}\label{eq:pegement2}
t\int_{\mathbb{R}^n} A_{\beta} \partial_t A^{\beta} dx = \frac{t}{2} \partial_t \int_{\mathbb{R}^n} A_{\beta} A^{\beta} dx,
\end{equation}
we finally obtain that
$$\int_{\mathbb{R}^n} \left( t A_{\beta} \partial_t A^{\beta}-\frac{1}{2}A_{\beta}A^{\beta}-tA_{\beta} \partial^{\beta} A_{0}-A_0^2 \right) dx=\frac{1}{2} \sum_{\beta=0}^n (t \partial_t-1)\|A_{\beta}\|_{L^2(\mathbb{R}^n)}^2 .$$

The divergence theorem applied to $P_{\mu}$ in $[0,t] \times \mathbb{R}^n$ gives, using Lemma \ref{momentumnull} and \ref{divFmod},

\begin{flalign*}
& \hspace{0mm} \int_{\Sigma_t} \tau_+^2 |\alpha|^2+\tau_-^2| \underline{\alpha}|^2+(\tau_+^2+\tau_-^2)(|\rho|^2+|\sigma|^2) dx  \leq 4\int_0^t \int_{\Sigma_s} |\nabla^{\mu} P_{\mu} | dx ds &
\end{flalign*}
 $$\hspace{0.2cm} + \int_{\Sigma_0} (1+r^2)(|\alpha|^2+|\underline{\alpha}|^2+2|\rho|^2+2|\sigma|^2) dx + 2(n-3)\sum_{\mu=0}^n \left( \int_{\Sigma_t} (1-t\partial_t) A_{\mu}^2 dx \right) .$$
It only remains to use the last lemma and the inequality
$$-t\partial_t \int_{\Sigma_t} A_{\mu}^2 dx \leq  \frac{1-n}{2}\| A_{\mu} \|^2_{L^2(\Sigma_t)}+\frac{1}{2}\| SA_{\mu} \|^2_{L^2(\Sigma_t)}$$
which ensues from \eqref{eq2.5}.

\end{proof}

This estimate justifies the introduction of the following norms.

\begin{Def}\label{norm2}

Let $G$ be a $2$-form defined on $[0,T[$ and $N \in \mathbb{N}$. We define, for all $t \in [0,T[$,

$$ \mathcal{E}[G](t) = \int_{\Sigma_t} \tau_+^2 |\alpha(G)|^2+\tau_-^2|\underline{\alpha}(G)|^2+(\tau_+^2+\tau_-^2)(|\rho (G) |^2+|\sigma (G) |^2) d x$$
and
$$\mathcal{E}_N[G](t)= \sum_{|\beta| \leq N} \mathcal{E}[\mathcal{L}_{Z^{\beta}}G](t),$$
with $Z^{\beta} \in \mathbb{K}^{|\beta|}$.
\end{Def}

We then deduce, using Propositions \ref{Commuelec}, \ref{proelec} and Lemma \ref{derivpotential}, an energy estimate for the electromagnetic field $F$.

\begin{Pro}\label{MoraF}

Let $A$ be a sufficiently regular potential in the Lorenz gauge of $F$. We have, for all $N \in \mathbb{N}$ and all $t \in [0,T[$,

\begin{flalign*}
& \mathcal{E}_N[F](t) - \mathcal{E}_N[F](0)-(n-3)\sum_{|\kappa| \leq N} \sum_{\mu=0}^n  \| S \mathcal{L}_{Z^{\kappa}}(A)_{\mu} \|^2_{L^2(\Sigma_t)} \lesssim &
\end{flalign*}
$$ \hspace{4.6cm}+ \sum_{|\beta|,|\gamma| \leq N} \int_0^t \int_{\Sigma_s} |e^k\overline{K}_0^{\nu} \mathcal{L}_{Z^{\beta}}(F)_{\nu \mu} J^{\mu}(\widehat{Z}^{\gamma}f_k) | dx ds $$
 $$ \hspace{3.4cm}+ \sum_{|\kappa| \leq N} \int_0^t \int_{\Sigma_s} s| \mathcal{L}_{Z^{\kappa}}(A)_{\mu} \square \mathcal{L}_{Z^{\kappa}}(A)^{\mu}| dx ds,$$ 
with $Z^{\beta} \in \mathbb{K}^{|\beta|}$, $Z^{\kappa} \in \mathbb{K}^{|\kappa|}$ and $\widehat{Z}^{\gamma} \in \widehat{\mathbb{P}}_0^{|\gamma|}$.
\end{Pro}

\subsubsection{Using $S$ as a multiplier}

The main difference with the previous case comes from the fact that the scaling is not a timelike vector field. Because of that we are not able to estimate all the null components of the electromagnetic field with this energy estimate. We start by introducing, for $A$ a potential of $G$ satisfying the Lorenz gauge,
$$Q_{\mu}=T(G)_{\mu \nu} S^{\nu}+\frac{n-3}{2}(A_{\beta} \partial_{\mu} A^{\beta}-A_{\beta} \partial^{\beta} A_{\mu} ).$$
As the potential $A$ satisfies the Lorenz gauge and since the conformal factor of the scaling is $2$, we have

\begin{equation}\label{divSelec}
\nabla^{\mu} Q_{\mu} = G_{\mu \nu} S^{\nu} J^{\mu}+\frac{n-3}{2}A_{\beta} \square A^{\beta}.
\end{equation}
We can now state the energy estimate.

\begin{Pro}\label{truc5}

For all $t \in [0,T[$,
$$ \int_{\Sigma_t}(t+r)|\alpha|^2+(t-r)|\underline{\alpha}|^2+2t(|\rho|^2+|\sigma|^2)dx+(n-3)\partial_t\sum_{\beta=0}^n \|A_{\beta}\|_{L^2(\Sigma_t)}^2  = $$ $$ \hspace{1.4cm} \int_{\Sigma_0} r(|\alpha|^2-|\underline{\alpha}|^2)dx+(n-3)\partial_t\sum_{\beta=0}^n \|A_{\beta}\|_{L^2(\Sigma_0)}^2+4 \int_0^t \int_{\Sigma_s} \nabla^{\mu} Q_{\mu} dxds.$$

\end{Pro}

\begin{proof}

Note first that we proved, during the proof of Proposition \ref{proelec} (see Equations \eqref{eq:pegement1} and \eqref{eq:pegement2}), $$ \int_{\R^n} A_{\beta} \partial_0 A^{\beta}-A_{\beta} \partial^{\beta} A_0dx=\frac{\partial_t}{2} \sum_{\beta=0}^n \|A_{\beta} \|_{L^2(\R^n)}^2 .$$
It then remains to apply the divergence theorem to $Q_{\mu}$ on $[0,T] \times \R^n$ (recall that $2S=(t+r)L+(t-r)\underline{L}$).

\end{proof}

Note that $(t-r)|\underline{\alpha}|^2$ is not necessarily non negative, which invites us to transform the equality in the following estimate. 

\begin{Pro}\label{proscal2}
For all $t \in [0,T]$,
\begin{flalign*}
& \int_{\Sigma_t}(1+|t-r|)|\underline{\alpha}|^2 dx \leq \int_{\Sigma_0}(1+r)(|\alpha|^2+|\underline{\alpha}|^2)+2|\rho|^2+2|\sigma|^2 dx &
\end{flalign*}
 $$\hspace{1.2cm}+(n-3)(n+2)\widetilde{\mathcal{E}}_0[A](0)+\frac{2}{1+t}\left(\mathcal{E}[F](t)+\frac{(n-3)(n+2)}{2}\widetilde{\mathcal{E}}_0[A](t)\right)$$ $$\hspace{1.2cm}+4\int_0^t \int_{\Sigma_s} |G_{0 \mu}J^{\mu}|+|S^{\nu} G_{\mu \nu} J^{\mu} | dx ds+2(n-3) \int_0^t \int_{\Sigma_s} |A_{\mu} \square A^{\mu} |dx ds.$$

\end{Pro}

\begin{proof}
Adding the energy identities of Propositions \ref{truc5} and \ref{energyt0}, we can obtain, 

\begin{flalign*}
& \int_{\Sigma_t}(1+|t-r|)|\underline{\alpha}|^2 dx \leq \int_{\Sigma_0}(1+r)(|\alpha|^2+|\underline{\alpha}|^2)+2|\rho|^2+2|\sigma|^2 dx &
\end{flalign*}
 $$\hspace{1.9cm}+\int_{\Sigma_t}(t+r)|\alpha|^2+2t(|\rho|^2+|\sigma|^2)dx+4\int_0^t \int_{\Sigma_s} |G_{0 \mu} J^{\mu}| +|\nabla^{\mu} Q_{\mu} | dx ds$$ 
 $$ \hspace{-0.2cm}+ (n-3)\left|\partial_t\sum_{\beta=0}^n \Bigg(\|A_{\beta}\|_{L^2(\R^n)}^2(0)- \|A_{\beta}\|_{L^2(\R^n)}^2(t)\Bigg)\right|.$$
The result then ensues from the three following inequalities. Using Definition \ref{norm2}, one has 
$$ (1+t)\int_{\Sigma_t}(t+r)|\alpha|^2+2t(|\rho|^2+|\sigma|^2)dx \leq 2\mathcal{E}[F](t).$$
According to \eqref{divSelec}, we have
$$\int_0^t \int_{\Sigma_s} |\nabla^{\mu} Q_{\mu} | dx ds \leq \int_0^t \int_{\Sigma_s} |S^{\nu} G_{\mu \nu} J^{\mu} | +\frac{(n-3)}{2}  |A_{\mu} \square A^{\mu} |dx ds.$$
Finally, Equation \eqref{eq2.5} gives us
$$(1+t)\left|\partial_t \|A_{\mu}\|_{L^2(\Sigma_s)}^2 \right| \leq \|S A_{\mu} \|_{L^2(\Sigma_s)}^2+ \|\partial_t A_{\mu} \|_{L^2(\Sigma_s)}^2+(n+2)\|A_{\mu} \|_{L^2(\Sigma_s)}^2.$$
\end{proof}

Let us introduce the following norms.

\begin{Def}\label{normS}
We define, for $N \in \mathbb{N}$ and $t \in [0,T[$,
$$\mathcal{E}_N^{S}[F](t)=\sum_{\begin{subarray}{l} Z^{\beta} \in \mathbb{K}^{|\beta|} \\ \hspace{1mm}|\beta| \leq N \end{subarray}} \int_{\Sigma_t} \tau_- |\underline{\alpha}(\mathcal{L}_{Z^{\beta}}(F)) |^2 dx.$$

\end{Def}

Commuting the equation satisfied by the electromagnetic field $F$ and using the previous energy estimate, we get the following proposition (see the commutation formulas of Proposition \ref{Commuelec} and Lemma \ref{derivpotential}).

\begin{Pro}\label{scalF}

Let $A$ a sufficiently regular potential of the the electromagnetic field $F$ in the Lorenz gauge. Then, for $N \in \mathbb{N}$, we have, for all $t \in [0,T[$,
\begin{flalign*}
& \mathcal{E}_N^{S}[F](t) -\mathcal{E}_N[F](0) \lesssim \widetilde{\mathcal{E}}_N[A](0)+\frac{1}{1+t}\Big(\widetilde{\mathcal{E}}_N[A](t) +\mathcal{E}_N[F](t)\Big) &
\end{flalign*}
$$\hspace{0.7cm}+\sum_{|\beta|,|\gamma| \leq N} |e^k| \int_0^t \int_{\Sigma_s} | \mathcal{L}_{Z^{\beta}}(F)_{0 \mu} J^{\mu}(\widehat{Z}^{\gamma}f_k) |+ |S^{\nu} \mathcal{L}_{Z^{\beta}}(F)_{\nu \mu} J^{\mu}(\widehat{Z}^{\gamma}f_k) | dx ds $$
 $$ \hspace{-4.4cm}+\sum_{|\beta| \leq N} \int_0^t \int_{\Sigma_s} |\mathcal{L}_{Z^{\beta}}(A)_{\mu} \square \mathcal{L}_{Z^{\beta}}(A)^{\mu} |dx ds,$$ 

with $Z^{\beta} \in \mathbb{K}^{|\beta|}$ and $\widehat{Z}^{\gamma} \in \K^{|\gamma|}$. 

\end{Pro}

Later, we will have, in the $4$ dimensional massless case, a strong loss on $\mathcal{E}_N[F]$ which will lead to a poor pointwise decay estimate on $|\underline{\alpha}|$. With this inequality, we will avoid the $\tau_+$-loss and we will have an extra $\tau_-$-decay (which is not given by Proposition \ref{energyt}).

\section{Some technical results}\label{section4}

\subsection{An integral estimate}

The following lemma is useful so as to estimate a quantity like
$$ \int_0^t \int_{\mathbb{R}^n} |u(s,x)| \int_v |f(s,x,v)| dv dx ds,$$
where we already have a bound on $\|u(s,.)\|_{L^2}$ and a pointwise decay estimate on $\int_v |f(s,x,v)| dv$.

\begin{Lem}\label{intesti}

Let $m \in \mathbb{N}^*$ and let $a$, $b \in \mathbb{R}$, such that $a+b >m$ and $b \neq 1$. Then
$$\exists \hspace{0.5mm} C_{a,b,m} >0, \hspace{0.5mm} \forall \hspace{0.5mm} t \in \mathbb{R}_+ \hspace{2mm} \int_0^{+ \infty} \frac{r^{m-1}}{\tau_+^a \tau_-^b}dr \leq C_{a,b,m} \frac{1+t^{b-1}}{1+t^{a+b-m}} .$$

\end{Lem}

A proof of this estimate can be found in $\cite{FJS}$, Appendix $B$.

\subsection{The null coordinates of $\nabla_v f$}

Let $f : [0,T[ \times \mathbb{R}^n \times \mathbb{R}^n$ be a smooth function. We designate by $(( \nabla_v f )^L,$ $ ( \nabla_v f )^{\underline{L}},  ( \nabla_v f )^B,...)$ the null components of $\nabla_v f$. Later, we will have to transform the $v$-derivatives in combinations of $\widehat{\mathbb{P}}_0$-derivatives. If we only use the relation 
\begin{equation}\label{transformpartialv}
v^0\partial_{v^k} = \widehat{\Omega}_{0k}-t\partial_k-x^k\partial_t,
\end{equation}
 we get that
\begin{equation}\label{eq:null0}
\left| \left( \nabla_v f \right)^L \right|, \left| \left( \nabla_v f \right)^{\underline{L}} \right|, \left| \left( \nabla_v f \right)^B \right| \leq \frac{\tau_+}{v^0} \sum_{\widehat{Z} \in \widehat{\mathbb{P}}_0 } |\widehat{Z} f |,
\end{equation}
which will not be good enough to close the energy estimates (for the Vlasov-Maxwell system).

We then use the following lemma.

\begin{Lem}\label{vderivL}

Let $f$ be a smooth function. We have
$$\left| \left( \nabla_v f \right)^L \right|, \left| \left( \nabla_v f \right)^{\underline{L}} \right| \leq \frac{\tau_-}{v^0} \sum_{\widehat{Z} \in \widehat{\mathbb{P}}_0}  |\widehat{Z} f |.$$

\end{Lem}

\begin{proof}

Since $( \nabla_v f )^0=0$ (by definition),
$$( \nabla_v f )^L=\frac{x^i}{r}\partial_{v^i} f.$$
Now, we use $\partial_{v^i}=\frac{1}{v^0}(\widehat{\Omega}_{0i}-t\partial_i-x_i\partial_t)$. As
$$\frac{x^i}{rv^0}(t\partial_i+x_i\partial_t)=\frac{1}{v^0}(t\partial_r+r\partial_t)=\frac{1}{v^0}(S+(r-t)\underline{L}),$$ we have
$$( \nabla_v f )^L = \frac{x^i }{rv^0}\widehat{\Omega}_{0i} f-\frac{1}{v^0}S f +\frac{t-r}{v^0}\underline{L}f .$$
It only remains to notice that $( \nabla_v f )^{\underline{L}}=-( \nabla_v f )^{L}$, since $( \nabla_v f )^0=0$.

\end{proof}

We are now interested in $\left( \nabla_v  f \right)^B$. During the study of the Vlasov equation, each time that \eqref{eq:null0} is not sufficient to close the estimates, $(\nabla_v f )^B$ is multiplied by $v^{\underline{L}}$, which reflects the null structure of the system. This leads us to study $v^{\underline{L}} \left( \nabla_v f \right)^B$.

\begin{Lem}\label{lemnull1}

For $1 \leq i \leq n $, we have $$2v^{\underline{L}} \frac{x^i}{r}=\frac{v^0x^i}{r}-v^i+\frac{z_{ij}x^j}{r^2},$$
where $z_{\mu \nu}=x^{\nu}v^{\mu}-x^{\mu}v^{\nu}$.

\end{Lem}

\begin{Rq}

If $\mu \neq \nu$, $\frac{z_{\mu \nu}}{v^0} \in \mathbf{k}_1$ and if $\mu = \nu$, then $z_{\mu \mu} =0$.

\end{Rq}

\begin{proof}

For simplicity, we take $i=1$. We have

\begin{eqnarray}
\nonumber 2v^{\underline{L}}\frac{x^1}{r} & = & \frac{x^1v^0}{r}-\frac{x^1}{r^2}x_iv^i\\ \nonumber 
& = & \frac{x^1v^0}{r}-v^1+\frac{z_{1j}x^j}{r^2} .
\end{eqnarray}

\end{proof}
And we obtain

\begin{Cor}\label{Cornull1}

Let $i$, $j \in \llbracket 1,n \rrbracket$  such that $i \neq j$. We have

$$2v^{\underline{L}}\left( \frac{x^i}{r}\partial_{v^j}-\frac{x^j}{r}\partial_{v^i} \right)=\left(\frac{x^i}{r}+\frac{z_{ik}x^k}{v^0r^2} \right)\widehat{\Omega}_{0j}-\left(\frac{x^j}{r}+\frac{z_{jk}x^k}{v^0r^2} \right)\widehat{\Omega}_{0i}-\widehat{\Omega}_{ij}$$ $$-\left(\frac{x^i(t-r)}{r}+\frac{tx^kz_{ik}}{r^2v^0} \right) \partial_j+\left(\frac{x^j(t-r)}{r}+ \frac{tx^kz_{jk}}{r^2v^0} \right) \partial_i-\frac{z_{ij}}{v^0}\partial_t.$$

\end{Cor}

\begin{proof}
By the previous lemma,

$$2v^{\underline{L}}\left( \frac{x^i}{r}\partial_{v^j}-\frac{x^j}{r}\partial_{v^i} \right)= \left(\frac{v^0x^i}{r}+\frac{z_{ik}x^k}{r^2} \right) \partial_{v^j}-\left(\frac{v^0x^j}{r}+\frac{z_{jk}x^k}{r^2} \right)\partial_{v^i}$$ $$-\widehat{\Omega}_{ij}+x^i\partial_j-x^j \partial_i.$$
Now, using the relation $v^0\partial_{v^k} = \widehat{\Omega}_{0k}-t\partial_k-x^k\partial_t$, we have

$$2v^{\underline{L}}\left( \frac{x^i}{r}\partial_{v^j}-\frac{x^j}{r}\partial_{v^i} \right) = \left( \frac{x^i}{r}+ \frac{z_{ik}x^k}{v^0r^2} \right)\widehat{\Omega}_{0j}-\left( \frac{x^j}{r}+ \frac{z_{jk}x^k}{v^0r^2} \right)\widehat{\Omega}_{0i}-\widehat{\Omega}_{ij}$$ $$+x^i\partial_j-x^j \partial_i-\frac{t}{v^0}\left(\frac{v^0x^i}{r}+\frac{z_{ik}x^k}{r^2} \right)\partial_j+\frac{t}{v^0}\left(\frac{v^0x^j}{r}+ \frac{z_{jk}x^k}{r^2} \right)\partial_i-\frac{z_{ij}}{v^0}\partial_t.$$
It remains to remark that $t\frac{v^0x^i}{r}-v^0x^i=v^0\frac{x^i}{r}(t-r)$.
\end{proof}

The naive estimation gave us
$$\left| v^{\underline{L}} \left( \nabla_v f \right)^B  \right| \lesssim |x| |\partial_t f |+\sum_{k=1}^n \left( |\widehat{\Omega}_{0k} f| +t |\partial_k f | \right) ,$$
whereas, with this lemma and the fact that $\left( \nabla_v f \right)^B$ is a combination with bounded coefficients of $\left( \frac{x^i}{r}\partial_{v^j}f-\frac{x^j}{r}\partial_{v^i}f \right)_{1 \leq i < j \leq n}$, we have 
\begin{equation}\label{eq:null1}
\left| v^{\underline{L}} \left( \nabla_v f \right)^B  \right| \lesssim \sum_{\widehat{Z} \in \widehat{\mathbb{P}} } |\widehat{Z} f | + \hspace{-1mm} \sum_{1 \leq i < j \leq n } \hspace{-1mm}  \frac{|z_{ij}|}{v^0}|\partial_t f |+ \sum_{k=1}^n (\tau_-+ \frac{t \sum_{i=1}^n |z_{ki}|}{rv^0}) |\partial_k f | .
\end{equation}
Therefore, with the last corollary, we transform a $t$-loss (and a $|x|$-loss) in a $\tau_-$-loss and a $\frac{t}{r}$-loss (thanks, among others, to the weights transported by the flow). It is particularly useful when we look for an estimate of $ \| \int_v | v^{\underline{L}}  ( \nabla_v \widehat{Z}^{\beta} f )^B  | dv  \|_{L^2_x}$ and we already have an estimate of $\int_v v^0 | z \widehat{Z}^{\delta} f | dv$. We can then use Lemma \ref{intesti}. One can also transform the $\frac{t}{r}$-loss.

\begin{Lem}\label{lemnull3}
For $1 \leq j \leq n$,
$$ \left| \frac{x^j(t-r)}{r}+ \frac{tx^kz_{jk}}{r^2v^0} \right| \lesssim \tau_-\sum_{z \in \mathbf{k}_1} |z|.$$
\end{Lem}

\begin{proof}
We obviously have $\tau_-^{-1} \left| \frac{x^j(t-r)}{r} \right| \leq  \frac{v^0}{v^0}$. For the second term, we need to study different cases. \\
If $r \leq 1$, then
$$\tau_-^{-1} \left|\frac{tx^kz_{jk}}{r^2v^0} \right| = \frac{t}{\tau_-}\frac{x_k}{r}\frac{x^kv^j-x^jv^k}{rv^0} \lesssim \frac{1}{v^0}\sum_{i=1}^n |v^i|.$$
Otherwise, $r \geq 1$, and
$$\tau_-^{-1} \left|  \frac{tx^kz_{jk}}{r^2v^0} \right| \leq \frac{t}{\tau_-r}\frac{x^k}{r}\frac{|z_{jk}|}{v^0}.$$
It remains to note that if $r \leq \frac{t}{2}$, $\tau_- \geq \frac{t}{2}$ and if $r \geq \frac{t}{2}$, then $\frac{t}{r} \leq 2 $.
\end{proof}

One then obtains the following result.

\begin{Pro}\label{nullresume}
We have
$$ \left| v^{\underline{L}} \left( \nabla_v f \right)^B  \right| \lesssim \tau_-\sum_{\widehat{Z} \in \widehat{\mathbb{P}} } \sum_{z \in \mathbf{k}_1}   |z\widehat{Z} f | .$$

\end{Pro}

Note that later, in Sections \ref{L2system} and \ref{L2systembis}, when we will establish an estimate on $\left\| \int_v |\widehat{Z}^{\beta}f| dv \right\|_{L^2_x}$, we will not be able to apply Propositions \ref{nullresume} or \ref{vderivL}. A vector $X$ will contain various derivatives of $f$ and we will split it in two vectors $H+G$ such that
$$T_F(H)=0, \hspace{2mm} \text{with} \hspace{3mm} H(0)=X(0), \hspace{3mm} \text{and} \hspace{3mm} T_F(G)=T_F(X), \hspace{2mm} \text{with} \hspace{3mm} G(0)=0.$$
Note yet that, for instance, if $X_{\mu}$ is $\partial_{\mu}f$ and $X_{S}$ is $S(f)$, we have $x^{\mu} X_{\mu}=X_S$ whereas we do not necessarily have $x^{\mu}G_{\mu}=G_S$.

\subsection{Some Sobolev inequalities}

The following results come from \cite{CK} and in order to be self-sufficient, we also recall their proof. We will use them to prove pointwise decay estimate for the electromagnetic field.

We first recall two classical Sobolev inequalities.

\begin{Lem}\label{ClassicalSob}

Let $u : \mathbb{R}^n \rightarrow \mathbb{R}$ be a sufficiently regular function. We have
$$\forall \hspace{0.5mm} x \in \mathbb{R}^n, \hspace{2mm} |u(x)| \lesssim \sum_{|\beta| \leq \frac{n+2}{2}} \| \partial^{\beta} u \|_{L^2_y(|y-x| \leq 1)}.$$

Let $v : \mathbb{S}^{n-1} \rightarrow \mathbb{R} $ a sufficiently regular function (where $\mathbb{S}^{n-1}$ is the unit sphere in $\mathbb{R}^{n}$). We have
$$\forall \hspace{0.5mm} \xi \in \mathbb{S}^{n-1}, \hspace{2mm} |v(\xi)| \lesssim \sum_{|\beta| \leq \frac{n}{2}} \| \nabla_{Z^{\beta}} v \|_{L^2(\mathbb{S}^{n-1})},$$
with $Z^{\beta} \in \mathbb{O}^{|\beta|}$.

\end{Lem}

In order to treat the interior of the light cone (or rather the domain in which $|x| \leq 1+\frac{1}{2}t$), we will use.

\begin{Lem}\label{decayint}

Let $U$ be a smooth tensor field defined in the Euclidian space $\mathbb{R}^n$. Then,
$$\forall \hspace{0.5mm} t \in \mathbb{R}_+, \hspace{2mm} \sup_{|x| \leq 1+\frac{t}{2}} |U(x)| \lesssim \frac{1}{(1+t)^{\frac{n}{2}}} \sum_{k=0}^{\frac{n+2}{2}} (1+t)^k \| \nabla^k U \|_{L^2 \left( \{|y| \leq 3+\frac{3}{4}t \} \right)}.$$

\end{Lem}

\begin{proof}

As it suffices to prove the result for each component of the tensor, we assume that $U$ is a scalar function. Let $t \in \mathbb{R}_+$ and $|x| \leq 1+ \frac{1}{2}t$. If $t \leq 1$, then $|x| \leq 2$, so, according to Lemma \ref{ClassicalSob},
$$|U(x)| \lesssim \sum_{|\beta| \leq \frac{n+2}{2}} \| \nabla^{\beta} U \|_{L^2_y(|y| \leq 3)}.$$

Now, if $t \geq 1$, we apply Lemma \ref{ClassicalSob} to $y \mapsto U(x+\frac{t}{4}y)$. It comes that (after a change of variables)

$$|U(x)| \lesssim \left( \frac{t}{4} \right)^{-\frac{n}{2}} \sum_{|\beta| \leq \frac{n+2}{2}} \left( \frac{t}{4} \right)^{|\beta|} \| \nabla^{\beta} U \|_{L^2_y(|y-x| \leq \frac{t}{4})}.$$

It remains to observe that $|y-x| \leq \frac{t}{4}$ imply $|y| \leq 1+\frac{3}{4}t$.

\end{proof}

For the other region ($|x| \geq 1+\frac{1}{2}t$), we have the following inequality.

\begin{Lem}\label{Sob}
Let $U$ be a sufficiently regular tensor field, which in particular vanishes at $\infty$, defined in the euclidian space $\mathbb{R}^n$. Then, for $t \in \R_+$,
$$\forall \hspace{0.5mm} x \neq 0, \hspace{2mm} |U(x)| \lesssim \frac{1}{|x|^{\frac{n-1}{2}}\tau_-^{\frac{1}{2}}} \left( \int_{|y| \geq |x|} |U(y)|^2_{\mathbb{O},\frac{n}{2}}+\tau_-^2|\nabla_{\partial_r} U(y) |^2_{\mathbb{O},\frac{n}{2}} dy \right)^{\frac{1}{2}}.$$

\end{Lem}

\begin{proof}

As $\sum_{|\beta| \leq k} |\nabla_{Z^{\beta}} U|^2 \lesssim |U|^2_{\mathbb{O},k}$, for $Z^{\beta} \in \mathbb{O}^{|\beta|}$, we only have to prove the result for each component of $U$ and we can assume that $U$ is a scalar function. 

Let $x \neq 0$ such that $x=r \xi$, with $r=|x|$ and $\xi \in \mathbb{S}^{n-1}$. Since $\partial_r \left( (\sqrt{\tau_-}U)^2 \right) = 2\sqrt{\tau_-}U\partial_r (\sqrt{\tau_-}U)$,
$$\tau_-|U(r \xi)|^2 \lesssim r^{-(n-1)} \int_r^{+\infty} |\sqrt{\tau_-}U(\lambda \xi)|\partial_r(\sqrt{\tau_-} U)(\lambda \xi)| \lambda^{n-1} d \lambda .$$
Therefore, an integration over $\mathbb{S}^{n-1}$ and the inequality $2|ab| \leq a^2+b^2$ gives us
$$\| U(r \xi) \|_{L^2_{\xi}( \mathbb{S}^{n-1})} \lesssim r^{-\frac{n-1}{2}}\tau_-^{-\frac{1}{2}}  \left( \int_{|y| \geq r} |U(y)|^2+\tau_-^2|\partial_r U(y) |^2 dy \right)^{\frac{1}{2}}.$$
As every vector field of $\mathbb{O}$ commute with $\partial_r$, we obtain, using Lemma \ref{ClassicalSob},

$$|U(x)| \lesssim r^{-\frac{n-1}{2}}\tau_-^{-\frac{1}{2}} \left( \int_{|y| \geq r} |U(y)|_{\mathbb{O},\frac{n}{2}}^2+\tau_-^2|\partial_r U(y) |_{\mathbb{O},\frac{n}{2}}^2 dy \right)^{\frac{1}{2}}.$$
\end{proof}

\subsection{Pointwise decay estimate for the null decomposition of the electromagnetic field}

In this section, we recall some inequalities coming from \cite{CK} between quantites linked to the null decomposition of a $2$-form (see Section \ref{notations} for its definition) and we then prove pointwise decay estimates on it. However, we cannot adapt the method used in \cite{CK} to establish, in dimension $3$, the optimal decay estimate on the null component $\alpha$. To circomvent this difficuty, we make crucial use of an electromagnetic potential satisfying the Lorenz gauge. We first introduce some notations.

\begin{Def}
Let $F$ be a $2$-form. We define its pointwise norm $|F|^{\#}$ by
$$|F|^{\#}=\sqrt{\tau_+^2|\alpha|^2+\tau_-^2|\underline{\alpha}|^2+(\tau_-^2+\tau_+^2)(|\rho|^2+|\sigma|^2)},$$
which is also equal to $\sqrt{4 T[F](\overline{K}_0,\partial_t)}$. \\
We also define, for $\mathbb{L}=\mathbb{O}$ or $\mathbb{L}=\mathbb{K}$ and $k \in \mathbb{N}$,

$$|F|^{\#}_{\mathbb{L},k}=\sqrt{ \sum_{|\beta| \leq k} \left( |\mathcal{L}_{Z^{\beta}}F|^{\#} \right)^2 },$$
with $Z^{\beta} \in \mathbb{L}^{|\beta|}$.

Similarly, we define
$$|F|=\sqrt{|\alpha|^2+|\underline{\alpha}|^2+2(|\rho|^2+|\sigma|^2)}$$
and
$$|F|_{\mathbb{L},k}=\sqrt{ \sum_{\begin{subarray}{l} Z^{\beta} \in \mathbb{L}^{|\beta|} \\ \hspace{1mm} |\beta| \leq k \end{subarray}}  |\mathcal{L}_{Z^{\beta}}F|^2 }.$$

\end{Def}

\begin{Rq}\label{rqgain}

By definition of $|F|^{\#}$, it comes that $\tau_- |F| \leq |F|^{\#}$.

\end{Rq}

We have the following inequality (cf Remark \ref{extradecay}).

\begin{Lem}\label{gainderiv}

Let $F$ be a $2$-form and $k$ a non-negative integer. Then
$$\forall \hspace{0.5mm} |\beta| = k, \hspace{2mm} |\nabla^{\beta} F|^{\#} \lesssim \tau_-^{-k} |F|^{\#}_{\mathbb{K},k}.$$

\end{Lem}

We also have, according to \cite{CK}.

\begin{Lem}\label{null}

Let $F$ be a $2$-form and $(\alpha, \underline{\alpha}, \rho, \sigma)$ its null decomposition. Then, for all $k \in \mathbb{N}$,
\begin{flalign*}
& \hspace{0mm} \sum_{l=0}^k\sum_{i+j=l} \tau_-^{2i}r^{2j}\Big(|\nabla_{\underline{L}}^i \nabla_{L}^j \underline{\alpha}|^2_{\mathbb{O},k-i-j}+|\nabla_{\underline{L}}^i \nabla_{L}^j \alpha|^2_{\mathbb{O},k-i-j} & \\
& \hspace{4cm}+|\nabla_{\underline{L}}^i \nabla_{L}^j \rho|^2_{\mathbb{O},k-i-j}+|\nabla_{\underline{L}}^i \nabla_{L}^j \sigma|^2_{\mathbb{O},k-i-j}\Big) \lesssim  |F|_{\mathbb{K},k}^2. &
\end{flalign*}
and
\begin{flalign*}
& \hspace{0mm} \sum_{l=0}^k\sum_{i+j=l} \tau_-^{2i}r^{2j}\Big(\tau_-^2|\nabla_{\underline{L}}^i \nabla_{L}^j \underline{\alpha}|^2_{\mathbb{O},k-i-j}+r^2(|\nabla_{\underline{L}}^i \nabla_{L}^j \alpha|^2_{\mathbb{O},k-i-j} & \\
& \hspace{3.3cm}+|\nabla_{\underline{L}}^i \nabla_{L}^j \rho|^2_{\mathbb{O},k-i-j}+|\nabla_{\underline{L}}^i \nabla_{L}^j \sigma|^2_{\mathbb{O},k-i-j} )\Big) \lesssim \left( |F|^{\#}_{\mathbb{K},k} \right)^2. &
\end{flalign*}

\end{Lem}

The first inequality is not proved in \cite{CK} but can be treated similarly as the second one.

The following corollary will be useful, particularly for the massless case in dimension $4$, to obtain an extra decay on $\underline{\alpha}$ away from the light cone.

\begin{Cor}\label{esti}

Using the same notations as in the previous lemma, we have, for $F$ a $2$-form,
$$|\sqrt{\tau_-}\underline{\alpha}|^2_{\mathbb{O},k}+\tau_-^2|\nabla_r (\sqrt{\tau_-}\underline{\alpha})|^2_{\mathbb{O},k-1} \lesssim \tau_- |F|^2_{\mathbb{K},k}.$$

\end{Cor}

\begin{proof}
One only has to use that
$$|\nabla_r \sqrt{\tau_-}| \leq \tau_-^{-\frac{1}{2}}, \hspace{5mm} 2\nabla_r=L-\underline{L}$$
and the previous lemma.

\end{proof}

Let us show how to establish pointwise decay estimates on the null decomposition of the electromagnetic field with these inequalities.

\begin{Pro}\label{decayFpointwise}
Let $G$ be a $2$-form and $J$ be a $1$-form, both defined on $[0,T[ \times \R^n$, such that
\begin{eqnarray}
\nonumber \nabla^{\mu} G_{\mu \nu} & =& J_{\nu}, \\ \nonumber
\nabla^{\mu} {}^* \! G_{ \mu \lambda_1 ... \lambda_{n-2} } & = & 0.
\end{eqnarray}
If $G$ and $J$ are sufficiently regular, we have, for all $(t,x) \in [0,T[ \times \R^n$,

$$\hspace{1mm} |\alpha(G)|(t,x), \hspace{2mm} |\rho(G)|(t,x)  , \hspace{2mm} |\sigma(G)|(t,x)  \lesssim \frac{\sqrt{ \mathcal{E}_{\frac{n+2}{2}}[G](t)}}{\tau_+^{\frac{n+1}{2}}\tau_-^{\frac{1}{2}}}, $$ 
\begin{equation}\label{eq:underalpha1}
\hspace{-4.5cm} |\underline{\alpha}(G)|(t,x) \lesssim \frac{\sqrt{ \mathcal{E}_{\frac{n+2}{2}}[G](t)}}{\tau_+^{\frac{n-1}{2}}\tau_-^{\frac{3}{2}}} 
\end{equation}
 and
\begin{equation}\label{eq:underalpha2}
\hspace{-2cm} |\underline{\alpha}(G)|(t,x) \lesssim \frac{\sqrt{\frac{\mathcal{E}_{\frac{n+2}{2}}[G](t)}{1+t} }+\sqrt{\mathcal{E}^S_{\frac{n+2}{2}}[G](t)}}{\tau_+^{\frac{n-1}{2}}\tau_-}.
\end{equation}
\end{Pro}

\begin{Rq}
When we will study the massless Vlasov-Maxwell system in dimension $n=4$, a strong $t$-loss on $ \mathcal{E}_{\frac{n+2}{2}}[G]$ will lead to a strong $\tau_+$-loss on the pointwise estimate $\eqref{eq:underalpha1}$. Since we will not need all the $\tau_-$ decay rate of \eqref{eq:underalpha1}, we will rather use \eqref{eq:underalpha2}.

\end{Rq}

\begin{proof}

Let us denote the null decomposition of $G$ by $(\alpha, \underline{\alpha}, \rho, \sigma)$. Let $(t,x) \in [0,T[ \times \R^n$.

First, we consider the case $|x| \leq 1+\frac{1}{2}t$.

As
$$  \int_{\Sigma_t} \left( |G|^{\#}_{\mathbb{K}, \frac{n+2}{2}} \right)^2 dx \lesssim \mathcal{E}_{\frac{n+2}{2}}[G](t),$$

Lemma \ref{gainderiv} and Remark \ref{rqgain} give us
$$ \sum_{|\beta| \leq \frac{n+2}{2}} \int_{\Sigma_t} \tau_-^{2|\beta|+2} |\nabla^{\beta} G|^2 dx \lesssim \mathcal{E}_{\frac{n+2}{2}}[G](t).$$
Moreover, 
$$\forall \hspace{0.5mm} (t,y) \in [0,T[ \times \mathbb{R}^n \hspace{1mm} \text{such that} \hspace{1mm} |y| \leq 3+\frac{3}{4}t, \hspace{2mm} \tau_-(t,y) \gtrsim 1+t .$$
Hence,
$$\sum_{|\beta| \leq \frac{n+2}{2}} \int_{|y| \leq 3+ \frac{3}{4}t} (1+t)^{2|\beta|+2} |\nabla^{\beta} G|^2 dy \lesssim \mathcal{E}_{\frac{n+2}{2}}[G](t).$$
Using Lemma \ref{decayint}, we obtain
$$|G(t,x)| \lesssim \frac{\sqrt{\mathcal{E}_{\frac{n+2}{2}}[G](t)} }{(1+t)^{\frac{n+2}{2}}}.$$

We consider now the case $|x| \geq 1+\frac{1}{2}t$.

According to Lemma \ref{null},
\begin{flalign*}
& \hspace{0mm} \sum_{l=0}^1\sum_{i+j=l}\int_{|y| \geq 1+ \frac{1}{2}t}  \tau_-^{2i}r^{2j}\Big(\tau_-^2|\nabla_{\underline{L}}^i \nabla_{L}^j \underline{\alpha}|^2_{\mathbb{O},\frac{n+2}{2}-i-j}+r^2(|\nabla_{\underline{L}}^i \nabla_{L}^j \alpha|^2_{\mathbb{O},\frac{n+2}{2}-i-j} & \\
& \hspace{2cm}+|\nabla_{\underline{L}}^i \nabla_{L}^j \rho|^2_{\mathbb{O},\frac{n+2}{2}-i-j}+|\nabla_{\underline{L}}^i \nabla_{L}^j \sigma|^2_{\mathbb{O},\frac{n+2}{2}-i-j}) \Big) dy  \lesssim \mathcal{E}_{\frac{n+2}{2}}[G](t).&
\end{flalign*}
Let $w$ be either $r\alpha$, $r\rho$, $r\sigma$ or $\tau_- \underline{\alpha}$. Since $\partial_r=\frac{L-\underline{L}}{2}$ and $|\partial_r(\tau_-)| \leq 1$, we have
$$\int_{|y| \geq 1+\frac{1}{2}t} |w|^2_{\mathbb{O},\frac{n+2}{2}}+\tau_-^2|\nabla_{\partial_r} w |_{\mathbb{O},\frac{n}{2}}^2 dy \lesssim \mathcal{E}_{\frac{n+2}{2}}[G](t).$$
Lemma \ref{Sob} then gives us
$$|w(t,x)| \lesssim \frac{\sqrt{ \mathcal{E}_{\frac{n+2}{2}}[G](t)}}{|x|^{\frac{n-1}{2}}\tau_-^{\frac{1}{2}}}.$$
Thus, 
$$|\alpha(t,x)|, \hspace{2mm} |\rho(t,x)|  , \hspace{2mm} |\sigma(t,x)|  \lesssim \frac{\sqrt{ \mathcal{E}_{\frac{n+2}{2}}[G](t)}}{|x|^{\frac{n+1}{2}}\tau_-^{\frac{1}{2}}} \hspace{2mm} \text{and} \hspace{2mm}  |\underline{\alpha}(t,x)| \lesssim \frac{\sqrt{ \mathcal{E}_{\frac{n+2}{2}}[G](t)}}{|x|^{\frac{n-1}{2}}\tau_-^{\frac{3}{2}}}.$$

We now prove \eqref{eq:underalpha2}. Using Corollary \ref{esti}, we have
\begin{flalign*}
& \hspace{0mm} \int_{|y| \geq 1+\frac{1}{2}t} |\sqrt{\tau_-}\underline{\alpha}|^2_{\mathbb{O},\frac{n+2}{2}}+\tau_-^2|\nabla_{\partial_r} (\sqrt{\tau_-}\underline{\alpha})|_{\mathbb{O},\frac{n}{2}}^2 dy  \lesssim & 
\end{flalign*}
$$\hspace{0.3cm} \sum_{|\beta| \leq \frac{n+2}{2}} \int_{\Sigma_t} \tau_-\left(|\alpha(\mathcal{L}_{Z^{\beta}}G)|^2+|\underline{\alpha}(\mathcal{L}_{Z^{\beta}}G)|^2+|\rho(\mathcal{L}_{Z^{\beta}}G)|^2+|\sigma(\mathcal{L}_{Z^{\beta}}G)|^2 \right)dx. $$

As, by Definition \ref{norm2},
$$\sum_{|\beta| \leq \frac{n+2}{2}} \int_{\Sigma_t} \tau_-\left(|\alpha(\mathcal{L}_{Z^{\beta}}G)|^2+|\rho(\mathcal{L}_{Z^{\beta}}G)|^2+|\sigma(\mathcal{L}_{Z^{\beta}}G)|^2 \right)dx \lesssim \frac{\mathcal{E}_{\frac{n+2}{2}}[G](t)}{1+t} $$
and, by Definition \ref{normS}
$$\sum_{|\beta| \leq \frac{n+2}{2}} \int_{\Sigma_t} \tau_-|\underline{\alpha}(\mathcal{L}_{Z^{\beta}}G)|^2dx \lesssim \mathcal{E}^S_{\frac{n+2}{2}}[G](t),$$
we obtain, again by Lemma \ref{Sob}
$$|\underline{\alpha}(t,x)| \lesssim \frac{\sqrt{\frac{\mathcal{E}_{\frac{n+2}{2}}[G](t)}{1+t} }+\sqrt{\mathcal{E}^S_{\frac{n+2}{2}}[G](t)}}{|x|^{\frac{n-1}{2}}\tau_-}.$$

\end{proof}

Our goal now is to show how to improve the decay estimate on $\alpha$, in the Lorenz gauge, near the light cone (we cannot reproduce the method used by \cite{CK} to treat the $3d$ case). We start by the following lemma.

\begin{Lem}\label{lorenznull}
Let $\mathcal{A}$ be a sufficiently regular current, defined on $[0,T] \times \R^n$, such that 
$$\partial^{\mu} \mathcal{A}_{\mu} =0 \hspace{3mm} \text{and} \hspace{3mm} \forall \hspace{1mm} t \in [0,T], \hspace{3mm} \widetilde{\mathcal{E}}_{\frac{n+2}{2}}[\mathcal{A}](t) \leq \mathcal{E}(t),$$
with $\mathcal{E} : [0,T] \rightarrow \R_+$ an increasing function. Then
$$|\mathcal{A}_{\underline{L}}|(t,x), \hspace{2mm} |\mathcal{A}_B|(t,x) \lesssim \frac{\mathcal{E}(t)}{\tau_+^{\frac{n-1}{2}}\tau_-^{\frac{1}{2}}} \hspace{3mm} \text{and} \hspace{3mm} |\mathcal{A}_L|(t,x) \lesssim \frac{\mathcal{E}(t)}{\tau_+^{\frac{n}{2}}}.$$
\end{Lem}

\begin{proof}
Using a classical $L^2$-Klainerman-Sobolev inequality, we have, $\forall \hspace{0.5mm} |\gamma| \leq 1, \hspace{0.5mm} 1 \leq \mu \leq n, \hspace{0.5mm} (t,x) \in [0,T] \times \R^n$,

$$|Z^{\gamma} \mathcal{A}_{\mu}|(t,x) \lesssim \frac{\sqrt{\widetilde{\mathcal{E}}_{\frac{n+2}{2}}[\mathcal{A}](t)}}{\tau_+^{\frac{n-1}{2}}\tau_-^{\frac{1}{2}}}.$$
We then have
\begin{equation}\label{bihanig1}
|(Z^{\gamma} \mathcal{A})_L|, \hspace{1mm} | (Z^{\gamma} \mathcal{A})_{\underline{L}}|, \hspace{1mm} |(Z^{\gamma} \mathcal{A})_B| \lesssim \frac{\sqrt{\mathcal{E}(t)}}{\tau_+^{\frac{n-1}{2}}\tau_-^{\frac{1}{2}}}.
\end{equation}

It then remains to improve the decay estimate on $\mathcal{A}_L$ near the light cone. Since, $\partial^{\mu} \mathcal{A}_{\mu}=0$,
$$ (\nabla^L \mathcal{A})_L+(\nabla^{\underline{L}} \mathcal{A})_{\underline{L}}+(\nabla^B \mathcal{A})_B=0.$$
So, as $\nabla_{\underline{L}} L =0$,
\begin{equation}\label{tra1}
-\nabla_{\underline{L}} \mathcal{A}_L-(\nabla_{L} \mathcal{A})_{\underline{L}}+(\nabla^B \mathcal{A})_B =0.
\end{equation}

If $r \leq \frac{t}{2}$ or $r \geq \frac{t}{2}$ and $t \leq 1$, the result comes from \eqref{bihanig1}. For the remaining case, $r \geq \frac{t}{2}$ and $t \geq 1$, note first that
$$|\underline{L} (\mathcal{A}_L) | \lesssim \frac{\sqrt{\mathcal{E}(t)}}{\tau_+^{\frac{n-1}{2}}\tau_-^{\frac{1}{2}}r} \lesssim \frac{\sqrt{\mathcal{E}(t)}}{\tau_+^{\frac{n+1}{2}}\tau_-^{\frac{1}{2}}}.$$
Indeed, using Remark \ref{extradecay} and \eqref{bihanig1}, we have
$$|(\nabla_{L} \mathcal{A})_{\underline{L}}|(t,x), \hspace{2mm} | (\nabla^B \mathcal{A})_B|(t,x) \lesssim \frac{\sqrt{\mathcal{E}(t)}}{\tau_+^{\frac{n-1}{2}}\tau_-^{\frac{1}{2}}r}$$
so that (using \eqref{tra1}), $\underline{L} (\mathcal{A}_L)$ satisfies also this decay rate. As for a sufficiently regular function $g$, $$g(t,r)=g(0,t+r)+\int_{u=-t-r}^{t-r} \underline{L}(g) du,$$ and since $\mathcal{E}$ is a increasing function, we have
\begin{eqnarray}
\nonumber
|\mathcal{A}_L|(t,r) & \leq & |\mathcal{A}_L|(0,t+r)+\int_{u=-t-r}^{t-r}|\underline{L} (\mathcal{A}_L)| du \\ \nonumber
& \lesssim & \frac{\sqrt{\mathcal{E}(0)}}{\tau_+^{\frac{n}{2}}}+\frac{\sqrt{\mathcal{E}(t)}}{\tau_+^{\frac{n+1}{2}}} \int_{u=-t-r}^{t-r} \tau_-^{-\frac{1}{2}} du \\ \nonumber
& \lesssim & \frac{\sqrt{\mathcal{E}(t)}}{\tau_+^{\frac{n}{2}}}.
\end{eqnarray}
\end{proof}

Finally, we obtain the following pointwise decay on $\alpha$.

\begin{Pro}\label{decayalphapointwise}
Let $G$ and $J$ be a sufficiently regular $2$-form and $1$-form (respectively), defined on $[0,T] \times \R^n$, such that
\begin{eqnarray}
\nonumber \nabla^{\mu} G_{\mu \nu} & = &  J_{\nu}, \\ \nonumber
    \nabla^{\mu} {}^* \! G_{\mu \lambda_1 ... \lambda_{n-2}} & = & 0.
 \end{eqnarray} 
 Let $A$ be a potential of $G$ in the Lorenz gauge such that $\widetilde{\mathcal{E}}_{\frac{n+4}{2}}[A](t) \leq \mathcal{E}(t)$. We suppose that
 $$|J|(t,x) \lesssim \frac{\theta(t)}{\tau_+^{n-1}\tau_-}$$ and that $\mathcal{E}$ and $\theta$ are increasing functions. Then, 
 $$ \forall \hspace{0.5mm} (t,x) \in [0,T] \times \mathbb{R}^n, \hspace{3mm} |\alpha(G)|(t,x) \lesssim \frac{\sqrt{\mathcal{E}(t)}}{\tau_+^{\frac{n+2}{2}}}+\frac{\theta(t)\log(\tau_-)}{\tau_+^{n-1}}.$$
\end{Pro}
\begin{Rq}
the functions $\mathcal{E}$ and $\theta$ will later be of the form $t \mapsto (1+t)^{a}$ or $t \mapsto \log^k(1+t)$.

\end{Rq}
\begin{proof}
We consider a spherical variable $B$. We have
\begin{eqnarray}
\nonumber \alpha_B(G) & = & (\partial_{\mu} A_{\nu}-\partial_{\nu} A_{\mu})_{BL} \\ \nonumber
& = & e_B(A)_L-L(A)_B \\ \nonumber
& = & \mathcal{L}_{e_B}(A)_L-\frac{1}{r}A_B-L(A)_B, 
\end{eqnarray}
since $e_B(A)_L=\mathcal{L}_{e_B}(A)_L-\frac{1}{r}A_B$. Indeed, as $e_B$ can be written as a linear combination of rescaled rotations (namely $\frac{\Omega_{ij}}{r}$), we only have to prove
$$\Omega_{ij}(A)_L=\mathcal{L}_{\Omega_{ij}}(A)_L-\frac{1}{r}A_{\Omega_{ij}} \hspace{3mm} \text{for all} \hspace{2mm} 1 \leq i < j \leq n.$$
Consider for instance $\Omega_{12}$. As $\Omega_{12}=x^1\partial_2-x^2 \partial_1$, we have
$$\mathcal{L}_{\Omega_{12}}(A)_1=\Omega_{12} A_1+A_2, \hspace{2mm} \mathcal{L}_{\Omega_{12}}(A)_2=\Omega_{12} A_2-A_1 $$ and $$ \mathcal{L}_{\Omega_{12}}(A)_k=\Omega_{12} A_k \hspace{2mm} \text{if} \hspace{2mm} k \geq 3,$$
so that
$$\mathcal{L}_{\Omega_{12}}(A)_L=\Omega_{12}(A)_L+\frac{1}{r}A_{\Omega_{12}}.$$
As $4\tau_- \geq  \tau_+$ if $ t \geq 2r$ or $t+r \leq 2$, we only have to consider the case\footnote{When $4\tau_- \geq \tau_+$, the result comes from Proposition \ref{decayFpointwise} or from \newline $|\partial A| \lesssim \tau_+^{-\frac{n-1}{2}} \tau_-^{-\frac{3}{2}}\sum_{Z \in \mathbb{K}} |Z A|$ (cf Remark \ref{extradecay}).} where $2r \geq  t$ and $t+r \geq 2$, so that $3r \geq \tau_+$.
Recall from Lemma \ref{derivpotential} that 
$$\forall \hspace{0.5mm} \Omega \in \mathbb{O}, \hspace{3mm} \partial^{\mu} \mathcal{L}_{\Omega} A_{\mu}=0.$$ So, using Lemma \ref{lorenznull} and that $e_B$ can be written as a linear combination of rescaled rotations, we have
$$|\mathcal{L}_{e_B}(A)_L|(t,x) \lesssim \frac{\sqrt{\mathcal{E}(t)}}{r\tau_+^{\frac{n}{2}}} \lesssim \frac{\sqrt{\mathcal{E}(t)}}{\tau_+^{\frac{n+2}{2}}}.$$

For the remaining term, rewritting the wave equation \eqref{eqLorenz} satisfied by $A$ in null coordinates, we have, for $0 \leq \mu \leq n$,
$$-\underline{L}L A_{\mu} +\nabla^C \nabla_C A_{\mu} + \frac{1}{r} L A_{\mu}-\frac{1}{r} \underline{L} A_{\mu} = J_{\mu}.$$
Hence
$$\underline{L}\left(\left(L+\frac{1}{r} \right)  A_{\mu} \right)= \nabla^C \nabla_C A_{\mu} + \frac{1}{r} L A_{\mu}+\underline{L} \left( \frac{1}{r} \right) A_{\mu}-J_{\mu}.$$
Now, note that, using a classical $L^2$ Klainerman-Sobolev inequality and Remark \ref{extradecay},
$$|\underline{L} \left( \frac{1}{r} \right) A_{\mu}|, \hspace{1mm}| \nabla^C \nabla_C A_{\mu}| \lesssim \frac{\sqrt{\mathcal{E}(t)}}{r^2 \tau_+^{\frac{n-1}{2}}\tau_-^{\frac{1}{2}}}, \hspace{2mm}  |\frac{1}{r} L A_{\mu}| \lesssim \frac{\sqrt{\mathcal{E}(t)}}{r\tau_+^{\frac{n+1}{2}}\tau_-^{\frac{1}{2}}}, \hspace{2mm} \left|J^{\mu} \right| \lesssim \frac{\theta(t)}{\tau_+^{n-1}\tau_-}$$
so that, as $3r \geq \tau_+$,
$$\left|\underline{L}\left(\left(L+\frac{1}{r} \right)  A_{\mu} \right) \right| \lesssim \frac{\sqrt{\mathcal{E}(t)}}{\tau_+^{\frac{n+3}{2}}\tau_-^{\frac{1}{2}}}+\frac{\theta(t)}{\tau_+^{n-1}\tau_-}.$$
Hence, as for a sufficiently regular function $g$, $$g(t,r)=g(0,t+r)+\int_{u=-t-r}^{t-r} \underline{L}(g) du,$$
we have (using that $\mathcal{E}$ and $\theta$ are increasing functions)
\begin{eqnarray}
\nonumber \left| \left( L+\frac{1}{r} \right)  A_{\mu}  \right|(t,x) & \lesssim & \frac{\sqrt{\mathcal{E}(0)}}{\tau_+^{\frac{n+2}{2}}}+\frac{\sqrt{\mathcal{E}(t)}}{\tau_+^{\frac{n+3}{2}}}\int_{-t-r}^{t-r} \frac{1}{\tau_-^{\frac{1}{2}}}du+\frac{\theta(t)}{\tau_+^{n-1}}\int_{-t-r}^{t-r} \frac{1}{\tau_-}du \\ \nonumber
& \lesssim & \frac{\sqrt{\mathcal{E}(t)}}{\tau_+^{\frac{n+2}{2}}}+\frac{\theta(t)\log(\tau_-)}{\tau_+^{n-1}},
\end{eqnarray}
implying
$$\left| \frac{1}{r}  A_{B}+L(A)_B  \right|(t,x) \lesssim \frac{\sqrt{\mathcal{E}(t)}}{\tau_+^{\frac{n+2}{2}}}+\frac{\theta(t)\log(\tau_-)}{\tau_+^{n-1}}.$$
\end{proof}
\begin{Rq}
In the context of the Vlasov-Maxwell system, using the null component $v^B$ of the velocity vector, we have a better pointwise estimate on the component $J_B$ of the source term, as $J_B$ is a linear combination of the terms $\int_v \frac{v^B}{v^0} \widehat{Z}^{\beta} f dv$. Since the dimension $n$ is such that $n \geq 4$, we do not need this extra decay (and we then worked with the Cartesian components of the source term in the proof of the previous proposition).
\end{Rq}

\subsection{A Grönwall inequality}

Later, when we will study the velocity support of the scalar field in the massless case, we will need the following variant of Grönwall's lemma.

\begin{Lem}\label{Gronwall}
Let $T>0$, $f$ and $g$ two continuous nonnegatives functions defined on $[0,T]$ and $C \geq 0$. If $$ \forall t \in [0,T], \hspace{1mm} f(t) \leq C+2\int_0^t g(s) \sqrt{f(s)} ds,$$
then $$\forall t \in [0,T], \hspace{1mm} f(t) \leq  \left( \sqrt{C}+\int_0^t g(s) ds \right)^2.$$

\end{Lem}

\begin{proof}

First, we suppose that $C >0$. Let $F: t \mapsto C+2\int_0^t g(s) \sqrt{f(s)} ds$. We have
$$F'(t) \leq 2g(t)\sqrt{F(t)}.$$
Since $C>0$, $F$ is nonnegative and we can divide by $2\sqrt{F(t)}$. Integrating the above, we obtain
$$\sqrt{F(t)} \leq \sqrt{C}+\int_0^t g(s) ds,$$
which implies the result. If $C=0$, then, for all $\epsilon >0$,
$$ \forall t \in [0,T], \hspace{1mm} f(t) \leq \epsilon+2\int_0^t g(s) \sqrt{f(s)} ds.$$
It only remains to apply the inequality in the case $C \neq 0$ and let $\epsilon$ tends to zero.

\end{proof}

\section{Decay estimate for the massive case}\label{section5}

Recall that, as we study massive particles in this section, $v^0=\sqrt{1+|v|^2}$. We will use the commutation vector fields of $\K$ and the weights of $\mathbf{k}_1$ preserved by the operator $T_1$ (see Subsections \ref{subsecvector} and \ref{poids} for their definitions). We fix for all this section a sufficiently regular $2$-form $F$ defined on $[0,T^*[ \times \R^n$ and we recall that we defined $T_F$ as the operator
$$T_F : g \mapsto v^{\mu} \partial_{\mu} g + F(v, \nabla_v g)$$
and that $\nabla_v g= (0, \partial_{v^1} g, ..., \partial_{v^n} g)$. The main result of this section is the following estimate.

\begin{Th}\label{decayestimate}
Let $T^*>0$ and $f : [0, T^*[ \times \R^n_x \times \R^n_v \rightarrow \R$ be a sufficiently regular function such that
$$ \sum_{z \in \mathbf{k}_1} \sum_{|\beta| \leq n} \int_{\Sigma_0} \int_{\R^n_v} \left| z \widehat{Z}^{\beta} f \right| dv dx   < + \infty, $$
 Then, for all $(t,x) \in [0,T^*[ \times \R^n$,
\begin{eqnarray}
\nonumber  \int_{\R^n_v} |f(t,x,v)| \frac{dv}{(v^0)^2} \hspace{1mm} \lesssim \hspace{1mm} \frac{1}{\tau_+^n} \sum_{z \in \mathbf{k}_1} \sum_{|\beta| \leq n}  \Bigg( \hspace{-3mm} & & \hspace{-3mm} \int_{\Sigma_0} \int_{\R^n_v} \left| z \widehat{Z}^{\beta} f \right| dv dx \\ \nonumber
 & & \hspace{-2mm}  + \int_0^{t} \int_{\Sigma_s} \int_{\R^n_v} \left| T_F \left(  z \widehat{Z}^{\beta} f \right) \right| \frac{dv}{v^0} dxds \Bigg). 
 \end{eqnarray}
\end{Th}

\begin{Rq}
Compared to Theorem $8$ in \cite{FJS}, the advantage is that the $L^1$ norms on the right hand side are taken on $\{t\} \times \R^n$ (or $\{0 \} \times \R^n$) and not on a hyperboloid. On the other hand, our estimate is not a pure Sobolev inequality (we applied the operator $T_F$ to $\widehat{Z}^{\beta} f$ to establish it).
\end{Rq}

\begin{Rq}
To simplify the notation, we took the mass to be $1$, but the estimate is true as long as the mass is strictly positive (the constant hidden in $\lesssim$ is however proportional to $\frac{1}{m^2}$).
\end{Rq}

\begin{Rq}\label{rqdecayestimate}
As we will need an estimate on $\int_v |f| dv$ in this article, we will apply Theorem \ref{decayestimate} to $(v^0)^2f$. Note that, since $T_1((v^0)^2 z)=0$, the spacetime integral given by Theorem \ref{decayestimate} can be bounded, in that case, by
$$\int_0^{t} \int_{\Sigma_s} \int_{\R^n_v} \left|v^0z T_F \left(\widehat{Z}^{\beta} f \right) \right|+\left| v^0 F(v, \nabla_v z ) \widehat{Z}^{\beta} f \right|+\left|zv^i F_{i0} \widehat{Z}^{\beta} f \right| dv dxds.$$
One can then use commutation formula of Proposition \ref{Commuf} in order to compute $T_F (\widehat{Z}^{\beta} f)$.
\end{Rq}

The proof is based on a partition of the spacetime. In the interior ($|x| \leq \frac{t}{2}$) and the exterior ($ t \leq |x|$) of the light cone, the proof relies on the Klainerman-Sobolev inequality of Theorem \ref{KS1}. In the exterior region, the lack of decay is compensated by using the weights $(x^i-t\frac{v^i}{v^0}) \in \mathbf{k}_1$ defined in Section \ref{poids}. For the remaining region, we work on subsets of $\R^{n+1}$ composed of a piece of an hyperboloid and a piece of a slice $t=constant$ as \cite{Georgiev} for the Klein-Gordon equation, mixing what is usually done for such problems.

\vspace{4mm}

\begin{tikzpicture}

\fill [color=gray!35] 
(0,0)
-- plot [domain=0:3] (\x, {sqrt(\x*\x+7)})
--(7,4)--(7,0)
--cycle;
\draw[color=black!100] (3,2) node[scale=1.5]{$D_a(T)$};
\draw (0,4) node[scale=1.5,left]{$T$};
\draw (0,0)--(0,{sqrt(7)});
\draw [-{Straight Barb[angle'=60,scale=3.5]}] (0,{sqrt(7)})--(0,5);
\draw (3,4)--(7,4);
\draw (0,0)--(7,0) node[scale=1.5,right]{$\Sigma_0$};
\draw [domain=0:3] plot (\x, {sqrt(\x*\x+7)});
\draw (-0.1,4)--(0,4);
\draw (0,{sqrt(7)}) node[scale=1.5,left]{$a$};
\draw (-0.1,{sqrt(7)})--(0,{sqrt(7)});
\draw (0,-0.5) node[scale=1.5]{$r=0$};
\node[align=center,font=\bfseries, yshift=-1em] (title) 
    at (current bounding box.south)
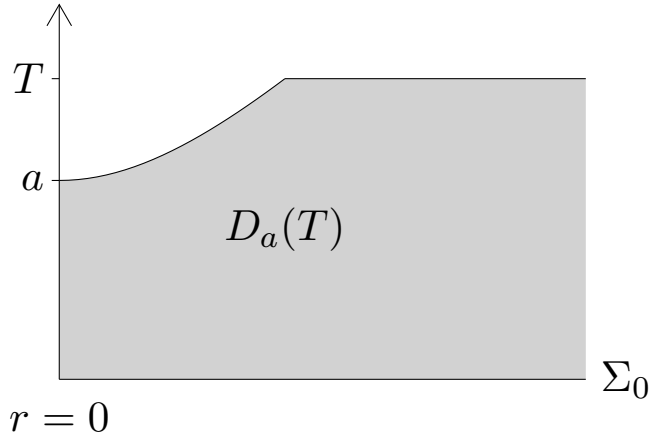
    {The set $D_a(T)$ and its boundary};
\end{tikzpicture}

\subsection{Sobolev inequalities}

We start by a Sobolev inequality independent of time.

\begin{Lem}

Let $g : \R^n_x \times \R^n_v \rightarrow \R$ be a sufficiently regular function. Then, for all $ x \in \R^n$,

$$|x|^n\int_{v \in \R^n} |g(x,v)| dv \lesssim  \sum_{\begin{subarray}{l} | \beta| \leq n-1 \\ \hspace{2mm} j \leq 1 \end{subarray}} \left\|  \int_{v \in \R^n} (r \partial_r)^{j}(|\widehat{\Omega}^{\beta} g|)(y,v) dv \right\|_{L^1(|y| \leq |x|)},$$
where $\Omega^{\beta} \in \mathbb{O}^{|\beta|}$.
\end{Lem}

During the proof of this lemma, we will use many time the following one dimensional Sobolev inequality. For $w \in W^{1,1}$, we have, for all $a \in \R$ and all $\delta \geq \eta >0$,
$$|w(a)| \leq C_{\eta} (\|w(y)\|_{L^1(a- \delta \leq y \leq a)}+\|w'(y)\|_{L^1(a- \delta \leq y \leq a)}),$$
with $C_{\eta}$ a positive constant depending only on $\eta$.

\begin{proof}
As there is nothing to prove when $x = 0$, we suppose $x \neq 0$. We start by introducing spherical coordinates. A point $y \in \R^n$ has for coordinates $(r,\theta)$, with $r=|y|$ and $\theta \in \mathbb{S}^{n-1}$. We denote by $(|x|,\omega)$ the spherical coordinates of $x$ and by $(\theta_1,...,\theta_{n-1})$ a local coordinate map in a neighbourhood of $\omega \in \mathbb{S}^{n-1}$ (by the symmetry of the sphere, we can suppose that the $\theta_i$ take their values in an interval of a size independent of $\omega$). Let $h$ be the function defined by $h(r,\theta,v)=g(|x|r\theta,v)$. By a one dimensional Sobolev inequality,
\begin{flalign*}
& \hspace{0.6cm} \int_{v \in \R^n} |h(1,\omega,v)| dv  \lesssim  \int_{\theta_1}  \left|\int_{v \in \R^n} |h|(1,\omega_1+\theta_1,\omega_2,...,\omega_n,v)dv \right| &
\end{flalign*}
$$\hspace{3.4cm}+ \left|\partial_{\theta_1}\int_{v \in \R^n} |h|(1,\omega_1+\theta_1,\omega_2,...,\omega_n,v) dv \right| d \theta_1 .$$
As $\partial_{\theta_1}$ is a linear combination of the rotation vector fields, Remark \ref{ineq:nolor} gives us
\begin{flalign*}
& \hspace{0cm} \left|\partial_{\theta_1}\int_{v \in \R^n} |h|(1,\omega_1+\theta_1,\omega_2,...,\omega_n,v) dv \right| \lesssim &
\end{flalign*} 
$$\hspace{5.3cm} \sum_{\Omega \in \mathbb{O}} \int_{v \in \R^n} |\widehat{\Omega}h|(1,\omega_1+\theta_1,\omega_2,...,\omega_n,v) dv.$$
Thus,
$$\int_{v \in \R^n} |h(1,\omega,v)| dv \lesssim  \sum_{\begin{subarray}{l} \Omega^{\beta} \in \mathbb{O}^{|\beta|} \\ \hspace{1mm} |\beta| \leq 1 \end{subarray}}\int_{\theta_1}  \int_{v \in \R^n} |\widehat{\Omega}^{\beta}h|(1,\omega_1+\theta_1,\omega_2,...,\omega_n,v)dv  d \theta_1 .$$
Using the same argument for the variables $\theta_2$, ..., $\theta_{n-2}$ and $\theta_{n-1}$, it comes
$$\int_{v \in \R^n} |h(1,\omega,v)| dv \lesssim  \sum_{\begin{subarray}{l} \Omega^{\beta} \in \mathbb{O}^{|\beta|} \\ |\beta| \leq n-1 \end{subarray}}\int_{\theta \in \mathbb{S}^{n-1}}  \int_{v \in \R^n} |\widehat{\Omega}^{\beta}h|(1,\theta,v)dv  d \theta.$$

The one dimensional Sobolev inequality, applied this time to the first variable, gives 

$$\int_{v \in \R^n} |h(1,\omega,v)| dv \lesssim  \sum_{j \leq 1}\sum_{\begin{subarray}{l} \Omega^{\beta} \in \mathbb{O}^{|\beta|} \\ |\beta| \leq n-1 \end{subarray}} \int_{\frac{1}{2}}^1 \left| \partial^j_r \int_{\theta \in \mathbb{S}^{n-1}}  \int_{v \in \R^n} |\widehat{\Omega}^{\beta}h|(r,\theta,v)dv  d \theta \right| dr.$$
Hence, as $\frac{1}{2} \leq r $,
\begin{flalign*}
& \hspace{0mm} \int_{v \in \R^n} |h(1,\omega,v)| dv \lesssim  &
\end{flalign*}
 $$\hspace{2.1cm} \sum_{j \leq 1}\sum_{\begin{subarray}{l} \Omega^{\beta} \in \mathbb{O}^{|\beta|} \\ |\beta| \leq n-1 \end{subarray}} \int_{\frac{1}{2}}^1 \int_{\theta \in \mathbb{S}^{n-1}} \left| (r\partial_r)^j  \int_{v \in \R^n} |\widehat{\Omega}^{\beta}h|(r,\theta,v)dv \right| d \theta  r^{n-1}dr,$$
which implies
$$\int_{v \in \R^n} |g(x,v)| dv \lesssim  \sum_{j \leq 1, | \beta| \leq n-1} \left\|  \int_{v \in \R^n} (r \partial_r)^{j}(|\widehat{\Omega}^{\beta} (g(|x|y,v))|) dv \right\|_{L^1(|y| \leq 1)},$$
It only remains to remark that, as $r\partial_r$ and $\Omega$ are homogeneous vector fields, 
$$(r \partial_r)^{j}\left(|\widehat{\Omega}^{\beta} \left(g(|x|y,v) \right)|\right)=(r \partial_r)^{j}\left(|\widehat{\Omega}^{\beta} g|\right)(|x|y,v)$$
and to make the change of variables $y'=|x|y$.
\end{proof}
\noindent We are now able to prove the following time dependent Sobolev inequality.

\begin{Lem}\label{pointwisehyperbo}

Let $g : [0,T^*[ \times \R^n_x \times \R^n_v \rightarrow \R$ a sufficiently regular function. For all $(t,x) \in [0,T^*[ \times \R^n$ such that $|x| \leq t$, we have
$$|x|^n\int_{v \in \R^n} |g(t,x,v)| dv \lesssim \sum_{|\beta| \leq n} \| \widehat{Z}^{\beta} g (\sqrt{|y|^2+a^2},y,v) \|_{L^1(|y| \leq |x|)L^1_v}  ,$$
with $a^2=t^2-|x|^2$ and $\widehat{Z}^{\beta} \in \K^{|\beta|}$ (more precisely the vector fields involved are either rotations or Lorentz boosts).

\end{Lem}

\begin{proof}
Let $(t,x) \in [0,T^*[ \times \R^n$ such that $|x| \leq t$ and $a^2=t^2-|x|^2$. We apply the previous lemma to $(y,v) \mapsto g(\sqrt{|y|^2+a^2},y,v)$ to get

\begin{flalign*}
& \hspace{0mm} |x|^n\int_{v \in \R^n} |g(t,x,v)| dv \lesssim &
\end{flalign*}
 $$ \hspace{2.1cm} \sum_{j \leq 1, | \beta| \leq n-1} \left\|  \int_{v \in \R^n} (r \partial_r)^{j}\left(|\widehat{\Omega}^{\beta} g|(\sqrt{|y|^2+a^2},y,v)\right) dv \right\|_{L^1(|y| \leq |x|)},$$
where we used that
$$\widehat{\Omega}^{\beta}\left( g(\sqrt{|y|^2+a^2},y,v) \right) = \widehat{\Omega}^{\beta}\left( g \right) (\sqrt{|y|^2+a^2},y,v) ,$$
since $\Omega(\sqrt{|y|^2+a^2})=0$ for all $\Omega \in \mathbb{O}$.

Now, we remark that
\begin{flalign*}
& \hspace{0mm} r \partial_r\left(|\widehat{\Omega}^{\beta} g|(\sqrt{|y|^2+a^2},y,v)\right)= &
\end{flalign*}
 $$\hspace{2.8cm} \frac{\widehat{\Omega}^{\beta} g(\sqrt{|y|^2+a^2},y,v)}{|\widehat{\Omega}^{\beta} g|(\sqrt{|y|^2+a^2},y,v)}\frac{y^i}{\sqrt{|y|^2+a^2}}\Omega_{0i}\widehat{\Omega}^{\beta}g(\sqrt{|y|^2+a^2},y,v).$$
Note also that, droping the dependance in $(\sqrt{|y|^2+a^2},y,v)$ of the functions considered,
$$\int_{v \in \R^n} \frac{\widehat{\Omega}^{\beta} g}{|\widehat{\Omega}^{\beta} g|}\frac{y^i}{\sqrt{|y|^2+a^2}}v^0\partial_{v^i}\widehat{\Omega}^{\beta}gdv=-\int_{v \in \R^n} \frac{v^i}{v^0}\frac{y_i}{\sqrt{|y|^2+a^2}}|\widehat{\Omega}^{\beta}g| dv.$$
It then comes that
\begin{flalign*}
& \hspace{0cm} |x|^n\int_{v \in \R^n} |g(t,x,v)| dv \lesssim  &
\end{flalign*}
 $$ \hspace{1.3cm} \sum_{| \beta| \leq n} \left\|\left( 1+\frac{|y|}{\sqrt{|y|^2+a^2}} \right)  \int_{v \in \R^n} |\widehat{Z}^{\beta} g|(\sqrt{|y|^2+a^2},y,v) dv \right\|_{L^1(|y| \leq |x|)},$$
which allows us to deduce the result.
\end{proof}

\subsection{An energy estimate}

Before starting the proof of Theorem \ref{decayestimate}, we establish the following lemma, which combined with our last Sobolev inequality, will give us the expected decay on the velocity average of the Vlasov field for a spacetime region.

\begin{Lem}\label{energyhyperbo}

Let $a>0$, $T \in ]a,T^*[$ and $g: [0,T^*[ \times \R^3_x \times \R^3_v \rightarrow \R$ be a sufficiently regular function. Then,
\begin{eqnarray}
\nonumber \int_{|y| \leq \sqrt{T^2-a^2}} \int_{\R^n_v} |g(\sqrt{|y|^2+a^2},y,v)| \frac{dv}{(v^0)^2} dy & \leq & 2\int_{\Sigma_0} \int_{\R^n_v} |g| dx dv \\ \nonumber
& & \hspace{-1cm} +2\int_0^T \int_{\Sigma_s} \int_{\R^n_v}\left| T_F(g) \right|\frac{dv}{v^0} dxds.
\end{eqnarray}
\end{Lem}

\begin{proof}
We use again the vector field $N^{\mu}(g):=\int_{\R^n_v} g \frac{v^{\mu}}{v^0} dv$ and recall from \eqref{eq:divN} that
$$ \partial_{\mu} N^{\mu}(|g|)= \int_{v \in \R^n} \frac{g}{|g|}\frac{T_F(g)}{v^0}-\frac{g}{|g|}F\left( \frac{v}{v^0},\nabla_v g \right) dv= \int_{v \in \R^n} \frac{g}{|g|}T_F(g) \frac{dv}{v^0}.$$
We now introduce the following subset of $\R_+ \times \R^n$ :
$$D_a(T)= \{(s,y) \in \mathbb{R}_+ \times \mathbb{R}^n \hspace{1mm} / \hspace{1mm} a^2 \geq s^2-|y|^2, \hspace{1mm} 0\leq s \leq T \} .$$
Denoting by $\nu$ is the outward pointing unit normal field to $\partial D_a(T)$, the divergence theorem (in $W^{1,1}$, for the euclidian space $\R^{n+1}$) gives us
$$\int_{\partial D_a(T)} \nu_{\mu} N^{\mu}(|g|) d \partial D_a(T) = \int_{D_a(T)}\int_{v \in \R^n} \frac{g}{|g|} T_F(g) \frac{dv}{v^0}dy d D_a(T).$$
The boundary term is equal to 
\begin{flalign*}
& \hspace{0mm} \int_{|y| \leq \sqrt{T^2-a^2}} \int_{v \in \R^n} \nu_{\mu} \frac{v^{\mu}}{v^0}|g|(\sqrt{|y|^2+a^2},y,v) dv d \lambda (y) &
\end{flalign*}
 $$ \hspace{4.4cm} +\|g\|_{L^1_x(|y| \geq \sqrt{T^2-a^2})L^1_v}(T)-\|g\|_{L^1_xL^1_v}(0),$$
where $d\lambda(y)$ is the surface measure on the hyperboloid $\{s^2-|y|^2=a^2 \}$. More precisely, on this hyperboloid\footnote{Here, $^ty$ denotes the transpose of $y$.},
$$d \lambda(y)= \sqrt{ det\left( I_n+\frac{1}{|y|^2+a^2} \;^tyy \right) } dy=\sqrt{\frac{2|y|^2+a^2}{|y|^2+a^2}} $$ and $$ \nu(y)=\frac{1}{\sqrt{2|y|^2+a^2}}(\sqrt{|y|^2+a^2},-y) .$$
We then deduce, as $D_a(T) \subset [0,T] \times \R^n$,
$$\int_{|y| \leq \sqrt{T^2-a^2}} \int_{\R^n_v} \left(\sqrt{|y|^2+a^2}v^0- y_iv^i \right)|g|(\sqrt{|y|^2+a^2},y,v) \frac{dvdy}{v^0\sqrt{|y|^2+a^2}} \leq $$ $$\hspace{5.1cm} \|g\|_{L^1_xL^1_v}(0) +\int_0^T \int_{\Sigma_s} \int_{\R^n_v} \left| T_F(g) \right| \frac{dv}{v^0}.$$
Finally, note that for $s=\sqrt{|y|^2+a^2} \geq |y|$, $$\frac{sv^0-y_iv^i}{s} \geq \frac{sv^0-|y||v|}{s} \geq s\frac{(v^0-|v|)(v^0+|v|)}{s(v^0+|v|)} \geq \frac{1}{2v^0} .$$
The result follows from a combination of the last two inequalities.
\end{proof}

\begin{Rq}
The lemma is also valid on the cone $s=|y|$, which means that the result is true for $a=0$, but we already knew it with Proposition \ref{energyfsimple}. 
\end{Rq}

\subsection{Proof of Theorem \ref{decayestimate}}

We consider a partition of the spacetime into four regions.
\begin{itemize}
\item The bounded region, $t+|x| \leq 2$, where a standard Sobolev inequality gives the result.
\item The interior of the light cone, where $|x| \leq \frac{t}{2}$.
\item The exterior of the light cone, where $t \leq |x|$ and $|x| \geq 1$.
\item The remaining region where $\frac{t}{2} \leq |x| \leq t$ and $t \geq 1$.
\end{itemize}

\subsubsection{The interior of the light cone}

Let $(t,x) \in [0,T^*[ \times \R^n$ such that $|x| \leq \frac{t}{2}$. Thus, $\tau_- \geq \frac{1}{3}\tau_+$ and the Klainerman-Sobolev inequality of Theorem \ref{KS1} gives
$$\int_{v \in \R^n} |f(t,x,v)| dv \lesssim \frac{1}{\tau_+^n} \sum_{|\beta| \leq n} \left\| \widehat{Z}^{\beta} f \right\|_{L^1_{x,v}} (t).$$
It only remains to apply Proposition \ref{energyfsimple}, which gives us
$$ \left\| \widehat{Z}^{\beta} f \right\|_{L^1_{x,v}} (t) \lesssim \left\| \widehat{Z}^{\beta} f \right\|_{L^1_{x,v}} (0)+ \int_0^t \int_{\Sigma_s} \int_{\R^n_v} \left| T_F \left( \widehat{Z}^{\beta} f \right) \right| \frac{dv}{v^0} dx ds.$$
\subsubsection{The exterior of the light cone}\label{indefinitedecay}

We use $(x^i-t\frac{v^i}{v^0}) \in \mathbf{k}_1$, for $1 \leq i \leq n$, which are solutions to the homogeneous relativistic transport equation. Let $(t,x) \in [0,T^*[ \times \R^n$ such that $ t \leq |x|$ and $|x| \geq 1$. By Theorem \ref{KS1}, we have
$$\int_{\R^n_v} \left| x^i-t\frac{v^i}{v^0} \right| |f|(t,x,v) dv  \lesssim \frac{1}{\tau_+^{n-1}} \sum_{z \in \mathbf{k}_1} \sum_{|\beta| \leq n} \|z \widehat{Z}^{\beta} f\|_{L^1_{x,v}}(t).$$
Since $|xv^0-tv| \geq v^0|x|-t|v| \geq \frac{|x|}{2v^0}$, it comes
\begin{eqnarray}
\nonumber |x|  \int_v |f|(t,x,v) \frac{dv}{(v^0)^2}  & \lesssim &  \int_v \left| x-t\frac{v}{v^0} \right| |f|(t,x,v) dv  \\ \nonumber 
& \lesssim &  \sum_{i=1}^n \int_v \left| x^i-t\frac{v^i}{v^0} \right||f|(t,x,v) dv .
\end{eqnarray}
Hence, using that $|x| \gtrsim \tau_+$ (recall that $|x| \geq 1$ and $t \leq |x|$ in the region studied) and applying Proposition \ref{energyfsimple}, we finally obtain
\begin{eqnarray}
\nonumber \int_v \hspace{-0.3mm} |f|(t,x,v) \frac{dv}{(v^0)^2} \hspace{-2.2mm} & \lesssim & \hspace{-2.2mm} \frac{1}{|x|\tau_+^{n-1}} \sum_{z \in \mathbf{k}_1} \sum_{|\beta| \leq n} \|z \widehat{Z}^{\beta} f\|_{L^1_{x,v}}(t) \\ \nonumber
& \lesssim & \hspace{-2.2mm} \frac{1}{\tau_+^n} \hspace{-0.4mm} \sum_{\begin{subarray}{}  |\beta| \leq n \\ z \in \mathbf{k}_1 \end{subarray}} \left\| z \widehat{Z}^{\beta} f \right\|_{L^1_{x,v}} \hspace{-1.5mm} (0) \hspace{-0.3mm} + \hspace{-0.3mm} \int_0^t \hspace{-0.4mm} \int_{\Sigma_s} \hspace{-0.4mm} \int_{v} \left| T_F \left(z \widehat{Z}^{\beta} f \right) \right| \frac{dv}{v^0} dx ds.
\end{eqnarray}
\subsubsection{The remaining region }

Let $(t,x) \in [0,T^*[ \times \R^n$ such that $ \frac{t}{2} \leq |x| \leq t$ and $t \geq 1$.
We start by applying Lemma \ref{energyhyperbo} to $\widehat{Z}^{\beta} f$, for all $|\beta| \leq n$, with $T=t$ and $a^2=t^2-|x|^2$. We have

\begin{flalign*}
& \hspace{0mm} \sum_{|\beta| \leq n} \int_{|y| \leq |x|} \int_{\R^n_v} |\widehat{Z}^{\beta} f(\sqrt{|y|^2+a^2},y,v)| \frac{dv}{(v^0)^2} dy \lesssim &
\end{flalign*}
 $$ \hspace{2.7cm} \sum_{|\beta| \leq n} \left( \int_{\Sigma_0} \int_{\R^n_v} \left| \widehat{Z}^{\beta}f \right| dv dx+\int_0^t \int_{\Sigma_s} \int_{\R^n_v} T_F(\widehat{Z}^{\beta}f) \frac{dv}{v^0} dx ds \right).$$
As $|\widehat{Z}^{\gamma} \left( (v^0)^{-2} \right) | \lesssim (v^0)^{-2}$, Lemma \ref{pointwisehyperbo} applied to $g=(v^0)^{-2} f$ allows us to bound by below the left hand side of the previous inequality by
$$ |x|^n \int_v |f|(t,x,v) \frac{dv}{(v^0)^2} .$$
The result follows from $|x|^n \gtrsim \tau_+^n$ (as $ |x| \geq \frac{t}{2} \geq \frac{1}{2}$).

\subsection{Improved decay for the derivatives of the velocity averages}\label{optidecayderiv}

Let us introduce the following vector fields.

\begin{Def}
For $1 \leq i \leq n$ and $1 \leq k,l \leq n$, with $k \neq l$, we consider
$$X_i=\frac{v^i}{v^0}\partial_t+\partial_i \hspace{3mm} \text{and} \hspace{3mm} Y_{kl}=\frac{v^k}{v^0} \partial_l-\frac{v^l}{v^0} \partial_k.$$

\end{Def}

\begin{Pro}
The vector fields $\frac{1}{v^0}T_1$, $X_i$ and $Y_{kl}$ are good derivatives (as the derivates tangential to the light cone $L$ and $e_B$, see Remark \ref{extradecay}), which means that if $W$ denotes one of them, we have, for a smooth function $f$,
$$\left|\int_v Wf dv\right| \lesssim \frac{1}{\tau_+}\left(\sum_{\widehat{Z} \in \K} \left| \int_v \widehat{Z} f dv \right|+\sum_{z \in \mathbf{k}_1} \int_v |z| |\nabla_{t,x} f| dv\right)$$
\end{Pro}

\begin{proof}
For $T_1$, we remark that
$$tT_1=v^0S+(tv^i-x^iv^0)\partial_i, \hspace{3mm} rT_1=tT_1+(r-t)T_1$$
and that $$\left| (r-t) \int_v \partial f dv \right| \lesssim \sum_{ \widehat{Z} \in \mathbb{K}} \left| \int_v Z f dv \right| \lesssim \sum_{\begin{subarray}{l} \widehat{Z}^{\beta} \in \K^{|\beta|} \\ \hspace{1mm} |\beta| \leq 1 \end{subarray}} \left| \int_v \widehat{Z}^{\beta} f dv \right|.$$ For $X_i$, that ensues from
$$tv^0X_i=v^0 \Omega_{0i}+(tv^i-x^iv^0)\partial_t \hspace{3mm} \text{and} \hspace{3mm} rX_i=tX_i +(r-t)X_i.$$
For $Y_{kl}$, that follows from
$$tv^0Y_{kl}=v^0\Omega_{kl}+(tv^k-x^kv^0)\partial_l-(tv^l-x^lv^0)\partial_k \hspace{3mm} \text{and} \hspace{3mm} rY_{kl}=tY_{kl} +(r-t)Y_{kl}.$$

\end{proof}
Finally, let us show how we can obtain extra decay on $\partial \int_v f dv$ if $f$ solves an equation such as $T_F(f)=0$.
\begin{Pro}
Let $f : [0,T] \times \R^n_x \times \R^n_v \rightarrow \R$ be a function such that
$$\forall \hspace{0.5mm} \widehat{Z} \in \K, \hspace{1mm} z \in \mathbf{k}_1, \hspace{3mm} \left| \int_v (v^0)^2z \widehat{Z} fdv\right| \lesssim \tau_+^{-n} .$$
Then, for all $0 \leq \mu \leq n$,
$$\left| \partial_{\mu} \int_v f dv \right| \lesssim \tau_+^{-n-1}.$$
\end{Pro}

\begin{proof}
As $$T_1=v^{\mu}\partial_{\mu}=v^0 \partial_t +v^iX_i-\frac{|v|^2}{v^0}\partial_t=v^iX_i+\frac{1}{v^0}\partial_t,$$
we have
$$\partial_t=v^0T_1-v^0v^iX_i.$$
Similarly
$$\partial_i=(v^0)^2X_i-v^iT_1-v^0v^kY_{ki}.$$
\end{proof}

\begin{Rq}
We can prove a similar proposition for derivatives of higher orders.

\end{Rq}

\section{The massive Vlasov-Maxwell equations}\label{section6}

\subsection{Global existence for small data}

The aim of this section is to prove Theorem \ref{intromasstheo}. We suppose that the dimension $n$ is at least $4$ and we consider the massive Vlasov-Maxwell system \eqref{syst1}-\eqref{syst3} with at least two species, so that $K \geq 2$. For simplicty, we suppose that $m_k=1$ for all $1 \leq k \leq K$.

To simplify the notation, we denote during this chapter the energy norm $\mathbb{E}^k_{M,q,1}$, introduced previously in Definition \ref{norm1}, by $\mathbb{E}^k_{M,q}$ . We also introduce the function $\chi$ defined on $\R_+$ by
$$ \chi ( s )= \log^3{(3+s)} \hspace{3mm} \text{if} \hspace{3mm} n=4 \hspace{3mm} \text{and} \hspace{3mm} \chi(s) = 1 \hspace{3mm} \text{if} \hspace{3mm} n \geq 5.$$ 

This is a more precise version of Theorem \ref{intromasstheo}.

\begin{Th}\label{massif}

Let $n \geq 4$, $K \geq 2$ and $N \geq \frac{5}{2}n+1$. Let $(f_0,F_0)$ be an initial data set for the massive Vlasov-Maxwell system. Let $(f,F)$ be the unique classical solution to the system and let $A$ be a potential in the Lorenz gauge. There exists $\epsilon >0$ such that\footnote{A smallness condition on $F$, which implies $\widetilde{\mathcal{E}}_N[A](0) \leq \epsilon$, is given in Proposition \ref{LorenzPot}.}, if 
\begin{flalign*}
& \hspace{25mm} \widetilde{\mathcal{E}}_N[A](0) \leq \epsilon, \hspace{20mm} \mathcal{E}_N[F](0) \leq \epsilon &
\end{flalign*}
and if, for all $1 \leq k \leq K$,
\begin{flalign*}
& \hspace{25mm} \mathbb{E}^2_{N+n,1}[f_k](0) \leq \epsilon, &
\end{flalign*} 
then $(f,F)$ exists globally in time and verifies the following estimates. 
\begin{itemize}
\item Energy bounds for $A$, $F$ and $f_k$: $ \forall$ $1 \leq k \leq K$ and $\forall \hspace{1mm} t \in \mathbb{R}_+$,
\begin{flalign*}
& \hspace{20mm} \widetilde{\mathcal{E}}_N[A](t) \lesssim \epsilon\chi(t), \hspace{15mm} \mathcal{E}_N[F](t) \lesssim \epsilon\chi(t), & \\
& \hspace{20mm} \mathbb{E}^2_{N}[f_k](t) \lesssim \epsilon \hspace{8mm} \text{and} \hspace{8mm}  \mathbb{E}^2_{N,1}[f_k](t) \lesssim \epsilon \chi^{\frac{1}{6}}(t). &
\end{flalign*}

\item Pointwise decay for the null decomposition of $\mathcal{L}_{Z^{\beta}}(F)$: $\forall$ $|\beta| \leq N-n$, $(t,x) \in \mathbb{R}_+ \times \mathbb{R}^n$,
\begin{flalign*}
& \hspace{7mm} |\alpha ( L_{Z^{\beta}} F)| \lesssim \sqrt{\epsilon} \sqrt{\chi(t)}\tau_+^{-\frac{n+2}{2}}, \hspace{9mm} |\underline{\alpha} ( L_{Z^{\beta}} F)| \lesssim \sqrt{\epsilon} \sqrt{\chi(t)}\tau_+^{-\frac{n-1}{2}}\tau_-^{-\frac{3}{2}}, & \\
& \hspace{7mm} |\rho (L_{Z^{\beta}} F)| \lesssim \sqrt{\epsilon} \sqrt{\chi(t)}\tau_+^{-\frac{n+1}{2}}\tau_-^{-\frac{1}{2}}, \hspace{3mm} |\sigma (L_{Z^{\beta}} F)| \lesssim \sqrt{\epsilon} \sqrt{\chi(t)}\tau_+^{-\frac{n+1}{2}}\tau_-^{-\frac{1}{2}}. &
\end{flalign*}
\item Pointwise decay for $\int_{v \in \mathbb{R}^n } | \widehat{Z}^{\beta} f_k | dv$: $$ \forall \hspace{0.5mm} |\beta| \leq N-\frac{3n+2}{2}, \hspace{0.5mm} (t,x) \in \mathbb{R}_+ \times \mathbb{R}^n, \hspace{1cm} \int_{v \in \mathbb{R}^n  } | \widehat{Z}^{\beta} f_k | dv \lesssim \frac{\epsilon}{\tau_+^{n}}.$$
\item Pointwise decay for $\int_{v \in \mathbb{R}^n  } | \widehat{Z}^{\beta} f_k |(v^0)^2 dv$ and $\int_{v \in \mathbb{R}^n  } |z\widehat{Z}^{\beta} f_k | (v^0)^2dv$: 
$$ \forall \hspace{0.5mm} |\beta| \leq N-n, \hspace{0.5mm} (t,x) \in \mathbb{R}_+ \times \mathbb{R}^n, \int_{v \in \mathbb{R}^n } | \widehat{Z}^{\beta} f _k|(v^0)^2 dv \lesssim \frac{\epsilon}{\tau_+^{n-1}\tau_-},$$
$$ \hspace{-3cm} \forall \hspace{0.5mm} |\beta| \leq N- \frac{3n+2}{2}, \hspace{0.5mm} z \in \mathbf{k}_1, \hspace{0.5mm} (t,x) \in \mathbb{R}_+ \times \mathbb{R}^n, $$ $$ \hspace{5.4cm} \int_{v \in \mathbb{R}^n } |z \widehat{Z}^{\beta} f_k | (v^0)^2dv \lesssim \frac{\epsilon}{\tau_+^{n-1}\tau_-}.$$
\item $L^2$ estimates on $\int_{v \in \mathbb{R}^n } | \widehat{Z}^{\beta} f_k | dv$:
$$ \forall \hspace{0.5mm} |\beta| \leq N, \hspace{0.5mm} t \in \mathbb{R}_+, \hspace{1cm} \left\| \int_{v \in \mathbb{R}^n  } | \widehat{Z}^{\beta} f_k | dv \right\|_{L^2(\Sigma_t)} \lesssim \frac{\epsilon\chi^{\frac{1}{6}}(t)}{(1+t)^{\frac{n}{2}}}.$$
\end{itemize}
\end{Th}

\begin{Rq}
In dimension $4$, if $N \geq 14$, we can take $\chi(t)=\log^2 (3+t)$ and avoid the $\log^{\frac{1}{2}}(3+t)$-loss on the $L^2$ estimate on $\int_v |\widehat{Z}^{\beta} f_k| dv$.
\end{Rq}

\subsection{Structure and beginning of the proof}

Let $(f_0,F_0)$ be an initial data set satisfying the assumptions of Theorem \ref{massif}. By a standard local well-posedness argument, there exists a unique maximal solution $(f,F)$ of the massive Vlasov-Maxwell system defined on $[0,T^*[$, with $T^* \in \R_+^* \cup \{+ \infty \}$.  

We consider the following bootstrap assumptions. Let $T$ be the largest time such that, $\forall$ $1 \leq k \leq K$ and $\forall \hspace{1mm} t \in [0,T]$,

\begin{equation}\label{bootF}
\hspace{-5cm} \mathcal{E}_N[F](t) \leq 2C\epsilon \chi(t), \hspace{5mm} \mathcal{E}_N^{S}[F](t) \leq 2\overline{C}\epsilon,
\end{equation}
\begin{equation}\label{bootf}
 \mathbb{E}^2_{N}[f_k](t) \leq 4\epsilon, \hspace{2mm} \mathbb{E}^2_{N-\frac{n+2}{2},1}[f_k](t) \leq 4\epsilon \hspace{2mm} \text{and} \hspace{2mm}  \mathbb{E}^2_{N,1}[f_k](t) \leq 4\epsilon \chi^{\frac{1}{6}}(t), 
\end{equation}

where $C$ and $\overline{C}$ are positive constants which will be specified during the proof. Note that by continuity, $T>0$. We now present our strategy to improve these bootstrap assumptions.

\begin{enumerate}
\item First, using the bootstrap assumptions, we obtain decay estimates for the null decomposition of $F$ (and its Lie derivatives) and for velocity averages of derivatives of $f_k$.

\item Next, we improve the bounds on the Vlasov fields energies by means of the energy estimates proved in Propositions \ref{energyf} and \ref{energypoids}. To bound the right hand side in these energy estimates, we make fundamental use of the null structure of the system and the pointwise decay estimates on $\rho$, $\sigma$, $\alpha$, $\underline{\alpha}$ and $\int_{v \in \mathbb{R}^n  } |z \widehat{Z}^{\beta} f_k | dv$.

\item Then, using Theorem \ref{decayestimate}, we improve the decay estimate on $\int_v \hspace{-0.5mm} |\widehat{Z}^{\beta} f_k | dv$ near the light cone.

\item In order to improve the estimates on the electromagnetic field energies, we establish an $L^2_x$ estimate for the velocity averages of the Vlasov fields (and its derivatives). For this purpose, we follow \cite{FJS} and we rewrite all the transport equations as an inhomogeneous system of transport equations. The velocity averages of the homogeneous part of the solution verify strong pointwise decay estimates (we use particularly the control that we have at our disposal on the initial data of $f$, for derivatives of order $N+n$ or less). The inhomegeneous part is decomposed into a product of an integrable function and a pointwise decaying function which gives us the expected estimate.

\item Finally, we bound the energy of the electromagnetic potential (which satisfy the Lorenz gauge) and we improve the estimates on the electromagnetic field energies with the energy estimates for the Maxwell equations (Propositions \ref{MoraF} and \ref{scalF}). We use again the null decomposition of $F$ (and its Lie derivatives), which, combined by the estimates on $ \left\| \tau_+ \int_{\mathbb{R}^n  } | \widehat{Z}^{\beta} f_k | dv \right\|_{L^2_x}$, gives us the improvement.

\end{enumerate}

\subsection{Step 1: Decay estimates}

Using the bootstrap assumption on $\mathcal{E}_N[F]$ and Proposition \ref{decayFpointwise}, one immediately obtains the following pointwise decay estimates on the electromagnetic field.

\begin{Pro}\label{decayFmass}

For all $t \in [0,T]$, $|\beta| \leq N- \frac{n+2}{2}$, we have
$$|\alpha ( \mathcal{L}_{Z^{\beta}} F)| \lesssim \sqrt{\epsilon} \sqrt{\chi(t)} \tau_+^{-\frac{n+1}{2}}\tau_-^{-\frac{1}{2}}, \hspace{2mm} |\underline{\alpha} ( \mathcal{L}_{Z^{\alpha}} F)| \lesssim \sqrt{\epsilon} \sqrt{\chi(t)} \tau_+^{-\frac{n-1}{2}}\tau_-^{-\frac{3}{2}}, $$
$$|\rho (\mathcal{L}_{Z^{\alpha}} F)| \lesssim \sqrt{\epsilon} \sqrt{\chi(t)} \tau_+^{-\frac{n+1}{2}}\tau_-^{-\frac{1}{2}}, \hspace{2mm} |\sigma (\mathcal{L}_{Z^{\alpha}} F)| \lesssim \sqrt{\epsilon} \sqrt{\chi(t)} \tau_+^{-\frac{n+1}{2}}\tau_-^{-\frac{1}{2}}.$$

\end{Pro}

\begin{Rq}
We will improve later the decay estimate on $\alpha(\mathcal{L}_{Z^{\beta}} F)$, for $|\beta| \leq N-n$, near the light cone (see Section \ref{sectionpotential}).
\end{Rq}

The pointwise decay estimates on the velocity averages of the Vlasov fields are given by Klainerman-Sobolev inequalities and the bootstrap assumptions on the $f_k$ energy norms. Using Theorem \ref{KS1}, we have that $\forall \hspace{1mm}|\beta| \leq N-n$, $ (t,x) \in [0,T] \times \R^n$, $1 \leq k \leq K$

\begin{equation}\label{decayf1}
 \int_{\R^n} |\widehat{Z}^{\beta} f_k|(v^0)^2 dv \lesssim  \frac{\mathbb{E}^2_N[f_k](t)}{\tau_+^{n-1} \tau_-} \lesssim \frac{\epsilon}{\tau_+^{n-1} \tau_-}.
\end{equation}

In the same spirit, using Corollary \ref{KS3}, we have that $\forall$ $|\beta| \leq N-\frac{3n+2}{2}$, $ z  \in \mathbf{k}_1$, $ (t,x) \in [0,T] \times \R^n$,

\begin{equation}\label{decayf2}
 \int_{\R^n} |z\widehat{Z}^{\beta} f_k| (v^0)^2dv \lesssim \frac{\mathbb{E}^2_{N-\frac{n+2}{2},1}[f_k](t)}{\tau_+^{n-1} \tau_-} \lesssim \frac{\epsilon}{\tau_+^{n-1} \tau_-},
\end{equation}

\subsection{Step 2: Improving the energy estimates for the transport equation}\label{step2energy}

We fix, for this section, $1 \leq k \leq K$. According to Proposition \ref{energyf}, $\mathbb{E}^2_{N}[f_k] \leq 3 \epsilon$ on $[0,T]$, for $\epsilon$ small enough, would follow if we prove
\begin{equation}\label{eq:noweight}
 \int_0^t \int_{\Sigma_s} \int_{v \in \R^n} |\mathcal{L}_{Z^{\beta_1}}(F)(v,\nabla_v \widehat{Z}^{\beta_2}f_k)| v^0dv dx ds \lesssim \epsilon^{\frac{3}{2}},
 \end{equation}
for all $|\beta_1|+|\beta_2| \leq N$, with $|\beta_2| \leq N-1$, and
 $$ \int_0^t \int_{\Sigma_s} \int_{v \in \R^n} |v^iF_{i0} \widehat{Z}^{\beta} f_k|   dvdxds \lesssim \epsilon^{\frac{3}{2}},$$
 for all $|\beta| \leq N$. The second integral is easy to bound. Using Proposition \ref{decayFmass} and the bootstrap assumption on $\mathbb{E}^2_N[f_k]$, we have
 \begin{eqnarray}
 \nonumber \int_0^t \int_{\Sigma_s} \int_{\R^n} |v^iF_{i0} \widehat{Z}^{\beta} f_k|   dvdxds & \lesssim & \int_0^t \|F \|_{L^{\infty}(\Sigma_s)} \mathbb{E}^2_N[f_k](s) ds \\ \nonumber
 & \lesssim & \epsilon^{\frac{3}{2}}.
 \end{eqnarray}
Similarly, according to Proposition \ref{energypoids}, $\mathbb{E}^2_{N,1}[f_k] \leq 3 \epsilon \chi^{\frac{1}{6}} (t)$ on $[0,T]$, for $\epsilon$ small enough, would follow if we prove
\begin{equation}\label{eq:weight}
 \int_0^t \int_{\Sigma_s} \int_{v \in \R^n} |zv^0\mathcal{L}_{Z^{\beta_1}}(F)(v,\nabla_v \widehat{Z}^{\beta_2}f_k)| dv dx ds \lesssim \epsilon^{\frac{3}{2}}\chi^{\frac{1}{6}} (t),
 \end{equation}
for all $z \in \mathbf{k}_1$ and $|\beta_1|+|\beta_2| \leq N$, with $|\beta_2| \leq N-1$,
\begin{equation}\label{eq:weight2}
\int_0^t \int_{\Sigma_s} \int_{v \in \R^n} |v^0F(v,\nabla_v z)\widehat{Z}^{\beta} f_k| dvdxds \lesssim \epsilon^{\frac{3}{2}}\chi^{\frac{1}{6}} (t),
\end{equation}
for all $z \in \mathbf{k}_1$, $|\beta| \leq N$ and
$$ \int_0^t \int_{\Sigma_s} \int_{v \in \R^n} |zv^iF_{i0} \widehat{Z}^{\beta} f_k|   dvdxds \lesssim \epsilon^{\frac{3}{2}},$$
for all $|\beta| \leq N$. Again, the last integral is easy to bound.

We fix $|\beta_1|+|\beta_2| \leq N$ (with $|\beta_2| \leq N-1$), $|\beta| \leq N$ and $z \in \mathbf{k}_1$. We denote respectively $\rho(\mathcal{L}_{Z^{\beta_1}}(F))$, $\sigma(\mathcal{L}_{Z^{\beta_1}}(F))$, $\alpha(\mathcal{L}_{Z^{\beta_1}}(F))$ and $\underline{\alpha}(\mathcal{L}_{Z^{\beta_1}}(F))$ by $\rho$, $\sigma$, $\alpha$ and $\underline{\alpha}$. We denote also $\widehat{Z}^{\beta_2}f_k$ by $g$ and $\widehat{Z}^{\beta} f_k$ by $h$. To unify the study of the remaining integrals, we introduce $b$, which could be equal to $0$ or $1$, $z_0=v^0$ and $z_b=v^0z$. The null decomposition of $\mathcal{L}_{Z^{\beta_1}}(F)(v,\nabla_v g)$ (for \eqref{eq:noweight} and \eqref{eq:weight}) or $F(v,\nabla_v z)$ (for \eqref{eq:weight2}) brings us to control the integral of the following terms.

The good terms 

\begin{equation}\label{nonlin1}
 \left|z_bv^{\underline{L}} \rho \left( \nabla_v g \right)^L\right|, \hspace{5mm} \left|v^0v^{\underline{L}} h\rho(F) \left( \nabla_v z \right)^L\right|,
\end{equation}

\begin{equation}
\left|z_bv^L \rho \left( \nabla_v g \right)^{\underline{L}}\right|, \hspace{5mm} \left|v^0v^L h \rho(F) \left( \nabla_v z \right)^{\underline{L}}\right|,
\end{equation}

\begin{equation}
\left|z_bv^B \sigma_{BD} \left( \nabla_v g \right)^D\right|, \hspace{5mm} \left|v^0v^B h \sigma(F)_{BD}\left( \nabla_v z \right)^D\right|,
\end{equation}

\begin{equation}
\left|z_bv^L \alpha_B \left( \nabla_v g \right)^{B}\right|, \hspace{5mm} \left|v^0v^L h \alpha(F)_B \left( \nabla_v z \right)^B \right|,
\end{equation}

\begin{equation}\label{nonlin5}
\left|z_bv^B \alpha_B \left( \nabla_v g \right)^{L}\right|, \hspace{5mm} \left|v^0v^B h \alpha(F)_B \left( \nabla_v z \right)^{L}\right|,
\end{equation}

and the bad terms 

\begin{equation}\label{nonlin6}
\left|z_bv^{\underline{L}} \underline{\alpha}_B \left( \nabla_v g \right)^B\right|, \hspace{5mm} \left|v^0v^{\underline{L}}  h\underline{\alpha}(F)_B \left( \nabla_v z \right)^B\right|,
\end{equation}

\begin{equation}\label{nonlin7}
\left|z_bv^B \underline{\alpha}_B \left( \nabla_v g \right)^{\underline{L}}\right|, \hspace{5mm} \left|v^0v^B h \underline{\alpha}(F)_B \left( \nabla_v z \right)^{\underline{L}} \right|.
\end{equation}

The study of $\mathbb{E}^2_N[f_k]$ corresponds to $b=0$ and, in this case, we only have to estimate the spacetime integral of each of the first terms of \eqref{nonlin1}-\eqref{nonlin7}. The study of $\mathbb{E}^2_{N,1}[f_k]$ corresponds to $b=1$. For both of them, when $|\beta_1| \leq N-\frac{n+2}{2}$ we can use the pointwise decay estimates on the electromagnetic field given by Proposition \ref{decayFmass}. When $|\beta_1| > N-\frac{n+2}{2}$, $|\beta_2| \leq N-\frac{3n+2}{2}$ (since $N \geq \frac{5}{2}n+1$), and we can then use the pointwise estimates \eqref{decayf1} and \eqref{decayf2} on the velocity averages of the Vlasov field.

For the part where $|\beta_1| \leq N-\frac{n+2}{2}$, our proof leads also to $\mathbb{E}_{N-\frac{n+2}{2},1}^2[f_k] \leq 3 \epsilon$, for $\epsilon$ small enough, on $[0,T]$.

\begin{Rq}

To simplify the argument we will sometimes denote $\mathbb{E}^2_{N}[f_k]$ by $\mathbb{E}^2_{N,0}[f_k]$.

\end{Rq}

\subsubsection{Estimating the $v$ derivatives}\label{sectionvderiv}

To deal with the $v$ derivatives of the Vlasov field, which do not commute with the relativistic transport operator, we recall \eqref{eq:null0}

\begin{equation}\label{eq:vderivbasic}
\left| \left( \nabla_v \psi \right)^{L} \right|, \hspace{1mm} \left| \left(\nabla_v \psi \right)^{\underline{L}} \right|, \hspace{1mm} \left| \left( \nabla_v \psi \right)^B \right| \lesssim \frac{\tau_+}{v^0}\sum_{\widehat{Z} \in \K} |\widehat{Z} \psi|.
\end{equation}
We will also use
\begin{equation}\label{eq:vderivL}
\left| \left( \nabla_v \psi \right)^{L} \right|, \hspace{1mm} \left| \left( \nabla_v \psi \right)^{\underline{L}} \right| \lesssim \frac{\tau_-}{v^0}\sum_{\widehat{Z} \in \K}|\widehat{Z}\psi|
\end{equation}
and
\begin{equation}\label{eq:vderivA}
\left| v^{\underline{L}} \left( \nabla_v \psi \right)^B  \right| \lesssim \tau_-\sum_{\widehat{Z} \in \widehat{\mathbb{P}} } \sum_{z \in \mathbf{k}_1}  |z\widehat{Z} \psi | ,
\end{equation}
which come from Lemma \ref{vderivL} and Proposition \ref{nullresume}. In order to reutilize certain estimates of this section, we will not use inequalities \eqref{eq:vderivL} and \eqref{eq:vderivA} in the case where we have a pointwise estimate on the electromagnetic field. We make this choice because we do not identify such null structures in the equations studied in Section \ref{L2system}, where we will make similar computations as in Subsection \ref{step2F}.

\subsubsection{If $|\beta_1| \leq N-\frac{n+2}{2}$}\label{step2F}

We start by treating the good terms. We use $\zeta$ to denote $\alpha$, $\rho$ or $\sigma$. Thus, according to Proposition \ref{decayFmass},
$$|\zeta| \lesssim \frac{\sqrt{\epsilon}\sqrt{\chi(t)}}{\tau_+^{\frac{n+1}{2}}\tau_-^{\frac{1}{2}}}.$$
Using \eqref{eq:vderivbasic}, we can bound by $\sum_{\widehat{Z} \in \K}\tau_+|\zeta||z_b\widehat{Z}g|$ each first term of \eqref{nonlin1}-\eqref{nonlin5} so that their integrals on $[0,t] \times \mathbb{R}^n_x \times \mathbb{R}^n_v$ are bounded by

\begin{equation}\label{boundgood}
\sum_{\widehat{Z} \in \widehat{\mathbb{P}}_0} \int_0^t \int_{\Sigma_s} \tau_+ |\zeta| \int_v |z_b\widehat{Z}g| dv dx ds.
\end{equation}
It remains to notice that $$\int_0^t \int_{\Sigma_s} \tau_+ |\zeta| \int_v |z_b\widehat{Z}g| dv dx ds \lesssim \int_0^t \sqrt{\epsilon}\frac{ \log^{\frac{3}{2}} (3+s)}{(1+s)^{\frac{3}{2}}} \mathbb{E}^{1}_{N,b}[f_k](s) ds \lesssim \epsilon^{\frac{3}{2}},$$
since $\mathbb{E}^{1}_{N,b}[f_k](s) \leq \mathbb{E}^{2}_{N,1}[f_k](s) \leq  4 \epsilon\log^{\frac{1}{2}}(3+s)$ for all $s \in [0,T]$. Similarly, each second term of \eqref{nonlin1}-\eqref{nonlin5} is bounded by $\sum_{\widehat{Z} \in \K}v^0\tau_+|\zeta(F)| |h| |\widehat{Z}(z)|$ and, using Lemma \ref{vectorweight}, their integral on $[0,t] \times \mathbb{R}^n_x \times \mathbb{R}^n_v$ are bounded by
$$\sum_{z' \in \mathbf{k}_1} \int_0^t \int_{\Sigma_s} \tau_+ |\zeta(F)| \int_v v^0|z'h| dv dxds.$$
Using the pointwise estimate on $\zeta$ and the bootstrap assumption \eqref{bootf}, one has
$$ \int_0^t \int_{\Sigma_s} \tau_+ |\zeta(F)| \int_v v^0|z'h| dv dxds \lesssim \int_0^t \sqrt{\epsilon}\frac{ \log^{\frac{3}{2}} (3+s)}{(1+s)^{\frac{3}{2}}} \mathbb{E}^{1}_{N,1}[f_k](s) ds \lesssim \epsilon^{\frac{3}{2}}.$$
We now study the bad terms. Recall that, according to Proposition \ref{decayFmass},
$$ |\underline{\alpha}| \lesssim \frac{\sqrt{\epsilon}\sqrt{\chi(t)}}{\tau_+^{\frac{n-1}{2}} \tau_-^{\frac{3}{2}}}.$$
Let us denote $[0,t] \times \R_x^n \times \R_v^n$ by $X_t$ and $dvdxds$ by $dX_t$. Then, using \eqref{eq:vderivbasic}, $\int_{C_u(t)}\int_v v^{\underline{L}}|z_b \widehat{Z} g| dv d C_u(t)  \leq \mathbb{E}^{2}_{N,b}[f_k](t) $ and the bootstrap assumption \eqref{bootf}, we have,
\begin{eqnarray}
\nonumber
   \int_{X_t} |z_bv^{\underline{L}} \underline{\alpha}_B \left( \nabla_v g \right)^B|dX_t \hspace{-2mm} & \lesssim & \hspace{-2mm} \sum_{\widehat{Z} \in \widehat{\mathbb{P}}_0} \int_0^t \int_{\Sigma_s}  \tau_+ |\underline{\alpha}| \int_v \frac{v^{\underline{L}}}{v^0}|z_b \widehat{Z} g| dvdxds \\ \nonumber
   & \lesssim & \hspace{-2mm} \sum_{\widehat{Z} \in \widehat{\mathbb{P}}_0} \int_{u=-\infty}^{t}  \int_{C_u(t)}\tau_+|\underline{\alpha}|\int_v v^{\underline{L}}|z_b \widehat{Z} g| dv dC_u(t) du \\  \nonumber
   & \lesssim & \hspace{-2mm} \sum_{\widehat{Z} \in \widehat{\mathbb{P}}_0} \int_{u=-\infty}^{t} \frac{1}{\tau_-^{\frac{3}{2}}} \int_{C_u(t)}\int_v v^{\underline{L}}| z_b\widehat{Z} g| dv dC_u(t) du \\  \nonumber
   & \lesssim & \hspace{-1mm} \epsilon^{\frac{1}{2}} \mathbb{E}^2_{N,b}[f_k](t)\int_{u=-\infty}^{+\infty} \frac{1}{\tau_-^{\frac{3}{2}}} du \\ \nonumber
   & \lesssim & \hspace{-1mm} \epsilon^{\frac{3}{2}} \hspace{2mm} \text{if} \hspace{2mm} b=0, \hspace{3mm} \epsilon^{\frac{3}{2}}\chi^{\frac{1}{6}}(t) \hspace{2mm} \text{if} \hspace{2mm} b=1. 
\end{eqnarray}
Finally, for the first term of \eqref{nonlin7}, we use successively \eqref{eq:vderivbasic}, the inequality $|v^B| \lesssim v^0 v^{\underline{L}}$ (which ensues from Proposition \ref{extradecay1}) as well as the bootstrap assumptions \eqref{bootf} to get

\begin{eqnarray}
\nonumber  
    \int_{X_t} | z_bv^B \underline{\alpha}_B \left( \nabla_v g \right)^{\underline{L}}|dX_t \hspace{-2mm} & \lesssim & \sum_{\widehat{Z} \in \widehat{\mathbb{P}}_0} \int_0^t \int_{\Sigma_s}  \int_v |\underline{\alpha}| v^0 v^{\underline{L}} \frac{\tau_+}{v^0}|z_b  \widehat{Z} g| dvdxds \\ \nonumber
    & \lesssim & \sum_{\widehat{Z} \in \K} \hspace{-0.1mm} \int_{u=-\infty}^{t} \hspace{-0.1mm} \int_{C_u(t)} \hspace{-0.2mm} \tau_+|\underline{\alpha}| \int_v v^{\underline{L}} |z_b \widehat{Z} g| dv d C_u(t) du \\ \nonumber
    & \lesssim & \epsilon^{\frac{1}{2}}\int_{u=-\infty}^{t} \tau_-^{-\frac{3}{2}} \mathbb{E}^2_{N,b}[f_k](t) du \\ \nonumber
   & \lesssim & \epsilon^{\frac{1}{2}} \mathbb{E}^2_{N,b}[f_k](t).
\end{eqnarray}
The integrals of the second terms of \eqref{nonlin6} and \eqref{nonlin7} are treated similarly. For instance, as $\sum_{\widehat{Z} \in \K} |\widehat{Z}(z)| \lesssim \sum_{z' \in \mathbf{k}_1} |z'|$, we have

\begin{eqnarray}
\nonumber  
    \int\limits_{X_t} |v^0 v^B h\underline{\alpha}_B \left( \nabla_v z \right)^{\underline{L}}|dX_t & \lesssim & \sum_{\widehat{Z} \in \widehat{\mathbb{P}}_0} \int_0^t \int_{\Sigma_s}  \tau_+ |\underline{\alpha}| \int_v \sqrt{v^Lv^{\underline{L}}} |h  \widehat{Z} (z)| dvdxds \\ \nonumber
    & \lesssim & \sum_{z' \in \mathbf{k}_1} \int_{-\infty}^{t}  \int_{C_u(t)} \tau_+ |\underline{\alpha}| \int_v v^{\underline{L}}v^0 |z' h| dv d C_u(t) du \\ \nonumber
   & \lesssim & \epsilon^{\frac{1}{2}} \mathbb{E}^2_{N,1}[f_k](t).
\end{eqnarray}

\subsubsection{If $|\beta_1| > N-\frac{n+2}{2}$}

In this case we cannot use Proposition \ref{decayFmass} anymore. As $|\beta_2| \leq \frac{n}{2}$, we can however use the pointwise estimates on the velocity averages of $v^0z_b\widehat{Z}^{\beta}g$ given by \eqref{decayf2}. This time, we only have to bound the first terms of \eqref{nonlin1}-\eqref{nonlin7}.

Again, we start by studying the good terms. Let us denote again $\alpha$, $\rho$ or $\sigma$ by $\zeta$ . Then, according to Definition \ref{norm2}, for all $s \in [0,T]$,

$$\int_{\Sigma_s} \tau_+^2 |\zeta|^2 dx \leq \mathcal{E}_N[F](s) \lesssim \epsilon\chi(s).$$
Recall that the integral of each first term of \eqref{nonlin1}-\eqref{nonlin5} can be bounded by \eqref{boundgood}. Using the Cauchy-Schwarz inequality and $$\left\|\int_v | z_b\widehat{Z} g| dv \right\|^2_{L^2(\Sigma_s)} \lesssim \epsilon^2 \int_0^{+\infty} \frac{r^{n-1}}{\tau_+^{2n-2}\tau_-^2}dr \lesssim \epsilon^2 (1+s)^{-(n-1)},$$
which comes from \eqref{decayf2} and Lemma \ref{intesti}, we have 
\begin{eqnarray}
\nonumber \int_0^t \int_{\Sigma_s} \tau_+ |\zeta| \int_v |z_b \widehat{Z} g| dv dx ds & \lesssim & \int_0^t \|\tau_+ \zeta\|_{L^2(\Sigma_s)}  \left\|\int_v |z_b \widehat{Z} g| dv \right\|_{L^2(\Sigma_s)} ds \\ \nonumber & \lesssim & \epsilon^{\frac{3}{2}}.
\end{eqnarray}

In order to close the estimates for the bad terms, we use \eqref{eq:vderivL} or \eqref{eq:vderivA}. The integral of the first term of \eqref{nonlin7} is then bounded by
$$  \sum_{\widehat{Z} \in \widehat{\mathbb{P}}_0}  \int_{0}^{ + \infty}  \| \tau_- \underline{\alpha} \|_{L^2(\Sigma_s)} \left\| \int_v \left|\frac{v^B}{v^0}z_b\widehat{Z} g \right| dv \right\|_{L^2(\Sigma_s)} ds.$$ 
Now, using \eqref{decayf2} and Lemma \ref{intesti}, we have
$$\left\| \int_v \left|\frac{v^B}{v^0}z_b\widehat{Z} g\right| dv \right\|_{L^2(\Sigma_s)} \lesssim \left\| \int_v |z_b\widehat{Z} g| dv \right\|_{L^2(\Sigma_s)} \lesssim \epsilon (1+s)^{-\frac{3}{2}}.$$
Since $\| \tau_- \underline{\alpha} \|^2_{L^2(\Sigma_s)} \leq \mathcal{E}_N[F](s) \lesssim \epsilon\chi(t)$,
$$ \int_0^t \int_{\Sigma_s} \int_v |z_bv^B \underline{\alpha}_B \left( \nabla_v g \right)^{\underline{L}}|dvdxds \lesssim \epsilon^{\frac{3}{2}}.$$
For the remaining term, $z_b v^{\underline{L}} \underline{\alpha}_B\left(\nabla_v g \right)^B$, we treat the two cases separately. First, if $b=0$, $\int_0^t \int_{\Sigma_s} \int_v v^0 v^{\underline{L}} |\underline{\alpha}_B\left(\nabla_v g \right)^B|dvdxds$ is bounded by

$$  \sum_{\widehat{Z} \in \widehat{\mathbb{P}}_0} \sum_{z' \in \mathbf{k}_1} \int_{0}^{ + \infty}  \| \tau_- \underline{\alpha} \|_{L^2(\Sigma_s)} \left\|  \int_v v^0|z'\widehat{Z}g| dv \right\|_{L^2(\Sigma_s)}ds .$$ 
Now, using \eqref{decayf2} and Lemma \ref{intesti}, we have

$$\left\| \int_v v^0|z'\widehat{Z}g| dv \right\|^2_{L^2(\Sigma_s)}  \lesssim \epsilon^2 (1+s)^{-(n-1)}.$$
Hence, as $\| \tau_- \underline{\alpha} \|^2_{L^2(\Sigma_s)} \leq \mathcal{E}_N[F](s) \lesssim \epsilon\chi(t)$, we obtain
$$ \int_0^t \int_{\Sigma_s} \int_v v^0 v^{\underline{L}} |\underline{\alpha}_B\left(\nabla_v g \right)^B|dvdxds \lesssim \epsilon^{\frac{3}{2}}.$$

Finally, if $b=1$, we have, by \eqref{eq:vderivbasic}, 
\begin{flalign*}
& \hspace{0cm} \int_0^t \int_{\Sigma_s} \int_v  v^0v^{\underline{L}} |z\underline{\alpha}_B\left(\nabla_v g \right)^B|dvdxds \lesssim &
\end{flalign*}
$$ \hspace{3cm} \sum_{\widehat{Z} \in \K} \int_0^t \int_{\Sigma_s}  \frac{\tau_-^{\frac{1}{2}}}{(1+s)^{\frac{n-3}{2}}}|\underline{\alpha}|\frac{\tau_+(1+s)^{\frac{n-3}{2}}}{\tau_-^{\frac{1}{2}}}\int_v v^{\underline{L}}| z\widehat{Z} g | dv dx ds.$$

By the Cauchy-Schwarz inequality (in $(s,x)$), the right-hand side of the previous inequality is bounded by
\begin{equation}\label{eq:logloss}
\left( \int_0^t \frac{\|\sqrt{\tau_-} |\underline{\alpha}| \|^2_{L^2(\Sigma_s)}}{(1+s)^{n-3}} ds \right)^{\frac{1}{2}}\sum_{\widehat{Z} \in \K} \left( \int_0^t \int_{\Sigma_s} \frac{\tau_+^{n-1}}{\tau_-}\left( \int_v v^{\underline{L}} |z \widehat{Z}g| dv \right)^2 dx ds \right)^{\frac{1}{2}}.
\end{equation}
By the bootstrap assumption\footnote{Note that if we used the bound on $\|\tau_- \underline{\alpha}\|_{L^2(\Sigma_s)}$ we would have in $4d$ an extra loss on $\mathbb{E}^2_{N,1}[f_k]$ which would lead to a $(1+t)^{\eta}$-loss for the electromagnetic energy.} \eqref{bootF}, $\|\sqrt{\tau_-} |\underline{\alpha}| \|^2_{L^2(\Sigma_s)} \lesssim \epsilon$, so
 $$\int_0^t \frac{\|\sqrt{\tau_-} |\underline{\alpha}| \|^2_{L^2(\Sigma_s)}}{(1+s)^{n-3}} ds \lesssim \epsilon \chi^{\frac{1}{3}}(t).$$
The second factor of \eqref{eq:logloss} is bounded by $\epsilon^2$. Indeed, as, by \eqref{decayf2},
$$\frac{\tau_+^{n-1}}{\tau_-}\left( \int_v v^{\underline{L}} |z \widehat{Z}g| dv \right)^2 \lesssim \frac{\epsilon }{\tau_-^2} \int_v v^{\underline{L}} |z \widehat{Z}g| dv,$$
we have
\begin{eqnarray}
\nonumber \int_0^t \int_{\Sigma_s} \frac{\tau_+^{n-1}}{\tau_-}\left( \int_v v^{\underline{L}} |z \widehat{Z}g| dv \right)^2 dx ds \hspace{-2mm} & \lesssim & \hspace{-2mm} \epsilon \hspace{-0.5mm} \int_{-\infty}^t \tau_-^{-2} \hspace{-0.5mm} \int_{C_u(t)} \int_v v^{\underline{L}}|z \widehat{Z}g| dv d C_u(t)du \\ \nonumber
& \lesssim & \hspace{-2mm} \epsilon \int_{u=-\infty}^t \tau_-^{-2} \mathbb{E}^2_{N-\frac{n+2}{2},1}[f_k](t) du \\ \nonumber
& \lesssim & \hspace{-2mm} \epsilon^2,
\end{eqnarray}
since $\mathbb{E}^2_{N-\frac{n+2}{2},1}[f_k](t) \leq 4 \epsilon$ by the bootstrap assumption \eqref{bootf}. Thus
$$\int_0^t \int_{\Sigma_s} \int_v  v^0v^{\underline{L}} |z\underline{\alpha}_B\left(\nabla_v g \right)^B|dvdxds \lesssim \epsilon^2 \chi^{\frac{1}{6}}(t).$$
This concludes the improvement of the bootstrap assumption \eqref{bootf}.

\subsection{Step 3: Improved decay estimates for velocity averages}\label{step3opti}
In this section, we improve the pointwise decay estimate on $\int_v |\widehat{Z}^{\beta} f_k| dv$ near the lightcone.
\begin{Pro}\label{decayopti}
We have, for all $1 \leq k \leq K$,
$$\forall \hspace{0.5mm} (t,x) \in [0,T] \times \R^n, \hspace{1mm} |\beta| \leq N- \frac{3n+2}{2}, \hspace{3mm} \int_{v \in \mathbb{R}^n} |\widehat{Z}^{\beta} f_k| dv \lesssim \frac{\epsilon}{\tau_+^n}$$

and
$$\forall \hspace{0.5mm} (t,x) \in [0,T] \times \R^n, \hspace{1mm} |\beta| \leq N-n, \hspace{3mm} \int_{v \in \mathbb{R}^n} |\widehat{Z}^{\beta} f_k| dv \lesssim \epsilon\frac{\chi^{\frac{1}{6}}(t)}{\tau_+^n}.$$
\end{Pro}

\begin{proof}
This ensues from Theorem \ref{decayestimate}, Remark \ref{rqdecayestimate} and the estimations made in Section \ref{step2energy}. The loss for the derivatives of higher order is linked to the loss on $\mathbb{E}^2_{N,1}[f_k]$.

\end{proof}

\subsection{Step 4: $L^2$ estimates for the velocity averages}\label{L2system}

In view of commutation formula of Propositions \ref{CommuA} and the energy estimates of Propositions \ref{energypotential}, \ref{MoraF}, we need to prove enough decay on \newline $\| \tau_+ \int_{\mathbb{R}^n  } |\widehat{Z}^{\beta} f_k | dv \|_{L^2_x}$ for all $|\beta| \leq N$. The goal of this section is to prove the following proposition.

\begin{Pro}\label{L2decay}
We have, for all $1 \leq k \leq K$, $|\beta| \leq N$ and for all $t \in [0,T]$,
$$\left\|\tau_+ \int_{\mathbb{R}^n} |\widehat{Z}^{\beta} f_k | dv \right\|_{L^2(\Sigma_t)} \lesssim \epsilon\frac{\chi^{\frac{1}{6}}(t)}{(1+t)^{\frac{n}{2}-1}} .$$
The $\log^{\frac{1}{2}}(3+t)$-loss (specific to the dimension $4$) can be removed for $|\beta| \leq N-\frac{3n+2}{2}$ or improved in a $\log^{\frac{1}{4}}(3+t)$-loss for $|\beta| \geq N-n+1$.
\end{Pro}

Note that if $|\beta| \leq N-n$, that ensues from Proposition \ref{decayopti} and Lemma \ref{intesti}. For the higher order derivatives, we follow the strategy used in \cite{FJS}, in Section $4.5.7$, to prove similar $L^2$ estimates. Let\footnote{If $n =4$ and $N \geq 14$, we can take $8 \leq M \leq N-6$ and avoid the $\log^{\frac{1}{2}}(3+t)$-loss for all derivatives.} $1 \leq k \leq K$ and $M \in \mathbb{N}$ such that $\frac{3n+4}{2} \leq M \leq N-n+1$. Let $I_1$ and $I_2$ be two sets defined as 
$$I_1= \{ \beta \hspace{1mm} \text{multi-index} \hspace{0.5mm} / \hspace{0.5mm} M \leq |\beta| \leq N \} \hspace{1mm} \text{and} \hspace{1mm} I_2= \{ \beta \hspace{1mm} \text{multi-index} \hspace{0.5mm} / \hspace{0.5mm} |\beta| \leq M-1 \}.$$
We consider an ordering on $I_i$, for $1 \leq i \leq 2$, so that $I_i=\{ \beta_{i,1},...,\beta_{i,|I_i|} \}$ and two vector valued fields $X$ and $Y$, of respective length $|I_1|$ and $|I_2|$, such that
$$X^j=\widehat{Z}^{\beta_{1,j}}f_k \hspace{3mm} \text{and} \hspace{3mm} Y^j=\widehat{Z}^{\beta_{2,j}} f_k.$$

\begin{Lem}\label{bilanL2}
There exists three matrices valued functions $A_1 :[0,T] \times \R^n \rightarrow \mathfrak M_{|I_1|}(\R)$, $A_2 :[0,T] \times \R^n \rightarrow  \mathfrak M_{|I_2|}(\R)$ and $B :[0,T] \times \R^n \rightarrow  \mathfrak M_{|I_1|,|I_2|}(\R)$ such that
$$T_F(X)+A_1X=B Y , \hspace{3mm} \text{and} \hspace{3mm} T_F(Y)=A_2Y.$$
If $1 \leq j \leq I_1$, $A_1$ and $B$ are such that $T_F(X^j)$ is a linear combination of $$\frac{v^{\mu}}{v^0}\mathcal{L}_{Z^{\gamma_1}}(F)_{\mu m}X^{\beta_{1,q}}, \hspace{2mm} t\frac{v^{\mu}}{v^0}\mathcal{L}_{Z^{\gamma_1}}(F)_{\mu m}X^{\beta_{1,q}}, \hspace{2mm} \frac{v^{\mu}}{v^0}\mathcal{L}_{Z^{\gamma_1}}(F)_{\mu i}x^iX^{\beta_{1,q}},$$
$$\frac{v^{\mu}}{v^0}\mathcal{L}_{Z^{\gamma_2}}(F)_{\mu m}Y^{\beta_{2,l}}, \hspace{2mm} t\frac{v^{\mu}}{v^0}\mathcal{L}_{Z^{\gamma_2}}(F)_{\mu m}Y^{\beta_{2,l}} \hspace{2mm} \text{and} \hspace{2mm} \frac{v^{\mu}}{v^0}\mathcal{L}_{Z^{\gamma_2}}(F)_{\mu i}x^iY^{\beta_{2,l}},$$
with $|\gamma_1| \leq N-\frac{3n+2}{2}$, $|\gamma_2| \leq N$, $1 \leq m \leq n$, $1 \leq q \leq |I_1|$ and $1 \leq l \leq |I_2|$.
Similarly, if $1 \leq j \leq I_2$, $A_2$ is such that $T_F(Y^j)$ is a linear combination of
$$\frac{v^{\mu}}{v^0}\mathcal{L}_{Z^{\gamma}}(F)_{\mu m}Y^{\beta_{2,l}}, \hspace{2mm} t\frac{v^{\mu}}{v^0}\mathcal{L}_{Z^{\gamma}}(F)_{\mu m}Y^{\beta_{2,l}} \hspace{2mm} \text{and} \hspace{2mm} \frac{v^{\mu}}{v^0}\mathcal{L}_{Z^{\gamma}}(F)_{\mu i}x^iY^{\beta_{2,l}},$$
with $|\gamma| \leq N-n$, $1 \leq m \leq n$ and $1 \leq l \leq |I_2|$. Note also, using Proposition \ref{decayopti}, that
$$\int_v |Y|_{\infty} dv \lesssim \epsilon\frac{\chi^{\frac{1}{6}}(t)}{\tau_+^{n}}.$$ 
\end{Lem}

\begin{proof}

Let $|\beta| \leq N$. According to commutation formula of Lemma \ref{Commuf}, $T_F(\widehat{Z}^{\beta} f_k)$ is a linear combination of terms such as $\mathcal{L}_{Z^{\gamma}}(F)(v,\nabla_v \widehat{Z}^{\delta}(f_k))$, with $|\gamma| + |\delta| \leq |\beta|$ and $|\delta| \leq |\beta|-1$. Replacing each $\partial_{v^i}\widehat{Z}^{\delta} f_k$ by $\frac{1}{v^0}(\widehat{\Omega}_{0i}\widehat{Z}^{\beta} f_k-t\partial_i\widehat{Z}^{\beta} f_k-x^i\partial_t\widehat{Z}^{\beta} f_k )$, the matrices naturally appear.

\end{proof}

Now, we split $X$ in $G+H$ where $G$ is the solution of the homogeneous system and $H$ is the solution to the inhomogeneous system,
$$\left\{
    \begin{array}{ll}
         T_F(H)+A H=0 \hspace{2mm}, \hspace{2mm} H(0,.,.)=X(0,.,.),\\
        T_F(G)+AG=BY \hspace{2mm}, \hspace{2mm} G(0,.,.)=0.
    \end{array}
\right.$$

The goal now is to prove $L^2$ estimates on the velocity averages of $H$ and $G$.

\subsubsection{The homogeneous part}

We start by the following commutation formula.

\begin{Lem}

Let $1   \leq i \leq |I_1| $ and consider $\widehat{Z}^{\delta} \in \widehat{\mathbb{P}}_0^{|\delta|}$, with $|\delta| \leq n$. Then, $T_F(\widehat{Z}^{\delta}H^i)$ can be written as a linear combination of terms of the form
$$\mathcal{L}_{Z^{\gamma}}(F)(v,W),$$
where $W$ is such that 
$$\forall \hspace{0.5mm} 0 \leq \mu \leq n, \hspace{2mm} |W^{\mu}| \lesssim \frac{\tau_+}{v^0} \sum_{|\theta| \leq n} \sum_{ q=1}^{|I_1|} |\widehat{Z}^{\theta}H^q|,$$
and where $|\gamma| \leq N-\frac{n+2}{2}$, so that the electromagnetic field can be estimated pointwise.
\end{Lem}

\begin{proof}
The proof is similar to the ones of Lemma \ref{Commufsimple} and Corollary \ref{Commuf}.

\end{proof}

We introduce the energy $\widetilde{\mathbb{E}}[H]$ of $H$.
$$\widetilde{\mathbb{E}}[H]=\sum_{i=1}^{|I_1|} \mathbb{E}^2_n[H^i]+ \mathbb{E}^2_{n,1}[H^i]$$
and we have the following lemma.

\begin{Lem}\label{L2hom}
If $\epsilon$ is small enough, we have, for all $t \in [0,T]$,
$$\widetilde{\mathbb{E}}[H](t) \leq 8\epsilon \hspace{3mm} \text{and} \hspace{3mm} \forall \hspace{0.5mm} 1 \leq i \leq |I_1|, \hspace{2mm} \left\|\int_v \tau_+ |H^i| dv \right\|_{L^2(\Sigma_t)} \lesssim \frac{\epsilon}{(1+t)^{\frac{n-2}{2}}}.$$

\end{Lem}
\begin{proof}
We follow here what we have done in Section \ref{step2F}. Since $\widetilde{\mathbb{E}}[H](0) \leq \frac{3}{2}\mathbb{E}^2_{N+n}[f_k](0) + \frac{3}{2}\mathbb{E}^2_{N+n,1}[f_k](0) \leq 3\epsilon$ for $\epsilon$ small enough, there exists $0 < \widetilde{T} \leq T$ such that 
$$\forall \hspace{0.5mm} t \in [0,\widetilde{T}], \hspace{3mm} \widetilde{\mathbb{E}}[H](t) \leq 8 \epsilon.$$
To improve this bootstrap assumption, for $\epsilon$ small enough, we only have to use the previous lemma and to follow Section \ref{step2F} (as we always estimated $\left|(\nabla_v w)^L\right|$, $\left|(\nabla_v w)^{\underline{L}}\right|$ and $\left|(\nabla_v w)^B\right|$ by $\frac{\tau_+}{v^0}\sum_{\widehat{Z} \in \K} |\widehat{Z}w|$). We can then take $\widetilde{T}=T$ and obtain, as in Section \ref{step3opti}, that
$$\forall \hspace{0.5mm} 1 \leq i \leq |I_1|, \hspace{1mm} (t,x) \in [0,T] \times \mathbb{R}^n, \hspace{2mm} \int_{v \in \mathbb{R}^n} |H^i(t,x,v)| dv \lesssim \frac{\epsilon}{\tau_+^n}.$$
The $L^2$ estimate then ensues from Lemma \ref{intesti}.
\end{proof}

\subsubsection{The inhomogeneous part}

Let us introduce $K$, the solution of $T_F(K)+A_1K+KA_2=B$ which verifies $K(0,.,.)=0$, and the function

$$|KKY|_{\infty} = \sum_{\begin{subarray}{l} \hspace{1mm} 1 \leq i \leq |I_1| \\ 1 \leq j,p \leq |I_2| \end{subarray}} | K^{j}_{i}|^2 |Y_{p}|.$$

$KY$ and $G$ are solutions of the same system,
\begin{eqnarray}
\nonumber T_F(KY)=T_F(K)Y+KT_F(Y)& = & BY-A_1KY-KA_2Y+KA_2Y \\ \nonumber
& = & BY-A_1KY.
\end{eqnarray}
As $KY(0,.,.)=0$ and $G(0,.,.)=0$, it comes that $KY=G$. For $1 \leq i \leq |I_1|$ and $1 \leq j,p \leq |I_2|$, $| K^{j}_{i}|^2 Y_{q}$ sastifies the equation
$$T_F\left( |K^j_i|^2 Y_p\right) = |K^j_i |^2(A_2)^q_pY_q-2\left((A_1)^q_i K^j_q +K^q_i (A_2)^j_q \right) K^j_i Y_p+2B^j_iK^j_iY_p,$$

which will allow us to estimate $$\mathbb{E}[|KKY|_{\infty}]:=\mathbb{E}^0_0[|KKY|_{\infty}].$$ 

We will then be able to bound $\left\| \tau_+ \int_{v \in \mathbb{R}^n} |G| dv \right\|_{L^2(\Sigma_t)}$ thanks to the estimates on $\int_{v \in \mathbb{R}^n} |Y| dv$ and $\mathbb{E}[|KKY|_{\infty}]$.

\begin{Lem}
We have, 
$$\forall \hspace{0.5mm} t \in [0,T], \hspace{3mm} \mathbb{E}[|KKY|_{\infty}] \leq  \epsilon.$$
\end{Lem}

\begin{proof}
We use again the continuity method. Let $T_0$ be the largest time such that $\mathbb{E}[|KKY|_{\infty}]\leq 2\epsilon$ for all $t \in [0,T_0]$ and let us prove, with the energy estimate of Proposition \ref{energyfsimple}, that for $\epsilon$ small enough, $\mathbb{E}[|KKY|_{\infty}] \leq  \epsilon $ on $[0,T_0]$. Let $t \in [0,T_0]$.

As for the estimate of $\widetilde{\mathbb{E}}[H]$ in the proof of Lemma \ref{L2hom}, we have
$$\int_0^t \int_{\Sigma_s} \int_v \frac{1}{v^0}\left| |K^j_i |^2(A_2)^q_pY_q-2\left((A_1)^q_i K^j_q +K^q_i (A_2)^j_q \right) K^j_i Y_p \right|dvdxds \lesssim \epsilon^{\frac{3}{2}}.$$
Next, we need to estimate the following integral,
\begin{equation}\label{eq:niver}
\int_0^t\int_{\Sigma_s} \int_v \frac{1}{v^0}|B^j_iK^j_iY_p| dv dx .
\end{equation}
The components of the matrix $B$ involve terms in which the electromagnetic field has too many derivatives to be estimated pointwise. Indeed, recall from Lemma \ref{bilanL2} that 
$$|B^j_iK^j_iY_p| \lesssim \sum_{m=1}^n\sum_{|\gamma| \leq N}\tau_+\left|\frac{v^{\mu}}{v^0} \mathcal{L}_{Z^{\gamma}}(F)_{\mu m}K^j_iY_p\right|.$$
We fix $|\gamma|$ and we denote the null decomposition of $\mathcal{L}_{Z^{\gamma}}(F)$ by $(\alpha, \underline{\alpha}, \rho, \sigma)$. In order to bound \eqref{eq:niver}, we bound the integral of the five following terms, given by the null decomposition of the velocity vector $v$ and $\mathcal{L}_{Z^{\gamma}}(F)$.
\begin{itemize}
\item The good terms
$$\tau_+|\alpha|\frac{|KY|}{v^0}, \hspace{3mm} \tau_+|\rho|\frac{|KY|}{v^0} \hspace{3mm} \text{and} \hspace{3mm} \tau_+|\sigma|\frac{|KY|}{v^0}.$$
\item The bad terms
$$\tau_+\frac{v^{\underline{L}}}{(v^0)^2}|\underline{\alpha}||KY| \hspace{3mm} \text{and} \hspace{3mm} \tau_+ \frac{|v^B|}{(v^0)^2}|\underline{\alpha}||KY|.$$
\end{itemize}

We start by bounding the integral on $\Sigma_s \times \R^n_v$ of the good terms. We use $\zeta$ to denote either $\alpha$, $\rho$ or $\sigma$. Using twice the Cauchy-Schwarz inequality (in $x$ and then in $v$), we have
\begin{eqnarray}
\nonumber \int_{\Sigma_s} \int_v \tau_+ |\zeta| \frac{|KY|}{v^0} dvdx & \lesssim &  \|\tau_+ |\zeta| \|_{L^2(\Sigma_s)}\left( \int_{\Sigma_s} \left(\int_v |KY| dv \right)^2 dx \right)^{\frac{1}{2}} \\ \nonumber
& \lesssim & \sqrt{\mathcal{E}_N[F](s)} \left( \int_{\Sigma_s} \int_v |Y| dv \int_v |KKY|_{\infty} dv dx \right)^{\frac{1}{2}} \\ \nonumber
& \lesssim & \sqrt{\mathcal{E}_N[F](s)}  \left\| \int_v |Y| dv \right\|_{L^{ \infty }(\Sigma_s)}^{\frac{1}{2}}\mathbb{E}[|KKY|_{\infty}]^{\frac{1}{2}}.
\end{eqnarray}

Using the bootstrap assumptions, on $\mathcal{E}_N[F]$ and $\mathbb{E}[|KKY|_{\infty}]$, and the pointwise estimate $\int_v |Y| dv \lesssim \epsilon \log^{\frac{1}{2}}(3+t)\tau_+^{-n}$ given in Lemma \ref{bilanL2}, we obtain
$$\int_0^t \int_{\Sigma_s} \int_v \tau_+ |\zeta| \frac{|KY|}{v^0} dvdx ds \lesssim \int_0^t \frac{\epsilon^{\frac{3}{2}} \log^2 (3+t)}{(1+s)^2}ds \lesssim \epsilon^{\frac{3}{2}}.$$

To unify the study of the bad terms, we use $\widetilde{v}$ to denote $v^{\underline{L}}$ or $v^B$. Using the Cauchy-Schwarz inequality (in $(s,x)$), the integral on $[0,t] \times \Sigma_s \times \R^n_v$ of a bad term is bounded by 
\begin{equation}\label{eq:esti1}
\hspace{-0.5cm} \left( \int_0^t \int_{\Sigma_s} \frac{\tau_-^2 |\underline{\alpha}|^2}{(1+s)^{\frac{3}{2}}}dxds  \int_0^t \int_{\Sigma_s} \frac{\tau_+^2 (1+s)^{\frac{3}{2}}}{\tau_-^2}\left(\int_v \left|\frac{\widetilde{v}}{v^0}\right| |KY|dv \right)^2 dxds \right)^{\frac{1}{2}}.
\end{equation}
As $\|\tau_- |\underline{\alpha}| \|^2_{L^2(\Sigma_s)} \lesssim \epsilon \log^3(3+t)$, we have
$$\int_0^t \int_{\Sigma_s} \frac{\tau_-^2 |\underline{\alpha}|^2}{(1+s)^{\frac{3}{2}}}dxds \lesssim \epsilon.$$
For the second factor of the product in \eqref{eq:esti1}, we first note that, by the Cauchy-Schwarz inequality, 
$$\left(\int_v \left|\frac{\widetilde{v}}{v^0}\right| |KY|dv \right)^2 \leq \int_v |Y|dv \int_v \left|\frac{\widetilde{v}}{v^0} \right|^2|KKY|_{\infty} dv.$$
Now, recall from Proposition \ref{extradecay1} that $|v^B| \lesssim \sqrt{v^L v^{\underline{L}}}$ so that $\left|\frac{\widetilde{v}}{v^0} \right|^2 \lesssim \frac{v^{\underline{L}}}{v^0}$. Using the pointwise decay estimate $\int_v |Y|dv \lesssim \epsilon\log^{\frac{1}{2}}(3+t)\tau_+^{-n}$, it comes
$$\left(\int_v \left|\frac{\widetilde{v}}{v^0}\right| |KY|dv \right)^2 \leq \epsilon\frac{\log^{\frac{1}{2}}(3+t)}{\tau_+^n} \int_v \frac{v^{\underline{L}}}{v^0}|KKY|_{\infty} dv.$$
As $\int_{C_u(t)} \int_v \frac{v^{\underline{L}}}{v^0} |KKY|_{\infty} dC_u(t) dv \leq \mathbb{E}[|KKY|_{\infty}](t) \leq 2\epsilon$, we obtain
$$\int_0^t \int_{\Sigma_s} \frac{\tau_+^2 (1+s)^{\frac{3}{2}}}{\tau_-^2}\left(\int_v \left|\frac{\widetilde{v}}{v^0}\right| |KY|dv \right)^2 dxds \lesssim \epsilon^2 \int_{u=-\infty}^{+\infty} \tau_-^{-2} du \lesssim \epsilon^2.$$
Hence,
$$\int_0^t \int_{\Sigma_s} \int_v \tau_+ \frac{|\widetilde{v}|}{(v^0)^2}|\underline{\alpha}| |KY| dv dx ds \lesssim \epsilon^{\frac{3}{2}}$$
and the energy estimate of Proposition \ref{energyfsimple} gives that, for $\epsilon$ small enough, $\mathbb{E}[|KKY|_{\infty}] \leq \epsilon$ on $[0,T_0]$.
\end{proof}

\begin{Rq}
A naive estimation of the bad terms in the previous lemma would lead to a $(1+t)^{\eta}$-loss which would affect the electromagnetic energy.
\end{Rq}

We are now able to prove the expected $L^2$ estimate on $\int_v |G| dv$.

\begin{Lem}\label{L2inhom}
If $\epsilon$ is small enough, we have,
$$\forall \hspace{0.5mm} t \in [0,T], \hspace{1mm} 1 \leq i \leq |I_1|, \hspace{2mm} \left\|\int_v \tau_+ |G^i| dv \right\|_{L^2(\Sigma_t)} \lesssim \frac{\epsilon\chi^{\frac{1}{12}}(t)}{(1+t)^{\frac{n-2}{2}}}.$$
\end{Lem}

\begin{proof}
Let $1 \leq i \leq |I_1|$. The Cauchy-Schwarz inequality (in $v$) gives us
$$\left\| \tau_+\int_v |G^i| dv \right\|_{L^2(\Sigma_t)} \lesssim \sum_{j=1}^{|I_2|} \left\|\tau_+^2 \int_v |Y_j| dv \int_v |(K_i^j)^2 Y_j| dv \right\|_{L^1(\Sigma_t)}^{\frac{1}{2}}.$$
Thus, using once again that $\int_v |Y_j|dv \lesssim \epsilon \chi^{\frac{1}{6}}(t) \tau_+^{-n}$, we obtain
$$\left\| \tau_+ \int_v | G^i| dv \right\|_{L^2(\Sigma_t)} \lesssim \frac{\epsilon\chi^{\frac{1}{12}}(t)}{(1+t)^{\frac{n}{2}-1}}.$$
\end{proof}
We can now conclude this section.
\begin{proof}[Proof of Proposition \ref{L2decay}]

As mentionned earlier, for $|\beta| \leq M-1$, the estimate ensues from Proposition \ref{decayopti} and Lemma \ref{intesti}. If $M \leq |\beta| \leq N$, as there exists $1 \leq i \leq |I_1|$ such that $\widehat{Z}^{\beta} f_k=H^i+G^i$, we have
$$\left\| \tau_+\int_v |\widehat{Z}^{\beta} f_k| dv \right\|_{L^2(\Sigma_t)} \leq \left\| \tau_+\int_v |H^i| dv \right\|_{L^2(\Sigma_t)}+\left\| \tau_+\int_v |G^i| dv \right\|_{L^2(\Sigma_t)}.$$
It then remains to use Lemmas \ref{L2hom} and \ref{L2inhom}.
\end{proof}

\subsection{Step 5: Improvement of the electromagnetic field energy estimates}

\subsubsection{The bound on the potential energy}\label{sectionpotential}

According to the energy estimate given by Proposition \ref{energypotential} and the commutation formula of Proposition \ref{CommuA}, we have, for all $t \in [0,T]$,

$$\sqrt{\widetilde{\mathcal{E}}_N[A]}(t) \lesssim \sqrt{\widetilde{\mathcal{E}}_N[A]}(0)+\sum_{|\gamma| \leq N} \int_0^t |e^k|\left\| \tau_+ \int_{\R^n} |\widehat{Z}^{\gamma} f_k| dv\right\|_{L^2(\Sigma_s)} ds.$$

Using the $L^2$ decay estimate of Proposition \ref{L2decay} and $\widetilde{\mathcal{E}}_N[A](0) \leq \epsilon $, we obtain, for $\epsilon$ small enough and if the constant $C$ is large enough, that
$$\forall \hspace{1mm} t \in [0,T], \hspace{3mm} \widetilde{\mathcal{E}}_N[A](t) \leq \frac{C}{2(n-3)} \epsilon \log^3{(3+t)} \hspace{3mm} \text{if} \hspace{2mm} n=4$$
and
$$\forall \hspace{1mm} t \in [0,T], \hspace{3mm} \widetilde{\mathcal{E}}_N[A](t) \leq \frac{C}{2(n-3)} \epsilon \hspace{3mm} \text{if} \hspace{2mm} n \geq 5.$$

We are now able, using Proposition \ref{decayalphapointwise}, to improve the pointwise decay estimate on $\alpha$.
$$\forall \hspace{1mm} |\beta| \leq N-n, \hspace{1mm} (t,x) \in [0,T] \times \R^n, \hspace{3mm} |\alpha(\mathcal{L}_{Z^{\beta}}(F)|(t,x) \lesssim \sqrt{\epsilon}\frac{\sqrt{\chi(t)}}{\tau_+^{\frac{n+2}{2}}}.$$

\subsubsection{Improvement of the electromagnetic field energy estimates}

Recall from Proposition \ref{MoraF} that
$$\mathcal{E}_N[F](t) \leq \mathcal{E}_N[F](0)+(n-3)\widetilde{\mathcal{E}}_N[A](t)+\varphi(t), $$
where $\varphi(t)$ is a linear combination of terms such that
\begin{equation}\label{eq:energyF}
 \int_0^t \int_{ \Sigma_s } | \overline{K}_0^{\nu} \mathcal{L}_{Z^{\beta}} (F)_{\mu \nu } J(\widehat{Z}^{\gamma} f_k)^{\mu}| dx ds \hspace{2mm} \text{and} \hspace{2mm} \int_0^t \int_{\Sigma_s} s |\mathcal{L}_{Z^{\delta}} A_{\mu} \square \mathcal{L}_{Z^{\delta}} A^{\mu} | dx ds,
 \end{equation}
 with $|\beta|$, $|\gamma|$, $|\delta| \leq N$ and $ 1 \leq k \leq K$. Then, if we could prove that each integrals of \eqref{eq:energyF} is bounded by $\epsilon^{\frac{3}{2}}\chi(t)$, we would have, for $\epsilon$ small enough and if the constant $C$ is large enough, $\mathcal{E}_N [F] \leq C \epsilon\chi(t)$ on $[0,T]$ since $\mathbb{E}_N[F](0) \leq \epsilon$ and $(n-3)\widetilde{\mathcal{E}}_N[A](t) \leq \frac{C}{2}\epsilon \chi(t)$.

We start by bounding the integrals involving the potential. Using Proposition \ref{CommuA} and the Cauchy-Schwarz inequality, we have, for $|\delta| \leq N$,

\begin{flalign*}
& \hspace{0cm} \int_0^t \int_{\Sigma_s} s |\mathcal{L}_{Z^{\delta}} A_{\mu} \square \mathcal{L}_{Z^{\delta}} A^{\mu} | dx ds \lesssim &
\end{flalign*}
 $$ \hspace{4.3cm} \sum_{k=1}^K \sum_{|\gamma| \leq |\delta|} \int_0^t \sqrt{\widetilde{\mathcal{E}}_N[A](s)} \left\|\tau_+\int_v |\widehat{Z}^{\gamma} f_k| dv \right\|_{L^2(\Sigma_s)} ds.$$
Using the $L^2$ estimate of Proposition \ref{L2decay} and that $\widetilde{\mathcal{E}}_N[A](s) \lesssim \epsilon \chi(s)$, it comes
\begin{eqnarray}
 \nonumber
\sum_{|\delta| \leq N} \int_0^t \int_{\Sigma_s} s |\mathcal{L}_{Z^{\delta}} A_{\mu} \square \mathcal{L}_{Z^{\delta}} A^{\mu} | dx ds & \lesssim & \epsilon^{\frac{3}{2}}\int_0^t \frac{\log^2(3+s)}{(1+s)^{\frac{n-2}{2}}} ds \\ \nonumber
& \lesssim & \epsilon^{\frac{3}{2}}\chi(t).
\end{eqnarray}

In order to estimate the remaining integrals of \eqref{eq:energyF}, we express \newline $\overline{K}_0^{\nu} \mathcal{L}_{Z^{\beta}} (F)_{\mu \nu } J(\widehat{Z}^{\gamma} f_k)^{\mu}$ in null coordinates. Dropping the dependance in $\mathcal{L}_{Z^{\beta}}(F)$ or $\widehat{Z}^{\gamma} f_k$, this gives us the four following terms :

\begin{equation}\label{nullterms}
\tau_+^2 \rho J^{\underline{L}}, \hspace{3mm} \tau_-^2 \rho J^{{L}}, \hspace{3mm} \tau_+^2 \alpha_B J^B,  \hspace{3mm} \text{and} \hspace{3mm} \tau_-^2 \underline{\alpha}_B J^B. 
\end{equation}

As $$J^{\underline{L}}=\int_v \frac{v^{\underline{L}}}{v^0} \widehat{Z}^{\gamma} f_k dv, \hspace{3mm} J^{{L}}=\int_v \frac{v^{{L}}}{v^0} \widehat{Z}^{\gamma} f_k dv \hspace{3mm} \text{and} \hspace{3mm} J^B=\int_v \frac{v^B}{v^0} \widehat{Z}^{\gamma} f_k dv,$$
we have,
$$|J^L|, \hspace{1mm} |J^{\underline{L}}|, \hspace{1mm} |J^B| \lesssim \int_{v \in \R^n} |\widehat{Z}^{\gamma} f_k| dv.$$
The integrals (on $[0,T] \times \mathbb{R}^n_x \times \mathbb{R}^n_v$) of each of the four terms of \eqref{nullterms} are then bounded, using the Cauchy-Schwarz inequality, by

$$\int_0^t \sqrt{\mathcal{E}_N[F]}(s) \left\| \tau_+ \int_v |\widehat{Z}^{\gamma} f_k| dv \right\|_{L^2(\Sigma_s)} ds.$$
By Proposition \ref{L2decay} and the bootstrap assumption \eqref{bootF},
\begin{eqnarray}
 \nonumber
\int_0^t \sqrt{\mathcal{E}_N[F]}(s) \left\| \tau_+ \int_v |\widehat{Z}^{\gamma} f_k| dv \right\|_{L^2(\Sigma_s)} ds & \lesssim & \int_0^t \sqrt{\epsilon \chi(s)}\frac{\epsilon\log^{\frac{1}{2}}(3+s)}{(1+s)^{ \frac{n-2}{2} }} ds \\ \nonumber
& \lesssim & \epsilon^{\frac{3}{2}}\chi(t).
\end{eqnarray}

Hence, $\mathcal{E}_N[F](t) \leq C\epsilon\chi(t)$ for all $t \in [0,T]$ if $\epsilon$ is small enough. 

We can prove in the same way, using in particular the energy estimate of Proposition \ref{scalF} and
$$\left\|  \int_v |\widehat{Z}^{\beta} f_k |dv \right\|_{L^2(\Sigma_t)} \leq \frac{1}{1+t} \left\| \tau_+ \int_v |\widehat{Z}^{\beta} f_k |dv \right\|_{L^2(\Sigma_t)},$$
that $\mathcal{E}_N^{S}[F] \leq \overline{C} \epsilon$ on $[0,T]$ if $\epsilon$ is small enough and the constant $\overline{C}$ is large enough. 
We then improve the bootstrap assumption \eqref{bootF}.

\section{The massless Vlasov-Maxwell equations}\label{section7}
\subsection{Global existence for small data}

The aim of this section is to prove Theorem \ref{intromasslessTheo}. We then consider the massless Vlasov-Maxwell system \eqref{syst1}-\eqref{syst3}, with at least two species\footnote{We recall that we take $K \geq 2$ since we suppose that the initial energy $\mathcal{E}[F]$ is finite, which implies that the plasma is electrically neutral (see Remark \ref{introneutral} for more details).}, in dimension $n \geq 4$. This means that $K \geq 2$ and $m_k=0$ for all $1 \leq k \leq K$.

To simplify the notation, we denote, during this chapter, $\mathbb{E}^0_{M}[f]$ by $\mathbb{E}_{M}[f]$ and $\mathbb{E}^{0}_{M,1,0}$ by $\mathbb{E}_{M,1}$. In view of Definition \ref{norm1} and $1 \in \mathbf{k}_0$, we have $$\mathbb{E}_M[f] \leq \mathbb{E}_{M,1}[f].$$ We introduce the functions $\chi$, defined on $\R_+$ by
$$\chi ( s )= 1+s \hspace{2mm} \text{if} \hspace{2mm} n=4, \hspace{2mm} \chi ( s )= \log^2{(3+s)} \hspace{2mm} \text{if} \hspace{2mm} n=5 \hspace{2mm} \text{and} \hspace{2mm} \chi(s) = 1 \hspace{2mm} \text{if} \hspace{2mm} n \geq 6,$$ 
and $\log^*$, defined on $\R_+$ by
$$\log^*=\log \hspace{2mm} \text{if} \hspace{2mm} n=4 \hspace{2mm} \text{and} \hspace{2mm} \log^*= 1 \hspace{2mm} \text{if} \hspace{2mm} n \geq 5.$$

We give a more precise version of Theorem \ref{intromasslessTheo}.

\begin{Th}\label{Theomassless}
Let $n \geq 4$, $K \geq 2$, $N \geq 6n+2$ if $n$ is even and $N \geq 6n+3$ is $n$ is odd, $0 < \eta < \frac{1}{2} $ if $n=4$ and $\eta=0$ if $n \geq 5$ and $R>0$. Let $(f_0,F_0)$ be an initial data set for the massless Vlasov-Maxwell system. Let $(f,F)$ be the unique classical solution to the system and let $A$ be a potential in the Lorenz gauge. There exists $\epsilon >0$ such that\footnote{We recall that a smallness condition on $F$, which implies $\widetilde{\mathcal{E}}_N[A](0) \leq \epsilon$, is given in Proposition \ref{LorenzPot}.}, if 

\begin{flalign*}
& \hspace{3cm} \widetilde{\mathcal{E}}_N[A](0) \leq \epsilon, \hspace{20mm} \mathcal{E}_N[F](0) \leq \epsilon &
\end{flalign*}

and if, for all $1 \leq k \leq K$,
\begin{flalign*}
& \hspace{3cm}  \mathrm{supp}(f_{0k}) \subset \{(x,v) \in \mathbb{R}^n_x \times \mathbb{R}^n_v \setminus  \{ 0 \} \hspace{1mm} / \hspace{1mm} |v| \geq R \} , &
\end{flalign*}
\begin{flalign*}
& \hspace{3cm} \mathbb{E}_{N+n,1}[f_k](0) \leq \epsilon, & 
\end{flalign*}

then $(f,F)$ exists globally in time and verifies the following estimates.

\begin{itemize}
\item Vanishing property for small velocities : for all $1 \leq k \leq K$,
$$\mathrm{supp} (f_k) \subset \left\{ (t,x,v) \in \mathbb{R}_+ \times \mathbb{R}^n_x \times \mathbb{R}^n_v \setminus \{ 0 \} \hspace{1mm} / \hspace{1mm} |v| \geq \frac{R}{2} \right\} .$$

\item Energy bounds for $F$ and $f_k$ : $\forall \hspace{1mm} 1 \leq k \leq K$ and $\forall \hspace{1mm} t \in \mathbb{R}_+$,

\begin{flalign*}
& \hspace{4mm} \mathcal{E}_N[F](t) \lesssim \epsilon\chi(t)(1+t)^{\eta}, \hspace{10mm} \mathcal{E}_{N-2n}[F](t) \lesssim \epsilon \chi(t), & \\
& \hspace{4mm} \mathbb{E}_{N}[f_k](t) \lesssim \epsilon \log^* (3+t), \hspace{10mm} \mathbb{E}_{N-n,1}[f_k](t) \lesssim \epsilon. &
\end{flalign*}

\item Pointwise decay for the null decomposition of $\mathcal{L}_{Z^{\beta}}(F)$ : $\forall$ $|\beta| \leq N-\frac{5n+4}{2}$, $(t,x) \in \mathbb{R}_+ \times \mathbb{R}^n$,
\begin{flalign*}
& \hspace{0.3cm} |\alpha ( \mathcal{L}_{Z^{\beta}} F)| \lesssim \sqrt{\epsilon\chi(t)} \tau_+^{-\frac{n+2}{2}}, \hspace{11mm} |\underline{\alpha} ( \mathcal{L}_{Z^{\beta}} F)| \lesssim \sqrt{\epsilon\chi(t)} \tau_+^{-\frac{n-1}{2}}\tau_-^{-\frac{3}{2}}, & \\ 
& \hspace{0.3cm} |\rho (\mathcal{L}_{Z^{\beta}} F)| \lesssim \sqrt{\epsilon\chi(t)} \tau_+^{-\frac{n+1}{2}}\tau_-^{-\frac{1}{2}}, \hspace{5mm} |\sigma (\mathcal{L}_{Z^{\beta}} F)| \lesssim \sqrt{\epsilon\chi(t)} \tau_+^{-\frac{n+1}{2}}\tau_-^{-\frac{1}{2}} &
\end{flalign*}
and
$$|\underline{\alpha} ( \mathcal{L}_{Z^{\beta}} F)| \lesssim \sqrt{\epsilon} \tau_+^{-\frac{n-1}{2}}\tau_-^{-1}.$$

\item Pointwise decay for $\int_{v \in \mathbb{R}^n \setminus \{0 \} } | z\widehat{Z}^{\beta} f_k | dv$ : $$ \forall \hspace{0.5mm} |\beta| \leq N-2n, \hspace{0.5mm} z \in \mathbf{k}_0, \hspace{0.5mm} (t,x) \in \mathbb{R}_+ \times \mathbb{R}^n, \hspace{5mm} \int_{v  } | z\widehat{Z}^{\beta} f_k |dv \lesssim \frac{\epsilon}{\tau_+^{{n-1}}\tau_-}.$$
\item $L^2$ estimates on $\int_{v \in \mathbb{R}^n \setminus \{0 \} } |\widehat{Z}^{\beta} f_k | dv$ :
$$ \forall \hspace{0.5mm} |\beta| \leq N, \hspace{0.5mm} t \in \mathbb{R}_+, \hspace{3mm} \left\| \int_{v \in \mathbb{R}^n \setminus \{0 \} } | \widehat{Z}^{\beta} f_k | dv \right\|_{L^2(\Sigma_t)} \lesssim \frac{\epsilon}{(1+t)^{\frac{n-1-\eta}{2}}}.$$

\item Energy bound for a potential $A$ satisfying the Lorenz gauge :
$$\forall \hspace{1mm} t \in \mathbb{R}_+, \hspace{3mm} \widetilde{\mathcal{E}}_N[A](t) \lesssim \epsilon \chi(t)(1+t)^{\eta} \hspace{3mm} \text{and} \hspace{3mm} \widetilde{\mathcal{E}}_{N-2n}[A](t) \lesssim \epsilon \chi(t).$$

\end{itemize}

\end{Th}

\subsection{Structure and beginning of the proof}

Let $(f_0,F_0)$ be an initial data set satisfying the assumptions of Theorem \ref{Theomassless}. By a standard local well-posedness argument, there exists a unique maximal solution $(f,F)$ of the massless Vlasov-Maxwell system defined on $[0,T^*[$, with $T^* \in \R_+^* \cup \{+ \infty \}$. 

We consider the following bootstrap assumptions. Let $T$ be the largest time such that, $\forall$ $1 \leq k \leq K$ and $\forall \hspace{1mm} t \in [0,T]$,

\begin{equation}\label{bootFbis}
 \hspace{-1.1cm} \mathcal{E}_N[F](t) \leq 2C\epsilon\chi(t)(1+t)^{\eta}, \hspace{8mm} \mathcal{E}_{N-2n}[F](t) \leq 2C\epsilon \chi(t),
 \end{equation}
 \begin{equation}\label{bootFS}
  \mathcal{E}^0_N[F](t) \leq 4\epsilon, \hspace{5mm} \mathcal{E}_N^{S}[F](t) \leq 2\overline{C} \epsilon(1+t)^{\eta}, \hspace{5mm} \mathcal{E}_{N-2n}^{S}[F](t) \leq 2\overline{C},
\end{equation}
\begin{equation}\label{bootA}
\hspace{-1.2cm}  \widetilde{\mathcal{E}}_N[A](t) \leq 2C\epsilon\chi(t)(1+t)^{\eta}, \hspace{8mm} \widetilde{\mathcal{E}}_{N-2n}[A](t) \leq 2C\epsilon\chi(t), 
\end{equation}
\begin{equation}\label{bootfbis}
\hspace{-2cm} \mathbb{E}_{N}[f_k](t) \leq 4\log^*(3+t), \hspace{3mm} \text{and} \hspace{3mm} \mathbb{E}_{N-n,1}[f_k](t) \leq 4\epsilon, 
\end{equation}

where $C$ and $\overline{C}$ are positive constants which will be specified during the proof. Note that by continuity, $T>0$. We now present our strategy to improve these bootstrap assumptions.

\begin{enumerate}
\item First, using the bootstrap assumptions, we obtain decay estimates for the null decomposition of $F$ (and its Lie derivatives) and for velocity averages of derivatives of $f_k$.

\item Then, we prove that $0$ is not in the closure of the Vlasov fields $v$-support. This follows from the study of the characteristics of the transport equation.

\item Next, we improve the bounds on the Vlasov fields energies by means of the energy estimates proved in Propositions \ref{energyf} and \ref{energypoids}. To bound the right hand side in these energy estimates, we make fundamental use of the null structure of the system and the pointwise decay estimates on $\rho$, $\sigma$, $\alpha$, $\underline{\alpha}$ and $\int_{\mathbb{R}^n \setminus \{ 0 \} } |z \widehat{Z}^{\beta} f_k | dv$.

\item In order to improve the estimates on the electromagnetic field energies, we establish an $L^2_x$ estimate for the velocity averages of the Vlasov fields (and its derivatives). For this purpose, we follow \cite{FJS} and we rewrite all the transport equations as an inhomogeneous system of transport equations. The velocity averages of the homogeneous part of the solution verify strong pointwise decay (we use particularly the control that we have at our disposal on the initial data of $f_k$, for derivatives of order $N+n$ or less). The inhomegeneous part is decomposed into a product of an integrable function and a pointwise decaying function which gives us the expected estimate.

\item Finally, we improve the estimates on the energies of the electromagnetic potential and the electromagnetic field, with the energy estimate for the Maxwell equations (using in particular Propositions \ref{MoraF} and \ref{scalF}). We use the null decomposition of $J(\widehat{Z}^{\gamma} f_k)^{\mu}\mathcal{L}_{Z^{\beta}}(F)_{\mu \nu} \overline{K}_0^{\nu}$, which, combined with $L^2_x$ estimates on quantities such as $ \int_{\mathbb{R}^n  } | \widehat{Z}^{\gamma} f_k | dv $, gives us the improvement.

\end{enumerate}

\subsection{Step 1: Decay estimates}\label{step1massless}

By the Klainerman-Sobolev inequality of Theorem \ref{KS1} and the bootstrap assumption \eqref{bootfbis}, we  have $\forall$ $|\beta| \leq N-n$, $(t,x) \in [0,T[ \times \R^n$, $1 \leq k \leq K$,
\begin{equation}\label{decayf1bis}
 \int_v |\widehat{Z}^{\beta} f_k| dv \lesssim \frac{\mathbb{E}_{N}[f_k](t)}{\tau_+^{n-1}\tau_-}  \lesssim \frac{\epsilon \log^*(3+t)}{\tau_+^{n-1} \tau_-}.
\end{equation}

In the same spirit\footnote{Note that the pointwise decay estimate \eqref{decayf2bis} implies \eqref{decayf1bis} for the lower order derivatives, taking $z=1$. }, using Corollary \ref{KS3}, we have $\forall$ $|\beta| \leq N-2n$, $z  \in \mathbf{k}_0$, $(t,x) \in [0,T] \times \R^n$,

\begin{equation}\label{decayf2bis}
 \int_v \left| z\widehat{Z}^{\beta} f_k \right| dv \lesssim \frac{\mathbb{E}_{N-n,1}[f_k](t)}{\tau_+^{n-1} \tau_-} \lesssim \frac{\epsilon}{\tau_+^{n-1} \tau_-}.
\end{equation}

We have improved decay estimates for the null components of the current $M_{\mu}:= \int_v \frac{v_{\mu}}{v^0} \widehat{Z}^{\beta} f_kdv$. For all $|\beta| \leq N-2n$, we have

\begin{equation}\label{extradecayL+}
\int_v \frac{v^{\underline{L}}}{v^0}| \widehat{Z}^{\beta} f_k| dv \lesssim \frac{\epsilon}{\tau_+^n \tau_-},
\end{equation}
\begin{equation}\label{extradecayL-}
\int_v \frac{v^L}{v^0}| \widehat{Z}^{\beta} f_k| dv \lesssim \frac{\epsilon}{\tau_+^{n-1} \tau_-^{2}}
\end{equation}
and
\begin{equation}\label{extradecayA}
\int_v \frac{|v^B|}{v^0}| \widehat{Z}^{\beta} f_k| dv \lesssim \frac{\epsilon}{\tau_+^n \tau_-}.
\end{equation}

This results from (see Proposition \ref{extradecay1})
$$\frac{v^{\underline{L}}}{v^0} \lesssim \frac{1}{\tau_+}\sum_{z \in \mathbf{k}_0}|z|, \hspace{5mm} \frac{v^{L}}{v^0} \lesssim \frac{1}{\tau_-}\sum_{z \in \mathbf{k}_0}|z| \hspace{5mm}  \text{and} \hspace{5mm} \left|\frac{v^B}{v^0}\right| \lesssim  \frac{1}{\tau_+}\sum_{z \in \mathbf{k}_0}|z|.$$

Using the bootstrap assumptions \eqref{bootFbis}, \eqref{bootFS}, \eqref{bootA}, Propositions \ref{decayFpointwise}, \ref{decayalphapointwise} and the pointwise decay estimate \eqref{decayf1bis}, we obtain.

\begin{Pro}\label{decayFmassless}

For all $t \in [0,T]$, $|\beta| \leq N- n$, we have
\begin{flalign*}
& \hspace{1cm} |\alpha ( \mathcal{L}_{Z^{\beta}} F)| \lesssim \frac{\sqrt{\epsilon\chi(t)(1+t)^{\eta}}}{ \tau_+^{\frac{n+2}{2}}}, \hspace{9mm} |\underline{\alpha} ( \mathcal{L}_{Z^{\beta}} F)| \lesssim \frac{\sqrt{\epsilon\chi(t)(1+t)^{\eta}}}{ \tau_+^{\frac{n-1}{2}}\tau_-^{\frac{3}{2}}}, & \\ 
& \hspace{1cm} |\rho (\mathcal{L}_{Z^{\beta}} F)| \lesssim \frac{\sqrt{\epsilon\chi(t)(1+t)^{\eta}}}{ \tau_+^{\frac{n+1}{2}}\tau_-^{\frac{1}{2}}}, \hspace{9mm} |\sigma (\mathcal{L}_{Z^{\beta}} F)| \lesssim \frac{\sqrt{\epsilon\chi(t)(1+t)^{\eta}}}{ \tau_+^{\frac{n+1}{2}}\tau_-^{\frac{1}{2}}} &
\end{flalign*}

and
$$|\underline{\alpha} ( \mathcal{L}_{Z^{\gamma}} F)| \lesssim \frac{\sqrt{\epsilon(1+t)^{\eta}}}{ \tau_+^{\frac{n-1}{2}}\tau_-}. $$

For all $t \in [0,T]$, $|\beta| \leq N-\frac{5n+2}{2}$, we have
\begin{flalign*}
& \hspace{1cm} |\underline{\alpha} ( \mathcal{L}_{Z^{\gamma}} F)| \lesssim \sqrt{\epsilon} \tau_+^{-\frac{n-1}{2}}\tau_-^{-1}, \hspace{12mm} |\underline{\alpha} ( \mathcal{L}_{Z^{\beta}} F)| \lesssim \sqrt{\epsilon\chi(t)} \tau_+^{-\frac{n-1}{2}}\tau_-^{-\frac{3}{2}}, & \\ 
& \hspace{1cm} |\rho (\mathcal{L}_{Z^{\beta}} F)| \lesssim \sqrt{\epsilon\chi(t)} \tau_+^{-\frac{n+1}{2}}\tau_-^{-\frac{1}{2}}, \hspace{5mm} |\sigma (\mathcal{L}_{Z^{\beta}} F)| \lesssim \sqrt{\epsilon\chi(t)} \tau_+^{-\frac{n+1}{2}}\tau_-^{-\frac{1}{2}}. &
\end{flalign*}

Finally, for all $t \in [0,T]$, $|\beta| \leq N-\frac{5n+4}{2}$,
$$|\alpha ( \mathcal{L}_{Z^{\gamma}} F)| \lesssim \sqrt{\epsilon\chi(t)} \tau_+^{-\frac{n+2}{2}}. $$

\end{Pro}

\begin{Rq}
The alternative estimate on $\underline{\alpha}$ is useful to avoid a $\tau_+$-loss when $n \leq 5$ and is particularly used in Section \ref{homsystembis}.
\end{Rq}

\begin{Rq}
We also have pointwise decay estimates if $|\beta| \leq N-\frac{n+2}{2}$ but the one on $\alpha$ is worse near the light cone (see Proposition \ref{decayFpointwise}).
\end{Rq}

\subsection{Step 2: the Vlasov fields vanishes for small velocities}

We recall that 
$$\forall \hspace{0.5mm} 1 \leq i \leq n, \hspace{2mm} E^i = F_{0i} ,$$
and that the transports equation of the Vlasov-Maxwell system can be rewritten 
$$v^{\mu} \partial_{\mu} f_k +  v^0E^i\partial_{v^i} f_k+v^j{F_{j}}^i \partial_{v^i} f_k =0.$$
We now fix $1 \leq k \leq K$ and we prove, under the bootstrap assumption, that if $f_k(t,x,v) \neq 0$, with $(t,x,v) \in [0,T] \times \mathbb{R}^n \times \mathbb{R}^n \setminus \{ 0 \} $, then $|v| \geq \frac{R}{2}$. During the argument, we will use various constants and we will all call them $C$ for simplicity. These constants will not depend on $\epsilon$ or on $T$.

Let $x \in \mathbb{R}^n$ and $|v| \geq R$. Let $(X,V)$ be the characteristics of the transport equation such that $(X(0),V(0))=(x,v)$. In particular

$$\forall \hspace{1mm}1 \leq i \leq n, \hspace{3mm} \frac{dV^i}{ds}=E^i(s,X)+\frac{V^j}{V^0} F_{ji}(s,X).$$

It follows that

$$\frac{d(|V|^2)}{ds}= 2\left<E(s,X),V \right>.$$

So,

\begin{equation}\label{eq:vlassupp}
|V(t) |^2=|v|^2+2\int_0^t \left<E(s,X(s)),V(s) \right> ds.
\end{equation}

We denote $|V(s)|^2$ by $g(s)$. By the Cauchy-Schwarz inequality, we have

$$g(t) \leq |v|^2+2\int_0^t |E(s,X(s)) | \sqrt{g(s)} ds.$$

We now use a Grönwall inequality (Lemma \ref{Gronwall}) and $|E(s,X(s))| \leq \frac{C \sqrt{\epsilon}}{(1+s)^{\frac{n-1}{2}}} $ (which come from Proposition \ref{decayFmassless}) to obtain $$ g(s) \leq \left(|v|+\int_0^t \frac{C \sqrt{\epsilon} ds}{(1+s)^{\frac{3}{2}}}\right)^2.$$

Thus,
$$|V(s)| \leq |v| + C \sqrt{\epsilon}.$$

Returning to \eqref{eq:vlassupp}, we obtain

$$|V(s) |^2 \geq |v|^2-2\int_0^t |E(s,X(s)) | |V(s) | ds.$$

Therefore, using again the pointwise estimate on $E$,

$$|V(s) |^2 \geq |v|^2-2C\sqrt{\epsilon}(|v|+C\sqrt{\epsilon}).$$

Finally, $$|V(s)|^2 \geq |v|(|v|-C\sqrt{\epsilon})-C\epsilon \geq \frac{1}{4}|v |^2,$$
if $\epsilon$ is sufficiently small so that $C \epsilon \leq \frac{R}{4}$ and $C \sqrt{\epsilon} \leq \frac{R}{2}$.

Then, if $(x,v)$ is such that $|v| \geq R$, $(X,V)$ is well defined on $[0,T]$ ($X$ is also bounded since $\left| \frac{d X}{ds} \right| =1$) and $|V| \geq \frac{R}{2}$. Consequently, we obtain.

\begin{Lem}\label{suppmassless}
$$\mathrm{supp}(f_{k \scriptscriptstyle{\vert [0,T]}}) \subset \{ (t,x,v) \in [0,T] \times \mathbb{R}^n \times \mathbb{R}^n \setminus \{ 0 \} \hspace{1mm} / \hspace{1mm} |v| \geq \frac{R}{2} \} .$$
\end{Lem}

In the remainder, we will then be able to use inequalities like
$$\frac{1}{v^0}|f_k(t,x,v)| \lesssim |f_k(t,x,v)| .$$
Sometimes, we will abusively use inequalities such that $$\frac{1}{v^0} \sum_{z \in \mathbf{k}_0} |z| \lesssim \sum_{z \in \mathbf{k}_0} |z|$$
because these quantities are always multiplied by $\widehat{Z}^{\beta} f_k$.

\subsection{Step 3: Improving the Energy estimates for the transport equations}

We fix for all this section $1 \leq k \leq K$. According to Proposition \ref{energyf}, $\mathbb{E}_{N}[f_k] \leq 3\epsilon \log^* (3+t) $ on $[0,T]$, for $\epsilon$ small enough, follows from

  $$ \int_0^t \int_{\Sigma_s} \int_v \left|\mathcal{L}_{Z^{\beta_1}}(F)\left(\frac{v}{v^0},\nabla_v \widehat{Z}^{\beta_2}f_k \right)\right| dv dx ds \lesssim \epsilon^{\frac{3}{2}}\log^* (3+t),$$
for all $|\beta_1|+|\beta_2| \leq N$, with $|\beta_2 | \leq N-1$.

Similarly, according to Proposition \ref{energypoids}, $\mathbb{E}_{N-n,1}[f_k] \leq 3 \epsilon$ on $[0,T]$, for $\epsilon$ small enough, follows from

  $$ \int_0^t \int_{\Sigma_s} \int_v |z|\left|\mathcal{L}_{Z^{\beta_1}}(F)\left(\frac{v}{v^0},\nabla_v \widehat{Z}^{\beta_2}f_k \right) \right| dv dx ds \lesssim \epsilon^{\frac{3}{2}},$$ and
$$\int_0^t \int_{\Sigma_s} \int_v \left| F\left(\frac{v}{v^0},\nabla_v z \right)\widehat{Z}^{\beta}f_k\right| dv dx ds \lesssim \epsilon^{\frac{3}{2}},$$
for all $z \in \mathbf{k}_0$, $|\beta_1|+|\beta_2| \leq N-n$ (with $|\beta_2 | \leq N-n-1$) and $|\beta| \leq N-n$.

To unify the study of $\mathbb{E}_{N}[f_k]$ and $\mathbb{E}_{N-n,1}[f_k]$, we consider $b$, which could be equal to $0$ or to $1$, $N_0=N$ and $N_{1}=N-n$. Now, we fix $z \in \mathbf{k}_0$, $|\beta_1|+|\beta_2| \leq N_b$ (with $|\beta_2| \leq N_b-1$) and $|\beta| \leq N-n$. We denote $\rho(\mathcal{L}_{Z^{\beta_1}}(F))$, $\sigma(\mathcal{L}_{Z^{\beta_1}}(F))$, $\alpha(\mathcal{L}_{Z^{\beta_1}}(F))$ and $\underline{\alpha}(\mathcal{L}_{Z^{\beta_1}}(F))$ by $\rho$, $\sigma$, $\alpha$ and $\underline{\alpha}$ (respesctively). We also denote $\widehat{Z}^{\beta_2}f_k$ by $g$ and $\widehat{Z}^{\beta}f_k$ by $h$. The null decomposition of $\mathcal{L}_{Z^{\beta_1}}(F)(v,\nabla_v g)$ or $F(v,\nabla_v z)$ brings us to control the integral of the following terms, with $z_0=1$ and $z_1=z$.

The terms involving $L$ or $\underline{L}$ components of $\nabla_v g$ or $\nabla_v z$ 

\begin{equation}\label{nonlin1bis}
 \left|z_b\frac{v^{\underline{L}}}{v^0} \rho \left( \nabla_v g \right)^L \right| , \hspace{5mm}  \left|h\frac{v^{\underline{L}}}{v^0} \rho(F) \left( \nabla_v z \right)^L\right|,
\end{equation}

\begin{equation}\label{nonlin2bis}
 \left|z_b\frac{v^{{L}}}{v^0} \rho \left( \nabla_v g \right)^{\underline{L}}\right| , \hspace{5mm} \left|h \frac{v^{{L}}}{v^0} \rho(F) \left( \nabla_v z \right)^{\underline{L}}\right|,
\end{equation}

\begin{equation}\label{nonlin3bis}
 \left|z_b\frac{v^B}{v^0} \alpha_B \left( \nabla_v g \right)^{L}\right| , \hspace{5mm} \left|h \frac{v^B}{v^0} \alpha_B(F) \left( \nabla_v z \right)^L\right|,
\end{equation}

\begin{equation}\label{nonlin4bis}
 \left|z_b\frac{v^B}{v^0} \underline{\alpha}_B \left( \nabla_v g \right)^{\underline{L}} \right|, \hspace{5mm}\left| h \frac{v^B}{v^0} \underline{\alpha}_B(F) \left( \nabla_v z \right)^{\underline{L}} \right|.
\end{equation}

The terms involving angular components of $ \nabla_v g $ or $\nabla_v z $

\begin{equation}\label{nonlin5bis}
\left| z_b\frac{v^{{L}}}{v^0} \alpha_B \left( \nabla_v g \right)^B \right|, \hspace{5mm} \left| h \frac{v^{{L}}}{v^0} \alpha_B(F) \left( \nabla_v z \right)^B \right|,
\end{equation}

\begin{equation}\label{nonlin6bis}
 \left|z_b\frac{v^B}{v^0} \sigma_{BD} \left( \nabla_v g \right)^D \right| , \hspace{5mm} \left|h \frac{v^B}{v^0} \sigma(F)_{BD} \left( \nabla_v z \right)^D \right|,
\end{equation}

\begin{equation}\label{nonlin7bis}
 \left|z_b\frac{v^{\underline{L}}}{v^0} \underline{\alpha}_B \left( \nabla_v g \right)^B \right|, \hspace{5mm} \left| h \frac{v^{\underline{L}}}{v^0} \underline{\alpha}_B(F) \left( \nabla_v z \right)^B \right|.
\end{equation}

The study of $\mathbb{E}_N[f_k]$ corresponds to $b=0$. In this case, we only have to estimate the spacetime integral of each of the first terms of \eqref{nonlin1bis}-\eqref{nonlin7bis}, but we need to consider two cases. When $|\beta_1| \leq N-n$ we can use the pointwise decay estimates on the electromagnetic field given by Proposition \ref{decayFmassless}. When $|\beta_1| > N-n$, $|\beta_2| \leq N-2n$ (since $N \geq 6n+2$), and we can then use the pointwise decay estimates on the velocity averages of the Vlasov field given in Section \ref{step1massless}.

In the study of $\mathbb{E}_{N-n,1}[f_k]$ (which corresponds to $b=1$ and where $z$ can be any weights of $\mathbf{k}_0$), we can always use a pointwise estimate on the electromagnetic field (as $|\beta_1| \leq N-n$), but we need to estimate the spacetime integral of all the terms of \eqref{nonlin1bis}-\eqref{nonlin7bis}.

\begin{Rq}

To simplify the argument we will sometimes denote $\mathbb{E}_{N}[f_k]$ by $\mathbb{E}_{N_0,0}[f_k]$ and $\mathbb{E}_{N-n,1}[f_k]$ by $\mathbb{E}_{N_1,1}[f_k]$.

\end{Rq}

\subsubsection{Estimating the $v$ derivatives}

In order to eliminate the $v$ derivatives, we use, as in Section \ref{sectionvderiv},
\begin{equation}\label{eq:vderivbasicbis}
 \left| \nabla_v w \right| \lesssim \frac{\tau_+}{v^0}\sum_{\widehat{Z} \in \K} |\widehat{Z} w|
\end{equation}
and

\begin{equation}\label{eq:vderivLbis}
\left| \left( \nabla_v w \right)^{L} \right|, \hspace{1mm} \left| \left( \nabla_v w \right)^{\underline{L}} \right| \lesssim \frac{\tau_-}{v^0}\sum_{\widehat{Z} \in \K}|\widehat{Z}w|.
\end{equation}

\subsubsection{If $|\beta_1| \leq N-n$}\label{pointwiseFbis}

We start by the terms involving $L$ or $\underline{L}$ components of $\nabla_v g$ or $\nabla_v z$. We use $\zeta$ to denote $\alpha$, $\underline{\alpha}$ or $\rho$. Thus, by Proposition \ref{decayFmassless},
$$|\zeta| \lesssim \frac{\sqrt{\epsilon}(1+t)^{\frac{\eta}{2}}}{\tau_+^{\frac{3}{2}}\tau_-}.$$
The integral on $[0,t] \times \mathbb{R}^n_x \times \left( \mathbb{R}^n_v \setminus \{0 \} \right)$ of each of the first terms of \eqref{nonlin1bis}-\eqref{nonlin4bis} are bounded by

$$\sum_{\widehat{Z} \in \widehat{\mathbb{P}}_0} \int_0^t \int_{\Sigma_s} \tau_- |\zeta| \int_v |z_b\widehat{Z}g| dv dx ds,$$
where we use in particular \eqref{eq:vderivLbis} and the fact that $\frac{1}{v^0} \lesssim 1$ on the support of $g$.
Using the bootstrap assumption \eqref{bootfbis}, we obtain
 $$\int_0^t \int_{\Sigma_s} \tau_- |\zeta| \int_v |z_b\widehat{Z}g| dv dx ds \lesssim \int_0^t \sqrt{\epsilon}(1+s)^{-\frac{3-\eta}{2}} \mathbb{E}_{N_b,b}[f_k](s) ds \lesssim \epsilon^{\frac{3}{2}}.$$

Similarly, the integrals of each of the second terms of \eqref{nonlin1bis}-\eqref{nonlin4bis} are bounded by

$$\sum_{\widehat{Z} \in \widehat{\mathbb{P}}_0} \int_0^t \int_{\Sigma_s} \tau_- |\zeta(F)| \int_v |h\widehat{Z}(|z|)| \frac{dv}{v^0} dx ds.$$
Using again the bootstrap assumption \eqref{bootfbis} and $\frac{1}{v^0} \lesssim 1$ on the support of $h$, one has
$$\int_0^t \int_{\Sigma_s} \tau_- |\zeta(F)| \int_v |h\widehat{Z}(|z|)| \frac{dv}{v^0} dx ds \lesssim \int_0^t \sqrt{\epsilon}(1+s)^{-\frac{3-\eta}{2}} \mathbb{E}_{N_1,1}[f_k](s) ds \lesssim \epsilon^{\frac{3}{2}},$$
since $|\widehat{Z}(|z|) | \lesssim \sum_{w \in \mathbf{k}_0} |w|$ by Lemma \ref{vectorweight}.

We now study the remaining terms. Using \eqref{eq:vderivbasicbis}, the pointwise decay estimates of Proposition \ref{decayFmassless}, that $\frac{1}{v^0} \lesssim 1$ on the support of $g$ and the bootstrap assumption \eqref{bootfbis}, we have
\begin{eqnarray}
\nonumber \int_0^t\int_{\Sigma_s} \int_v  \left|z_b \frac{v^L}{v^0} \alpha_B \left( \nabla_v g \right)^B\right| dvdxds & \lesssim & \sum_{\widehat{Z} \in \K}\int_0^t \int_{\Sigma_s} \tau_+ |\alpha| \int_v  |z_b\widehat{Z}  g | dv dx ds \\ \nonumber
& \lesssim & \sqrt{\epsilon} \int_0^t (1+s)^{-\frac{3-\eta}{2}} \mathbb{E}_{N_b,b}[f_k](s)ds \\ \nonumber
& \lesssim & \epsilon^{\frac{3}{2}}.
\end{eqnarray}

The second term of \eqref{nonlin5bis} can be treated similarly. For the second term of \eqref{nonlin6bis} (as the first one can be treated in a similar way), we have, using \eqref{eq:vderivbasicbis}, Lemma \ref{vectorweight} and $|v^B| \lesssim \sqrt{v^Lv^{\underline{L}}}$ (which comes from Proposition \ref{extradecay1}),

\begin{flalign*}
& \hspace{0mm} \int_0^t \int_{\Sigma_s} \int_v \left|h \frac{v^B}{v^0}\sigma(F)_{BD}\left( \nabla_v |z| \right)^D \right| dv dx ds  \lesssim & 
\end{flalign*}
$$ \hspace{5.1cm} \sum_{z' \in \mathbf{k}_0} \int_0^t \int_{\Sigma_s} \tau_+ |\sigma(F)|\int_v \frac{\sqrt{v^Lv^{\underline{L}}}}{(v^0)^2} |z'| |h| dv dx ds . $$
Since, by Proposition \ref{decayFmassless}, $|\sigma(F)| \lesssim \sqrt{\epsilon}\tau_+^{-2}\tau_-^{-\frac{1}{2}}(1+t)^{\frac{\eta}{2}}$, one has, using the Cauchy-Schwarz inequality (in $(s,x,v)$), that the right hand side of the previous inequality is bounded by the product of
$$\hspace{-5mm} \sum_{z' \in \mathbf{k}_0}\left( \sqrt{\epsilon} \int_0^t (1+s)^{-\frac{3-\eta}{2}} \int_{\Sigma_s} \int_v \frac{v^L}{(v^0)^3}|z'| |h| dvdxds \right)^{\frac{1}{2}}$$ with $$\sum_{z' \in \mathbf{k}_0} \left(  \sqrt{\epsilon} \int_{u=-\infty}^{t} \tau_-^{-\frac{3-\eta}{2}} \int_{C_u(t)} \int_v \frac{v^{\underline{L}}}{v^0}|z'| |h| dvdC_u(t)du \right)^{\frac{1}{2}}.$$
The first factor is bounded by $\epsilon^{\frac{3}{4}}$ since $\frac{v^L}{(v^0)^3} \lesssim 1$ on the support of $h$ and $\int_{\Sigma_s} \int_v|z'| |h| dvdx \leq 4 \epsilon$ by the bootstrap assumption \ref{bootfbis}. The same is true for the second factor since $\int_{C_u(t)} \int_v \frac{v^{\underline{L}}}{v^0}|z'| |h| dvdC_u(t) \leq \mathbb{E}_{N-n,1}[f_k](t) \leq 4 \epsilon$, still by the bootstrap assumption \eqref{bootfbis}.

Finally, let us treat, for instance, the first term of \eqref{nonlin7bis}. Using the same ingredients as before, namely \eqref{eq:vderivbasicbis}, that $\frac{1}{v^0} \lesssim 1$ on the support of $g$ and the bootstrap assumption \eqref{bootfbis}, we have,

\begin{eqnarray}
\nonumber  \int_0^t \hspace{-0.5mm} \int_{\Sigma_s} \int_v \left|z_b\frac{v^{\underline{L}}}{v^0} \underline{\alpha}_B \left( \nabla_v g \right)^B\right|dsdxdv  \hspace{-2mm} & \lesssim & \hspace{-2mm} \sum_{\widehat{Z} \in \widehat{\mathbb{P}}_0} \int_0^t \hspace{-0.5mm} \int_{\Sigma_s} \hspace{-1mm} \tau_+ |\underline{\alpha}| \int_v \frac{v^{\underline{L}}}{(v^0)^2}|z_b \widehat{Z} g| dvdxds \\ \nonumber  
 &  \lesssim  & \hspace{-2mm} \int_{u=-\infty}^{t} \|\tau_+|\underline{\alpha}| \|_{L^{\infty}(C_u(t))} \mathbb{E}_{N_b,b}[f_k](t) du  \\ \nonumber
 & \lesssim & \hspace{-2mm} \epsilon^{\frac{3}{2}} \int_{-\infty}^{+\infty} \hspace{-1mm} \frac{1}{\tau_-^{\frac{3-\eta}{2}}} du(\delta_{1,b} \hspace{-0.5mm}+ \hspace{-0.5mm}\delta_{0,b}\log^* (3+t)). 
 \end{eqnarray}

\begin{Rq}

If we used \eqref{eq:vderivbasicbis} instead of \eqref{eq:vderivLbis} to estimate \eqref{nonlin2bis} and \eqref{nonlin4bis}, it would give us a $(1+t)^{\eta}$-loss on the energies (as in the proof of Lemma \ref{L2hombis} below). The weight $v^B$ could be used to avoid this loss in \eqref{nonlin4bis}.

\end{Rq}

\subsubsection{If $|\beta_1| > N-n$}

We study again the integrals of the first terms of \eqref{nonlin1bis}-\eqref{nonlin7bis}, but this time when $|\beta_1| > N- n$, so that $|\beta_2| \leq N-2n$, and $z_b=1$. We then use the pointwise estimate on the velocity averages of the Vlasov fields. This time, we study the terms involving $\alpha$, $\rho$ and $\sigma$ together\footnote{Note that except for \eqref{nonlin5bis}, we could bound all this terms without the $\log^* (1+t)$-loss.} and we finish with the two terms involving $\underline{\alpha}$. Note that as we use the extra decay given by $v^L$, $v^{\underline{L}}$ and $v^B$, we cannot close the estimate for $\mathbb{E}_{N,1}[f_k]$ with our method.

Let us denote this time $\alpha$, $\rho$ or $\sigma$ by $\zeta$. Then, by the bootstrap assumption \eqref{bootFbis},

$$\forall \hspace{1mm} t \in [0,T], \hspace{3mm} \int_{C_u(t)} |\zeta|^2 dx  \lesssim \epsilon.$$

All first terms of \eqref{nonlin1bis}-\eqref{nonlin7bis} involving $\alpha$, $\rho$ or $\sigma$ have their integral on $[0,t] \times \R^n_x \times \R^n_v$ bounded by

$$M:=\sum_{\widehat{Z} \in \K} \int_0^t \int_{\Sigma_s}  |\zeta| \tau_+\int_v \left|\frac{\widetilde{v}}{v^0} \widehat{Z} g\right| dv dx ds,$$
where we used \eqref{eq:vderivbasicbis}, that $\frac{1}{v^0} \lesssim 1$ on the support of $g$ and where $\widetilde{v}$ denotes either $v^L$, $v^{\underline{L}}$ or $v^B$.

Using the pointwise decay estimate on $\int_v \left|\frac{\widetilde{v}}{v^0} \widehat{Z}g\right| dv$, given by \eqref{extradecayL+}, \eqref{extradecayL-} or \eqref{extradecayA}, and the Cauchy-Schwarz inequality (on the $u=constant$ integrals), we have

\begin{eqnarray}
\nonumber M & \lesssim & \int_{u=-\infty}^{t} \left( \int_{C_u(t)} |\zeta|^2 dx \right)^{\frac{1}{2}} \left( \int_{C_u(t)} \frac{\epsilon^2}{\tau_+^{2n-4} \tau_-^4} d C_u(t) \right)^{\frac{1}{2}} du \\ \nonumber
& \lesssim & \int_{u=-\infty}^{t} \frac{\epsilon^{\frac{3}{2}}}{\tau_-^2} \left| \int_{\underline{u}=0}^{2t-u} \frac{r^{n-1}}{\tau_+^{2n-4} } d \underline{u} \right|^{\frac{1}{2}} du \\ \nonumber
& \lesssim & \epsilon^{\frac{3}{2}} \int_{u=-\infty}^t \frac{\log^* (1+2t-u)}{\tau_-^{2}} du \\ \nonumber
& \lesssim & \epsilon^{\frac{3}{2}} \log^*(1+t) \int_{u=-\infty}^{+\infty} \tau_-^{-\frac{3}{2}} du.
\end{eqnarray}

We now study the two remaining terms, which involve $\underline{\alpha}$. We start by \eqref{nonlin4bis}. Using \eqref{eq:vderivLbis}, we obtain that $\int_0^t \int_{\Sigma_s} \int_v |\frac{v^B}{v^0} \underline{\alpha}_B \left( \nabla_v g \right)^{\underline{L}}|dvdxds$ is bounded by
 $$ \sum_{\widehat{Z} \in \widehat{\mathbb{P}}_0} \int_{0}^{+ \infty}  \| \tau_- \underline{\alpha}\|_{L^2(\Sigma_s)}  \left\| \int_v \left|\frac{v^B}{(v^0)^2}\widehat{Z} g \right| dv \right\|_{L^2(\Sigma_s)} ds .$$
 Using \eqref{extradecayA}, Lemma \ref{intesti} and that $\frac{1}{v^0} \lesssim 1$ on the support of $\widehat{Z}g$, we have $$\left\| \int_v \left|\frac{v^B}{(v^0)^2}\widehat{Z} g \right| dv \right\|_{L^2(\Sigma_s)} \lesssim \frac{\epsilon}{(1+s)^{\frac{n+1}{2}}}.$$ 
 By the bootstrap assumption \ref{bootFbis}, $\| \tau_- \underline{\alpha}\|_{L^2(\Sigma_s)} \lesssim \sqrt{\epsilon \chi(t)(1+t)^{\eta} }$, so
$$\int_0^t \int_{\Sigma_s} \int_v \left|\frac{v^B}{v^0} \underline{\alpha}_B \left( \nabla_v g \right)^{\underline{L}}\right|dvdxds \lesssim \epsilon^{\frac{3}{2}}.$$

Finally,
\begin{flalign*}
& \hspace{0cm} \int_0^t \int_{\Sigma_s} \int_v \left|\frac{v^{\underline{L}}}{v^0} \underline{\alpha}_B \left( \nabla_v g \right)^B \right|dvdxds  \lesssim & 
\end{flalign*}
$$ \hspace{3cm} \sum_{\widehat{Z} \in \widehat{\mathbb{P}}_0} \int_{0}^{ t}  \| \underline{\alpha} \|_{L^2(\Sigma_s)}  \left\| \tau_+ \int_v  \left|\frac{v^{\underline{L}}}{(v^0)^2}\widehat{Z}g \right| dv \right\|_{L^2(\Sigma_s)}ds .$$ 
Now, using the pointwise estimates \eqref{extradecayL+} and Lemma \ref{intesti}, we have
$$\left\| \tau_+ \int_v \frac{v^{\underline{L}}}{(v^0)^2}|\widehat{Z}g| dv \right\|^2_{L^2(\Sigma_s)} \lesssim \epsilon^2 \int_0^{+\infty} \frac{r^{n-1}}{\tau_+^{2n-2}\tau_-^2}dr \lesssim \epsilon (1+s)^{-(n-1)}.$$
As, by the bootstrap assumption \eqref{bootFS}, $\|  \underline{\alpha} \|^2_{L^2(\Sigma_s)} \lesssim \epsilon$,
$$ \int_0^t \int_{\Sigma_s} \int_v \left|\frac{v^{\underline{L}}}{v^0} \underline{\alpha}_B \left( \nabla_v g \right)^B \right|dvdxds \lesssim \epsilon^{\frac{3}{2}}.$$
This concludes the improvement of the bootstrap assumption \eqref{bootfbis}.

\subsection{Step 4: $L^2$ estimates for the velocity averages}\label{L2systembis}

As for the massive case, to close the energy estimates on the electromagnetic field, we need enough decay on quantities such as $\| \int_{v } |\widehat{Z}^{\beta} f_k | dv \|_{L^2_x}$ for all $|\beta| \leq N$. If $|\beta| \leq N-2n$, strong $L^2$ decay estimates can already be obtained on $\int_v \frac{v^{\underline{L}}}{v^0}|\widehat{Z}^{\beta} f_k|dv$, for instance,  combining \eqref{extradecayL+} and Lemma \ref{intesti}.

We fix, for the remaining of this section, $1 \leq k \leq K$. Following the strategy of \cite{FJS} (see Section $4.5.7$), for a similar problem, we introduce $M \in \mathbb{N}$ such that $ \frac{7n+4}{2} \leq M \leq N-\frac{5}{2}n$. Let $I_1$ and $I_2$ be defined as $$I_1= \{ \beta \hspace{1mm} \text{multi-index} \hspace{0.5mm} / \hspace{0.5mm} M \leq |\beta| \leq N \} \hspace{1mm} \text{and} \hspace{1mm} I_2= \{ \beta \hspace{1mm} \text{multi-index} \hspace{0.5mm} / \hspace{0.5mm} |\beta| \leq M-1 \}.$$
We consider an ordering on $I_i$, for $1 \leq i \leq 2$, so that $I_i=\{ \beta_{i,1},...,\beta_{i,|I_i|} \}$ and two vector valued fields $X$ and $Y$, of respective length $|I_1|$ and $|I_2|$, such that
$$X^j=\widehat{Z}^{\beta_{1,j}}f_k \hspace{3mm} \text{and} \hspace{3mm} Y^j=\widehat{Z}^{\beta_{2,j}} f_k.$$ 

\begin{Lem}\label{bilanL2bis}
There exists three matrices valued functions $A_1 :[0,T] \times \R^n \rightarrow \mathfrak M_{|I_1|}(\R)$, $A_2 :[0,T] \times \R^n \rightarrow \in \mathfrak M_{|I_2|}(\R)$ and $B :[0,T] \times \R^n \rightarrow \in \mathfrak M_{|I_1|,|I_2|}(\R)$ such that
$$T_F(X)+A_1X=B Y , \hspace{3mm} \text{and} \hspace{3mm} T_F(Y)=A_2Y.$$
If $1 \leq j \leq I_1$, $A_1$ and $B$ are such that $T_F(X^j)$ is a linear combination of $$\frac{v^{\mu}}{v^0}\mathcal{L}_{Z^{\gamma_1}}(F)_{\mu m}X^{\beta_{1,q}}, \hspace{2mm} t\frac{v^{\mu}}{v^0}\mathcal{L}_{Z^{\gamma_1}}(F)_{\mu m}X^{\beta_{1,q}}, \hspace{2mm} \frac{v^{\mu}}{v^0}\mathcal{L}_{Z^{\gamma_1}}(F)_{\mu i}x^iX^{\beta_{1,q}},$$
$$\frac{v^{\mu}}{v^0}\mathcal{L}_{Z^{\gamma_2}}(F)_{\mu m}Y^{\beta_{2,l}}, \hspace{2mm} t\frac{v^{\mu}}{v^0}\mathcal{L}_{Z^{\gamma_2}}(F)_{\mu m}Y^{\beta_{2,l}} \hspace{2mm} \text{and} \hspace{2mm} \frac{v^{\mu}}{v^0}\mathcal{L}_{Z^{\gamma_2}}(F)_{\mu i}x^iY^{\beta_{2,l}},$$
with $|\gamma_1| \leq N-\frac{7n+2}{2}$, $|\gamma_2| \leq N$, $1 \leq m \leq n$, $1 \leq q \leq |I_1|$ and $1 \leq l \leq |I_2|$.
Similarly, if $1 \leq j \leq I_2$, $A_2$ is such that $T_F(Y^j)$ is a linear combination of
$$\frac{v^{\mu}}{v^0}\mathcal{L}_{Z^{\gamma}}(F)_{\mu m}Y^{\beta_{2,l}}, \hspace{2mm} t\frac{v^{\mu}}{v^0}\mathcal{L}_{Z^{\gamma}}(F)_{\mu m}Y^{\beta_{2,l}} \hspace{2mm} \text{and} \hspace{2mm} \frac{v^{\mu}}{v^0}\mathcal{L}_{Z^{\gamma}}(F)_{\mu i}x^iY^{\beta_{2,l}},$$
with $|\gamma| \leq N-\frac{5n+2}{2}$, $1 \leq m \leq n$ and $1 \leq l \leq |I_2|$. Moreover,
$$ \forall \hspace{0.5mm} z \in \mathbf{k}_0, \hspace{3mm} \int_v |z||Y|_{\infty} dv \lesssim \frac{\epsilon}{\tau_+^{n-1}\tau_-}.$$ 
\end{Lem}

\begin{proof}

Let $|\beta| \leq N$. According to commutation formula of Lemma \ref{Commuf}, $T_F(\widehat{Z}^{\beta} f_k)$ is a linear combination of terms such as $\mathcal{L}_{Z^{\gamma}}(F)(v,\nabla_v \widehat{Z}^{\delta}(f_k))$, with $|\gamma| + |\delta| \leq |\beta|$ and $|\delta| \leq |\beta|-1$. Replacing each $\partial_{v^i}\widehat{Z}^{\delta} f_k$ by $\frac{1}{v^0}(\widehat{\Omega}_{0i}\widehat{Z}^{\beta} f_k-t\partial_i\widehat{Z}^{\beta} f_k-x^i\partial_t\widehat{Z}^{\beta} f_k )$, the matrices naturally appear. The decay estimates ensue from the definition of $Y$ and \eqref{decayf2bis}.

\end{proof}

Now, we split $X$ in $G+H$ where $G$ is the solution of the homogeneous system and $H$ is the solution to the inhomogeneous system,
$$\left\{
    \begin{array}{ll}
         T_F(H)+A H=0 \hspace{2mm}, \hspace{2mm} H(0,.,.)=X(0,.,.),\\
        T_F(G)+AG=BY \hspace{2mm}, \hspace{2mm} G(0,.,.)=0.
    \end{array}
\right.$$
We will prove below that $G=KY$ (with $K$ a well chosen matrix), which implies, in view of the velocity support of $X$ and $Y$, that $H$ and $G$ vanish if $|v| \leq \frac{R}{2}$.

The goal now is to prove $L^2$ estimates on the velocity averages of $H$ and $G$.

\subsubsection{The homogeneous part}\label{homsystembis}

As for the massive case, we have the following commutation formula.

\begin{Lem}\label{Commusystbis}

Let $1   \leq i \leq |I_1| $ and consider $\widehat{Z}^{\delta} \in \widehat{\mathbb{P}}_0^{|\delta|}$, with $|\delta| \leq n$. Then, $T_F(\widehat{Z}^{\delta}H^i)$ can be written as a linear combination of terms of the form
$$\mathcal{L}_{Z^{\gamma}}(F)(v,W),$$
where $W$ is such that 
$$\forall \hspace{0.5mm} 0 \leq \mu \leq n, \hspace{2mm} |W^{\mu}| \lesssim \frac{\tau_+}{v^0} \sum_{|\theta| \leq n} \sum_{ q=1}^{|I_1|} |\widehat{Z}^{\theta}H^q|,$$
and where $|\gamma| \leq N-\frac{5n+2}{2}$, so that we can use the sharpest estimates of Proposition \ref{decayFmassless}, except for $\alpha$.
\end{Lem}

We introduce the energy $\widetilde{\mathbb{E}}_1[H]$ defined by
$$\widetilde{\mathbb{E}}_1[H]=\sum_{ q=1}^{|I_1|} \mathbb{E}_{n,1}[H^q].$$
Note that for $\epsilon$ small enough,
$$\widetilde{\mathbb{E}}_1[H](0) \leq 2\mathbb{E}_{N+n,1}[f](0) \leq 2 \epsilon. $$

\begin{Lem}\label{L2hombis}
If $\epsilon$ is small enough, we have
$$\forall \hspace{0.5mm} t \in [0,T], \hspace{3mm} \widetilde{E}_1[H](t) \leq 6\epsilon(1+t)^{\frac{\eta}{2}}.$$
Moreover, 
$$\forall \hspace{0.5mm} 1 \leq i \leq |I_1|, \hspace{1mm} z \in \mathbf{k}_0, \hspace{1mm} (t,x) \in [0,T]\times \R^n, \hspace{3mm} \int_v |zH^i| dv \lesssim \epsilon \frac{(1+t)^{\frac{\eta}{2}}}{\tau_+^{n-1}\tau_-}.$$

\end{Lem}

\begin{proof}
 We use again the continuity method. Since, for $\epsilon$ small enough, $\widetilde{\mathbb{E}}_1[H](0) \leq 2\epsilon$, there exists a larger time $0<\widetilde{T} \leq T$ such that 
$$\forall \hspace{0.5mm} t \in [0, \widetilde{T}], \hspace{3mm} \widetilde{\mathbb{E}}_1[H](t) \leq 6\epsilon(1+t)^{\frac{\eta}{2}} .$$

Following the argument of Section \ref{pointwiseFbis}, we almost get that for $\epsilon$ small enough, $\widetilde{\mathbb{E}}_1[H] \leq 5 \epsilon(1+t)^{\eta} $ on $[0,\widetilde{T}]$. In fact, using Lemma \ref{Commusystbis}, we have that $T_F(H^{\beta})$ is a linear combination of terms like $\mathcal{L}_{Z^{\gamma}}(F)(v,W)$, with $|\gamma| \leq N-\frac{5n+2}{2}$. Thus we can use the null decomposition of the velocity vector and the electromagnetic field (and use its pointwises estimates) and then make similar computations as in Section \ref{pointwiseFbis}. As we cannot use \eqref{eq:vderivLbis} (the algebraic relations between $S \widehat{Z}^{\beta} f$ and $\partial_{\mu} \widehat{Z}^{\beta} f$ ($\mu \in \llbracket 0, n \rrbracket$), for instance, are not necessarily conserved by the decomposition $X=H+G$), we need to reexamine the terms corresponding to \eqref{nonlin1bis}-\eqref{nonlin4bis}. For instance, for the terms analogous to one of \eqref{nonlin4bis}, we have to prove, for $z \in \mathbf{k}_0$,
\begin{equation}\label{eq:badtermbis}
\int_0^t \int_{\Sigma_s} \int_{v } \tau_+ | \underline{\alpha}| \left| z\frac{v^B}{(v^0)^2} \widehat{Z}^{\theta} H^{q} \right| dv dx ds \lesssim \epsilon^{\frac{3}{2}}(1+t)^{\frac{\eta}{2}}.
\end{equation}
As $|v^B| \lesssim \sqrt{|v^L v^{\underline{L}}}$ by Proposition \ref{extradecay1} and as $\tau_+ |\underline{\alpha}| \lesssim \frac{\sqrt{\epsilon}}{\tau_+^{\frac{n-3}{2}}\tau_-}$, we have, by the Cauchy-Schwarz inequality (in $(s,x,v)$), that \eqref{eq:badtermbis} is bounded by the product of
$$\hspace{-2.4cm} \left(\int_0^t \int_{\Sigma_s} \frac{\epsilon}{\tau_+^{n-3}} \int_{v }  \left|z \widehat{Z}^{\theta} H^q \right| dv dx ds \right)^{\frac{1}{2}}$$ with $$\left(\int_{u=-\infty}^{t}  \frac{1}{\tau_-^2} \int_{C_u(t)} \int_{v}  \frac{v^Lv^{\underline{L}}}{(v^0)^4}  \left|z \widehat{Z}^{\theta} H^q \right| dv d C_u(t) du \right)^{\frac{1}{2}}.$$

The first factor is bounded by 
$$ \left( \int_0^t \frac{\epsilon}{1+s} \widetilde{\mathbb{E}}_1[H](s) ds \right)^{\frac{1}{2}} \lesssim \epsilon(1+t)^{\frac{\eta}{4}} ,$$
and the other one, since $\frac{v^L}{(v^0)^3} \lesssim 1$ on the support of $H$, by
$$ \sqrt{\widetilde{\mathbb{E}}_1[H](t)} \left(\int_{u=-\infty}^{+ \infty}  \frac{1}{\tau_-^2}  du\right)^{\frac{1}{2}} \lesssim \sqrt{\epsilon} (1+t)^{\frac{\eta}{4}} .$$

The other terms are easier to bound. Let us study also the terms analogous to one of \eqref{nonlin2bis}, as there are also the cause of the $(1+t)^{\frac{\eta}{2}}$-loss\footnote{Note that we could use that $\sqrt{\tau_+\tau_-}|v^B| \lesssim v^0\sum_{z \in \mathbf{k}_0} |z|$ in \eqref{nonlin4bis} to obtain a better bound in \eqref{eq:badtermbis} for an other energy of $H$. On the other hand, the loss coming from \eqref{nonlin2bis} could not be avoided with such techniques.}.
\begin{eqnarray}
\nonumber \int_0^t \int_{\Sigma_s} \int_v \tau_+|\rho| \left|z\frac{v^L}{(v^0)^2}\widehat{Z}^{\theta} H^{\xi} \right| dv dx ds & \lesssim & \sqrt{\epsilon}\int_0^t (1+s)^{-\frac{n-2}{2}}  \widetilde{\mathbb{E}}_1[H](s)ds \\ \nonumber
& \lesssim & \epsilon^{\frac{3}{2}}(1+t)^{\frac{\eta}{2}}.
\end{eqnarray}
The pointwise estimate on $\int_v |z||H^i|dv$ then ensues from the Klainerman-Sobolev inequality of Corollary \ref{KS3}.

\end{proof}

\subsubsection{The inhomogeneous part}\label{inhomsystembis}

As in the massive case, let us introduce $K$, the solution of $T_F(K)+A_1K+KA_2=B$ which verifies $K(0,.,.)=0$, and the function

$$|KKY|_{\infty} = \sum_{\begin{subarray}{l} \hspace{1mm} 1 \leq i \leq |I_1| \\ 1 \leq j,p \leq |I_2| \end{subarray}} | K^{j}_{i}|^2 |Y_{p}|.$$

$KY$ and $G$ are solutions of the same system,
\begin{eqnarray}
\nonumber T_F(KY)=T_F(K)Y+KT_F(Y)& = & BY-A_1KY-KA_2Y+KA_2Y \\ \nonumber
& = & BY-A_1KY.
\end{eqnarray}
As $KY(0,.,.)=0$ and $G(0,.,.)=0$, it comes that $KY=G$. For $1 \leq i \leq |I_1|$ and $1 \leq j,p \leq |I_2|$, $| K^{j}_{i}|^2 Y_p$ sastifies the equation
$$T_F\left( |K^j_i|^2 Y_p\right) = |K^j_i |^2(A_2)^q_pY_q-2\left((A_1)^q_i K^j_q +K^q_i (A_2)^j_q \right) K^j_i Y_p+2B^j_iK^j_iY_p,$$

which will allow us to estimate $$\mathbb{E}[|KKY|_{\infty}]:=\mathbb{E}_{0,1}[|KKY|_{\infty}].$$ 

We will then be able to prove $L^2$ estimates for  $\int_{v \in \mathbb{R}^n} |G| dv $ thanks to the estimates on $\int_{v \in \mathbb{R}^n} |Y| dv$ and on $\mathbb{E}[|KKY|_{\infty}]$.

\begin{Lem}\label{L2inhombis}
We have, if $\epsilon$ is small enough,
$$\forall \hspace{0.5mm} t \in [0,T], \hspace{3mm} \mathbb{E}[|KKY|_{\infty}] \leq  \epsilon(1+t)^{\eta}.$$
\end{Lem}

\begin{proof}
Let $\widetilde{T}>0$ be the largest time such that $\mathbb{E}[|KKY|_{\infty}](t)\leq 2\epsilon(1+t)^{\eta}$ for all $t \in [0,\widetilde{T}]$ and let us prove, with Proposition \ref{energyfsimple}, that for $\epsilon$ small enough, $\mathbb{E}[|KKY|_{\infty}](t) \leq  \epsilon(1+t)^{\eta} $ for all $t \in [0,\widetilde{T}]$. It will follow that $\widetilde{T}=T$.
As for the estimate of $\widetilde{\mathbb{E}}_1[H]$ in the proof of Lemma \ref{L2hombis},
 $$\hspace{-0.9cm} \int_0^t \int_{\Sigma_s} \int_v \left| |K^j_i |^2(A_2)^q_pY_q-2\left((A_1)^q_i K^j_q +K^q_i (A_2)^j_q \right) K^j_i Y_p \right|\frac{|z|}{v^0}dvdxds \lesssim $$
 $$ \hspace{11cm} \epsilon^{\frac{3}{2}}(1+t)^{\eta}$$
 and
 $$ \int_0^t \int_{\Sigma_s} \int_v \left| F \left( \frac{v}{v^0}, \nabla_v \left( |z| \right) \right) \right| |KKY|_{\infty} dv dx ds \lesssim \epsilon^{\frac{3}{2}}.$$
Next, we need to estimate the following integral,
\begin{equation}\label{eq:niverbis}
\int_0^t\int_{\Sigma_s} \int_v \frac{|z|}{v^0}|B^j_iK^j_iY_p| dv dx .
\end{equation}
Recall from Lemma \ref{bilanL2bis} that 
$$|B^j_iK^j_iY_p| \lesssim \sum_{m=1}^n\sum_{|\gamma| \leq N}\tau_+\left|\frac{v^{\mu}}{v^0} \mathcal{L}_{Z^{\gamma}}(F)_{\mu m}K^j_iY_p\right|.$$
The components of the matrix $B$ involve terms in which the electromagnetic field has too many derivatives to be estimated pointwise. We fix $|\gamma|$ and we denote the null decomposition of $\mathcal{L}_{Z^{\gamma}}(F)$ by $(\alpha, \underline{\alpha}, \rho, \sigma)$. In order to bound \eqref{eq:niverbis}, we bound the integral of the five following terms, given by the null decomposition of the velocity vector $v$ and $\mathcal{L}_{Z^{\gamma}}(F)$.
\begin{itemize}
\item The terms which do not involve $\underline{\alpha}$
$$\tau_+|\alpha||z|\frac{|KY|}{v^0}, \hspace{3mm} \tau_+|\rho||z|\frac{|KY|}{v^0} \hspace{3mm} \text{and} \hspace{3mm} \tau_+|\sigma||z|\frac{|KY|}{v^0}.$$
\item The terms involving $\underline{\alpha}$
$$\tau_+|\underline{\alpha}|\frac{v^{\underline{L}}}{(v^0)^2}|z||KY| \hspace{3mm} \text{and} \hspace{3mm} \tau_+ |\underline{\alpha}|\frac{|v^B|}{(v^0)^2}|z||KY|.$$
\end{itemize}

We start by bounding the integral on $\Sigma_s \times (\R^n_v \setminus \{ 0 \} )$ of the good terms. We use $\zeta$ to denote either $\alpha$, $\rho$ or $\sigma$. Using twice the Cauchy-Schwarz inequality (in $x$ and then in $v$) and that $\frac{1}{v^0} \lesssim 1$ on the support of $Y$, we have
\begin{eqnarray}
\nonumber \int_{\Sigma_s} \int_v \tau_+ |\zeta| \frac{|zKY|}{v^0} dvdx  & \lesssim &    \|\tau_+ |\zeta| \|_{L^2(\Sigma_s)}\left( \int_{\Sigma_s} \left(\int_v \left|zKY\right| dv \right)^2 dx \right)^{\frac{1}{2}} \\ \nonumber
& \lesssim &  \left| \mathcal{E}_N[F](s) \int_{\Sigma_s} \int_v \left|zY\right| dv \int_v \left|zKKY\right| dv dx \right|^{\frac{1}{2}} \\ \nonumber
& \lesssim & \left|\mathcal{E}_N[F](s) \left\| \int_v \left|zY\right| dv \right\|_{L^{ \infty }(\Sigma_s)}\mathbb{E}[|KKY|_{\infty}]\right|^{\frac{1}{2}}.
\end{eqnarray}

Using the bootstrap assumptions, on $\mathcal{E}_N[F]$ and $\mathbb{E}[|KKY|_{\infty}]$, and the pointwise decay estimate $\int_v \left|zY\right| dv \lesssim \epsilon \tau_+^{-n+1}\tau_-^{-1}$ given in Lemma \ref{bilanL2bis}, we obtain
$$\int_0^t \int_{\Sigma_s} \int_v \tau_+ |\zeta| \frac{|zKY|}{v^0} dvdx ds \lesssim \int_0^t \frac{\epsilon^{\frac{3}{2}} \sqrt{\chi(t)}}{(1+s)^{\frac{3}{2}-\eta}}ds \lesssim \epsilon^{\frac{3}{2}}(1+t)^{\eta}.$$

As in the massive case, to unify the study of the terms involving $\underline{\alpha}$, we use $\widetilde{v}$ to denote $v^{\underline{L}}$ or $v^B$. Using the Cauchy-Schwarz inequality (in $(s,x)$), we have
\begin{flalign*}
& \hspace{0mm} \int_0^t \int_{\Sigma_s} \tau_+ |\underline{\alpha}| \int_v \frac{|\widetilde{v}|}{(v^0)^2}|z||KY| dv dx ds \lesssim &
\end{flalign*}
\begin{equation}\label{eq:esti1bis}
\left| \int_0^t \int_{\Sigma_s} \frac{\tau_- |\underline{\alpha}|^2}{(1+s)^{n-3}}dxds  \int_0^t \int_{\Sigma_s} \frac{\tau_+^2 (1+s)^{n-3}}{\tau_-}\left|\int_v \left|\frac{\widetilde{v}z}{(v^0)^2}KY\right|dv \right|^2 dxds \right|^{\frac{1}{2}}.
\end{equation}
As, by the bootstrap assumption \ref{bootFS}, $\|\sqrt{\tau_-} |\underline{\alpha}| \|^2_{L^2(\Sigma_s)} \lesssim \epsilon (1+s)^{\eta}$, we have
$$\int_0^t \int_{\Sigma_s} \frac{\tau_- |\underline{\alpha}|^2}{(1+s)^{n-3}}dxds \lesssim \epsilon(1+t)^{\eta}.$$
For the second factor of the product in \eqref{eq:esti1bis}, we first note that, by the Cauchy-Schwarz inequality and that $\frac{1}{(v^0)^2} \lesssim 1$ on the support of $Y$, 
$$\left(\int_v \left|\frac{\widetilde{v}z}{(v^0)^2} KY\right|dv \right)^2 \leq \int_v \left|zY\right|dv \int_v \left|\frac{\widetilde{v}}{v^0} \right|^2|z||KKY|_{\infty} dv.$$
Now, recall from Proposition \ref{extradecay1} that $|v^B| \lesssim \sqrt{v^L v^{\underline{L}}}$ so that $\left|\frac{\widetilde{v}}{v^0} \right|^2 \lesssim \frac{v^{\underline{L}}}{v^0}$. Using the pointwise estimate $\int_v \left|zY\right|dv \lesssim \epsilon\tau_+^{-n+1}\tau_-^{-1}$, it comes
$$\left(\int_v \left|\frac{\widetilde{v}z}{(v^0)^2} KY\right|dv \right)^2 \lesssim \frac{\epsilon}{\tau_+^{n-1}\tau_-} \int_v \frac{v^{\underline{L}}}{v^0}|z||KKY|_{\infty} dv.$$
As $\int_{C_u(t)} \int_v \frac{v^{\underline{L}}}{v^0}|z| |KKY|_{\infty} dC_u(t) dv \leq \mathbb{E}[|KKY|_{\infty}](t) \leq 2\epsilon(1+t)^{\eta}$, we obtain
$$\int_0^t \int_{\Sigma_s} \frac{\tau_+^2 (1+s)^{n-3}}{\tau_-}\left|\int_v \left|\frac{\widetilde{v}z}{(v^0)^2} KY\right|dv \right|^2 dxds \lesssim \epsilon^2(1+t)^{\eta} \int_{u=-\infty}^{t} \tau_-^{-2} du .$$
Hence,
$$\int_0^t \int_{\Sigma_s} \int_v \tau_+ \frac{|\widetilde{v}|}{(v^0)^2}|\underline{\alpha}| |zKY| dv dx ds \lesssim \epsilon^2(1+t)^{\eta}$$
and the energy estimate of Proposition \ref{energyfsimple} gives that, for $\epsilon$ small enough, $\mathbb{E}[|KKY|_{\infty}] \leq \epsilon(1+t)^{\eta}$ on $[0,\widetilde{T}]$.
\end{proof}

\subsubsection{The $L^2$ estimates}

We start with the following proposition.

\begin{Pro}\label{basicL2bis}
We have,
$$ \forall \hspace{0.5mm} |\beta| \leq N, \hspace{0.5mm} t \in [0,T], \left\| \int_{v \in \mathbb{R}^n \setminus \{0 \} } | \widehat{Z}^{\beta} f_k | dv \right\|_{L^2(\Sigma_t)} \lesssim \frac{\epsilon}{(1+t)^{\frac{n-1-\eta}{2}}}$$
and
$$ \forall \hspace{0.5mm} |\beta| \leq N, \hspace{0.5mm} t \in [0,T], \left\|\tau_+ \int_{v \in \mathbb{R}^n \setminus \{0 \} } | \widehat{Z}^{\beta} f_k | dv \right\|_{L^2(\Sigma_t)} \lesssim \frac{\epsilon}{(1+t)^{\frac{n-3-\eta}{2}}}$$

We can remove the $(1+t)^{\frac{\eta}{2}}$-loss if $|\beta| \leq N-2n$.
\end{Pro}

\begin{proof}
Let $1 \leq k \leq K$. The first inequality ensues from the second one since $1+t \leq \tau_+$. If $|\beta| \leq N-2n$, we only have to use the pointwise estimate \eqref{decayf1bis} and Lemma \ref{intesti}. If $|\beta| > N-2n$, recall that there exists $1 \leq i \leq |I_1|$ such that $\widehat{Z}^{\beta} f_k=H^i+G^i$. For $1 \leq i \leq |I_1|$, Lemmas \ref{L2hombis} and \ref{intesti} imply 
$$\left\| \tau_+ \int_v | H^i| dv \right\|_{L^2(\Sigma_t)} \lesssim \frac{\epsilon}{(1+t)^{\frac{n-3-\eta}{2}}}.$$ 
Moreover, as $G=KY$, we have, by the Cauchy-Schwarz inequality (in $v$),
$$\left\|\tau_+ \int_v | G^i| dv \right\|_{L^2(\Sigma_t)} \leq \left\|\tau_+^2 \int_v |Y|_{\infty}dv \int_v |K^j_i|^2|Y_j|dv \right\|_{L^1(\Sigma_t)}^{\frac{1}{2}}.$$
As, by Lemmas \ref{bilanL2bis} and \ref{L2inhombis}, $$\left\|\tau_+^2\int_v |Y|_{\infty} dv \right\|_{L^{\infty}(\Sigma_t)} \lesssim \frac{\epsilon}{(1+t)^{n-3}} \hspace{1mm} \text{and} \hspace{1mm} \left\|\int_v |K^j_i|^2|Y_j|dv \right\|_{L^1(\Sigma_t)}^{\frac{1}{2}} \leq \epsilon^{\frac{1}{2}}(1+t)^{\frac{\eta}{2}},$$
we have
$$\left\|\tau_+ \int_v | G^i| dv \right\|_{L^2(\Sigma_t)} \lesssim \frac{\epsilon}{(1+t)^{\frac{n-3-\eta}{2}}}.$$

\end{proof}

These inequalities will not be sufficient to close the estimate on the energy $\mathcal{E}^{S}_{N-\frac{n+2}{2}}[F]$ in the next section. This is why we prove the following proposition.

\begin{Pro}\label{L2nullesti}
For all $|\beta| \leq N$ and all $t \in [0,T]$, we have :

\begin{eqnarray}
\nonumber \left\| \int_v \frac{v^{\underline{L}}}{v^0}|\widehat{Z}^{\beta}f_k| dv \right\|_{L^2(\Sigma_t)} & \lesssim & \frac{\epsilon}{(1+t)^{\frac{n+1-\eta}{2}}}, \\ \nonumber
\left\| \frac{\tau_-}{\tau_+} \int_v \frac{v^{{L}}}{v^0}|\widehat{Z}^{\beta}f_k| dv \right\|_{L^2(\Sigma_t)} & \lesssim & \frac{\epsilon}{(1+t)^{\frac{n+1-\eta}{2}}}, \\ \nonumber
\left\|  \int_v \left|\frac{v^B}{v^0}\widehat{Z}^{\beta}f_k\right| dv \right\|_{L^2(\Sigma_t)} & \lesssim & \frac{\epsilon}{(1+t)^{\frac{n+1-\eta}{2}}}.
\end{eqnarray}

We can remove the $(1+t)^{\frac{\eta}{2}}$-loss if $|\beta| \leq N-2n$.

\end{Pro}

\begin{proof}
If $|\beta| \leq N-2n$, these inequalities are implied by the pointwise estimates \eqref{extradecayL+}, \eqref{extradecayL-} and \eqref{extradecayA} and Lemma \ref{intesti}. 

If $|\beta| > N-2n$, we prove in the same way that these inequalities are true if we replace $\widehat{Z}^{\beta} f_k$ by $H^i$, with $1 \leq i \leq |I_1|$ such that $\widehat{Z}^{\beta} f_k=H^i+G^i$. It then only remains to consider $G^i$. Recall that by Proposition \ref{extradecay1} and Lemma \ref{bilanL2bis},
$$\frac{v^{\underline{L}}}{v^0} \leq \tau_+^{-1} \sum_{z \in \mathbf{k}_0} |z|, \hspace{3mm}  \frac{v^L}{v^0} \leq \tau_-^{-1} \sum_{z \in \mathbf{k}_0} |z| \hspace{3mm} \text{and} \hspace{3mm} \int_v |z||Y|_{\infty} dv \lesssim \frac{\epsilon}{\tau_+^{n-1}\tau_-}.$$ Hence, using also $G=KY$, the Cauchy-Schwarz inequality (in $v$) and \newline $\mathbb{E}[|KKY|_{\infty}](t) \leq 2\epsilon (1+t)^{\eta}$, we have

\begin{eqnarray}
\nonumber  \left\| \int_v \frac{v^{\underline{L}}}{v^0}|G^i| dv \right\|_{L^2(\Sigma_t)} & \lesssim &  \left\| \int_v \frac{v^{\underline{L}}}{v^0} |Y|_{\infty} dv \int_v \frac{v^{\underline{L}}}{v^0}|(K_i^j)^2 Y_j| dv \right\|_{L^1(\Sigma_t)}^{\frac{1}{2}} \\ \nonumber
 & \lesssim & \sum_{z \in \mathbf{k}_0 } \left\| \tau_+^{-2}\int_v  |z||Y|_{\infty} dv \int_v |z| |KKY|_{\infty} dv \right\|_{L^1(\Sigma_t)}^{\frac{1}{2}} 
\\ \nonumber & \lesssim & \epsilon(1+t)^{-\frac{n+1-\eta}{2}},
\end{eqnarray}
\begin{eqnarray}
\nonumber  \left\| \frac{\tau_-}{\tau_+} \int_v \frac{v^L}{v^0}|G^i| dv \right\|_{L^2(\Sigma_t)} & \lesssim &  \left\| \frac{\tau_-^2}{\tau_+^2}\int_v  \frac{v^L}{v^0}|Y|_{\infty} dv \int_v \frac{v^L}{v^0}|(K_i^j)^2 Y_j| dv \right\|_{L^1(\Sigma_t)}^{\frac{1}{2}} \\ \nonumber 
& \lesssim & \sum_{z \in \mathbf{k}_0 } \left\| \tau_+^{-2}\int_v  |z||Y|_{\infty} dv \int_v |z| |KKY|_{\infty} dv \right\|_{L^1}^{\frac{1}{2}} 
\\ \nonumber & \lesssim & \epsilon(1+t)^{-\frac{n+1-\eta}{2}}.
\end{eqnarray}

The remaining estimate can be proved in a similar way, using $|v^B| \lesssim \frac{v^0}{\tau_+} \sum_{z \in \mathbf{k}_0} |z|$ (see Proposition \ref{extradecay1}).

\end{proof}

\subsection{Step 5: Improvement of the electromagnetic field estimates}
\subsubsection{Improvement of the energies estimates for the potential}

According to the energy estimate of Proposition \ref{energypotential} and the commutation formula of Proposition \ref{CommuA}, one has, for all $t \in [0,T]$,

$$\sqrt{\widetilde{\mathcal{E}}_N[A]}(t) \lesssim \sqrt{\widetilde{\mathcal{E}}_N[A]}(0)+\sum_{|\gamma| \leq N} \int_0^t \left\| \tau_+ e^k\int_{\R^n} |\widehat{Z}^{\gamma} f_k| dv\right\|_{L^2(\R^n)} ds.$$

As $ \widetilde{\mathcal{E}}_N[A](0) \leq \epsilon$ and $\left\| \tau_+ e^k\int_{\R^n} |\widehat{Z}^{\gamma} f_k| dv \right\|_{L^2(\R^n)} \lesssim \epsilon (1+t)^{-\frac{n-3-\eta}{2}}$ (see Proposition \ref{basicL2bis}), we have, for $\epsilon$ small enough and if the constant $C$ is large enough,
$$\forall \hspace{1mm} t \in [0,T], \hspace{3mm} \widetilde{\mathcal{E}}_N[A](t) \leq \frac{C}{2(n-3)} \epsilon\chi(t)(1+t)^{\eta},$$
with $\chi$ such that
$$\chi ( s )= 1+s \hspace{2mm} \text{if} \hspace{2mm} n=4, \hspace{2mm} \chi ( s )= \log^2{(3+s)} \hspace{2mm} \text{if} \hspace{2mm} n=5 \hspace{2mm} \text{and} \hspace{2mm} \chi(s) = 1 \hspace{2mm} \text{if} \hspace{2mm} n \geq 6.$$ 

Similarly, using \eqref{decayf1bis} and Lemma \ref{intesti}, we obtain
$$\forall \hspace{1mm} t \in [0,T], \hspace{3mm} \widetilde{\mathcal{E}}_{N-2n}[A](t) \leq \frac{C}{2(n-3)} \epsilon\chi(t).$$
This concludes the improvement of the bootstrap assumption \eqref{bootA}.

\subsubsection{Improvement of the estimate on $\mathcal{E}^0_N[F]$}

Recall from Proposition \ref{energyt} that, for all $t \in [0,T]$,
$$\mathcal{E}^0_N[F](t) - 2\mathcal{E}^0_N[F](0) \lesssim  \sum_{|\beta|,|\gamma| \leq N} |e^k|\int_0^t \int_{ \Sigma_s } | \mathcal{L}_{Z^{\beta}} (F)_{0 \nu } J(\widehat{Z}^{\gamma} f_k)^{\nu}| dx ds.$$
As, by the Cauchy-Schwarz inequality, the bootstrap assumption \eqref{bootFS} and the $L^2$ estimates of Proposition \ref{basicL2bis},
\begin{eqnarray}
\nonumber \int_0^t \int_{ \Sigma_s } | \mathcal{L}_{Z^{\beta}} (F)_{0 \nu } J(\widehat{Z}^{\gamma} f_k)^{\nu}| dx ds \hspace{-1mm} & \lesssim & \int_0^t \|\mathcal{L}_{Z^{\beta}} (F)\|_{L^2(\Sigma_s)} \|J(\widehat{Z}^{\gamma} f_k) \|_{L^2(\Sigma_s)} ds \\ \nonumber
& \lesssim & \int_0^t \sqrt{\mathcal{E}^0_N[F](s)} \left\| \int_v |\widehat{Z} f_k dv \right\|_{L^2(\Sigma_s)} ds \\ \nonumber
& \lesssim & \int_0^t \epsilon^{\frac{1}{2}} \epsilon(1+s)^{\frac{n-1-\eta}{2}} ds \\ \nonumber & \lesssim  & \epsilon^{\frac{3}{2}},
\end{eqnarray}
we have, for $\epsilon$ small enough, $\mathcal{E}^0_N[F] \leq 3\epsilon$ on $[0,T]$.

\subsubsection{Improvement of the estimates on $\mathcal{E}_N[F]$ and $\mathcal{E}_{N-2n}[F]$}

Recall from Proposition \ref{MoraF} that
$$\mathcal{E}_N[F](t) \leq \mathcal{E}_N[F](0)+(n-3)\widetilde{\mathcal{E}}_N[A](t)+\varphi(t), $$
where $\varphi(t)$ is a linear combination of terms such that
\begin{equation}\label{eq:energyFbis}
 \int_0^t \int_{ \Sigma_s } | \overline{K}_0^{\nu} \mathcal{L}_{Z^{\beta}} F_{\mu \nu } J(\widehat{Z}^{\gamma} f_k)^{\mu}| dx ds \hspace{2mm} \text{and} \hspace{2mm} \int_0^t \int_{\Sigma_s} s |\mathcal{L}_{Z^{\delta}} A_{\mu} \square \mathcal{L}_{Z^{\delta}} A^{\mu} | dx ds,
 \end{equation}
 with $|\beta|$, $|\gamma|$, $|\delta| \leq N$ and $ 1 \leq k \leq K$. Then, if we could prove that each integrals of \eqref{eq:energyFbis} is bounded by $\epsilon^{\frac{3}{2}}\chi(t)(1+t)^{\eta}$, we would have, for $\epsilon$ small enough and $C$ large enough, $\mathcal{E}_N [F] \leq C \epsilon\chi(t)(1+t)^{\eta}$ on $[0,T]$ since $\mathbb{E}_N[F](0) \leq \epsilon$ and $(n-3)\widetilde{\mathcal{E}}_N[A](t) \leq \frac{C}{2}\epsilon \chi(t)(1+t)^{\eta}$.

\begin{Rq}
We could estimate the integrals of \eqref{eq:energyFbis} with a better bound (the computations are similar to those done below in Section \ref{alphabarre1}, but this would not give us a better estimate on $\mathcal{E}_N[F]$ because of the $\chi(t)(1+t)^{\eta}$-loss on $\widetilde{\mathcal{E}}_N[A]$. 
\end{Rq}

We start by bounding the integrals involving the potential. Using Proposition \ref{CommuA} and the Cauchy-Schwarz inequality, we have, for $|\delta| \leq N$,

\begin{flalign*}
& \hspace{0cm} \int_0^t \int_{\Sigma_s} s |\mathcal{L}_{Z^{\delta}} A_{\mu} \square \mathcal{L}_{Z^{\delta}} A^{\mu} | dx ds \lesssim &
\end{flalign*}
 $$ \hspace{4.2cm} \sum_{k=1}^K \sum_{|\gamma| \leq |\delta|} \int_0^t \sqrt{\widetilde{\mathcal{E}}_N[A](s)} \left\|\tau_+\int_v |\widehat{Z}^{\gamma} f_k| dv \right\|_{L^2(\Sigma_s)} ds.$$
Using the $L^2$ estimates of Proposition \ref{basicL2bis} and that $\widetilde{\mathcal{E}}_N[A](s) \lesssim \epsilon \chi(s)(1+t)^{\eta}$, it comes
\begin{eqnarray}
 \nonumber
\sum_{|\delta| \leq N} \int_0^t \int_{\Sigma_s} s |\mathcal{L}_{Z^{\delta}} A_{\mu} \square \mathcal{L}_{Z^{\delta}} A^{\mu} | dx ds & \lesssim & \epsilon^{\frac{3}{2}}\int_0^t \frac{\sqrt{\chi(s)}}{(1+s)^{\frac{n-3}{2}-\eta}} ds \\ \nonumber
& \lesssim & \epsilon^{\frac{3}{2}}\chi(t)(1+t)^{\eta}.
\end{eqnarray}

In order to estimate the remaining integrals of \eqref{eq:energyFbis}, we express \newline $\overline{K}_0^{\nu} \mathcal{L}_{Z^{\beta}} (F)_{\mu \nu } J(\widehat{Z}^{\gamma} f_k)^{\mu}$ in null coordinates. Dropping the dependance in $\mathcal{L}_{Z^{\beta}}(F)$ or $\widehat{Z}^{\gamma} f_k$, this gives us the four following terms :

\begin{equation}\label{nulltermsbis}
\tau_+^2 \rho J^{\underline{L}}, \hspace{3mm} \tau_-^2 \rho J^{{L}}, \hspace{3mm} \tau_+^2 \alpha_B J^B,  \hspace{3mm} \text{and} \hspace{3mm} \tau_-^2 \underline{\alpha}_B J^B. 
\end{equation}

As $$J^{\underline{L}}=\int_v \frac{v^{\underline{L}}}{v^0} \widehat{Z}^{\gamma} f_k dv, \hspace{3mm} J^{{L}}=\int_v \frac{v^{{L}}}{v^0} \widehat{Z}^{\gamma} f_k dv \hspace{3mm} \text{and} \hspace{3mm} J^B=\int_v \frac{v^B}{v^0} \widehat{Z}^{\gamma} f_k dv,$$
we have,
$$|J^L|, \hspace{1mm} |J^{\underline{L}}|, \hspace{1mm} |J^B| \lesssim \int_{v} |\widehat{Z}^{\gamma} f_k| dv.$$
The integrals (on $[0,T] \times \mathbb{R}^n_x \times \left( \mathbb{R}^n_v \setminus \{0\} \right)$) of each of the four terms of \eqref{nulltermsbis} are then bounded, using the Cauchy-Schwarz inequality (in $x$), by

$$\int_0^t \sqrt{\mathcal{E}_N[F]}(s) \left\| \tau_+ \int_v |\widehat{Z}^{\gamma} f_k| dv \right\|_{L^2(\Sigma_s)} ds.$$
By Proposition \ref{basicL2bis} and the bootstrap assumption \eqref{bootFbis},
\begin{eqnarray}
 \nonumber
\int_0^t \sqrt{\mathcal{E}_N[F]}(s) \left\| \tau_+ \int_v |\widehat{Z}^{\gamma} f_k| dv \right\|_{L^2(\Sigma_s)} ds & \lesssim & \int_0^t \sqrt{\epsilon \chi(s)}\frac{\epsilon}{(1+s)^{ \frac{n-3}{2}-\eta }} ds \\ \nonumber
& \lesssim & \epsilon^{\frac{3}{2}}\chi(t)(1+t)^{\eta}.
\end{eqnarray}

Hence, $\mathcal{E}_N[F](t) \leq C\epsilon\chi(t)(1+t)^{\eta}$ for all $t \in [0,T]$ if $\epsilon$ is small enough. We can prove exactly in the same way that $\mathcal{E}_{N-2n}[F](t) \leq C\epsilon\chi(t)$ for all $t \in [0,T]$ if $\epsilon$ is small enough.

We then improve the bootstrap assumption \eqref{bootFbis}.

\subsubsection{Improvement of the estimates on $\mathcal{E}^S_N[F]$ and $\mathcal{E}^S_{N-2n}[F]$}\label{alphabarre1}

Recall from Propositions \ref{scalF} and \ref{CommuA} that, for all $t \in [0,T]$,
$$\mathcal{E}^S_N[F](t) \leq \mathcal{E}_N[F](0)+C_n\left(\widetilde{\mathcal{E}}_N[A](0)+\frac{\widetilde{\mathcal{E}}_N[A](t)}{1+t}+\frac{\mathcal{E}_N[F](t)}{1+t}\right)+\psi(t),$$
where $C_n$ is a positive constant and where $\psi(t)$ is a linear combination of terms such as
\begin{equation}\label{machin1516}
\int_0^t \int_{\Sigma_s} |\mathcal{L}_{Z^{\beta}}(F)_{0 \mu} J^{\mu}(\widehat{Z}^{\gamma}f_k) |+|S^{\nu} \mathcal{L}_{Z^{\beta}}(F)_{\nu \mu} J^{\mu}(\widehat{Z}^{\gamma}f_k) | dx ds ,
\end{equation}
with $|\beta|$, $|\gamma| \leq N$ and $1 \leq k \leq K$, and 
\begin{equation}\label{machin1515}
\int_0^t \int_{\Sigma_s} \left|\mathcal{L}_{Z^{\beta}}(A)_{\mu} \int_v \frac{v^{\mu}}{v^0}\widehat{Z}^{\gamma} f_k \right|dx ds,
\end{equation}
with $|\beta|$, $|\gamma| \leq N$ and $1 \leq k \leq K$.

Let $|\beta|+|\gamma| \leq N$ and $1 \leq k \leq K$. We denote the null decomposition of $ \mathcal{L}_{Z^{\beta}}(F)$ by $(\alpha, \underline{\alpha}, \rho, \sigma)$, $\widehat{Z}^{\gamma}f_k$ by $g$ and $J(\widehat{Z}^{\gamma}f_k)$ by $J$. Expressing $\mathcal{L}_{Z^{\beta}}(F)_{0 \mu} J^{\mu}(g)$ and $S^{\nu} \mathcal{L}_{Z^{\beta}}(F)_{\nu \mu} J^{\mu}(g)$ in null components, \eqref{machin1516} would be bounded by $\epsilon^{\frac{3}{2}}$ if

\begin{flalign*}
& \hspace{6mm} \int_0^t\int_{\Sigma_s} |\tau_+ \rho J^{\underline{L}}|dxds \lesssim \epsilon^{\frac{3}{2}}, \hspace{18mm} \int_0^t\int_{\Sigma_s} |\tau_- \rho J^L|dxds \lesssim \epsilon^{\frac{3}{2}}, \\
& \hspace{6mm} \int_0^t\int_{\Sigma_s} |\tau_+ \alpha J^B|dxds \lesssim \epsilon^{\frac{3}{2}} \hspace{6mm} \text{and} \hspace{6mm} \int_0^t\int_{\Sigma_s} |\tau_+ \underline{\alpha} J^B|dxds \lesssim \epsilon^{\frac{3}{2}}. 
\end{flalign*}
By the Cauchy-Schwarz inequality,

$$\int_0^t \int_{\Sigma_s} \tau_+ |\rho J^{\underline{L}}| dx ds \lesssim \int_{0}^{t} \|\tau_+ \rho \|_{L^2(\Sigma_s)} \left\| \int_v \frac{v^{\underline{L}}}{v^0}|g| dv \right\|_{L^2(\Sigma_s)} ds .$$

Since, by the bootstrap assumption \eqref{bootFbis},  $\|\tau_+ \rho \|^2_{L^2(\Sigma_s)} \lesssim\epsilon\chi(s)(1+s)^{\eta}$ and, according to Proposition \ref{L2nullesti}, $\left\|  \int_v \frac{v^{\underline{L}}}{v^0}|g| dv \right\|_{L^2(\Sigma_s)} \lesssim \epsilon (1+s)^{-\frac{n+1-\eta}{2}}$, it comes that
 $$\int_0^t \int_{\Sigma_s} \tau_+^2 |\rho J^{\underline{L}}| dx ds \lesssim \epsilon^{\frac{3}{2}} \int_0^t \frac{\sqrt{\chi(t)}}{(1+s)^{\frac{n+1}{2}-\eta}} ds \lesssim \epsilon^{\frac{3}{2}}.$$
The other terms are treated similarly.
$$\int_0^t \int_{\Sigma_s} \tau_- |\rho J^{{L}}| dx ds \lesssim \int_{0}^{t} \| \tau_+ \rho \|_{L^2(\Sigma_s)} \left\| \frac{\tau_-}{\tau_+} \int_v \frac{v^{{L}}}{v^0}|g| dv \right\|_{L^2(\Sigma_s)} ds \lesssim \epsilon^{\frac{3}{2}} ,$$

$$\int_0^t \int_{\Sigma_s} \tau_+ |\alpha_B J^B| dx ds \lesssim \int_{0}^{t} \|\tau_+ \alpha \|_{L^2(\Sigma_s)} \left\| \int_v \frac{v^B}{v^0}|g| dv \right\|_{L^2(\Sigma_s)} ds \lesssim \epsilon^{\frac{3}{2}},$$

$$\int_0^t \int_{\Sigma_s} \tau_- |\underline{\alpha}_B J^B| dx ds \lesssim \int_{0}^{t} \|\tau_- \underline{\alpha} \|_{L^2(\Sigma_s)} \left\| \int_v \frac{v^B}{v^0}|g| dv \right\|_{L^2(\Sigma_s)} ds \lesssim \epsilon^{\frac{3}{2}}.$$

Denoting $\mathcal{L}_{Z^{\beta}}(A)$ by $B$, \eqref{machin1515} would be bounded by $\epsilon^{\frac{3}{2}}(1+t)^{\eta}$ if we prove that

\begin{equation}\label{eq5}
\int_0^t \int_{\Sigma_s} |B_L J^L| dx ds \lesssim \epsilon^{\frac{3}{2}} (1+t)^{\eta},
\end{equation}

\begin{equation}\label{eq6}
\int_0^t \int_{\Sigma_s} |B_{\underline{L}} J^{\underline{L}}| dx ds \lesssim \epsilon^{\frac{3}{2}} \hspace{3mm} \text{and} \hspace{3mm} \int_0^t \int_{\Sigma_s} |B_{D} J^{D}| dx ds \lesssim \epsilon^{\frac{3}{2}}.
\end{equation}

Let us show \eqref{eq6} first. Using Proposition \ref{L2nullesti} and the bound on $\widetilde{\mathcal{E}}_N[A]$, we have $$\|B\|_{L^2(\Sigma_s)} \lesssim \sqrt{\epsilon \chi(s)(1+t)^{\eta}} \hspace{3mm} \text{and} \hspace{3mm} \|J^D\|_{L^2(\Sigma_s)}+\|J^{\underline{L}}\|_{L^2(\Sigma_s)} \lesssim \frac{\epsilon}{(1+s)^{\frac{n-\eta}{2}}}.$$ Hence, by the Cauchy-Schwarz inequality,

$$\int_0^t \int_{\Sigma_s} |B_{\underline{L}} J^{\underline{L}}|+|B_DJ^D| dx ds \lesssim \epsilon^{\frac{3}{2}}\int_0^t \frac{\sqrt{\chi(s)}}{ (1+s)^{2-\eta}} ds \lesssim \epsilon^{\frac{3}{2}}.$$

For \eqref{eq5}, we have

\begin{eqnarray}
\nonumber \int_0^t \int_{\Sigma_s} |B_L J^L| dx ds & \lesssim & \int_0^t \|B_L\|_{L^2(\Sigma_s)} \left\| \int_v |g| dv\right\|_{L^2(\Sigma_s)} ds \\ \nonumber
& \lesssim & \epsilon^{\frac{3}{2}} \int_0^t \frac{\sqrt{\chi(s)}}{(1+s)^{\frac{n-1}{2}-\eta}}ds. \\ \nonumber
& \lesssim & \epsilon^{\frac{3}{2}}(1+t)^{\eta}.
\end{eqnarray}

Hence, if $\epsilon$ is small enough and $\overline{C}$ large enough, we have $\mathcal{E}^S_N[F] \leq \overline{C} \epsilon (1+t)^{\eta}$ for all $t \in [0,T]$.

In view of the above, $\mathcal{E}^S_{N-2n}[F] \leq \overline{C}\epsilon$ on $[0,T]$, for $\epsilon$ small enough, would follow if we improve the bound in \eqref{eq5} from $\epsilon^{\frac{3}{2}}(1+t)^{\eta}$ to $\epsilon^{\frac{3}{2}}$, when $|\beta| \leq N-\frac{n+2}{2}$. To do this, we use a pointwise estimate on $B_L$ and we keep $J^L$ in $L^1$-norm. By Lemma \ref{lorenznull}, we have

$$|B_L(t,x)| \lesssim \frac{\sqrt{\epsilon \chi(t)(1+t)^{\eta}}}{\tau_+^{\frac{n}{2}}},$$
which implies 
\begin{eqnarray}
\nonumber \int_0^t \int_{\Sigma_s} |B_L J^L| dx ds & \lesssim & \int_0^t \|B_L\|_{L^{\infty}(\Sigma_s)} \|g\|_{L^1(\Sigma_s)}ds \\ \nonumber & \lesssim & \epsilon^{\frac{3}{2}} \int_0^t \frac{\sqrt{\chi(s)}\log^*(3+s)}{(1+s)^{\frac{n-\eta}{2}}}ds \lesssim \epsilon^{\frac{3}{2}}.
\end{eqnarray}

This concludes the improvement of the bootstrap assumption \eqref{bootFS}.

\section{Non existence}\label{section8}

We show in this chapter the following proposition. Let us denote $(1,...,1)$ by $\vec{u}$ and we recall that $E^i=F_{0i}$.

\begin{Pro}\label{nolocal}

Let the dimension $n$ be such that $n \geq 2$ and let $\chi : \R \rightarrow \R_+$ be a function of class $C^{\infty}$ such that $\chi=1$ on $]-\infty,1]$ and $\chi=0$ on $[3,+\infty[$. We suppose also that $\chi$ is decreasing on $[1,3]$. Let also $M \in \R_+$ such that $M^{-1}= \int_{v \in \R^n} \chi(|v|^2) dv$.

The Vlasov-Maxwell system \eqref{syst1}-\eqref{syst3}, with two species ($K=2$), $e_1=1$, $e_2=-1$, $m_1 =0$, $m_2 \in \R_+$ and the initial data 

$$E_0:x \mapsto 10\chi(2)^{-1}\chi\left(2\frac{r^2}{n}\right)\vec{u}, \hspace{3mm} {F_0}_{ij}=0 \hspace{3mm} \text{for all} \hspace{3mm} 1 \leq i,j\leq n,$$
$$f_{01}=M\Big(div(E_0)+\|div(E_0)\|_{L^{\infty}(\R^n)}\Big)\chi\left( \frac{2r^2}{3n} \right) \chi(|v|^2),$$ and $$f_{02}=M \|div(E_0)\|_{L^{\infty}(\R^n)} \chi\left( \frac{2r^2}{3n} \right) \chi(|v|^2),$$

do not admit a $C^1$ local solution, provided\footnote{Note that such a function $\chi$ exists. Recall for instance the classical construction of cut-off functions} $w \mapsto w \chi'(2w^2)$ is not constant on a neighborhood of $1$.

\end{Pro}

\begin{Rq}
Note that the initial data satisfy the constaint equations. Indeed, 
$$\int_v f_{01}-f_{02}dv=div(E_0) \chi\left(\frac{2r^2}{3n}\right)$$and $x \mapsto \chi\left(\frac{2r^2}{3n}\right)$ is equal to $1$ on the support of $E_0$. The other ones, $\nabla_{[i} {F_0}_{jk]}=0$, are obvious to check.
\end{Rq}

\begin{Rq}

There is uniqueness for a such Cauchy problem in the class of the local $C^1$ functions. Indeed, let $(f_1,f_2,F)$ and $(g_1,g_2,G)$ be two such solutions on $[0,T]$. As $f_i$ and $g_i$ are the unique $C^1$ solution of $T_{(-1)^{i+1}F}(h)=0$ and $T_{(-1)^{i+1}G}(h)=0$ on $[0,T]$, respectively, we obtain with the method of characteristics that they both vanish for $|x| \geq \frac{3}{\sqrt{2}}\sqrt{n}+T$. In view of the wave equations \eqref{waveelec} and \eqref{wavemagn}, the same is true for $F$ and $G$. All the integrals considered below will then be finite. We have
\begin{eqnarray}
\nonumber T_F(f_q-g_q) & = & (G-F)(v,\nabla_v g_q), \\ \nonumber
\nabla^{\mu} (F-G)_{\mu \nu} & = & e^q J(f_q-g_q), \\ \nonumber
\nabla^{\mu} {}^* \! (F-G)_{\mu \alpha_1 ... \alpha_{n-2}} & = & 0.
\end{eqnarray}
Using Propositions \ref{energyt} and \ref{energyfsimple}, we obtain
\begin{eqnarray}
\nonumber h(t) & := & \sum_{q=1}^2 \mathbb{E}_0[f_q-g_q](t)+\sqrt{\mathcal{E}^0_0[F-G](t)} \\ \nonumber
& \lesssim & \int_0^t h(s)\left(1+ \left\|\int_v |e^k\nabla_v g_k|dv \right\|_{L^2(\Sigma_s)} \right)ds.
\end{eqnarray}
The Grönwall lemma gives us that $h=0$ on $[0,T]$, implying $(f,F)=(g,G)$.
\end{Rq}

The strategy of the proof of Proposition \ref{nolocal} is to construct, for all $T_0>0$, a characteristic of the system such that its velocity part vanish in a time less than $T_0$. For this, we make crucial use of the colinearity of $y \mapsto E(t,y\vec{u})$ and $\vec{u}$ which is a corollary of the following subsection.

\subsection{A symmetry property for the Vlasov-Maxwell system}

To lighten the notations, we use $x_{(ij)}$, if $i \neq j$, to denote \newline $(x^1,...,x^{i-1},x^j,x^{i+1},...,x^{j-1},x^i,x^{j+1},...,x^n)$.

\begin{Pro}\label{sym}

We consider the $n$ dimensional Vlasov-Maxwell system, with $K$ species,

\begin{eqnarray}
  \nonumber  T_{m_q}(f_q)  +e_qv^0E^i\partial_{v^i}f_q+e_qv^i {F_i}^j \partial_{v^j} f_q & = & 0 ,\\ \nonumber \nabla^{\mu} F_{\mu \nu} & = & e^q J(f_q)_{\nu}, \\ \nonumber
    \nabla^{\mu} {}^* \! F_{\mu \lambda_1 ... \lambda_{n-2}} & = & 0  , 
\end{eqnarray}

with the initial smooth data $f_q(0,.,.)=f_{0q}$, $F(0,.)=F_0$. We suppose that the initial data satisfy the following symmetry relations
\begin{eqnarray}
 \nonumber f_{q0}(x_{(ik)},v_{(ik)}) & = & f_{q0}(x,v), \hspace{2mm} i \neq j, \\ \nonumber E^i_0(x_{(ik)}) & = & E^k_0(x), \hspace{2mm} i \neq k, \\ \nonumber
\nonumber E^i_0(x_{(kl)}) & = & E^i_0(x), \hspace{2mm} l \neq i, \hspace{2mm} k \neq i. \\ \nonumber (F_{kl})_0(x_{(kl)}) & = & -(F_{kl})_0(x), \\ \nonumber
(F_{kl})_0(x_{(ik)}) & = & (F_{il})_0(x), \hspace{2mm} l \neq k,i,  \\ \nonumber (F_{kl})_0(x_{(ij)}) & = & (F_{kl})_0(x), \hspace{2mm} i \neq k,l, \hspace{2mm} j \neq k,l.
\end{eqnarray}

If there is a unique classical solution $(f_1,...,f_K,F)$ on $[0,T[$, then \newline $(f_1(t,.,.),...,f_K(t,.,.),F(t,.))$ satisfies also these symmetries.
\end{Pro}

\begin{proof}

To simplify the notation, we suppose that $K=1$, $e_q=1$ and we consider the transposition $\tau=(12)$. We denote $(y^2,y^1,y^3,...,y^n)$ by $y_{\tau}$, $m_1$ by $m$ and $f_1$ by $f$. Let $g$ and $G$ be defined by 
\begin{eqnarray}
\nonumber g(t,x,v) & := & f(t,x_{\tau},v_{\tau}), \\ \nonumber
G_{02}(t,x) & := & E^1(t,x_{\tau}), \\ \nonumber
G_{01}(t,x) & := & E^2(t,x_{\tau}),\\ \nonumber
G_{0k}(t,x) & := & E^k(t,x_{\tau}), \hspace{2mm} k \geq 3, \\ \nonumber
G_{12}(t,x) & := & -F_{12}(t,x_{\tau}), \\ \nonumber
G_{1k}(t,x) & := & F_{2k}(t,x_{\tau}), \hspace{2mm} k \neq 1, \hspace{2mm} k \neq 2, \\ \nonumber
G_{2k}(t,x) & := & F_{1k}(t,x_{\tau}), \hspace{2mm} k \neq 1, \hspace{2mm} k \neq 2, \\ \nonumber
G_{kl}(t,x) & := & F_{kl}(t,x_{\tau}), \hspace{2mm} k,l \geq 3
\end{eqnarray}
and let $D^k=G_{0k}$. We want to prove that $(g,G)=(f,F)$. By assumption, this is true for $t=0$ and, by uniqueness, it will be true for $t<T$ if we can prove that $(g,G)$ is solution to the same system as $(f,F)$.

\vspace{2mm}
\textbf{Propagation of symmetry for the Maxwell equations} 
\vspace{2mm}

Let us prove first that $\nabla^{\mu} G_{\mu \nu}= J(g)_{\nu}$. As $J(h)^{\nu}= \int_v \frac{v^{\nu}}{v^0}hdv$, we have, by the change of variables $v'=v_{\tau}$,
$$J(g)^{1}(t,x)=J(f)^{2}(t,x_{\tau}), \hspace{3mm} J(g)^{2}(t,x)=J(f)^{1}(t,x_{\tau}) $$ and $$ J(g)^{\nu}(t,x)=J(f)^{\nu}(t,x_{\tau}) \hspace{3mm} \text{if } \hspace{3mm} \nu \neq 1,2.$$
The equation $\nabla^i G_{i0}=J(g)_0$ then comes from
$$\partial_1 D^1(t,x)=\partial_2 E^2(t,x_{\tau}), \hspace{3mm} \partial_2 D^2(t,x)=\partial_1 E^1(t,x_{\tau}) \hspace{3mm} \text{and} \hspace{3mm} \nabla_i E^i=J(f)_0.$$
As 
\begin{eqnarray}
\nonumber \nabla^{\mu} G_{\mu 1}(t,x) & = & -\partial_t E^2(t,x_{\tau})-\partial_2\left(F_{21}(t,x_{\tau})\right)+\sum_{i=3}^n\nabla^i\left(F_{i 2}(t,x_{\tau})\right) \\ \nonumber
& = & -\partial_t E^2(t,x_{\tau})-\partial_1 F_{21}(t,x_{\tau})+\sum_{i=3}^n\nabla^iF_{i 2}(t,x_{\tau}) \\ \nonumber
& = & \nabla^{\mu} F_{\mu 2}(t,x_{\tau}),
\end{eqnarray}
we have $\nabla^{\mu} G_{\mu 1}= J(g)_1.$ The equation $\nabla^{\mu} G_{\mu 2}(t,x)= J(g)_2$ can be obtained similarly. The remaining equations are obtained from
\begin{eqnarray}
\nonumber \nabla^{j} G_{j k}(t,x) & = & \partial_1\left(F_{2k}(t,x_{\tau})\right)+\partial_2\left(F_{1k}(t,x_{\tau})\right)+\sum_{i=3}^n\nabla^i\left(F_{i k}(t,x_{\tau})\right) \\ \nonumber
& = &  \partial_2 F_{2k}(t,x_{\tau})+\partial_1 F_{1k}(t,x_{\tau})+\sum_{i=3}^n\nabla^iF_{i 2}(t,x_{\tau}) \\ \nonumber
& = & \nabla^{j} F_{j k}(t,x_{\tau})
\end{eqnarray}
and $\partial_t D^k(t,x)=\partial_t E^k(t,x_{\tau})$, for $k \geq 3$. For the other part of the Maxwell equations, recall from Proposition \ref{equivalsyst} that it is equivalent to prove 
$$ \nabla_{[ \lambda} G_{\mu \nu ]}=0.$$
We have
\begin{eqnarray}
\nonumber \nabla_{[ 1} G_{2 3 ]}(t,x) & = & \partial_1 \left( F_{13}(t,x_{\tau})\right)+\partial_2 \left( F_{32}(t,x_{\tau})\right)-\partial_3 \left( F_{12}(t,x_{\tau})\right) \\ \nonumber & = & \partial_2  F_{13}(t,x_{\tau})+\partial_1  F_{32}(t,x_{\tau})+\partial_3 F_{21}(t,x_{\tau}) \\ \nonumber
& = &  \nabla_{[ 2} F_{1 3 ]}(t,x_{\tau}) \\ \nonumber
&=& 0.
\end{eqnarray}
The other equations can be obtained in the same way.

\vspace{2mm}
\textbf{Propagation of symmetry for the Vlasov equation}
\vspace{2mm}

We have
\begin{eqnarray}
\nonumber T_m(g)(t,x,v) & = & v^1 \partial_2 f(t,x_{\tau},v_{\tau})+v^2 \partial_1 f(t,x_{\tau},v_{\tau}) + \sum_{\begin{subarray}{l} \hspace{1mm} \mu=0 \\  \mu \neq 1 ,2  \end{subarray}}^nv^{\mu} \partial_{\mu} f(t,x_{\tau},v_{\tau}) \\ \nonumber
& = & T_m(f)(t,x_{\tau},v_{\tau} ) .
\end{eqnarray}

Moreover, as
$$\partial_{v^1}g (t,x,v)= \partial_{v^2}f(t,x_{\tau},v_{\tau}) \hspace{3mm} \text{and} \hspace{3mm} \partial_{v^2}g (t,x,v)= \partial_{v^1}f(t,x_{\tau},v_{\tau}),$$ $$ D^i(t,x)\partial_{v^i} g(t,x,v)=E^i(t,x_{\tau})\partial_{v^i} f(t,x_{\tau},v_{\tau}).$$ Finally,
\begin{eqnarray}
\nonumber \left(v^kG_{k1} \partial_{v^1} g\right)(t,x,v) & = & \left(-v^2 F_{21}(t,x_{\tau})+\sum_{k=3}^n v^k F_{k2}(t,x_{\tau}) \right) \partial_{v^2} f(t,x_{\tau},v_{\tau}) \\ \nonumber
& = & \left(v^k F_{k2} \partial_{v^2} f \right)(t,x_{\tau},v_{\tau}),
\end{eqnarray}
and more generally
$$\left(v^kG_{kj} \partial_{v^j} g\right)(t,x,v)= \left(v^k F_{k \tau(j)} \partial_{v^{\tau(j)}} f \right)(t,x_{\tau},v_{\tau}).$$

We then deduce, $$T_G(g)=0, \hspace{3mm} \text{as} \hspace{3mm} T_F(f)(t,x_{\tau},v_{\tau})=0.$$

\vspace{2mm}
\textbf{The symmetries are propagated over time}
\vspace{2mm} 

We then proved that $(g,G)$ satisfies the same system as $(f,F)$. As $(f,F)=(g,G)$ at $t=0$, we have, by the uniqueness of the solution, that $(f,F)=(g,G)$ for all $t \in [0,T[$. 

\end{proof}

\begin{Rq}
More generally, from the above proof, $(f,F) \mapsto (g,G)$ maps $C^1$ solutions of the Vlasov-Maxwell system to $C^1$ solutions of the Vlasov-Maxwell system.
\end{Rq}

\subsection{Proof of Proposition \ref{nolocal}}

Let us suppose that the system admits a local $C^1$ solution on $[0,T]$, with $T>0$, which is then necessarily unique. We will reduce $T$ later if necessary, but we already assume that $T \leq 1$.

\subsubsection{Some informations on the electromagnetic fields around $\vec{u}$}

We start by the study of the solution around $\vec{u}$. Let us introduce $M_0:=20 \chi \left( 2 \right)^{-1}$ and $(B_{ij})_{1 \leq i,j\leq n}$ the $2$-form defined by $B_{ij}=F_{ij}$.

\begin{Pro}\label{basicprop}

Reducing $T$ if necessary, we have the following properties.

\begin{enumerate}
\item Local bounds on the field: $\forall \hspace{0.5mm} t\in [0,T]$,
$$\forall \hspace{0.5mm}  |x| \leq \sqrt{n}+2T, \hspace{1mm}  1\leq i \leq 3, \hspace{2mm} 5 \leq E^i(t,x) \leq M_0 , \hspace{2mm} |\partial_t E(t,x)| \leq 1$$
and
\begin{equation}\label{boundmagnB}
\forall \hspace{0.5mm} t\in [0,T], \hspace{1mm} |x-\vec{u}| \leq 2T, \hspace{2mm} |B(t,x)| \leq \frac{1}{4}.
\end{equation}
\item The field is locally-Lipschitz: $\exists \hspace{0.5mm} L>0$, $\forall \hspace{0.5mm} t\in [0,T], \hspace{0.5mm} \hspace{0.5mm} |x|, \hspace{0.5mm} |y| \leq \sqrt{n}+2T$,
\begin{equation}\label{LipschitzboundL}
 |E(t,x)-E(t,y)|+|B(t,x)-B(t,y)| \leq L|x-y|.
\end{equation}
\item Specific behaviour along the $\vec{u}$-direction:
$$\forall \hspace{0.5mm}  y \in \R, \hspace{1mm}t \in [0,T], \hspace{3mm} E(t,y\vec{u}) =E^1(t,y\vec{u})\vec{u} \hspace{3mm} \text{and} \hspace{3mm} B(t,y\vec{u})=0.$$ 
\end{enumerate}

\end{Pro}

\begin{proof}
In view of the initial data, we have $B(0,\vec{u})=0$ and
$$\forall \hspace{0.5mm} |y| \leq \sqrt{n}, \hspace{1mm} 1 \leq i \leq n, \hspace{3mm} 10 \leq E^i(0,y) \leq \frac{M_0}{2}, \hspace{3mm} \partial_t E(0,y)=0.$$
The point $1$ then ensues, taking $T$ smaller if necessary, from the uniform continuity of the electromagnetic field on every compact subset of $[0,T] \times \R^n$. The point $2$ comes from the mean value theorem, as $E$ and $B$ are $C^1$ and the point $3$ follows from Proposition \ref{sym}.

\end{proof}

\subsubsection{The method of charateristics fails}

Let us denote by $(X(s,t,x,v),V(s,t,x,v))$ the value at $s$ of the characteristic, for the transport equation \eqref{syst1} satisfied by $f_1$, which was equal to $(x,v)$ at $t$. Let $\eta \in ]0,T[$ and
$$X_{\eta}:(s,t) \mapsto X(s,t,\vec{u}, \eta \vec{u}), \hspace{3mm}  V_{\eta}:(s,t) \mapsto V(s,t,\vec{u}, \eta \vec{u}).$$
We now fix $t \in [0,T[$. $(X_{\eta}(.,t),V_{\eta}(.,t))$ is well defined on a neighborhood of $t$ and we have, denoting $\frac{v}{|v|}$ by $\widehat{v}$,

\begin{eqnarray}\label{characsyst1}
   \frac{d X_{\eta}(.,t)}{ds}(s) & = & \widehat{V_{\eta}}(s), \\ \label{characsyst2}
    \frac{dV^j_{\eta}(.,t)}{ds}(s) & = & E^j(s,X_{\eta}(s,t))+\widehat{V_{\eta}}^i(s)  {F_i}^j(s,X_{\eta}(s,t)).     
\end{eqnarray}

\begin{Lem}\label{staycoli}

$X_{\eta}(.,t)$, $V_{\eta}(.,t)$ and $E$ (along $X_{\eta}(.,t)$) stay collinear to $\vec{u}$. We have, as long as $V_{\eta}$ stay positive, $X_{\eta}(s,t)=\left(1+\frac{1}{\sqrt{n}}(s-t)\right) \vec{u}$ and$$ V_{\eta}(s,t)=\eta \vec{u}+\int_t^s E^1\left(s',\left(1+\frac{1}{\sqrt{n}}(s-t)\right) \vec{u} \right) ds' \vec{u}.$$
\end{Lem}

\begin{proof}
We start by a change of coordinates. We consider an orthonormal system $(u_i)_{1 \leq i \leq n}$ such that $u_1=\frac{1}{\sqrt{n}}\vec{u}$ and we denote by $\widetilde{X}^i$ and $\widetilde{V}^i$ the coordinates of $X_{\eta}(.,t)$ and $V_{\eta}(.,t)$ in this basis. Then, for all $1 \leq i \leq n$,
$$\frac{d\widetilde{X}^i}{ds}(s) = \frac{\widetilde{V}^i(s)}{|V|(s)}$$
and, for $i \geq 2$, $\widetilde{X}^i(0)=0$ and $\widetilde{V}^i(0)=0$. We remark, using Proposition \ref{basicprop}, that if $\widetilde{X}^i=0$ for $i \geq 2$, then $E(s,X_{\eta}(s,t))=E^1(s,X_{\eta}(s,t))\vec{u}$ and $B(s,X_{\eta}(s,t))=0$. Consider now the solution of the following system
\begin{eqnarray}
\nonumber \frac{d r}{ds} & = & \frac{w}{|w|}, \\ \nonumber
\frac{d w}{ds} & = & \sqrt{n}E^1\left(s,\frac{r}{\sqrt{n}},...,\frac{r}{\sqrt{n}}\right), 
\end{eqnarray}
with the initial data $r(t)=\sqrt{n}$ and $w(t)=\eta \sqrt{n}$.  The solution exists as long long as $w \neq 0$ and we have $$\frac{r(s)}{\sqrt{n}}=1-\frac{t-s}{\sqrt{n}} \hspace{2mm} \text{and} \hspace{2mm} \frac{w(s)}{\sqrt{n}} = \eta +\int_t^s E^1\left(s',1-\frac{t-s'}{\sqrt{n}},...,1-\frac{t-s'}{\sqrt{n}})\right) ds'.$$ 
By uniqueness of the solution of the system \eqref{characsyst1}-\eqref{characsyst2}, we have $$(\widetilde{X}^1,..., \widetilde{X}^n, \widetilde{V}^1,...,\widetilde{V}^n)=(r,0,...,0,w,,0,...,0),$$
which implies the result.
\end{proof}

We now try to estimate the time when $V_{\eta}$ vanishes.

\begin{Pro}\label{characzero}
The exists $0 < \eta_0 < T$ such that for all $\eta \in ]0,\eta_0[$, there exists $T_{\eta}$ such that if $t<T_{\eta}$, $(X_{\eta}(.,t),V_{\eta}(.,t))$ is well defined on $[0,t]$ and if $T_{\eta} \leq t <T$ there exists $\tau_{\eta}(t) \leq t$ such that $(X_{\eta}(.,t),V_{\eta}(.,t))$ is well defined on $]t-\tau_{\eta}(t),t]$ and
$$\lim_{s \rightarrow \left(t-\tau_{\eta}(t)\right)^+} V_{\eta}(s,t)=0, \hspace{5mm} \lim_{s \rightarrow \left(t-\tau_{\eta}(t)\right)^+} X_{\eta}(s,t)= \left(1-\frac{\tau_{\eta}(t)}{\sqrt{n}} \right) \vec{u}.$$
Moreover, $t \mapsto t-\tau_{\eta}(t)$ is in $C^0([T_{\eta},T[) \cap C^1(]T_{\eta},T[)$, vanishes at $T_{\eta}$, and such that
$$ \forall \hspace{0.5mm} t \in ]T_{\eta},T[, \hspace{3mm}  \frac{4}{M_0}  \leq \frac{\partial (t-\tau_{\eta})}{\partial t}(t) \leq  \frac{M_0+1}{5}  .$$
\end{Pro}

\begin{proof}

We fix $\eta \in ]0,T[$. Noting, by \eqref{characsyst1}, that $$|X_{\eta}(s,t)-\vec{u}| \leq |t-s| \leq T,$$ we obtain by Proposition \ref{basicprop}, as $X_{\eta}$ and $\vec{u}$ are collinear, that $E(s,X_{\eta}(s,t))= E^1(s,X_{\eta}(s,t)) \vec{u}$. Hence, if $t \in [0,T[$, only two situations can occur. Either $(X_{\eta}(.,t),V_{\eta}(.,t)$ is well defined on $[0,t]$, or there exists $\tau_{\eta}(t)<t$ such that
$$\lim_{s \rightarrow \left(t-\tau_{\eta}(t) \right)^+} V_{\eta}(s,t)=0,$$
and the characteristic is well defined on $]t-\tau_{\eta}(t),t]$.
Now, consider $$g_{\eta} : (s,t) \mapsto \eta+\int_t^{s} E^1\left(s',\left(1-\frac{t-s'}{\sqrt{n}} \right) \vec{u} \right) ds'$$
so that, by Lemma \ref{staycoli}, if $t \in ]0,T[$ and $s$ is near to $t$, $g_{\eta}(s,t)$ is equal to $V^i_{\eta}(s,t)$ (for all $1 \leq i \leq n$). For all $t \in ]0,T[$, $g_{\eta}(.,t)$ stricly increases on $[0,t]$, as $E^1>0$ by Proposition \ref{basicprop}. As
$$\int_t^{s} E^1\left(s',\left(1-\frac{t-s'}{\sqrt{n}} \right) \vec{u} \right) ds'=-\int_0^{t-s} E^1\left(t-s',\left(1-\frac{s'}{\sqrt{n}} \right) \vec{u} \right) ds'$$
and
$$\frac{\partial g_{\eta}}{\partial t}(s,t)=-E^1\left(s,\left(1-\frac{t-s}{\sqrt{n}} \right) \vec{u} \right)-\int_0^{t-s} \partial_t E^1\left(t-s',\left(1-\frac{s'}{\sqrt{n}} \right) \vec{u} \right) ds',$$
we have
$$\frac{\partial g_{\eta}}{\partial t}(s,t) \leq -4,$$
so that $g_{\eta}(s,.)$ is strictly decreasing on $[s,T[$. Moreover, by the bounds given on $E^1$ in Proposition \ref{basicprop}, if $t<\frac{\eta}{M_0}$, $g_{\eta}(.,t)$ does not vanish on $[0,t]$ and vanishes exactly one time, in $t-\tau_{\eta}(t)$, if $t \geq \frac{\eta}{5}$. Then, if $\eta$ is small enough, there exists $t \in ]0,T[$ such that $g_{\eta}(.,t)$ vanishes in $t-\tau_{\eta}(t)$. Let $t_1$ be a such time and let $t_2>t_1$. We have
$$0=g_{\eta}(t_1-\tau_{\eta}(t_1),t_1)>g_{\eta}(t_1-\tau_{\eta}(t_1),t_2),$$
implying the existence of $t_2-\tau_{\eta}(t_2)$ and $t_1-\tau_{\eta}(t_1)<t_2-\tau_{\eta}(t_2)$, since
$$g_{\eta}(t_1-\tau_{\eta}(t_1),t_2)<0=g_{\eta}(t_2-\tau_{\eta}(t_2),t_2). $$ Hence, $T_{\eta}$ exists\footnote{More precisely, $T_{\eta} = \sup \{t \in ]0,T[ \hspace{1mm} / \hspace{1mm} g_{\eta}(.,t) >0 \hspace{2mm} \text{on} \hspace{2mm} [0,t] \}$.} and $t \mapsto t-\tau_{\eta}(t)$ strictly increases on $[T_{\eta},T[$, vanishes in $T_{\eta}$ and tends to zero as $t \rightarrow T_{\eta}$. The fact that $t \mapsto t- \tau_{\eta}(t)$ is in $C^1(]T_{\eta},T[)$ follows from the implicit function theorem, as $g_{\eta}(t-\tau_{\eta}(t),t)=0$ and $\frac{\partial g_{\eta}}{\partial s}(s,t) \geq 5$. Furthermore, dropping the dependance in $t$ of $\tau_{\eta}$,
$$\frac{\partial (t-\tau_{\eta})}{\partial t}(t)=\frac{E^1\left(t-\tau_{\eta},\left(1-\frac{\tau_{\eta}}{\sqrt{n}} \right) \vec{u} \right)+\int_0^{\tau_{\eta}} \partial_t E^1\left(t-s',\left(1-\frac{s'}{\sqrt{n}} \right) \vec{u} \right) ds'}{E^1\left(t-\tau_{\eta},\left(1-\frac{\tau_{\eta}}{\sqrt{n}} \right) \vec{u} \right)},$$
which, by Proposition \ref{basicprop}, implies the last statement.

\end{proof}

\begin{Rq}\label{rq:Teta}
Note that, if $0 < \eta < \eta_0$, $\tau_{\eta}(T_{\eta})=T_{\eta}$ and then $g_{\eta}(0,T_{\eta})=0$. 

Later, we will use again that $g_{\eta}(0,.)$ is strictly decreasing on $[0,T[$.

\end{Rq}

\subsubsection{The contradiction}

We fix again $\eta \in ]0,\eta_0[$. As $V_{\eta}(.,t)$ is not defined on $[0,t-\tau_{\eta}(t)]$ if $t > T_{\eta}$, we cannot directly express $f_1(t,\vec{u},\eta \vec{u})$ in terms of $f_{01}$ by the method of the characteristics.

If $t \geq T_{\eta}$, we extend $X_{\eta}(.,t)$ and  $V_{\eta}(.,t)$ on $[0,t-\tau_{\eta}(t)]$ by
$$ X_{\eta}(s,t) = \left(1+\frac{t-s-2\tau_{\eta}(t)}{\sqrt{n}}\right)\vec{u} \hspace{2mm} \text{and} \hspace{1.5mm} V_{\eta}(s,t) = \eta \vec{u}+\int_t^s E(s',X_{\eta}(s',t)) ds'.$$

\begin{Rq}
If $t > t - \tau_{\eta}(t)$, $\frac{dX_{\eta}}{ds}(s,t)=\frac{\vec{u}}{\sqrt{n}}.$ We extend $X_{\eta}(.,t)$ on $[0,t-\tau_{\eta}(t)]$ in order that
$$\frac{dX_{\eta}}{ds}(s,t)=-\frac{\vec{u}}{\sqrt{n}}.$$
We then extend $V_{\eta}(.,t)$ such that \eqref{characsyst2} remains true on $[0,t-\tau_{\eta}(t)]$.
\end{Rq}

We have the following result.
\begin{Lem}
\begin{equation}\label{prolongcharac}
\forall \hspace{0.5mm} t \in [0,T[, \hspace{3mm} f_1(t,\vec{u},\eta \vec{u})=f_{01}(X_{\eta}(0,t),V_{\eta}(0,t)).
\end{equation}
\end{Lem}

\begin{proof}
If $t<T_{\eta}$, this follows from the method of characteristics. In order to prove the result for $t \geq T_{\eta}$, we consider $\epsilon > 0$, $v_{\epsilon}=(0,...,0,\epsilon)$,
$$X_{\eta,t}^{\epsilon} : s \mapsto X(s,t-\tau_{\eta}(t),X_{\eta}(t-\tau_{\eta}(t),t),v_{\epsilon})$$ and $$ V_{\eta,t}^{\epsilon} : t \mapsto V(s,t-\tau_{\eta}(t),X_{\eta}(t-\tau_{\eta}(t),t),v_{\epsilon}).$$
Proposition \ref{basicprop} gives us that $X_{\eta,t}^{\epsilon}$ and $V_{\eta,t}^{\epsilon}$ are well defined on $[0,T[$. Indeed, as, by \eqref{characsyst1}, 
$$ |X_{\eta,t}^{\epsilon}(s)-\vec{u}| \leq |X_{\eta,t}^{\epsilon}(s)-X_{\eta}(t-\tau_{\eta}(t),t)| + |X_{\eta}(t-\tau_{\eta}(t),t)-\vec{u}| \leq 2T,$$
it comes
$$\forall \hspace{0.5mm} 1 \leq i \leq n, \hspace{3mm} 5 \leq E^i(s,X_{\eta,t}^{\epsilon}(s)) \leq M_0 \hspace{3mm} \text{and} \hspace{3mm} |B(s,X_{\eta,t}^{\epsilon}(s))| \leq \frac{1}{4},$$
so that $V_{\eta,t}^{\epsilon}$ cannot vanish. Now, the method of characteristics gives us, for all $t \in [0,T[$,
$$f_1(t,X_{\eta,t}^{\epsilon}(t),V_{\eta,t}^{\epsilon}(t))=f_{01}(X_{\eta,t}^{\epsilon}(0),V_{\eta,t}^{\epsilon}(0)).$$
Then, the result, for $t \geq T_{\eta}$, follows from the continuity of $f_1$ and the following proposition.
\end{proof}
\begin{Pro}

We have
$$\lim_{\epsilon \rightarrow 0} \|X_{\eta}(.,t)-X_{\eta,t}^{\epsilon} \|_{L^{\infty}([0,t])}+\|V_{\eta}(.,t)-V_{\eta,t}^{\epsilon} \|_{L^{\infty}([0,t])}=0.$$

\end{Pro}

\begin{proof}

On the one hand, as 
\begin{equation}\label{eq:vwvw}
\forall \hspace{0.5mm} v, \hspace{1mm} w \in \R^n \setminus \{0 \}, \hspace{3mm} \left| \frac{v}{|v|}-\frac{w}{|w|} \right| \leq \frac{2}{|w|}|v-w|,
\end{equation}
we have
$$|X_{\eta}(s,t)-X_{\eta,t}^{\epsilon}(s)| \leq  \left| \int_{t-\tau_{\eta}(t)}^s \frac{2}{|V_{\eta,t}^{\epsilon}(w)|}|V_{\eta}(w,t)-V_{\eta,t}^{\epsilon}(w)| dw \right|.$$

On the other hand, note first that for $s < T$ and $|x|$, $|y| \leq \sqrt{n}+2T$, we have, by the local Lipschitz property of the electromagnetic field \eqref{LipschitzboundL},
$$|E(s,x)-E(s,y)+\widehat{v}^i {B_i}(s,x)-\widehat{w}^i{B_i}(s,y)| \leq L|x-y|+|\widehat{v}-\widehat{w}||B(s,x)|.$$
Then, using \eqref{characsyst2}, \eqref{eq:vwvw} and the bound \eqref{boundmagnB} on the magnetic field,
\begin{flalign*}
& \hspace{0mm} |V_{\eta}(s,t)-V_{\eta,t}^{\epsilon}(s)| \leq &
\end{flalign*}
$$\hspace{1cm} \epsilon + \left| \int_{t-\tau_{\eta}(t)}^s L|X_{\eta}(w,t)-X_{\eta,t}^{\epsilon}(w)|+\frac{1}{2|V_{\eta,t}^{\epsilon}(w)|}|V_{\eta}(w,t)-V_{\eta,t}^{\epsilon}(w)| dw \right| .$$

Hence, by the Grönwall lemma, for all $s \in [0,T[$,
\begin{equation}\label{eq:unif}
|X_{\eta}(s,t)-X_{\eta,t}^{\epsilon}(s)|+|V_{\eta}(s,t)-V_{\eta,t}^{\epsilon}(s)| \leq \epsilon \exp{\left( \left| \int_{t-\tau_{\eta}(t)}^s L+\frac{5}{2|V_{\eta,t}^{\epsilon}(w)|} dw \right| \right) }.
\end{equation}

We now prove that, $\exists$ $a>0$, $b>\frac{5}{2}$ such that $\forall \hspace{0.5mm} w \in [0,T[$,
\begin{equation}\label{eq:characVesti}
 |V_{\eta,t}^{\epsilon}(w)| \geq a \epsilon+b|t-\tau_{\eta}(t)-w|.
\end{equation}

Recall that $5 \leq E^j \leq M_0$ and $ |B| \leq 1$ around $\vec{u}$ (see Proposition \ref{basicprop}) and
$$V_{\eta,t}^{\epsilon,j}(w)=v^j_{\epsilon}+\int^w_{t-\tau_{\eta}(t)} E^j(s,X_{\eta,t}^{\epsilon} (s))+\widehat{V_{\eta,t}^{\epsilon}}^i(s){B_i}^j(s,X_{\eta,t}^{\epsilon} (s))ds.$$
Hence, we have.
\begin{itemize}

\item If $ w \geq t-\tau_{\eta}(t)$,
$$ V_{\eta,t}^{\epsilon,j}(w) \geq \delta_{j,n} \epsilon+(5-1)(w-t+\tau_{\eta}(t))$$ so that
 $$|V_{\eta,t}^{\epsilon}(w)|^2 \geq \epsilon^2+n(5-1)^2(w-t+\tau_{\eta}(t))^2 \geq \frac{1}{2}(\epsilon+4\sqrt{n}|w-t+\tau_{\eta}(t)|)^2.$$

\item If $ t-\tau_{\eta}(t)-\frac{\epsilon}{2(M_0+1)} \leq w \leq t-\tau_{\eta}(t)$,
$$ V_{\eta,t}^{\epsilon,j}(w) \leq -(5-1)(t-\tau_{\eta}(t)-w) \hspace{2mm}
\text{for} \hspace{2mm} 1 \leq j \leq n-1 \hspace{2mm} \text{and} \hspace{2mm} V_{\eta,t}^{\epsilon,n}(w) \geq \frac{\epsilon}{2},$$
so
  \begin{eqnarray}
  \nonumber  |V_{\eta,t}^{\epsilon}(w)|^2 & \geq & \frac{\epsilon^2}{4}+(n-1)(5-1)^2(t-\tau_{\eta}(t)-w)^2  \\ \nonumber
  & \geq & \frac{1}{2}\left( \frac{\epsilon}{2}+4\sqrt{n-1}|t-\tau_{\eta}(t)-w| \right)^2
  \end{eqnarray}
\item If $w \leq t-\tau_{\eta}(t)-\frac{\epsilon}{2(M_0+1)}$, then, for $1 \leq j \leq n-1$,
$$ V_{\eta,t}^{\epsilon,j}(w) \leq -(5-1)(t-\tau_{\eta}(t)-w) \leq -\frac{4}{3}\left|\frac{\epsilon}{2(M_0+1)}+2(t-\tau_{\eta}(t)-w)\right|.$$
It comes,
 $$ |V_{\eta,t}^{\epsilon}(w)|^2  \geq \frac{16}{9}(n-1)\left( \frac{\epsilon}{2(M_0+1)}+2|t-\tau_{\eta}(t)-w| \right)^2.$$
\end{itemize}

Inequality \eqref{eq:characVesti} then holds with $$a=\min{\left(\frac{1}{2\sqrt{2}},\frac{2\sqrt{n-1}}{3(M_0+1)} \right)} \hspace{3mm} \text{and} \hspace{3mm} b=\frac{8}{3}\sqrt{n-1} .$$

We now prove that the right hand side of \eqref{eq:unif} tends uniformly to zero in $s$, on $[0,t]$. As, by \eqref{eq:characVesti},
$$\left|\int_{t-\tau_{\eta}(t)}^s \frac{5}{2|V_{\eta,t}^{\epsilon}(w)|} dw \right| \leq  \frac{5}{2b} \log \left( 1+\frac{b\max(t-\tau_{\eta}(t),\tau_{\eta}(t))}{a\epsilon} \right) ,$$
we have, since $\max(t-\tau_{\eta}(t),\tau_{\eta}(t)) \leq t$,
$$\epsilon \exp{\left( \left| \int_{t-\tau_{\eta}(t)}^s \frac{5}{2|V_{\eta}^{\epsilon}(w)|} dw \right| \right) } \leq \exp \left( {\frac{2b-5}{2b}\log ( \epsilon )+\frac{5}{2b}\log \left(\epsilon+\frac{bt}{a} \right) } \right).$$
We then deduce, as $2b>5$, that
$$\lim_{\epsilon \rightarrow 0} \max_{s \in [0,t]} \epsilon \exp{\left( \left| \int_{t-\tau_{\eta}(t)}^s L+ \frac{5}{2|V_{\eta}^{\epsilon}(w)|} dw \right| \right) } =0,$$
which implies the result.

\end{proof}

Differenciating \eqref{prolongcharac} in $t$ for $t<T_{\eta}$ gives us
$$\partial_t f_1(t,\vec{u},\eta \vec{u})=\sum_{i=1}^n -\frac{1}{\sqrt{n}}\partial_i f_{01}\left( \left(1-\frac{t}{\sqrt{n}}\right) \vec{u},V_{\eta}(0,t) \right)$$ 
\begin{flalign*}
& \hspace{4.3cm}  +\frac{dV^i_{\eta}}{dt}(0,t) \partial_{v^i} f_{01}\left( \left(1-\frac{t}{\sqrt{n}}\right) \vec{u},V_{\eta}(0,t) \right). &
 \end{flalign*}
Doing the same for $t>T_{\eta}$ gives

\begin{flalign*}
& \partial_t f_1(t,\vec{u},\eta \vec{u}) =\sum_{i=1}^n \frac{dV^i_{\eta}}{dt}(0,t) \partial_{v^i} f_{01}\left( \left(1+\frac{t-2\tau_{\eta}(t)}{\sqrt{n}}\right) \vec{u},V_{\eta}(0,t) \right) &
\end{flalign*}

$$ \hspace{1.9cm} +\frac{1}{\sqrt{n}}\left( 2\frac{\partial (t-\tau_{\eta})}{\partial t}(t)-1 \right)\partial_i f_{01}\left( \left(1+\frac{t-2\tau_{\eta}(t)}{\sqrt{n}}\right) \vec{u},V_{\eta}(0,t) \right). $$

Recall from Proposition \ref{characzero} that $t \mapsto \frac{\partial (t-\tau_{\eta})}{\partial t}(t)$ is defined on $]T_{\eta},T[$ and takes its values in $[\frac{4}{M_0},\frac{M_0+1}{5}]$. Hence, there exists a sequence $(t_n)$, with $t_n \rightarrow T_{\eta}$, such that, $$ \exists C>0, \hspace{3mm} \lim_{t_n \rightarrow T_{\eta}} \frac{\partial (t-\tau_{\eta})}{\partial t}(t_n) = C.$$
Using that $f_1$ and $f_{01}$ are $C^1$ and taking the limit $t_n \rightarrow T_{\eta}$ in the two last equations, we obtain
$$2C \sum_{i=1}^n \partial_i f_{01}\left( \left(1-\frac{T_{\eta}}{\sqrt{n}}\right) \vec{u},0 \right)=0$$
and thus
\begin{equation}\label{eq:lastone}
\sum_{i=1}^n \partial_i f_{01}\left( \left(1-\frac{T_{\eta}}{\sqrt{n}}\right) \vec{u},0 \right)=0.
\end{equation}
Finally, we need the following proposition.

\begin{Pro}
The function $\eta \mapsto T_{\eta}$ is defined on $]0,\eta_0[$, strictly increasing, continuous and such that
$$\lim_{\eta \rightarrow 0} T_{\eta} =0.$$
\end{Pro}

\begin{proof}

We recall (see Remark \eqref{rq:Teta}) that $T_{\eta}$ is defined by the implicit equation
$$ g_{\eta}(0,T_{\eta})=\eta -\int_0^{T_{\eta}} E^1\left( w, \left( 1-\frac{T_{\eta}-w}{\sqrt{n}} \right) \vec{u} \right) dw=0.$$
Let $0<\eta_1<\eta_2<T$. We have
$$g_{\eta_1}(0,T_{\eta_2})<g_{\eta_2}(0,T_{\eta_2})=0,$$
so
$$g_{\eta_1}(0,T_{\eta_2})<g_{\eta_1}(0,T_{\eta_1})=0.$$
Since $g_{\eta_1}(0,.)$ strictly decreases (see again Remark \eqref{rq:Teta}, $T_{\eta_2} > T_{\eta_1}$, which means that $\eta \mapsto T_{\eta}$ is strictly increasing. As $E^1$ is bounded away from $0$ on the domain of integration,  $T_{\eta}$ tends to $0$ as $\eta \rightarrow 0$. The continuity ensues from the implicit function theorem.

\end{proof}

Using Equation \eqref{eq:lastone} and the last proposition, we can find $T^*>0$ such that $w \mapsto f_{01}((1-w)\vec{u},0)$ is constant on $]0,T^*[$. However, there exists $C_0 >0$ and $C_1 >0$ such that
$$f_{01}((1-w)\vec{u},0)=C_0+C_1(1-w)\chi'\left(2(1-w)^2\right)$$
for all $0<w<T^*$, and $w \mapsto (1-w)\chi'\left(2(1-w)^2\right)$ is not constant around $0$.

\renewcommand{\refname}{References}
\bibliographystyle{abbrv}
\bibliography{biblio}

\begin{thebibliography}{10}

\bibitem{Bardos}
C.~Bardos and P.~Degond.
\newblock Global existence for the {V}lasov-{P}oisson equation in {$3$} space
  variables with small initial data.
\newblock {\em Ann. Inst. H. Poincar\'{e} Anal. Non Lin\'{e}aire},
  2(2):101--118, 1985.

\bibitem{Bieri}
L.~Bieri, S.~Miao, and S.~Shahshahani.
\newblock Asymptotic properties of solutions of the {M}axwell {K}lein {G}ordon
  equation with small data.
\newblock {\em Comm. Anal. Geom.}, 25(1):25--96, 2017.

\bibitem{BGP2}
F.~Bouchut, F.~Golse, and C.~Pallard.
\newblock Classical solutions and the {G}lassey-{S}trauss theorem for the 3{D}
  {V}lasov-{M}axwell system.
\newblock {\em Arch. Ration. Mech. Anal.}, 170(1):1--15, 2003.

\bibitem{BGP}
F.~Bouchut, F.~Golse, and C.~Pallard.
\newblock Nonresonant smoothing for coupled wave + transport equations and the
  {V}lasov-{M}axwell system.
\newblock {\em Rev. Mat. Iberoamericana}, 20(3):865--892, 2004.

\bibitem{Yukawa}
S.-H. Choi, S.-Y. Ha, and H.~Lee.
\newblock Dispersion estimates for the two-dimensional {V}lasov-{Y}ukawa system
  with small data.
\newblock {\em J. Differential Equations}, 250(1):515--550, 2011.

\bibitem{CK}
D.~Christodoulou and S.~Klainerman.
\newblock Asymptotic properties of linear field equations in {M}inkowski space.
\newblock {\em Comm. Pure Appl. Math.}, 43(2):137--199, 1990.

\bibitem{FJS2}
D.~Fajman, J.~Joudioux, and J.~Smulevici.
\newblock {Sharp asymptotics for small data solutions of the Vlasov-Nordstr\"om
  system in three dimensions}.
\newblock arXiv:1704.05353, 2017.

\bibitem{FJS3}
D.~Fajman, J.~Joudioux, and J.~Smulevici.
\newblock {The Stability of the Minkowski space for the Einstein-Vlasov
  system}.
\newblock arXiv:1707.06141, 2017.

\bibitem{FJS}
D.~Fajman, J.~Joudioux, and J.~Smulevici.
\newblock A vector field method for relativistic transport equations with
  applications.
\newblock {\em Anal. PDE}, 10(7):1539--1612, 2017.

\bibitem{Georgiev}
V.~Georgiev.
\newblock Decay estimates for the {K}lein-{G}ordon equation.
\newblock {\em Comm. Partial Differential Equations}, 17(7-8):1111--1139, 1992.

\bibitem{GlSch2.5}
R.~Glassey and J.~Schaeffer.
\newblock The ``two and one-half-dimensional'' relativistic {V}lasov {M}axwell
  system.
\newblock {\em Comm. Math. Phys.}, 185(2):257--284, 1997.

\bibitem{Glassey}
R.~T. Glassey.
\newblock {\em The {C}auchy problem in kinetic theory}.
\newblock Society for Industrial and Applied Mathematics (SIAM), Philadelphia,
  PA, 1996.

\bibitem{GSc}
R.~T. Glassey and J.~W. Schaeffer.
\newblock Global existence for the relativistic {V}lasov-{M}axwell system with
  nearly neutral initial data.
\newblock {\em Comm. Math. Phys.}, 119(3):353--384, 1988.

\bibitem{GlStrauss}
R.~T. Glassey and W.~A. Strauss.
\newblock Singularity formation in a collisionless plasma could occur only at
  high velocities.
\newblock {\em Arch. Rational Mech. Anal.}, 92(1):59--90, 1986.

\bibitem{GSt}
R.~T. Glassey and W.~A. Strauss.
\newblock Absence of shocks in an initially dilute collisionless plasma.
\newblock {\em Comm. Math. Phys.}, 113(2):191--208, 1987.

\bibitem{GlStracrit}
R.~T. Glassey and W.~A. Strauss.
\newblock Large velocities in the relativistic {V}lasov-{M}axwell equations.
\newblock {\em J. Fac. Sci. Univ. Tokyo Sect. IA Math.}, 36(3):615--627, 1989.

\bibitem{Lindblad}
{Hans Lindblad and Martin Taylor}.
\newblock {Global stability of Minkowski space for the Einstein--Vlasov system
  in the harmonic gauge}.
\newblock arXiv:1707.06079, 2017.

\bibitem{HRV}
H.~J. Hwang, A.~Rendall, and J.~J.~L. Vel\'{a}zquez.
\newblock Optimal gradient estimates and asymptotic behaviour for the
  {V}lasov-{P}oisson system with small initial data.
\newblock {\em Arch. Ration. Mech. Anal.}, 200(1):313--360, 2011.

\bibitem{Kl85}
S.~Klainerman.
\newblock Uniform decay estimates and the {L}orentz invariance of the classical
  wave equation.
\newblock {\em Comm. Pure Appl. Math.}, 38(3):321--332, 1985.

\bibitem{Kl93}
S.~Klainerman.
\newblock Remark on the asymptotic behavior of the {K}lein-{G}ordon equation in
  {${\bf R}^{n+1}$}.
\newblock {\em Comm. Pure Appl. Math.}, 46(2):137--144, 1993.

\bibitem{KlSta}
S.~Klainerman and G.~Staffilani.
\newblock A new approach to study the {V}lasov-{M}axwell system.
\newblock {\em Commun. Pure Appl. Anal.}, 1(1):103--125, 2002.

\bibitem{LukStrain}
J.~Luk and R.~M. Strain.
\newblock Strichartz estimates and moment bounds for the relativistic
  {V}lasov-{M}axwell system.
\newblock {\em Arch. Ration. Mech. Anal.}, 219(1):445--552, 2016.

\bibitem{Pallard1}
C.~Pallard.
\newblock On the boundedness of the momentum support of solutions to the
  relativistic {V}lasov-{M}axwell system.
\newblock {\em Indiana Univ. Math. J.}, 54(5):1395--1409, 2005.

\bibitem{Pallard2}
C.~Pallard.
\newblock A refined existence criterion for the relativistic {V}lasov-{M}axwell
  system.
\newblock {\em Commun. Math. Sci.}, 13(2):347--354, 2015.

\bibitem{Rein2}
G.~Rein.
\newblock Generic global solutions of the relativistic {V}lasov-{M}axwell
  system of plasma physics.
\newblock {\em Comm. Math. Phys.}, 135(1):41--78, 1990.

\bibitem{Sc}
J.~Schaeffer.
\newblock A small data theorem for collisionless plasma that includes high
  velocity particles.
\newblock {\em Indiana Univ. Math. J.}, 53(1):1--34, 2004.

\bibitem{Poisson}
J.~Smulevici.
\newblock Small data solutions of the {V}lasov-{P}oisson system and the vector
  field method.
\newblock {\em Ann. PDE}, 2(2):Art. 11, 55, 2016.

\bibitem{Sogge}
C.~D. Sogge.
\newblock {\em Lectures on nonlinear wave equations}.
\newblock Monographs in Analysis, II. International Press, Boston, MA, 1995.

\bibitem{SAI}
R.~Sospedra-Alfonso and R.~Illner.
\newblock Classical solvability of the relativistic {V}lasov-{M}axwell system
  with bounded spatial density.
\newblock {\em Math. Methods Appl. Sci.}, 33(6):751--757, 2010.

\bibitem{Yang}
S.~Yang.
\newblock Decay of solutions of {M}axwell-{K}lein-{G}ordon equations with
  arbitrary {M}axwell field.
\newblock {\em Anal. PDE}, 9(8):1829--1902, 2016.

\end{thebibliography}

\end{document}